\documentclass[12pt,a4paper]{amsart}
\usepackage{amssymb}
\usepackage{amsfonts}
\usepackage{amsmath}
\usepackage[mathscr]{eucal}
\usepackage{mathrsfs}

\newtheorem{thm}{Theorem}[section]
\newtheorem{prop}[thm]{Proposition}
\newtheorem{lem}[thm]{Lemma}
\newtheorem{cor}[thm]{Corollary}

\numberwithin{equation}{section}

\def\JJ{{\Bbb J}}

\def\N{{\Bbb N}}
\def\Z{{\Bbb Z}}
\def\Q{{\Bbb Q}}
\def\R{{\Bbb R}}
\def\C{{\Bbb C}}
\def\PP{{\Bbb P}}
\def\EE{{\Bbb E}}
\def\II{{\Bbb I}}
\def\A{{\Bbb A}}

\def\emp{\varnothing}
\def\fa{{\frak a}}

\def\fd{{\frak d}}

\def\ff{{\frak f}}
\def\fg{{\frak g}}

\def\fp{{\frak p}}

\def\fu{{\mathsf u}}

\def\fw{{\mathsf w}}

\def\fz{{\frak z}}
\def\fA{{\frak A}}

\def\fS{{\frak S}}

\def\fF{{\frak F}}

\def\fJ{{\frak J}}

\def\fW{{\frak W}}

\def\sa{{\mathsf a}}
\def\sG{\mathsf G}
\def\sH{\mathsf H}
\def\sM{\mathsf M}
\def\sN{\mathsf N}
\def\sP{\mathsf P}
\def\sR{\mathsf R}
\def\sQ{\mathsf Q}

\def\sm{\mathsf m}
\def\sn{\mathsf n}

\def\sf{\mathsf f}

\def\sc{\mathsf c}

\def\su{\mathsf u}
\def\sr{\mathsf r}

\def\sT{{\mathsf T}}
\def\cO{\frak o}
\def\cL{\mathscr L}
\def\cB{{\mathscr B}}

\def\cK{{\mathcal K}}

\def\cH{{\mathscr H}}
\def\cM{{\mathscr M}}
\def\cN{{\mathcal N}}

\def\cS{{\frak S}}

\def\cV{{\mathcal V}}

\def\cT{{\mathcal T}}

\def\cG{{\mathcal G}}

\def\cL{{\mathscr L}}
\def\cI{{\mathscr I}}
\def\cW{{\mathcal W}}
\def\cU{{\mathcal U}}
\def\cJ{{\mathcal J}}
\def\Re{{\operatorname {Re}}}
\def\Im{{\operatorname {Im}}}
\def\tr{{\operatorname{tr}}}

\def\GL{{\operatorname {GL}}}

\def\SO{{\operatorname{SO}}}

\def\diag{{\operatorname {diag}}}
\def\sgn{{\operatorname {sgn}}}
\def\Ad{{\operatorname{Ad}}} 
\def\vol{{\operatorname{vol}}}

\def\leq{\leqslant}
\def\geq{\geqslant}
\def\bsl{\backslash}
\def\ch{{\cosh\,}}
\def\sh{{\sinh\,}}

\def\bS{{\bold S}}
\def\e{\varepsilon}

\def\sw{\mathsf w}

\def\fD{{\frak D}}

\def\cE{{\mathscr E}}

\def\d{{\rm{d}}}

\def\fX{{\frak X}}

\def\cD{\mathscr D}

\def\fin{{\rm{\bf {f}}}}
\def\bK{{\bf K}}

\def\b1{{\bold 1}}

\def\B{{\Bbb B}}

\def\F{{\mathbb F}}

\def\SS{{\mathbf R}}

\def\ee{{\bf{e}}}
\def\cnt{{\sc}}
\def\rex{\sr}

\def\sB{{\mathsf B}}
\def\XX{{\mathbb X}}
\def\cBB{{\mathcal B}}
\def\MM{{\mathbb M}}
\def\calW{{\mathscr W}}
\def\Ocal{{\mathcal O}}

\def\cJ{{\mathscr J}}

\def\T{{\rm T}}

\def\Ical{{\mathcal I}}
\def\rQ{{\rm Q}}
\def\fJ{{\frak I}}
\def\ss{{\mathsf s}}
\def\ccU{{\mathcal N}}

\newcommand{\sslash}{\mathbin{/\mkern-6mu/}}

\title{Spectral average of central values of automorphic $L$-functions for holomorphic cusp forms on ${\rm SO}_0(m,2)$ II}

\author{Masao Tsuzuki}
\address{Faculty of Science and Technology, Sophia University, Kioi-cho 7-1 Chiyoda-ku Tokyo, 102-8554, Japan}
\email{m-tsuduk@sophia.ac.jp}

\begin{document}

\maketitle

\begin{abstract} Given a maximal integral lattice $\cL$ of signature $(m+, 2-)$ with an odd $m\geq 3$, we consider the holomorphic cusp forms $F$ of weight $l$ on the bounded symmetric domain of type IV of dimension $m$ with respect to the discriminant subgroup of the orthogonal group $O(\cL)$ defined by $\cL$. Under a non-negativity assumption on the central $L$-values, we prove an equidistribution result of Satake parameters in an ensemble constructed from the central values of standard $L$-functions and the square of the Whittaker-Bessel periods.
\end{abstract}

\section{Introduction}
The central values of automorphic $L$-functions are of special concern in many ways. Some interesting features of central $L$-values are only revealed not by dealing with a single $L$-function but rather by studying a family of $L$-functions as a statistical object. There are countless works done in this direction for standard $L$-functions on ${\bf{GL}}_2$. However, for higher degree $L$-functions, there still remains a huge room for similar researches to be conducted. In \cite{KST}, Kowalski-Saha-Tsimerman established, among other things, an asymptotic formula for a weighted average of the spinor $L$-values for Siegel cusp forms on ${\mathbf{Sp}}_2(\Z)$ of growing weights with the weight factor constructed from the Bessel period of Siegel cusp forms. They concerned themselves with the $L$-values at points on the convergent range of the Euler products. Later, Blomer \cite{Blomer} computed not only a similar asymptotic formula for central values of the $L$-series but also a second moment asymptotic formula to apply them to the problem of non-vanishing of central spinor $L$-values of Siegel cusp forms on ${\mathbf {Sp}}_2(\Z)$ of growing even weight. In this paper, we pursue a higher dimensional generalization of the asymptotic formula for the orthogonal groups of signature $(2,m)$. The special orthogonal group ${\rm SO}_0(2,m)$ of real rank $2$ admits the holomorphic discrete series representations. Thus we have a class of holomorphic automorphic forms for ${\rm SO}_0(2,m)$ which may be viewed as a generalization of the Siegel modular forms of genus $2$ through the accidental isomorphism ${\rm SO}(2,3)\cong {\bf{PGSp}}_2(\R)$. An arithmetic theory of holomorphic modular forms and the standard automorhic $L$-functions on the orthogonal group associated with a maximal even-integral lattice of signature $(2,m)$ has been developed by Sugano (\cite{Sugano85}, \cite{Sugano95}) and Murase-Sugano (\cite{MS94}, \cite{MS98}). We consider a weighted average for the central values of standard $L$-functions for holomorphic cuspidal Hecke eigenforms on ${\rm SO}_0(2,m)$, and prove an asymptotic formula for the average by a new summation formula to be developed in this article, relying on the integral representation of $L$-functions due to Andrianov (\cite{And1}, \cite{And2}) and Sugano (\cite{Sugano85}). A novelty of our formula lies in that the weighting factor to form the average involves an arbitrary Hecke operator; this allows us to prove an equidistribution result of Satake parameters of cusp forms with growing weights in an ensemble constructed from the central values of standard $L$-functions and the square of the Whittaker-Bessel periods of cusp forms when the degree of the orthogonal group is odd. This is an analogue of the equidistribution results of Satake parameters of automorphic representations originally proved for ${\mathbf {GL}}_2$ by Serre (\cite{Serre}) and Corney-Duke-Farmer (\cite{CDF}), and later generalized to higher rank groups by Sauvageot (\cite{Sauvageot}), Shin (\cite{Shin}) and by Kim-Yamauchi-Wakatsuki (\cite{KYW}). As for the proofs, we heavily rely on the results of our previous paper \cite{Tsud2011-1}; when they are cited, errata if any will be given on the footnotes. Let us explain our main result precisely, introducing notation which will be used in this paper. 

\subsection{Description of main result}\label{Intro}
Let $m \geq 3$ be an integer greater and $Q_0\in {\bf {GL}}_{m-2}(\Q)$ a positive definite even-integral symmetric matrix of degree $m-2$. Set $$Q_1=\left[\begin{smallmatrix} {} & {} & 1 \\ {} & Q_0 &{} \\ {1} & {} & {} \end{smallmatrix} \right]\in {\mathbf {GL}}_{m}(\Q), \quad Q=\left[\begin{smallmatrix} {} & {} & 1 \\ {} & Q_1 &{} \\ {1} & {} & {} \end{smallmatrix} \right] \in {\mathbf {GL}}_{m+2}(\Q).$$ 
Let $\sG={\bf O}(Q)$ be the orthogonal group of $Q$, which is an algebraic $\Q$-group formed by the linear automorphisms of the quadratic form $Q[X]={}^tX QX$ on the $(m+2)$-dimensional $\Q$-vector space $V=\Q^{m+2}$. The $\Q$-bi-linear form $\langle\,,\,\rangle$ associated with $Q$ is given as $\langle X,Y\rangle={}^tX Q Y$ for $X,Y\in V$. Let $V_1\cong \Q^{m}$ denote the orthogonal complement of the hyperbolic plane spanned by the isotropic vectors $\e_1=\left[\begin{smallmatrix} 1 \\ 0_{m} \\ 0 \end{smallmatrix}\right]$, $\e_1'=\left[\begin{smallmatrix}0 \\ 0_{m} \\ 1 \end{smallmatrix} \right]$ in $V$. Consider the open set $\tilde \cD=\{\fz=X+\sqrt{-1} Y|\,X,Y\in V_1(\R),\,Q[Y]<0\,\}$ of $V_1(\C)=V_{1}\otimes \C$. For $g\in \sG(\R)$ and $\fz \in \tilde\cD$, there exists a unique point $g\langle \fz\rangle \in \tilde\cD$ and a scalar $J(g,\fz)\in \C^\times$ such that 
\begin{align}
g\,\left[\begin{smallmatrix} -Q[\fz]/2 \\ \fz \\ 1 \end{smallmatrix} \right]=J(g,\fz)\, \left[\begin{smallmatrix} -Q[g\langle \fz \rangle]/2 \\ g\langle \fz\rangle \\ 1 \end{smallmatrix} \right]. 
 \label{Intro-f0}
\end{align}
Fix a point $\fz_0=\sqrt{2}\eta_0^{-}/i\in \tilde\cD$ with 
\begin{align*}
\text{$\eta_0^{-}\in V_{1}(\R)$ such that $Q[\eta_0^{-}]=-1$,}
\end{align*}
 and let $\cD$ be the connected component of $\tilde \cD$ containing $\fz_0$. Then the mapping $(g,\fz) \mapsto g\langle \fz \rangle$ is a transitive action of $\sG(\R)^0\cong {\rm SO}_0(m,2)$ on $\cD$ by holomorphic automorphisms. We fix a $\sG(\R)^0$-invariant K\"{a}hler structure on $\cD$ by demanding that the associated $2$-form is $2^{-1}\sqrt{-1}\,\partial \bar\partial\, Q[\Im(\fz)]$. Let $\sG(\R)^+\subset \sG(\R)$ be the subgroup of index $2$ formed by all those elements of $\sG(\R)$ which maps $\cD$ onto itself. Set $\sG(\Q)^{+}=\sG(\Q)\cap \sG(\R)^+$. We suppose that $\cL=\Z^{m+2}$ is a maximal integral lattice in the quadratic space $(V,Q)$, i.e., $2^{-1}Q[\cL]\subset \Z$ and if a $\Z$-lattice $\cM\subset V$ satisfies $\cL\subset \cM$ and $2^{-1}Q[\cM]\subset \Z$ then $\cM=\cL$. Let $\sG(\A_\fin)$ be the group of finite adeles of $\sG$, and $\bK_\fin$ the subgroup of all those elements of $\sG(\A_\fin)$ which leave the lattice $\cL$ stable. Let $\bK_\fin^*$ be the kernel of the natural homomorphism $\bK_\fin\rightarrow {\bf Aut}(\cL^*/\cL)$, where $\cL^*$ is the dual lattice of $\cL$ in $V$; $\bK_\fin$ is a maximal compact subgroup of $\sG(\A_\fin)$, which is open as well, and $\bK_\fin^*$ is a subgroup of finite index in $\bK_\fin$ whose properties are fully investigated by Murase-Sugano \cite{MS98}. Given $l\in \N^*$, we say that a function ${\rm F}:\cD\times \sG(\A_\fin)\rightarrow \C$ is a holomorphic cuspform of weight $l$ if it is bounded, the function $\fz\mapsto {\rm F}(\fz,g_\fin)$ on $\cD$ is holomorphic for any $g_\fin \in\sG(\A_\fin)$, and it possesses the automorphy 
$$
{\rm F}(\gamma\langle \fz\rangle,\gamma g_\fin k)=J(\gamma,\fz)^{l}\,{\rm F}(\fz,g_\fin), \quad \gamma\in \sG(\Q)^{+}, (\fz,g_\fin)\in \cD\times \sG(\A_\fin),\,k\in \bK_\fin^*.
$$
The $\C$-vector space of all such functions ${\rm F}$, denoted by $S_l(\bK_\fin^*)$, becomes a finite dimensional Hilbert space when endowed with the inner-product 
$$
({\rm F}|{\rm F}_1)_{\sG}=\int_{\sG(\Q)^+\bsl (\cD\times \sG(\A_\fin))}{\rm F}(\fz,g_\fin)\overline{{\rm F}_1(\fz,g_\fin)}\,\d \mu_{\cD}(\fz)\,\d g_\fin,
$$ 
where $\d \mu_{\cD}$ is the $\sG(\R)^0$-invariant measure on $\cD$ associated with the K\"{a}hler volume form and $\d g_\fin$ is the Haar measure on $\sG(\A_\fin)$ such that $\vol(\bK_\fin^*)=1$. Let $\cH(\sG(\A_\fin)\sslash \bK_\fin^*)$ be the Hecke algebra for the pair $(\sG(\A_\fin),\bK_\fin^*)$, i.e., the convolution algebra of all the finite $\C$-linear combinations of the characteristic functions of double $\bK_\fin^*$-cosets in $\sG(\A_\fin)$. In this general setting where $\cL$ is not necessarily self-dual, Murase-Sugano \cite{MS98} proved a version of the Satake isomorphism which describes a fine structure of the Hecke algebra $\cH(\sG(\A_\fin)\sslash \bK_\fin^*)$ as well as its center $\cH^{+}(\sG(\A_\fin)\sslash \bK_\fin^*)$. Moreover, for any prime number $p$ they gave a definition of the standard local $L$-factor $L(\lambda_p,s)$ associated with a character $\lambda_p$ of $\cH^{+}(\sG(\Q_p)\sslash\bK_p^*)$ which reduces to the common one by Langlands when $\cL$ is self-dual over $\Z_p$ (\cite[\S 1.4]{MS98}). We let the algebra $\cH^+(\sG(\A_\fin)\sslash\bK_\fin^*)$ act on the space $S_l(\bK_\fin^*)$ by the convolution product on $\sG(\A_\fin)$. Since these Hecke operators turn out to form a commuting family of normal operators, we can fix an orthonormal basis $\cB_l^+$ of $S_l(\bK_\fin^*)$ consisting of joint eigenfucntions of Hecke operators from $\cH^+(\sG(\A_\fin)\sslash \bK_\fin^*)$. For ${\rm F}\in \cB_l^{+}$, let $L_\fin({\rm F},s)$ be the standard $L$-function of ${\rm F}$, which is defined to be an analytic continuation of the Euler product $\prod_{p\in \fin}L(\lambda_{F,p},s)$ for $\Re(s)$ large with $\lambda_{F,p}:\cH^{+}(\sG(\Q_p)\sslash \bK_p)\rightarrow \C$ being the eigencharacter of ${\rm F}$ at $p$. From the functional equation of $L_\fin({\rm F},s)$ and the knowledge of the location of possible poles of $L_\fin ({\rm F},s)$ (\cite{MS}), we see that $L_\fin({\rm F},s)$ is regular at the point $s=1/2$ so that we can speak of the central value $L_\fin({\rm F},1/2)$ for any ${\rm F} \in \cB_l^{+}$, an investigation of whose statistical behavior as $l\rightarrow \infty$ is one of our objectives in this article. 

The orthogonal group $\sG_1={\bf O}(Q_1)$ of $Q_1$ is viewed as a $\Q$-subgroup of $\sG$ consisting of all the elements which leave the vectors $\e_1$ and $\e_1'$ invariant. Set $\cL_1=\cL\cap V_1$ and $\cL_1^{*}=\{\eta \in V_1|\,\langle \cL_1, \eta \rangle\subset \Z\}$ the dual lattice of $\cL_1$ in $V_1$. We fix a vector $\xi \in V_1$ with the following properties:
\begin{itemize}
\item[(i)] $Q[\xi]<0$ and $\langle \xi,\eta_0^{-}\rangle<0$. 
\item[(ii)] $\xi$ is primitive in $\cL_1^*$.
\item[(iii)] $\xi$ is reduced, i.e., $\cL_1^\xi=\cL_1\cap V_1^\xi$ is a maximal integral lattice in the quadratic space $(V_1^\xi,Q|V_1^\xi)$, where $V_{1}^{\xi}=\{X\in V_1|\,\langle X,\xi \rangle=0\}$. 
\end{itemize}
Let $\fd(\cL)$ (resp. $\fd(\cL_1^\xi)$) be the absolute value of the Gram determinant of a $\Z$-basis of $\cL$ (resp. $\cL_1^\xi$). 
Let $\sG^\xi_1$ denote the orthogonal group of the quadratic space $(V_1^\xi,Q|V_1^\xi)$, identified with the stabilizer of the vector $\xi$ in $\sG_1$. Since $V_1^\xi(\R)$ is a positive definite subspace, the real points $\sG_1^\xi(\R)$ is compact. Let $\bK_{1,\fin}^{\xi}$ be the group of all those elements of $\sG_1^\xi(\A_\fin)$ which leave the lattice $\cL_1^\xi$ stable and $\bK_{1,\fin}^{\xi*}$ the kernel of the homomorphism $\bK_{1,\fin}^\xi \rightarrow {\bf Aut}(\cL_1^{\xi*}/\cL_{1}^\xi)$. Then the space $\sG_1^\xi(\Q)\bsl \sG_1^\xi(\A)/\sG_1^\xi(\R)\bK_{1,\fin}^{\xi*}$ is a finite set. Let $\cV(\xi)$ denote the space of all the functions $f:\sG_1^\xi(\A)\rightarrow \C$ such that
\begin{align*}
f(\delta h u_\infty)=f(h), \quad (\delta,h,u_\infty)\in \sG_1^\xi(\Q) \times \sG_1^\xi(\A)\times \sG_1^\xi(\R).
\end{align*} 
The space $\cV(\xi)$ endowed with the action of $\sG_1^\xi(\A_\fin)$ by the right-translation becomes a smooth representation of $\sG_1^\xi(\A_\fin)$. Let $\cV(\xi;\bK_{1,\fin}^{\xi*})$ be the $\bK_{1,\fin}^{\xi*}$-fixed vectors of $\cV(\xi)$. Let $\{u_j\}_{j=1}^{h}$ be a complete set of representatives of $\sG_1^\xi(\Q)\bsl \sG_1^\xi(\A_\fin)/\bK_{1,\fin}^*$ and set
$$e_j^\xi=\#(\sG_1^\xi(\Q)\cap u_j \bK_{1,\fin}^{\xi *} u_j^{-1})\quad (1\leq j \leq h). 
$$
Then $\cV(\xi;\bK_{1,\fin}^{\xi*})$ is a finite dimensional Hilbert space with the inner-product
\begin{align}
(f|f_1)_{\sG_1^\xi}=\sum_{j=1}^{h}f(u_j)\,\bar f_1(u_j)/e_{j}^\xi, \quad f,\,f_1\in \cV(\xi;\bK_{1,\fin}^{\xi*}).
 \label{InnerProdStab}
\end{align}
Let $\cU\subset \cV(\xi)$ be an irreducible $\sG_1^\xi(\A_\fin)$-submodule of $\cV(\xi)$ such that $\cU(\bK_{1,\fin}^{\xi*}):=\cU\cap \cV(\xi;\bK_{1,\fin}^{\xi*})$ is non-zero. Then the Hecke algebra $\cH^{+}(\sG_1^\xi(\Q_p)\sslash \bK_{1,p}^{\xi*})$ at a prime number $p$ acts on $\cU(\bK_{1,\fin}^{\xi*})$ by a character $C_p^\cU:\cH^{+}(\sG_1^\xi(\A_\fin)\sslash \bK_{1,\fin}^{\xi*}) \rightarrow \C$. The basic properties of the standard $L$-function $L_\fin(\cU,s):=\prod_{p \in \fin}L(C_p^\cU,s)$ of $\cU$ (including the definition of the Euler factor $L(C_p^{\cU},s)$ at all $p$) has been established in \cite{MS94}, \cite{MS98}; among other things, it is shown that $L_\fin(\cU,s)$ has a possible simple pole at $s=1$ when $m$ is odd but is holomorphic at $s=1$ when $m$ is even. In \S\ref{subsubsec: xiEigenFtn}, we define a certain involutive operator $\tau_\fin^{\xi}$ on $\cV(\xi;\bK_{1,\fin}^{\xi*})$ which makes $\cU(\bK_{1,\fin}^{\xi*})$ stable. Fix an orthonormal basis $\cB(\cU;\bK_{1,\fin}^{\xi*})$ of $\cU(\bK_{1,\fin}^{\xi*})$ consisting of eignevectors of $\tau_{\fin}^{\xi}$. Set $\cB(\cU;\bK_{1,\fin}^{\xi*})^{\e}=\{f\in \cB(\cU;\bK_{1,\fin}^{\xi*})|\tau_\fin(f)=\e \,f\,\}$ and $d^{\e}(\cU):=\dim \cB(\cU;\bK_{1,\fin}^{\xi*})^{\e}$ for $\e\in \{+.-\}$. Define
$$
\chi(\cU):=\frac{d^{+}(\cU)-d^{-}(\cU)}{d^{+}(\cU)+d^{-}(\cU)}.  
$$
Then we have $\chi(\cU)\in \{-1,0,1\}$ or equivalently either $d^{+}(\cU)=0$, $d^{-}(\cU)=0$ or $d^{+}(\cU)=d^{-}(\cU)$, from Lemma~\ref{cUtrace-L}. For ${\rm F} \in S_l(\bK_\fin^*)$ with the Fourier expansion $$
{\rm F}(\fz, g_\fin)=\sum_{\eta \in(\cL_1\otimes \Q)\cap \sqrt{-1}\cD} a_{{\rm F}}(g_\fin;\,\eta)\,\exp(2\pi \sqrt{-1} \langle \fz,\eta\rangle), \qquad (\fz, g_\fin)\in \cD\times \sG(\A_\fin)
$$
(see \cite[\S 3.2]{Tsud2011-1}) and $f\in \cB(\cU;\bK_{1,\fin}^{\xi*})$, we form the average of the Fourier coefficients $a_{{\rm F}}(g_\fin;\,\xi)$, 
$$
a_{{\rm F}}^f(\xi)=\sum_{j=1}^{h} {f(u_j)}\,a_{{\rm F}}(u_j;\,\xi)/e_j^\xi.
$$ When $m=3$, ${\rm F}$ corresponds to a Siegel cusp form on ${\bf{Sp}}_2(\Z)$ and $a_{{\rm F}}^f(\xi)$ coincides with the average of the Fourier coefficients of the Siegel cusp form over an ideal class group of an elliptic torus, or what amounts to the same, the global Bessel function (\cite{PTB}, \cite{BFF}). The refined Gan-Gross-Prasad conjecture posed by Ichino-Ikeda (\cite{IchinoIkeda}) in the co-dimension $1$ case was extended by Liu (\cite{Liu}) in higher codimensional case; the conjecture predicts the quantity $|a_{{\rm F}}^{f}(\xi)|^2$, the norm-square of the Bessel period on ${\bf SO}(m+2)\times {\bf SO}(m-1)$, should be related to the central value of the convolution $L$-function of ${\rm F}$ and $f$. An important case of the conjecture on the special Bessel period on ${\bf SO}(m+2)\times {\bf SO}(2)$ for an odd $m$ has been proved by Furusawa-Morimoto \cite{FurusawaMorimoto} recently. For the Siegel modular case, Liu's conjecture is further refined by \cite{DPSS}. The quantity $a_{{\rm F}}^f(\xi)$ also plays a role in the integral representation of the $L$-function (\cite{Sugano85}). We call $a_{{\rm F}}^f(\xi)$ the $(\xi,f)$ Whittaker-Bessel coefficient of ${\rm F}$. For our purpose, it is enlightening to introduce the rescaled Whittaker-Bessel coefficient by \begin{align}
\fa_{{\rm F}}^f(\xi)=(4\pi\sqrt{2|Q[\xi]|})^{\rho-l+1/2}\,\Gamma\left(2l-\rho\right)^{1/2} \,a_{{\rm F}}^f(\xi)
 \label{RSFC}
\end{align}
for $f\in \cB(\cU;\bK_{1,\fin}^{\xi*})$, where $\rho=(m-1)/2$.

Let $S$ be a finite set of prime numbers such that 
\begin{align}
\text{$p\in S$ is prime to $\#(\cL^*/\cL)$, and $Q[\xi]\in \Z_p^\times$ for all $p\in S$,}
\label{conditionS}
\end{align}
 so that $\cL_{p}=\cL\otimes \Z_p$ (resp . $\cL_{1,p}^{\xi}=\cL_{1}^\xi\otimes \Z_p$) is a self-dual $\Z_p$-lattice in $V(\Q_p)$ (resp. $V_1^\xi(\Q_p)$) and $\bK_p=\bK_p^*$ (resp. $\bK_{1,p}^\xi=\bK_{1,p}^{\xi*}$) for all $p\in S$. Let us fix $p\in S$ for a while. Then both $\sG$ and $\sG_1^\xi$ are quasi-split over $\Q_p$. 
Let 
$$\fX_p=(\C/2\pi \sqrt{-1} (\log p)^{-1}\Z)^{\ell_p}, \quad \fX_p^0=(\sqrt{-1} \R/2\pi \sqrt{-1} (\log p)^{-1}\Z)^{\ell_p},
$$ 
where $\ell_p$ is the Witt index of $V(\Q_p)$, and $\pi_p^{\sG}(\nu)\,(\nu\in \fX_p)$ the $\bK_p$-spherical series of $\sG(\Q_p)$. Then the spherical Fourier transform of $\phi_p\in \cH(\sG(\Q_p)\sslash \bK_p)$ at $\nu \in \fX_p$ is defined to be the eigenvalue $\hat\phi_p(\nu)$ of the operator $\int_{\sG(\Q_p)} \phi_p(g)\, \pi_p^{\sG}(\nu)(g)\,\d g$ on the $\bK_p$-fixed vectors in $\pi^{\sG}_p(\nu)$, where $\d g$ is the Haar measure on $\sG(\Q_p)$ such that $\vol(\bK_p)=1$. Let $\fX_p^{0+}$ denote the set of $\nu\in \fX_p$ such that $\pi_p^{\sG}(\nu)$ is unitarizable. Then $\fX_p^{0}\subset \fX_p^{0+}$. A similar construction is applied to $\sH:=\sG_1^\xi$ to yield the $\bK_{1,p}^{\xi}$-spherical representation $\pi_p^{\sH}(z)$ $(z\in \fX_p(\xi))$ of $\sH(\Q_p)$ and the Fourier transform $\hat\varphi(z)$ of $\varphi \in \cH(\sG_1^\xi(\Q_p)\sslash \bK_{1,p}^\xi)$ at $z\in \fX_p(\xi)$, where $\fX_p(\xi)=(\C/2\pi \sqrt{-1}(\log p)^{-1}\Z)^{\ell_p(\xi)}$ with $\ell_p(\xi)$ the Witt index of $V_1^\xi(\Q_p)$. Let $W_{\Q_p}^{\sG}$ be the restricted Weyl group of $\sG$. Let $\d \mu_{p}^{\rm Pl}$ be the spherical Plancherel measure on $\fX_p^0/W_{\Q_p}^{\sG}$ corresponding to $\d g$ (\cite{Macdonald}). For our purpose, the explicit formula \eqref{Intro-f2} of $\mu_{p}^{\rm Pl}$ is not that important, whereas its non-negativity is crucial. For $z\in \fX_p^0(\xi)$, define a Radon measure $\Lambda_p^{\xi,(z)}$ on $\fX^{0+}_p$ supported on the tempered locus $\fX_p^{0}$ by setting
\begin{align}
\Lambda_p^{\xi,(z)}(\alpha) &=\frac{\Delta_{\sG^0,p}\,\zeta_p(1)^{\epsilon-1}}{L(1,\pi_p^{\sH^0}(z);{\rm Ad})\,L(1,\pi_p^{\sH^0}(z);{\rm Std})} 
 \label{SpectMeasure}
\\
&\quad \times \int_{\fX_p^{0}/W_{\Q_p}^{\sG}}\alpha(\nu)\,\frac{L\left(\frac{1}{2},\pi_p^{\sH^0}(z)\boxtimes \pi_p^{\sG^0}(\nu)\right)L\left(\frac{1}{2},\pi_p^{\sG^0}(\nu);{\rm Std}\right)}{L(1,\pi_p^{\sG^0}(\nu);{\rm Ad})}{}\,\d\mu_p^{\rm Pl}(\nu)
 \notag
\end{align}
for any $\alpha \in C(\fX_p^{0+}/W_{\Q_p}^{\sG})$. (Here $\sG^{0}$ and $\sH^0$ denote the identity component of $\sG$ and $\sH$, respectively and the local $L$-factors are defined for spherical representations $\pi_{p}^{\sG^0}(\nu)$ and $\pi_p^{\sH^0}(z)$ of special orthogonal groups $\sG^{0}(\Q_p)$ and $\sH^0(\Q_p)$. For other unexplained notation, we refer to \S~\ref{EVSfactro}. From $z\in \fX_p^{0}(\xi)$, the measure $\Lambda_p^{\xi,(z)}$ is easily seen to be non-negative.) The product group $W_S^{\sG}=\prod_{p\in S}W_{\Q_p}^{\sG}$ acts on $\fX_S=\prod_{p\in S}\fX_p$ and its compact subsets $\fX_S^{0}=\prod_{p\in S}\fX_p^{0}$ and $\fX_S^{0+}=\prod_{p\in S}\fX_S^{0+}$; the orbit spaces $\fX_S^{0}/W_{S}^{\sG}$ and $\fX_S^{0+}/W_{S}^{\sG}$ are referred to as the unramified tempered dual and the unramified unitary dual of $\sG(\Q_S)$, respectively. The spectral parameter of ${\rm F}$ at $p$ is defined to be the point $\nu_p=\nu_p({\rm F})\in \fX_p^{0+}/W_{\Q_p}^{\sG}$ such that ${\rm F}*\phi_p=\hat\phi_p(\nu_p)\,{\rm F}$ for all $\phi_p\in \cH(\sG(\Q_p)\sslash \bK_p)$. Similarly, we have the spectral parameter $z_p=z_p^{\cU} \in \fX_p^{0+}(\xi)$ of $\cU$ at $p\in S$ determined by the relation $C_p^{\cU}(\varphi_p)=\hat\varphi_p(z_p)$ for all $\varphi_p\in \cH(\sG_1^\xi(\Q_p)\sslash \bK_{1,p}^\xi)$. Set $\nu_S(F)=\{\nu_p\}_{p\in S} \in \fX_S^{0+}/W_S^\sG$ and $z_S=\{z_p\}_{p\in S}$. We say that $\cU$ is tempered over $S$ if $z_S$ belongs to $\fX_S^0(\xi)$. Depending on $l$, we define a linear functional $\mu_l^{\xi,\cU}$ on the space of continuous functions on $\fX_S^{0+}/W_S^\sG$ as 
\begin{align}
\mu_{l}^{\xi,\cU}(\alpha) =
\frac{{\mathbf \Gamma}(l)}{4l^m\,\dim(\cU(\bK_\fin^{\xi*}))} \sum_{F\in \cB_l^+}\sum_{f\in \cB(\cU;\bK_{1,\fin}^{\xi*})^{(-1)^{l}}} \alpha(\nu_{S}(F))\,L_\fin(F,1/2)\,|\fa_F^f(\xi)|^2, \quad \alpha \in C(\fX_S^{0+}/W_S^{\sG}),
\label{}
\end{align}
where 
$$
{\mathbf \Gamma}(l)=\frac{l^{m}\,\Gamma(l-\rho-1/2)\,\Gamma(l-2\rho)}{\Gamma(l-\rho/2)\,\Gamma(l-\rho/2+1/2)}.
$$
It is confirmed that ${\mathbf \Gamma}(l)=1+O(l^{-1})$ by Stirling's formula. When $\cU$ is tempered over $S$, the product measure of $\Lambda_p^{\xi,(z_p)}$ $(p\in S)$ is denoted by ${\mathbf\Lambda}^{\xi,(z_S)}$, i.e., ${\mathbf\Lambda}^{\xi,(z_S)}=\bigotimes_{p\in S}\Lambda_p^{\xi,(z_p)}$. 
\begin{thm} \label{Maintheorem}
Suppose that $\cU$ is tempered over $S$ if $S\not=\emp$ and that $L_\fin(\cU,s)$ is regular at $s=1$. Let $l\in \N$. For any $\phi=\otimes_{p<\infty}\phi_p \in \cH^{+}(\sG(\A_\fin)\sslash \bK_\fin^{*})$ with $\phi_p={\rm ch}_{\bK_p}$ for almost all $p$, there exists a constant $C_{\phi}>1$ such that 
{\allowdisplaybreaks\begin{align}
\mu^{\xi,\cU}_l(\widehat\phi_{S})
=c_\cL(\xi,f)\,\{1+(-1)^{l}\chi(\cU)\}
{\mathbf\Lambda}^{\xi,(z_S)}(\widehat{\phi}_{S})  
+{O}(C_{\phi}^{-l}), \quad (l\rightarrow +\infty),
 \label{Maintheorem-1}
\end{align}}where $\widehat \phi_{S}(\nu)=\prod_{p\in S}\hat \phi_p(\nu_p)$ for $\nu \in \fX_S^{0+}$, 
{\allowdisplaybreaks\begin{align*} 
&b_\cL(\xi)=2\,\delta(2\xi \in \cL_1)\,(2^{-1}\fd(\cL))^{-1/2}\left(\tfrac{\pi}{4}\right)^{-\rho}, \\
&c_\cL(\xi,\cU)=b_\cL(\xi)\,
\begin{cases}
L_\fin(\cU,1), \quad & ({\text{$m$: odd}}), \\
L'_\fin(\cU,1)-d_\cL(\xi)\,L_\fin(\cU,1), \quad & ({\text{$m$: even}}),
\end{cases}
 \\
&d_{\cL}(\xi)=\tfrac{-1}{2}\log(2^{-1}\fd(\cL_1^\xi))+\left(\tfrac{m}{2}-1\right)\,\log(2\pi)-\sum_{j=1}^{m/2-1}\tfrac{\Gamma'}{\Gamma}\left(\tfrac{m+1}{2}-j\right).
\end{align*}}
\end{thm}
In the case when $L_\fin(\cU,s)$ has a simple pole at $s=1$, which happens only when $m$ is odd, we have the following. 
\begin{thm}\label{Maintheorem0} 
Suppose that $\cU$ is tempered over $S$ if $S\not=\emp$ and that $L_\fin(\cU,s)$ has a simple pole at $s=1$. Let $\epsilon \in \{+,-\}$ be such that $d^{\e}(\cU)>0$, and set $\N(\epsilon)=\{l\in \N|\,\epsilon (-1)^l=1\}$. Then for any $\phi=\otimes_{p<\infty}\phi_p \in \cH^{+}(\sG(\A_\fin)\sslash \bK_\fin^{*})$ with $\phi_p={\rm ch}_{\bK_p}$ for almost all $p$,  
\begin{align*}
& \lim_{\substack{l\rightarrow \infty \\ l\in \N(\epsilon)}}(\log l)^{-1}\,\mu_{l}^{\xi,\cU}(\widehat{\phi_S})=b_{\cL}(\xi)\,\{1+(-1)^{\epsilon}\chi(\cU)\}\,{\rm Res}_{s=1}L_\fin (\cU,s)\times {\mathbf\Lambda}^{\xi,(z_S)}(\widehat{\phi}_{S}).
\end{align*}
\end{thm}
Applying this to $S=\emp$, we immediately obtain
\begin{cor}
Suppose that $m$ is odd and $\xi \in \cL_1$. Suppose $L_{\fin}(\cU,s)$ has a simple pole at $s=1$. Let $\epsilon\in \{0,1\}$ be such that $d^{\e}(\cU)>0$. Then, there exists $L\in \N(\epsilon)$ with the following property: For any integer $l\geq L$ in $\N(\epsilon)$ there exist ${\rm F}\in S_l(\bK_\fin^*)$ and $f\in \cB(\cU;\bK_{1,\fin}^{\xi*})^{(-1)^{l}}$ such that $L_{\fin}({\rm F},1/2)\not=0$ and $\fa_{{\rm F}}^{f}(\xi)\not=0$. 
\end{cor}
Since $\dim_\C S_l(\bK_\fin^*)=\sharp \cB_l \asymp l^{m}$ by the Hirzebruch-Mumford proportionality theorem (\cite{Mumford}, \cite{GHS}), Theorem~\ref{Maintheorem} provides a weak evidence toward a variant of the Lindel\"{o}f conjecture $L_\fin({\rm F},1/2)\,|\fa_{{\rm F}}^f(\xi)|^2=O(l^\epsilon)$ ($l\rightarrow \infty$) for any $\epsilon$ , showing its validity in average. We also remark that when $m$ is odd the point $s=1/2$ is critical in the sense that both $\Gamma_{\cL}(s)$ and $\Gamma_{\cL}(1-s)$ are regular at $s=1/2$. Let $\cB_{l}^{+}(\natural)$ denote the set of $F\in \cB_l^{+}$ such that $\pi_F$ is tempered, i.e., the local representations $\pi_{F,v}$ is tempered for all places, where $\pi_{F}\cong \bigotimes_{v}\pi_{F,v}$ is the cuspidal representation of $\sG(\A)$ generated by $F$, and set $\cB_l^{+}(\flat)=\cB_l^{+}-\cB_l^+(\natural)$. \footnote{Is seen that $\pi_{F}$ is irreducible as a $\sG(\A_\fin)\times (\fg,\bK_\infty)$-module. ({\it cf}. Proposition~\ref{AutoRep-P} and \cite[Theorem 3.1]{NPS}.} There seems to be a good reason to expect that the following assertions are true (\cite{FurusawaMorimoto1}, \cite{FurusawaMorimoto}, \cite{QU}, \cite{Lapid}, \cite{KST}, \cite{Tsud2019-2}):
\begin{align}
&\text{$|\fa_{F}^{f}(\xi)|^2\,L_\fin ({\rm F},1/2)\geq 0$ for all ${\rm F}\in \cB_l^{+}(\natural)$ and $f\in \cB(\cU;\bK_{1,\fin}^{\xi*})$.} 
 \label{non-negativityL} \\ 
&\lim_{l\rightarrow \infty}\frac{{\bf \Gamma}(m)}{4l^m}\sum_{f\in \cB(\cU;\bK_{1,\fin}^{\xi*})} \sum_{F\in \cB_l^{+}(\flat)}|L_{\fin}(F,1/2)|\,|\fa_{F}^{f}(\xi)|^2=0.   
\label{NTvanishing}
\end{align}
Conditionally upon these, we have the following. 

\begin{thm} \label{MAINTHM3}
 Suppose that $m$ is odd and $\xi \in \cL_{1}$. Let $\cU$ be an irreducible $\sG_1^{\xi}(\A_\fin)$-submodule of $\cV(\xi)$ with $\bK_{1,\fin}^{\xi*}$-fixed vectors. Suppose that $\cU$ is tempered over $S$ and that \eqref{non-negativityL} and \eqref{NTvanishing} are shown to be true. Let $\epsilon\in \{+,-\}$ be such that $d^{\e}(\cU)>0$. If $L_{\fin}(\cU,s)$ is regular (resp. has a simple pole) at $s=1$, then as $l\in \N(\epsilon)$ grows to infinity, the measure $\mu_l^{\xi,\cU}$ (resp. $\mu_l^{\xi,\cU}(\log l)^{-1}$) on $\fX_S^{0+}/W_{S}^{\sG}$ converges $*$-weakly to the measure $\{1+\epsilon\,\chi(\cU)\} b_{\cL}(\xi)L(\cU)\,{\mathbf\Lambda}^{\xi,(z_S)}$, where $L(\cU):=L_{\fin}(\cU,1)$ (resp. $L(\cU):={\rm Res}_{s=1}L_{\fin}(\cU,s)$).  
\end{thm}

\begin{cor} \label{MAINTHM2} 
Retain all the assumptions of Theorem~\ref{MAINTHM3}. Suppose $L_\fin(\cU,1)\not=0$. Then given a non-empty $W_S^{\sG}$-stable open subset $\cN$ of $\fX_S^0$, we can find $L\in \N$ with the following property: For any $l\in \N(\epsilon)$ with $l\geq L$, there exists ${\rm F}\in \cB_l^+$ and $f\in \cB(\cU;\bK_{1,\fin}^{\xi*})^{\epsilon} \not=0$ such that $L_\fin({\rm F},1/2)\not=0$, $\fa_{{\rm F}}^f(\xi)\not=0$, and $\nu_{S}({\rm F})\in \cN$. If $S=\emp$, then we have the same conclusion without assuming the temperedness of $\cU$ over $S$ as well as \eqref{non-negativityL} and \eqref{NTvanishing}. 
\end{cor}
The following example is worth recording here. Suppose $m$ is odd and let $\cU_0$ be the space of constant functions on $\sG_1^{\xi}(\A_\fin)$. From \cite[(5.33)]{Sugano95}, the $L$-function $L_\fin(\cU_0,s)$ is a product of Riemann zeta functions $\zeta(s+j-(m-1)/2)\,(1\leq j \leq m-2)$, the Dirichlet $L$-function $L(s,\chi)$ with $\chi$ being the Kronecker character of a certain quadratic extension of $\Q$, and a finite Euler product non-zero at $s=1$. From this, we have that $L_\fin(\cU_0,s)$ is a zero at $s=1$ if and only if $m\geq 13$, that $L_{\fin}(\cU_0,s)$ is regular and non-zero at $s=1$ if and only if $m=9,11$, and that $L_{\fin}(\cU_0,s)$ has a simple pole at $s=1$ if and only if $m=3,5,7$.

\subsection{Structure of paper}
We explain the structure of this paper, giving an overview of our method to prove the main theorem. Our method is completely free from a higher rank analogue of the Petersson formula of Fourier coefficients ({\it cf}. \cite{KST} and \cite{Blomer}); this means that a similar technique might carry over to the case of non-quasi split unitary groups. We start with the Poincar\'{e} series $\hat\F_l^{f,\xi}(\phi|\beta;g)$ which is a slight modification of a similar function constructed in our previous paper \cite{Tsud2011-1}, where $\beta$ is an auxiliary entire function for smoothing. In \S~\ref{subsecSPECEXP}, by means of the Rankin-Selberg integral, the $(\xi,\bar f)$ Whittaker-Bessel coefficient of ${\hat\F_l^{f,\xi}(\phi|\beta)} \in S_{l}(\bK_{\fin}^*)$ is explicitly computed in the form $\int_{(c)}\beta(s)\II_l(\phi|s)\d s$ with $\II_l(\phi|s)$ being an average of $\widehat\phi_S(\nu_S({\rm F}))a_{{\rm F}}^{\bar f}(\xi)|L({\rm F},s+1/2)$ over ${\rm F}\in \cB_l^{+}$ (Proposition~\ref{SpectralExp-P} and Lemma~\ref{IISpectEx}); to obtain this, we can largely follow the arguments in \cite{Tsud2011-1}. Our substantial task in this article is to compute $a_{\hat\F_l^{f,\xi}(\phi|\beta)}$ differently by separating the summation \eqref{PrS} to sub-series according to the $(\sP^\xi,\sP)$-double coset of $\gamma\in \sG(\Q)$. After recalling basic notation in \S\ref{Prelim}, in \S\ref{DouCosDec}, we study the structure of the double coset space $\sP^\xi(\Q)\bsl \sG(\Q)/ \sP(\Q)$ and show that it consists of $4$ elements, which are represented by particular rational points $1$, $\sw_0$, $\sw_1$ and $\sn(\xi)\sw_0$ in $\sG(\Q)$. Thus, as we shall see in \S~\ref{ProofMTHM}, $a_{\hat\F_l^{f,\xi}(\phi|\beta)}^{\bar f}(\xi)$ is written as a sum of four terms $\hat\JJ_{l}(\su,\phi|\beta)$ labeled by those coset representatives $\su \in \{1,\sw_0,\sw_1,\sn(\xi)\sw_0\}$. The term $\hat\JJ_{1}(\sw_0,\phi|\beta)$ is further divided to a sum of two terms $\hat \JJ_{l}^{\rm{sing}}(\sw_0,\phi|\beta)$ and $\hat \JJ_{l}^{\rm{reg}}(\sw_0,\phi|\beta)$, referred to as the singular term and the regular term, respectively. Let $\hat\JJ_{l}(\beta)$ be one of these 5 terms viewed as a linear functional in $\beta$. Spending 5 sections from \S\ref{JJidentity} to \S\ref{JJbsnxi}, we show that there exists an entire function $\JJ_l(s)$ on the vertical strip $\Re(s)\in (\rho,l-3\rho-1)$ such that $\hat\JJ_l(\beta)=\int_{(c)}\beta(s)\JJ_l(s)\, \d s$ for sufficiently many $\beta$, and hence obtain a ``trace-formula'' which equates $\II_l(\phi|s)$ with a sum of those $5$ entire functions $\JJ_l(1,\phi|s)$, $\JJ_l^{\rm sing}(\sw_0,\phi|s)$, $\JJ_l^{\rm reg}(\sw_0,\phi|s)$ and $\JJ_{l}(\sw_1,\phi|s)$, and $\JJ_l(\sn(\xi)\sw_0,\phi|s)$ on the vertical strip $\Re(s)\in (\rho,l-3\rho-1)$. Actually, in \S~\ref{JJidentity} and \S~\ref{JJw0sing}, we determine an explicit formula of $\JJ_l(1,\phi|s)$ and $\JJ_l^{\rm sing}(\sw_0,\phi|s)$, which shows their entireness on $\C$ together with the relation $\JJ_l(1,\phi|s)=\JJ_{l}^{\rm sing}(\sw_0,\phi|-s)$. Here Liu's formula of the regularized Bessel period for spherical matrix coefficients (\cite{Liu}) is of crucial importance; at this point, the temperedness of $f$ over $S$ is necessary. Although the remaining 3 terms are much more complicated to be exactly evaluated, an intensive analysis of Archimedes integrals involving Bessel functions made in \S\ref{JJw0regular}, \S\ref{JJbsne} and \S\ref{JJbsnxi} allows us to have their majorant $C^{-l}$ which is exponential decay as the growing weight $l$; the necessary formulas for special functions are collected in Appendix \S\ref{APP} for convenience. Prior to these highly technical sections, we complete the proof of Theorems~\ref{Maintheorem} and \ref{MAINTHM3} and Corollary~\ref{MAINTHM2} in \S~\ref{ProofMTHM} arguing like this. By the trace formula mentioned above, the sum of three terms $\SS_l(\phi|s)=-\frac{\Gamma(l-\rho)}{(\sqrt{8|Q[\xi]|}\pi)^{l-\rho}}\{\JJ_l^{\rm reg}(\sw_0,\phi|s)+\JJ_{l}(\sw_1,\phi|s)+\JJ_{l}(\sn(\xi)\sw_0,\phi|s)\}$ has a holomorphic continuation to $\C$ satisfying the functional equation $\SS_l(\phi|s)=\SS_l(\phi|-s)$, which  extends the bound $\SS_l(\phi|s)\ll C^{-l}$ on $\Re(s)\in [q,q']$ inside the strip $\Re(s)\in (\rho,l-3\rho-1)$ to the opposite side $\Re(s)\in [-q',-q]$. Then by Phragmen-Lindel\"{o}f convexity theorem the same bound can be interpolated to the estimation on the strip $|\Re(s)|\leq q'$ (Theorem~\ref{ErP1}). We obtain Theorem~\ref{Maintheorem} by looking the trace formula identity and the estimate $|\SS_l(\phi|s)|\ll C^{-l}$ at $s=0$. In Appendix 2, we first collect basic materials from \cite{MS98} and include a detailed exposition on the representations generated by Hecke eigenvectors. In Appendix 3, we establish a bound of the Eisenstein series on rank one orthogonal groups, which was used in the proof of Lemma~\ref{IISpectEx}. Once Theorem~\ref{Maintheorem} is established, then Corollary~\ref{MAINTHM3} is shown by a familiar argument of approximation under the assumptions \eqref{non-negativityL} and \eqref{NTvanishing}. 

To prove Theorem~\ref{MAINTHM3} unconditionally, it might be better to obtain a weight aspect asymptotic for the second moment of the central values $L_{\fin}({\rm F},1/2)$ with ${\rm F}\in \cB_l$. For $m=3$, such a bound for the average without Hecke operators is deducible from the result of \cite{Blomer}. We hope to return to this issue in a future work. 

Finally, we record notation and conventions in this article. Set $\N^{*}=\{n\in \Z|n>0\}$ and $\N=\N^{*}\cup\{0\}$. Let us denote by $\fin$ the set of all prime numbers. Let $\A$ and $\A_\fin$ denote the adele ring of $\Q$ and the subring of finite adeles, respectively. Let $\psi$ denote a character of $\A$ such that $\psi(a)=1$ for all $a\in \Q$ and $\psi(x)=e^{2\pi i x}$ for all $x\in \R$. For an idele $x=x_\infty x_\fin \in \A^\times$, we set $|x_\fin|_\fin=\prod_{p\in \fin} |x_p|_p$ and $|x|_\A=|x_\fin|_\fin\,|x_\infty|_\infty$. Set $\Gamma_\C(s)=(2\pi)^{-s}\Gamma(s)$, with $\Gamma(s)$ the gamma function. Note that our definition of $\Gamma_\C(s)$ is different from the usual one by the factor $2$. 

\section{Preliminary} \label{Prelim}

\subsection{Quadratic lattices} \label{sec:QUALAT}
Let $\e_0=\left[\begin{smallmatrix}0\\ 1 \\ 0_{m-2} \\ 0 \\ 0 \end{smallmatrix}\right]$, $\e_0'=\left[\begin{smallmatrix}0\\0\\ 0_{m-2} \\ 1 \\ 0\end{smallmatrix}\right]$ and $V_0=<\e_1,\e_0,\e_0',\e_1'>^{\bot}_\Q$. Then the quadratic space $V$ is an orthogonal direct sum of $\Q$-anisotropic space $V_0$ and two hyperbolic planes $<\e_1,\e_1'>_\Q$, $<\e_0,\e_0'>_\Q$. Set $\cL^\xi=\cL\cap \xi^\bot$ and $\cL_1^\xi=\xi^\bot \cap \cL_1$. For any commutative ring $R$ and any $\Z$-module $\cM$, set $\cM_R=\cM\otimes_{\Z}R$. If $R$ is a commutative $\Q$-algebra, we set $V(R)=\cL_R$; similarly, we define $V_{1}(R)$, $V_{0}(R)$, $V^\xi(R)$ and $V_{1}^\xi(R)$ viewing them as $\Q$-vector spaces with quadratic forms.

Let $\xi\in V_{1}$ be as in \S~\ref{Intro}. Recall that we have already imposed conditions (i), (ii) and (iii) on $\xi$; in addition to these, we further suppose 
\begin{itemize}  
\item[(iv)] $\langle \e_0,\xi\rangle=1$\footnote{This is not really a restriction on $\xi$; indeed, given $\xi$ primitive in $\cL_1$, we can find an isotropic primitive vector $\e_0\in \cL_1$ such that $\langle \e_0,\xi\rangle=1$}.
\end{itemize}
The condition (ii) and (iv) imply that $\xi$ is of the form 
\begin{align}
\xi=a\,\e_0+\sa+\e_0',
\quad a\in \Z,\,\sa \in \cL_{1}^{*}.
\label{xiForm}
\end{align} 
The $\hat \Z$-modules $\cL_{\hat \Z}$, $\cL_{1,\hat\Z}$ and $\cL_{1,\hat \Z}^\xi$ are abbreviated to $\cL_{\fin}$, $\cL_{1,\fin}$ and $\cL_{1,\fin}^\xi$, respectively. 

Let $\fd(\cL)$ denote the absolute value of the Gram determinant of $\cL$. i.e., $\fd(\cL)=|\det(\langle v_i,v_j \rangle)|$ with $\{v_j\}$ any $\Z$-basis of $\cL$. Then $\fd(\cL)=\#(\cL^*/\cL)$. For any prime number $p$, we set $\fd_p(\cL)=\#(\cL_p^*/\cL_p)$; then $\fd_p(\cL)=1$ for almost all $p$ and $\fd(\cL)=\prod_{p<\infty}\fd_p(\cL)$. Similarly, we define $\fd(\cL_{1}^\xi)$, $\fd(\cL_1)$, $\fd_p(\cL_1^\xi)$ and $\fd_p(\cL_1)$.

\begin{lem} \label{IndexL}
\begin{align*}
\fd(\cL_1)=|Q[\xi]|^{-1}\,\fd(\cL_1^{\xi}).  
\end{align*}
\end{lem}
\begin{proof}
Let $N$ be the smallest positive integer such that $N\xi \in \cL_1$. From $\xi \in \cL_1^{*}$, we have $N\,Q[\xi]=\langle N\xi,\xi\rangle \in \Z$. Set $\cM=\Z N\xi+\cL_1^{\xi}$, which is a sublattice of $\cL_1$ of full rank. By the integrality, we have the sequence of inclusions $\cM\subset \cL_1\subset \cL_1^*\subset \cM^{*}$, which yields $[\cM^*:\cM]=[\cM^*:\cL_1^*][\cL_1^*:\cL_1][\cL_1:\cM]$. Since the Pontrjagin dual of the finite abelian group $\cL_1/\cM$ is isomorphic to $\cM^*/\cL_1^*$, we have $[\cM^*:\cL_1^*]=[\cL_1:\cM]$. Since $\cM^*=\Z\,Q[\xi]^{-1}N^{-1}\,\xi+\cL_{1}^{\xi *}$, we have $[\cM^*:\cM]=|Q[\xi]|N^2\,[\cL_1^{\xi *}:\cL_1^\xi]$. Since $\cM=\Z N\xi+\cL_1^{\xi}$, $\cL_1=\Z \e_0+\cL_1^{\xi}$ and $\langle \e_0,\xi\rangle=1$, we easily see $[\cL_1:\cM]=|\langle N\xi,\xi\rangle|=N|Q[\xi]|$. Therefore, 
$$
|Q[\xi]|N^2[\cL_1^{\xi *}:\cL_1^{\xi}]=(N|Q[\xi]|)^{2}\times [\cL_1^*:\cL_1],
$$
or equivalently $[\cL_1^{\xi *}: \cL_1^\xi]=|Q[\xi]|\,[\cL_1^*:\cL_1]$ as required. 
\end{proof}

\subsection{Algebraic groups} 
Let $\sG_0$, $\sG_1$ and $\sG$ be orthogonal groups of $(\cL_0,Q_0)$, $(\cL_1,Q_1)$ and $(\cL,Q)$, respectively. Let $\sG^\xi$ and $\sG_1^\xi$ be the orthogonal groups of $(\cL^{\xi},Q|\cL^\xi)$ and $(\cL_1^{\xi},Q_1|\cL_1^{\xi})$, respectively. Let $\sP$ be the stabilizer of the isotropic line $\Q\e_1$ in $\sG$. For any $\Q$-algebra $R$, the set $\sP(R)$ consists of all the elements $\sm(r;h)\sm(X)$ with $r\in R^\times$, $h\in \sG_1(R)$ and $X\in V_1(R)$, where $\sm(r;h)\in \sG(R)$ and $\sn(X)\in \sG(R)$ are defined as  
\allowdisplaybreaks{
\begin{align*}
\sm(r;h)&={\rm diag}(r,h,r^{-1}), \quad \sn(X)=\left[\begin{smallmatrix} 1 & {-{}^tX Q_1} & {-2^{-1}Q_1[X]} \\ {0} & {1_m} & {X} \\ {0} & {0} & {1} \end{smallmatrix} \right].
\end{align*}
}
We have the Levi decomposition $\sP=\sM\,\sN$, where $\sM(R)=\{\sm(r;h)|\,r\in R^\times, \, h\in \sG_1(R)\,\}$ and $\sN(R)=\{\sn(X)|\,X\in V_{1}(R) \,\}$ for any $R$ as above. Set $\sP^\xi=\sG^\xi\cap \sP$, $\sM^{\xi}=\sG^\xi\cap \sM$, and $\sN^\xi=\sG^\xi\cap \sN$. Then $\sP^\xi$ is a $\Q$-parabolic subgroup of $\sG^\xi$ with the Levi decomposition $\sP^\xi=\sM^\xi\sN^\xi$; we have $\sM^\xi(R)=\{\sm(t;h_0)|t\in R^\times,\,h_0\in \sG^\xi_1(R)\}$ and $\sN^\xi(R)=\{\sn(Z)|Z\in V^\xi_1(R)\}$.

\subsection{Root systems and Weyl groups } \label{sectionBD}
Let $k$ be a field of characteristic different from $2$. We fix a $k$-basis of $V_{0,k}$ and suppose $Q_0=\left[\begin{smallmatrix} 0 & 0 & J_{\ell} \\ 0 & R & 0 \\ J_{\ell} & 0 & 0 \end{smallmatrix}\right]  
$ with $J_{\ell}$ being an anti-diagonal matrix of degree $\ell$ whose non-zero entries are all $1$ and $R$ a regular symmetric matrix of degree $n_0\geq 0$ which is $k$-anisotropic. We have $m-2=n_0+2\ell$. Let $\sT=\{t=\diag(t_1,\dots,t_{\ell+2},1_{n_0},t_{\ell+2}^{-1},\dots,t_1^{-1})|\,t_j\in {\mathbf{GL}}_1\,(1\leq j\leq \ell+2)\}$ be the maximal $k$-split torus of $\sG$ over $k$, where $1_{n_0}$ does not occur when $n_0=0$. Let $\eta_j\in X^{*}(\sT)$ $(1\leq j \leq \ell+2)$ be the $\Z$-basis of the $k$-rational character group of $\sT$ such that $\eta_j(t)=t_j\,(t\in T_k)$. Then the root systems of $(\sT,\sG_k)$ and $(\sT,\sM_k)$ are respectively given as
\begin{align*}
\Sigma(\sT,\sG_k)&=\{\pm(\eta_i\pm \eta_j)|1\leq i<j\leq \ell+2\}\cup \Sigma^{\rm short},  \\
\Sigma(\sT,\sM_k)&=\{\pm(\eta_i\pm \eta_j)|2\leq i<j\leq \ell+2\}\cup (\Sigma^{\rm short}\cap \{\pm \eta_j|\,2\leq j\leq \ell+2\}), 
\end{align*}
where $\Sigma^{\rm short}=\{\pm \eta_j|\,1\leq j\leq \ell+2\}$ if $n_0>0$ and $\Sigma^{\rm short}=\emp$ if $n_0=0$. Let $\sB$ be the Borel subgroup of upper-triangular matrices in $\sG_k=\sG\times_\Z k$, and $\Sigma^{+}(\sT,\sG_k)$ the corresponding set of positive roots. Let $\rho_{\sB_k}$ be the half-sum of elements of $\Sigma^{+}(\sT,\sG_k)$; then a computation yields $\rho_{\sB_k}=\sum_{j=1}^{\ell+2}(\frac{m+2}{2}-j)\eta_j$. Let $W_k^{\sG}$ (resp. $W_k^{\sM}$) is the Weyl group of the restricted root system of $\sG$ (resp. $\sM$) over $k$. We identify $W_k^{\sG}$ with the semi-direct product $S_{\ell+2} \ltimes (\Z/2\Z)^{\ell+2}$, i.e., 
$$
(\sigma,v)\cdot (\sigma',v')=(\sigma\sigma',v^{\sigma'}+v'), \quad 
(v,v'\in (\Z/2\Z)^{\ell+2},\,\sigma,\sigma'\in S_{\ell+2},
$$
where $(v^{\sigma'})_{i}=v_{\sigma'(i)}$ for $v=(v_i)_{i=1}^{\ell+2}$. An element $w=(\sigma,v)$ acts on $X^*(\sT)$ as
$$
w\eta_j=(-1)^{v_j}\eta_{\sigma(j)}, \quad 1\leq j \leq \ell+2.
$$
The Weyl group $W_k^{\sM}$ is identified with the subgroup $\{w\in W_k^{\sG}|\,w\,\eta_1=\eta_1\}$. We have the Bruhat decomposition $\sG(k)=\bigcup_{w\in W_k^{\sM}\bsl W_k^{\sG}/W_k^{\sM}}\sP(k)w\sP(k)$. 

\begin{lem} \label{BruhatRepr} The double coset space $W^{\sM}_k\bsl W_k^{\sG}/W_k^{\sM}$ is represented by the following three elements: 
\begin{align*}
w_{1}^{+}=(e,(0,0,\dots,0)), \quad w_{1}^{-}=(e,(1,0,\dots,0)), \quad w_{2}^{+}=((12),(0,0,\dots,0)),
\end{align*}
where $e$ is the identity of $S_{\ell+2}$ and $(12)\in S_{\ell+2}$ denotes the permutation of $1,2$. 
\end{lem}
\begin{proof}
The space $W^{\sM}_k\bsl W_k^{\sG}/W_k^{\sM}$ has a complete set of representatives given by elements $w\in W_k^{\sG}$ such that $w\alpha>0$ and $w^{-1}\alpha>0$ for all $\alpha\in \Sigma^{+}(\sT,\sM_k)$. An element $w=(\sigma,v)\in W_k$ with $w (\Sigma^{+}(\sT,\sM_k))\subset \Sigma^{+}(\sT,\sM_k)$ is one of the following form:
\begin{align}
w_a^{+}:=(\sigma_a,(0,0,\dots,0)), \quad w_a^{-}:=(\sigma_a,(1,0,\dots,0)
)\quad (1\leq a\leq \ell+2)
 \label{BruhatRepr-1}
\end{align}
where $\sigma_a\in S_{\ell+2}$ is the unique element such that $\sigma_a(j)<\sigma_a(j+1)$ for all $j\in [2,\ell+2]$ and $\sigma_{a}(1)=a$, or explicitly $\sigma_1=e$ and 
$$
\sigma_a=\left(\begin{matrix} 
1& 2& \cdots & a-1 & a  & a+1 & \cdots &\ell+2 \\ 
a & 1 &\cdots & a-2 & a-1& a+1 & \cdots &\ell+2
\end{matrix}\right) \quad (a>1)
$$
Hence $(w_a^{-})^{-1}=(\sigma_a^{-1},(0,1,0,\dots,0))$ and $(w_a^{-})^{-1}\eta_{2}=-\eta_{\sigma_a^{-1}(2)}<0$ if $a>1$, and $(w_a^{+})^{-1}=(\sigma_a^{-1},(0,\dots,0))$ and $(w_a^{+})^{-1}(\eta_{a-1}-\eta_a)=\eta_a-\eta_1<0$ if $a>2$. Thus among the elements \eqref{BruhatRepr-1}, only the three elements $w_1^{+}$, $w_2^{+}$ and $w_1^{-}$ satisfy the condition $w^{-1}(\Sigma^{+}(\sT,\sM_k))\subset \Sigma^{+}(\sT,\sM_k)$.
\end{proof}

Define $\fw_0,\,\sw_1\in \sG(\Q)$ by  
\begin{align*}&\fw_0\,\e_1=-\e_1', \quad \fw_0\,\e_1'=-\e_1,\quad \fw_0|V_1(\Q)={\rm id}, \\
&\fw_1\,\e_1=\e_0,\quad \fw_1\e_0=\e_1,\quad \fw_1\e_1'=\e_0', \quad \fw_1\e_0'=\e_1', \quad \fw_1|V_0(\Q)={\rm id}.
\end{align*}
Then for any $k$ of characteristic $0$, the Weyl group elements $1$, $\fw_0$ and $\fw_1$ corresponds to $w_{1}^{+}$, $w_{1}^{-}$ and $w_2^{+}$, respectively. Thus 
\begin{align}
\sG(k)=\sP(k)\bigsqcup \sP(k)\fw_1\sP(k)\bigsqcup \sP(k)\fw_0\sP(k).
\label{BruhatDecom}
\end{align}

\subsection{Open compact subgroups} \label{sec: OCSG}
 For any $p\in \fin$, set $\bK_{1,p}=\sG_1(\Z_p)$ and $\bK_p=\sG(\Z_p)$; then these are maximal compact subgroups of $\Q_p$-points $\sG_1(\Q_p)$ and $\sG(\Q_p)$, respectively. Set 
$$
\bK_\fin=\prod_{p\in \fin}\bK_{p}, \quad \bK_{1,\fin}=\prod_{p\in \fin} \bK_{1,p}. 
$$ 
The discriminant subgroups of $\bK_{p}$ and $\bK_{1,p}$ are denoted by $\bK_{p}^{*}$ and $\bK_{1,p}^{*}$, respectively (\cite[2.1]{Tsud2011-1}). Let $\bK_{\fin}^*$ (resp. $\bK_{1,\fin}^*$) be the direct product of $\bK_{p}^*$ (resp. $\bK_{1,p}^*$) over all $p\in \fin$. Let $\bK_{1,p}^{\xi}$ be the stabilizer of $\cL_{1,p}^{\xi}$ in $\sG_1^{\xi}(\Q_p)$ and $\bK_{1,p}^{\xi *}$ the discriminant subgroup of $\bK_{1,p}^\xi$. The direct product of $\bK_{1,p}^\xi$ (resp. $\bK_{1,p}^{\xi *}$) over all $p \in \fin$ is denoted by $\bK_{1,\fin}^{\xi *}$ (resp. $\bK_{1,\fin}^{\xi *}$). 

For a prime number $p$, the quotient group $\bK_p/\bK_p^*$ is denoted by $E_p$. From \cite[Proposition 1.1]{MS98}, the finite group $E_p$ is isomorphic to $\{1\}$, $\Z/2\Z$, or the dihedral group $D_{2(p+1)}$ of order $2(p+1)$. Since $E_{p}$ is trivial if $\cL_{p}=\cL_{p}^{*}$, the direct products $E_\fin:=\prod_{p\in \fin}E_{p}$ is a finite group isomorphic to the quotient $\bK_{\fin}/\bK_{\fin}^*$. Similarly, we set $E_p(\xi):=\bK_{1,p}^{\xi}/\bK_{1,p}^{\xi *}$ for $p\in \fin$ and define $E_\fin(\xi)=\prod_{p\in \fin}E_{p}(\xi)\cong \bK_{1,\fin}^{\xi}/\bK_{\fin}^{\xi *}$.

\subsection{Real Lie groups} \label{RealLie}
Let $G_1\cong {\rm {SO}}_0(m-1,1)$ and $G\cong{\rm {SO}}_0(m,2)$ be the identity connected components of the real points $\sG_1(\R)$ and $\sG(\R)$, respectively. Recall the definition of $\cD$ and $\fz_0$ from \S~\ref{Intro}. The mapping $g\mapsto g\langle \fz_0\rangle$ yields an isomorphism from the homogeneous space $G/\bK_\infty$ onto the domain $\cD$, where $\bK_\infty$ is the stabilizer of the base point $\fz_0$. Note that $\bK_\infty \cong \SO(m)\times \SO(2)$ is a maximal compact subgroup of $G$. 
Set 
\begin{align}
\eta_1^{\pm}=({\pm \e_1+\e'_1})/{\sqrt{2}}, \qquad 
v_0^{\C}=\eta_0^{-}+i\,\eta_1^{-}
\label{v0C}
\end{align}
Then, $\bK_\infty$ coincides with the stabilizer in $G$ of the totally negative subspace $<\eta_0^{-},\eta_1^{-}>_\R$. For $g\in \sG(\R)$, we have the relations
$$
-\sqrt{2}i\,\langle \e_1,gv_0^{\C}\rangle=J(g,\fz_0), \quad -\sqrt{2}i\,\langle \eta, gv_0^\C\rangle=\langle \eta,g\langle \fz_0\rangle\rangle \quad (\eta \in V_1(\R)).
$$
Let $\sG(\R)^+$ be the index two subgroup of $\sG(\R)$ stabilizing the component $\cD \subset \tilde \cD$. Since $\cD$ is the set of points $\fz\in \tilde\cD$ such that $\Im \langle \eta_0^{-},\fz\rangle>0$, we have that $\sG(\R)^{+}=\{g\in \sG(\R)|\Im\,\langle \eta_0^{-}, g\langle \fz_0\rangle\rangle>0\}$.  For any subgroup $H\subset \sG(\R)$, set $H^{+}=H \cap \sG(\R)^{+}$. Then, 
\begin{align*}
&\sG(\R)^+=\{g\in \sG(\R)|\,\Im(\langle \eta_0^{-},g\,v_0^{\C}\rangle/\langle \e_1,g\,v_0^{\C}\rangle)>0\,\}, \\
&\sM(\R)^{+}=\{\sm(t;\,g_1)|\,t\in \R^\times,\,g_1\in \sG_1(\R),\,t\,\langle g_1\eta_0^{-},\eta_0^{-}\rangle<0\,\}, \\
 &\sG_1^\xi(\R)\subset \sG(\R)^{+}, \qquad \sN(\R)\subset \sG(\R)^{+}.   
\end{align*}
For any anisotropic vector $\eta\in V(\R)$, let $\varrho_{\eta} \in \sG(\R)$ denote the reflection of $V({\R})$ with respect to $\eta$. Then $\sm(-1;1_{m-1})=\varrho_{\eta_1^{+}}\varrho_{\eta_{1}^{-}}$, and we have the disjoint decompositions:  
\begin{align}
\sG(\R)&=\sG(\R)^{+}\cup \sm(-1;1_{m-1})\,\sG(\R)^{+}, 
\label{G=G+}
\\
\sG(\R)^{+}&=G \cup \varrho_{\eta_1^{+}}\,G, \quad \sG_1(\R)^{+}=G_1 ^+ \cup \varrho_{\eta_1^{+}}\,G_1^{+}. 
 \notag
\end{align}
Let $\tilde\bK_\infty^+:={\rm Stab}_{\sG(\R)^+}(\fz_0)$ be the maximal compact subgroup of $\sG(\R)^+$ containing $\bK_\infty$. Then $\tilde \bK_\infty^{+}=\{1,\varrho_{\eta_1^{+}}\}\bK_\infty$. Set 
\begin{align}
\Delta:=|Q[\xi]|,  \qquad \xi_0^{-}:=\Delta^{-1/2}\,\xi. 
 \label{DelXi0}
\end{align}
We fix a point $b_\infty\in G_1$ such that 
\begin{align}
b_\infty\,\eta_0^{-}=\xi_0^{-}
 \label{b-infty}
\end{align}
once and for all (\cite[3.5.2]{Tsud2011-1}). Then, $b_\infty$ fixes the vectors $\eta_1^{+}$ and $\eta_{1}^{-}$, and $b_\infty \bK_\infty b_\infty^{-1}$ is a maximal compact subgroup of $G$ stabilizing the subspace $<\xi_0^{-},\eta_1^{-}>
_\R$. The intersection $\bK_\infty^\xi=b_\infty \bK_\infty b_\infty^{-1} \cap \sG^\xi(\R)$ is a maximal compact subgroup of $G^\xi=\sG^\xi(\R)^\circ\cong {\rm {SO}}_0(m,1)$ containing $G_1^\xi=\sG_1^\xi(\R)^\circ\cong{\rm {SO}}(m-1)$. Define an element $\kappa_\infty$ of $\sG(\R)$ as
\begin{align}
\kappa_\infty:=\varrho_{\eta_{1}^+}\,\varrho_{\eta_1^{-}}\,\varrho_{\eta_0^{-}}.\label{kappa-infty}
\end{align}
\begin{lem} \label{kappa-eta_1+}
We have $\kappa_\infty \varrho_{\eta_1^{+}}\in \bK_\infty$ and $J(\kappa_\infty \varrho_{\eta_1^{+}},\fz_0)=-1$. 
\end{lem}
\begin{proof} The claim follows from the relations $\kappa_\infty \varrho_{\eta_1^{+}}v_0^{\C}=-v_0^{\C}$ and $\det (\kappa_\infty \varrho_{\eta_1^{+}})=+1$.  
\end{proof}

\subsection{Norm functions on vector groups} \label{NormFTN}
For $p\in \fin$ and $X=(X(j))_{j=1}^{m+2}\in V(\Q_p)$, set 
$$\|X\|_p=\max\{|X(j)|_p|\,1\leq j \leq m+2\}.$$
Define $$V(\A_\fin)^{\#}=\{X=(X_p)\in V(\A_\fin)|\,\|X_p\|_p=1\,\text{for almost all $p\in \fin$}\}. 
$$ 
Note that $\sG(\A_\fin)V(\Q)\subset V(\A_\fin)^{\#}$. We define a norm of $X=(X_p)_{p\in \fin}\in V(\A_\fin)^{\#}$ to be the product $\|X\|_\fin=\prod_{p\in \fin}\|X_p\|_p$. We have $\|k X\|_\fin=\|X\|_\fin $ for all $k \in \bK_\fin$ and $X\in V(\A_\fin)^{\#}$. 

Set
$$
\xi_0^{+}:=-\Delta^{1/2}\e_0-\xi_0^{-} \quad \in V_1(\R).
$$
By \eqref{DelXi0}, the following relations are easily confirmed. 
\begin{align}
&\e_0=-\Delta^{-1/2}(\xi_0^{+}+\xi_0^{-}), 
 \label{xiplus0}
\\
&Q[\xi_0^{+}]=+1,\quad Q[\xi_0^{-}]=-1, \quad \langle \xi_0^{+},\xi_0^{-}\rangle=0. 
\label{xiplus0pm}
\end{align}
We have the orthogonal decompositions 
\begin{align}
V(\R)&=<\eta_1^{+},\eta_1^{-}>_{\R} \oplus V_1(\R), 
\notag
\\
V_1(\R)&=<\xi_0^{+},\xi_0^{-}>_\R\oplus W  \quad \text{with $W=<\xi_0^{+},\xi_0^{-}>_\R^\bot\cap V_1(\R)$.}
 \label{NormFTN-f0}
\end{align} 
Set $$\eta_0^{+}:=b_\infty^{-1} \xi_{0}^{+}, \qquad W':=b_{\infty}^{-1}(W).$$
 Note that $Q$ is positive definite both on $W$ and on $W'$. The minimal majorant of the quadratic form $(Q, V(\R))$ corresponding to $\bK_\infty$ is defined as 
\begin{align*}
\|X\|_\infty=\bigl(Q[Z]^{1/2}+x_{+}^2+x_{-}^{2}+y_{+}^2+y_{-}^2\bigr)^{1/2}
\end{align*}
for $X=Z+x_{+}\eta_1^{+}+x_{-}\eta_1^{-}+y_{+}\eta_0^{+}+x_{-}\eta_0^{-}$ with $Z\in W'$ and $x_{\pm},y_{\pm}\in \R$. Then $\|k X\|_\infty=\|X\|_\infty$ for $X\in V(\R)$ and $k\in \bK_\infty$. For $(X_\infty,X_\fin)\in V(\R)\times V(\A_\fin)^\#$, we set $\|X\|_\A=\|X_\infty\|_\infty\,\|X_\fin\|_\fin$. 

\subsection{Haar measures}
Let $U\subset V$ be a non-degenerate $\Q$-subspace. For any place $p$ of $\Q$, we endow the space $U(\Q_p)=U\otimes_\Q \Q_p$ with the Haar measure which is self-dual with respect to the bi-character $\psi_p(\langle X,Y\rangle)$. If $p<\infty$, then for any integral $\Z$-lattice $\cM$ of $U$, we easily see that the $\Z_p$-lattice $\cM_{p}=\cM\otimes_{\Z}\Z_p$ has the measure $\fd_p(\cM)^{-1/2}=\#(\cM_p^*/\cM_p)^{-1/2}$. If $p=\infty$, then the self-dual measure on $U(\R)$ is transfered to the Lebesgue measure on $\R^{\dim U}$ by the $\R$-isomorphism $U(\R)\cong \R^{\dim U}$ determined by an orthogonal $\R$-basis $\{u_j\}$ such that $|Q[u_j]|=1$. 

Let $p$ be a prime. We define a Haar measure on $\sN(\Q_p)$ by $\d n=\fd_p(\cL)^{1/2} \d X$. Similarly, a Haar measure of $\sN^\xi(\Q_p)$ is normalized by means of the $\Z$-lattice $\cL_1^\xi$. We fix a Haar measure on $\sG^\xi(\Q_p)$ (resp. $\sG_1^\xi(\Q_p)$) such that $\bK_p^{\xi*}$ (resp. $\bK_{1,p}^{\xi*}$) has the measure $1$ \footnote{Erratum : On \cite[line 4 (p.2423)]{Tsud2011-1}, (resp. $\bK_{1,p}^{\xi}$) should be (resp. $\bK_{1,p}^{\xi*}$).}. The compact groups $\bK_{p}^{\xi*}$ and $\bK_{1,p}^{\xi*}$ are endowed with the probability Haar measures. By the Iwasawa decomposition $\sG^\xi(\Q_p)=\sP^\xi(\Q_p)\bK_{p}^{\xi*}$, we have the integral formula\footnote{Note : The formula in \cite[p.2423]{Tsud2011-1} contains typos.}
\begin{align*}
&\int_{\sG^\xi(\Q_p)}f(h)\,\d h \\
& \quad =\fd_p(\cL_{1}^\xi)^{1/2}\int_{V_1^\xi(\Q_p)}\d Z\int_{\Q_p^\times} \d^\times t \int_{\sG_1^\xi(\Q_p)} \d h_0 \int_{u\in \bK_{p}^{\xi*}}f(\sn(Z)\sm(t;h_0)u)|t|_p^{-(m-1)}\,\d u. 
\end{align*} 
This is confirmed by the relations $\sG^\xi(\Q_p)\cap \bK_p^\xi=\bK_{1,p}^{\xi*}$ and $\sN^\xi(\Q_p)\cap \bK_p^{\xi*}=\sN^\xi(\Q_p)\cap \bK_p^\xi=\{\sn(Z)|\,Z\in \cL_{1,p}^\xi\}$. On the real group $G^\xi=(\sG^\xi(\R))^0\cong {\rm {SO}}_{0}(m,1)$, we define a Haar measure $\d h$ by 
$$
\int_{G^\xi}f(h)\d h=2^{-\rho} \int_{V_1^\xi(\R)}\int_{0}^{\infty}\int_{\bK_\infty^\xi}f(\sn(Z)\sm(t;1_{m})u)\,t^{-2\rho}\d Z\,\d^\times t\,\d u
$$
for any integrable function $f(h)$ on $G^\xi$, where $\d Z$ is the self-dual measure on $V_1^\xi(\R)$ and $\d u$ is the probability Haar measure on $\bK_\infty^\xi\cong {\rm SO}(m)$. On the compact group $\sG^\xi_1(\R)$ we use the probability Haar measure. The Haar measure on $\sG^\xi(\A)$ (resp. $\sG^\xi_1(\A)$) is defined by the product of the Haar measures on $\sG^\xi(\Q_p)$ (resp. $\sG^\xi(\Q_p)$).  

On the Lie group $\sN(\R)$ (resp. $\sN^\xi(\R))$, we use the Haar measure defined by $\d n=\fd(\cL_1)^{-1/2}\,\d X$ (resp. $\d n_0=\fd(\cL_{1}^{\xi})^{-1/2}\d Z$) for $n=\sn(X),\,X\in V_1(\R)$ (resp. $n_0=\sn(Z),\,Z\in V_1^\xi(\R)$). We endow the adelization $\sN(\A)$ (resp. $\sN^\xi(\A)$) with the Haar measure such that $\sN(\A)/\sN(\Q)$ (resp. $\sN^\xi(\A)/\sN^\xi(\Q)$) has the measure $1$, or what amount to the same, the product measure of the Haar measures on $\sN(\Q_p)$ (resp. $\sN^\xi(\A)$) over all the places $p$. We use the Haar measures on the groups $G=\sG(\R)^0$, $\sG(\Q_p)$ with $p\in \fin$ and $\sG(\A)$ normalized as in \cite[\S 2.3]{Tsud2011-1}.  

\subsection{Local Hecke algebras and spectral parameters} \label{SpectralParameter}
Let $S$ be a finite set of prime numbers such that $\cL_p=\cL_p^{*}$. Then $\bK_p=\bK_p^*$ for all $p\in S$ and the group $\sG$ is unramified over $\Q_p$. Thus the Hecke algebra $\cH^{+}(\sG(\Q_p)\sslash \bK_p^*)$ used by Murase-Sugano (\cite{MS98}) reduces to the common one $\cH(\sG(\Q_p)\sslash \bK_p)$ for $p\in S$. Let us fix a prime $p\in S$ for a while and fix a Witt decomposition of $\cL_p$ as in \S~\ref{sec:LocalTheory}; it determines an orthogonal decomposition of $V(\Q_p)$ to $\ell_p$ copies of the hyperbolic plane $\langle e_j,e_{-j}\rangle_{\Q_p}$ with a pair of isotropic vectors $e_{\pm j}$ such that $Q(e_j,e_{-j})=1$ and a maximal $\Q_p$-ansiotropic subspace $W_p\subset V(\Q_p)$ such that $W_p\cap \cL_p=\{X\in W_p|2^{-1}Q[X]\in \Z_p\}$, where $\ell_p$ is the Witt index of $V(\Q_p)$. Note that $n_{0,p}:=\dim_{\Q_p}(W_p)\in \{0,1,2\}$ from our assumption $\cL_p=\cL_p^*$, and that $\ell_p=(m+1)/2$ if $n_{0,1}=1$, $\ell_{p}=(m+2)/2$ if $n_{0,p}=0$ and $\ell_{p}=m/2$ if $n_{0,p}=2$. For $t\in (\GL_1)^{\ell_p}$ and $h\in {\rm O}(W_p)$, let $d(t;h)\in \sG$ be the element such that $d(t,h):e_{\pm j}\mapsto t_{j}^{\pm 1}e_{\pm j}$ and $d(t,h)X=h(X)$ for $X\in W_p$. These points $d(t,h)$ forms a Levi subgroup of the minimal $\Q_p$-parabolic subgroup of $\sB\subset \sG$ which stabilizes the isotropic flag $\langle e_1,\dots,e_j\rangle_{\Q_p}\,(1\leq j\leq \ell_p)$. Let $$
\fX_p:=(\C/2\pi i (\log p)^{-1}\Z)^{\ell_p}. 
$$
The unramified principal series representation of $\sG(\Q_p)$ is defined to be the normalized parabolic induction ${\rm Ind}_{\sB_p(\Q_p)}^{\sG(\Q_p)}(\chi_\nu)$ $(\nu \in \fX_p)$, where $\chi_{\nu}$ is a character of the Levi subgroup of $\sB(\Q_p)$ defined as $\chi_\nu(d(t;h))=\prod_{j=1}^{\ell_p}|t_j|_p^{\nu_j}$. Let $\pi_{p}^{\sG}(\nu)$ be the unique irreducible $\bK_p$-spherical subquotient of ${\rm Ind}_{\sB_p(\Q_p)}^{\sG(\Q_p)}(\chi_\nu)$. The spherical Fourier transform of $\phi_p \in \cH(\sG(\Q_p)\sslash \bK_p)$ is defined to be the function $\nu \mapsto \hat\phi(\nu)$ on $\fX_p$ such that $\pi_p^{\sG}(\nu)(\phi)$ acts on the $\bK_p$-invariant part by the scalar $\hat\phi(\nu)$. 
%
We have the Satake isomorphism from the Weyl group orbits $\fX_p/W_{\Q_p}^{\sG}$ onto the set ${\rm Hom}_{\C-{\rm alg}}(\cH(\sG(\Q_p)\sslash \bK_p),\C)$ by assigning the $\C$-algebra homomorphism $\lambda^{(\nu)}: \phi \mapsto \hat\phi(\nu)$ of $\cH(\sG(\Q_p)\sslash \bK_p)$ to a point $\nu\in \fX_p/W_{\Q_p}^{\sG}$; the point $\nu$ is called the Satake parameter of $\lambda^{(\nu)}$. Let $\fX_p^{+0}$ be the set of $\nu\in \fX_p$ such that the representation $\pi^{\sG}_p(\nu)$ is unitarizable. 
Since the Hecke algebra $\cH^{+}(\sG(\A_\fin)\sslash \bK_\fin^*)$, which is commutative, acts on the finite dimensional space $\fS_l(\bK_\fin^*)$ by normal operators commuting with the involution $\tau_\infty(\varrho_{\eta_1^{+}})$, there exists an orthonormal basis $\cB_l^{+}$ of $\fS_l(\bK_\fin^*)^{+}$ consisting of joint eigenfunctions of the Hecke operators from $\cH^{+}(\sG(\A_\fin)\sslash \bK_\fin^*)$. For each $F\in \cB_l^{+}$, let $\lambda_F:\cH^{+}(\sG(\A_\fin),\bK_\fin^*)\rightarrow \C$ denote the $\C$-algebra homomorphism such that $$
F*\phi=\lambda_F(\phi)\,F, \quad \phi \in \cH^+(\sG(\A_\fin)\sslash \bK_\fin^*).$$ 
Since the Hecke operator of a real valued $\phi$ is self-adjoint, we have the relation 
$$
\overline{\lambda_{F}(\phi)}=\lambda_{F}(\bar \phi), \quad \phi \in \cH^{+}(\sG(\A_\fin)\sslash\bK_\fin^*). 
$$
Let $F\in \cB^{+}_l$ and $\lambda_F:\cH^{+}(\sG(\A_\fin)\sslash \bK_\fin^*)\rightarrow\C$ the eigencharacter of $F$. There exists a collection of characters $\lambda_{F,p}:\cH(\sG(\Q_p)\sslash\bK_p)\rightarrow \C$ $(p\in \fin)$ such that $\lambda_{F}=\otimes_{p\in \fin}\lambda_{F,p}$. For $p \in S$, let $\nu_{p}(F)=(\nu_j)_{j=1}^{\ell_p} \in \fX_p/W_{\Q_p}^{\sG}$ be the Satake parameter of $\lambda_{F,p}$, in terms of which the local $L$-factor of $\lambda_{F,p}$ is given by \begin{align}
L(s,\lambda_{F,p})=\prod_{j=1}^{\ell_p}(1-p^{-\nu_j-s})^{-1}(1-p^{\nu_j-s})^{-1}\times \begin{cases} 1 &(n_{0,p}=0,1), \\
 (1-p^{-2s})^{-1} &(n_{0,p}=2).
\end{cases}
\label{Def-LocalLfactor}
\end{align} 
For $p\in \fin-S$, the definition of $L(s,\lambda_{F,p})$ is given by Murase-Sugano (\cite[\S1.4]{MS98}). Our assumption $\cL_p=\cL_p^*$ means $(n_{0},\partial)=(1,0),(0,0),(2,0)$ in \cite[(1.18)]{MS98}, in which case their definition \cite[(1.16)]{MS98} agrees with the formula \eqref{Def-LocalLfactor}.\footnote{Erratum: The remark on the last three lines of the paragraph \cite[\S2.1]{Tsud2011-1}, which is not correct, should be replaced with this sentence.} Then the $L$-function of $F\in \cB_l^+$ is defined to be the Euler product
$$
L_\fin(F,s):=\prod_{p\in \fin}L(s,\lambda_{F,p})
$$
By a well-known reasoning, the absolute convergence of the Euler product $L_\fin(F,s)$ on $\Re(s)>\rho+3/2$ is obtained from the square-integrability of $F$. The completed $L$-function is defined by $L(F,s)=\Gamma_{\cL}(l,s)\,L_\fin(F,s)$ with 
\begin{align} 
 \Gamma_\cL(l,s)&=\Gamma_{\C}(s-m/2+l)\prod_{j=1}^{[m/2]}\Gamma_{\C}(s+m/2-j)(2^{\epsilon-1}\fd(\cL))^{s/2}. 
 \label{ZetaEuler-f1}
\end{align}

\subsection{Automorphic forms on the stabilizer}
The center of an orthogonal group consists of the identity matrix and the scalar matrix with $-1$'s on the diagonal. Let $\cnt^{\sG_1}$ (resp. $\cnt^{\sG_1^\xi}$ and $\cnt^{\sG}$) be the nontrivial element of the center of $\sG_1(\Q)$ (resp. $\sG_1^\xi(\Q)$ and $\sG(\Q)$). The image of $\cnt^{\sG_1}$ in $\sG_1(\A)$ is denoted by $\cnt_\fin^{\sG_1}\,\cnt_\infty^{\sG_1}$ with $\cnt_\fin^{\sG_1}=(\cnt_{v}^{\sG_1})_{v\in \fin}$, where $\cnt_{v}^{\sG_1}$ is the non trivial central element of $\sG_1(\Q_v)$. Similarly we write $\cnt^{\sG_1^\xi}=\cnt_\fin^{\sG_1^\xi}\,\cnt_{\infty}^{\sG_1^\xi}$ and $\cnt^{\sG}=\cnt_\fin^{\sG}\,\cnt_{\infty}^{\sG}$ in the adele groups. Let $\rex^\xi\in \sG_1(\Q)$ denote the reflexion of $V_1$ with respect to the vector $\xi\in V_1$, i.e., 
\begin{align}
 \rex^\xi(Y)=Y-2\,Q[\xi]^{-1}\langle Y,\xi \rangle\,\xi, \qquad Y \in V_1.
 \label{rex-def}
\end{align}
The image of $\rex^\xi$ in $\sG_1(\Q_v)$ and that in $\sG_1(\A_\fin)$ are denoted by $\rex_v^\xi$ and $\rex_\fin^\xi$, respectively. We have the relation
\begin{align}
\cnt^{\sG_1}=\cnt^{\sG_1^\xi}\,\rex^\xi=\rex^\xi\,\cnt^{\sG_1^\xi}. 
 \label{cnt-ref}
\end{align}

\begin{lem} \label{JuneL-0}
Let $p$ be a prime number. 
\begin{itemize}
\item[(1)] The element $\rex^\xi_p$ belongs to $\sG_1^\xi(\Q_p)\,\bK_{1,p}^{*}$ if and only if $2\xi \in \cL_{1,p}$. 
\item[(2)] If $p$ is odd, $\xi \in \cL_{1,p}$ and $Q[\xi]\in \Z_p^\times$, then $\rex_p^\xi \in \bK_{1,p}^*$.  
\end{itemize}
\end{lem}
\begin{proof}
(1) From the relation \eqref{cnt-ref}, we have $\rex^\xi_p \in \sG_1^\xi(\Q_p)\,\bK_{1,p}^{*}$ if and only if $\cnt_p^{\sG_1} \in \sG_1^\xi(\Q_p)\,\bK_{1,p}^{*}$. Obviously, $\cnt_p^{\sG_1} \in \bK_{1,p}$; thus, from \cite[Proposition 2.7 (ii)]{MS98}, $\cnt_p^{\sG_1} \in \sG_1^\xi(\Q_p)\,\bK_{1,p}^{*}$ if and only if $\cnt_p^{\sG_1}(\xi)-\xi \in \cL_{1,p}$. Since $\cnt_p^{\sG_1}(\xi)=-\xi$, we are done. 

(2) If $p$ is an odd prime, $\xi\in \cL_{1,p}$ and $Q[\xi]\in \Z_p^\times$, then $\rex_p^\xi(\cL_{1,p})\subset \cL_{1,p}$ follows from \eqref{rex-def}. Hence $\rex_p^\xi \in \bK_{1,p}$. Moreover, $\rex_p^\xi(Y)-Y=Q[\xi]^{-1}\langle Y,\xi\rangle\,(2\xi) \in \cL_{1,p}$ for any $Y\in \cL_{1,p}^*$. This means $\rex_p^\xi \in \bK_{1,p}^*$. \end{proof}

\subsubsection{} 
In this paragraph, we consider a prime $p$ where $2\xi\in \cL_{1,p}$; then, by Lemma~\ref{JuneL-0} (1), we write the element $\rex_p^{\xi}$ as
\begin{align}
 \rex_p^{\xi}=h^{\xi}_p\,k_{p}^{\xi} \qquad (h_p^{\xi}\in \sG_1^\xi(\Q_p),\,k_p^{\xi} \in \bK_{1,p}^{*}). 
 \label{rex-dec}
\end{align}

\begin{lem} \label{JuneL-1}
We have that $h_{p}^{\xi}k_{p}^{\xi}=k_p^\xi h_p^\xi$. Moreover, $h_p^{\xi} \in \bK_{1,p}^{\xi}$ and $(h_p^\xi)^{2}=(k_{p}^\xi)^{-2}\in \bK_{1,p}^{\xi *}$.  
\end{lem}
\begin{proof} Since $h_p^{\xi}\in \sG_1^{\xi}(\Q_p)$, we have $\cnt_p^{\sG_1^{\xi}}h_{p}^{\xi}=h_p^{\xi}\cnt^{\sG_1^\xi}$. This relation combined with \eqref{cnt-ref} and \eqref{rex-dec} yields $\cnt_p^{\sG_1^\xi} k_p^{\xi}=k_p^\xi\cnt_p^{\cG_1^\xi}$. Since $\cnt_p^{\sG_1}$ is commutative with $k_p^{\xi}\in \sG_1(\Q_p)$, we have $\rex_p^{\xi}\cnt_p^{\sG_1^\xi}k_p^\xi=k_p^\xi \rex_p^\xi \cnt^{\sG_1^\xi}$ from \eqref{cnt-ref}, which becomes $h_p^\xi k_p^\xi \cnt^{\sG_1^\xi}k_p^\xi=k_p^\xi h_p^\xi k_p^\xi \cnt^{\sG_1^\xi}$ by \eqref{rex-dec}. Applying $\cnt_p^{\sG_1^\xi} k_p^{\xi}=k_p^\xi\cnt_p^{\sG_1^\xi}$ to the right-hand side of the last equality, we obtain $ h_p^\xi k_p^\xi \cnt^{\sG_1^\xi}k_p^\xi=k_p^\xi h_p^\xi \cnt^{\sG_1^\xi} k_p^\xi$, or equivalently $h_p^\xi k_p^\xi=k_p^\xi h_p^\xi$ as desired.  

Let $Y\in \cL_{1,p}^{\xi}$. Since $\rex_p^\xi(Y)=Y$ and $k_p^\xi \in \bK_{1,p}$, we have $(h_p^{\xi})^{-1}(Y)=k_p^{\xi}(\rex_p^{\xi}(Y))=k_p^\xi(Y) \in \cL_{1,p}$ on the one hand. On the other hand, from $h_p^\xi\in \sG_1^\xi(\Q_p)$, we also have $(h_p^{\xi})^{-1}(Y) \in V_{1,p}^{\xi}$. Hence, $(h_p^\xi)^{-1}(Y) \in V_{1,p}^\xi\cap \cL_{1,p}=\cL_{1,p}^\xi$ for any $Y\in \cL_{1,p}^\xi$, or equivalently $(h_{p}^\xi)^{-1}(\cL_{1,p}^\xi)\subset \cL_{1,p}^\xi$. Thus, $h_{p}^\xi \in \bK_{1,p}^\xi$. 

Since $k_p^\xi(\xi)=(h_p^\xi)^{-1}(\rex_p^\xi(\xi))=h_p^\xi(-\xi)=-\xi$, we have$(k_p^\xi)^2(\xi)=\xi$, or equivalently $(k_p^\xi)^2\in \sG_1^\xi(\Q_p)$. Thus, $(k_p^\xi)^2\in \sG_1^\xi(\Q_p)\cap \bK_{1,p}^*=\bK_{1,p}^{\xi *}$ by \cite[Proposition 2.3]{MS98}. The shows the second assertion. Since $\rex_p^\xi$ is an involution and since $h_p^\xi$ and $k_p^\xi$ are commutative with each other as shown above, the square of \eqref{rex-dec} becomes ${\rm Id}=(h_p^\xi)^2(k_p^\xi)^2$. Hence $(h_p^{\xi})^2=(k_p^\xi)^{-2}\in \bK_{1,p}^{\xi *}$. 
\end{proof}

\subsubsection{} \label{subsubsec: xiEigenFtn}
Let $\cV(\xi)$ be the space of all those smooth functions $f:\sG_1^\xi(\A)\rightarrow \C$ such that $f(\delta h u_\infty)=f(h)$ $(\delta\in \sG_1^\xi(\Q),\,h\in \sG_1^{\xi}(\A),\,u_\infty \in \sG_1^\xi(\R))$, endowed with the inner product on the  space $\sG_1^\xi(\Q)\bsl \sG_1^\xi(\A)$ with the quotient measure. We view $\cV(\xi)$ as a $\sG_1^\xi(\A_\fin)$-module by the right-translations. Let $\cV(\xi;\bK_{1,\fin}^{\xi*})$ be the $\bK_{1,\fin}^{\xi*}$-fixed vectors in $\cV(\xi)$. The induced inner product on $\cV(\xi;\bK_{1,\fin}^{\xi*})$ is given by \eqref{InnerProdStab}. We have a well-defined action $\tau$ of the finite group $E_{\fin}(\xi)$ (see \S\ref{sec: OCSG}) on the space $\cV(\xi;\bK_{1,\fin}^{\xi*})$ induced by the right-translation., i.e., $[\tau(\dot u)\,f](h)=f(h u)$ for $f\in \cV(\xi;\bK_{1,\fin}^{\xi*})$ and $\dot u\in E_\fin(\xi)$ represented by $u\in \bK_{1,\fin}^{\xi}$. This action of $E_\fin(\xi)$ is also defined through the action of the Hecke algebra $\cH(\sG_1^\xi(\A_\fin)\sslash \bK_{1,\fin}^{\xi*})$ as $\tau(\dot u)\,f=f*\check \phi_{u}$ where $\phi_{u} \in \cH(\sG_1^\xi(\A_\fin)\sslash \bK_{1,\fin}^{\xi*})$ is the characteristic function of the coset $u\bK_{1,\fin}^{\xi*}=\bK_{1,\fin}u\bK_{1,\fin}^{\xi*}$. Hence the operators $\tau(u)\,(u \in E_\fin(\xi))$ commutes with the Hecke operators from $\cH^{+}(\sG_1^\xi(\A_\fin)\sslash \bK_{1,\fin}^{\xi *})$ which is the center of $\cH(\sG_1^\xi(\A_\fin)\sslash \bK_{1,\fin}^{\xi *})$.  

For each prime number $p$ such that $2\xi \in \cL_{1,p}$, let $h_p^{\xi} \in \sG_1^\xi(\Q_p)$ be an element which fits in the formula \eqref{cnt-ref}; we have that $h_p^\xi \in \bK_{1,p}^\xi$ yields a $2$-torsion element of $E_p(\xi)=\bK_{1,p}^\xi/\bK_{1,p}^{\xi *}$ by Lemma~\ref{JuneL-1}. We set $h_\fin^{\xi}=(h_p^\xi)_{p\in \fin}$ with $h_{p}^\xi=1$ at a prime $p$ where $2\xi\not\in \cL_{1,p}$, and $\dot h^\xi\in E_\fin(\xi)$ the coset of $h_\fin^\xi$. Thus we have an involutive operator $\tau_\fin^\xi:=\tau(\dot h_\fin^\xi)$ on $\cV(\xi;\bK_{1,\fin}^{\xi *})$ commuting with all the Hecke operators $\cH^{+}(\sG_1^\xi(\A_\fin)\sslash \bK_{1,\fin}^{\xi *})$. 

\subsubsection{} \label{subsubsec: xiEigenFtn-2}
For materials in this and the next paragraphs, we refer to \S~\ref{sec:LocalTheory}. Let $p$ be a prime number. Let us fix a Witt decomposition of $\cL_{1,p}^{\xi}$ as in \S\ref{sec:LocalTheory}; it determines an orthogonal decomposition of $V_1^\xi(\Q_p)$ to $\ell_p(\xi)$ copies of hyperbolic plane $\Q_p v_j+\Q_p v_{-j}$ $(1\leq j \leq \ell_p(\xi))$ with isotropic vectors $v_{\pm j}$ such that $Q(v_j,v_{-j})=1$ and a maximal $\Q_p$-anisotropic subspace $W_p(\xi) \subset V_1^\xi(\Q_p)$ such that $\cL_{1,p}^{\xi}\cap W_p(\xi)=\{X\in W_p(\xi)|2^{-1} Q[X]\in \Z_p\}$, where $\ell_p(\xi)$ is the Witt index of $V_1^\xi(\Q_p)$. For each $a\in (\Q_p^\times)^{\ell_p(\xi)}$ and $u \in {\rm O}(W_p(\xi))$, define $d(a;u)\in \sG_1^\xi(\Q_p)$ as the linear endomorphism such that $v_{\pm j}\mapsto a_j^{\pm 1}v_{\pm j}$ $(1\leq j \leq \ell_p(\xi))$ and $X\mapsto u(X)$ for $X\in W_p(\xi)$. Set
$$
\fX_p(\xi):=(\C/2\pi i (\log p)^{-1}\Z)^{\ell_p(\xi)}\times \widehat{E_p(\xi)}, $$
where $E_p(\xi)=\bK_{1,p}^{\xi}/\bK_{1,p}^{\xi*}$, which is known to be isomorphic to ${\rm O}(W_p(\xi))/{\rm O}(W_p(\xi))\cap \bK_{1,p}^{\xi*}$. The restricted Weyl group $W_{\Q_p}^{\xi}$ of $\sG_1^{\xi}(\Q_p)$ acts on $\fX_p(\xi)$ by permutations of the $\C/2\pi i (\log p)^{-1}\Z$-factors. For $(z;\rho)\in \fX_p(\xi)$, let ${\rm I}(z;\rho)$ be the minimal principal series representation $\sG_1^{\xi}(\Q_p)$ induced from the representation $d(a;u)\mapsto \rho(u)\,\prod_{j=1}^{\ell_p(\xi)}|a_j|_p^{z_j}$ of the Levi subgroup of the minimal parabolic subgroup stabilizing the isotropic flag $\langle v_1,\dots,v_j\rangle_{\Q_p}$ $(1\leq j\leq \ell_p(\xi))$. For each $(z;\rho)\in \fX_p(\xi)$ with $z\in (\C/2\pi i (\log p)^{-1}\Z)^{\ell_p(\xi)}$ and $\rho \in \widehat{E_p(\xi)}$, we define $\pi_p^{\sG^\xi_1}(z;\rho)$ to be the unique irreducible subquotient $\pi$ of ${\rm I}(z;\rho)$ such that the natural representation of $E_p(\xi)$ on $\pi^{\bK_{1,p}^{\xi*}}$ is isomorphic to $\rho$; the equivalence class of $\pi_p^{\sG_1^{\xi}}(z;\rho)$ depends only on the $W_{\Q_p}^\xi$-orbit of $(z;\rho)$. 

\subsubsection{} \label{subsubsec: cU-def}
Let $\cU$ be an irreducible smooth $\sG_1^\xi(\A_\fin)$-submodule of $\cV(\xi)$ and $\cU(\bK_{1,\fin}^{\xi *})$ the space of $\bK_{1,\fin}^{\xi*}$-fixed vectors in $\cU$. There corresponds a unique family $(z_{p}^{\cU},\rho_p^{\cU}) \in \fX_p(\xi)/W_{\Q_p}^\xi$ $(p\in \fin)$ such that $\cU$, as a representation of $\sG_1^\fin(\A_\fin)$, is isomorphic to the restricted tensor product $\bigotimes_{p\in \fin} \pi_p^{\sG_1^\xi}(z_p^\cU;\rho_p^\cU)$. The space $\cU(\bK_{1,\fin}^{\xi *})$ is an irreducible $E_{\fin}(\xi)$-subspace which is isomorphic to $\rho^\cU:=\bigotimes_{p\in \fin}\rho_p^\cU$ as representations of $E_\fin(\xi)$ (Proposition~\ref{AutoRep-P}). For each $p\in \fin$, let $C_{p}^{\cU}$ be the character of $\cH^{+}(\sG_1^\xi(\Q_p)\sslash \bK_{1,p}^{\xi*})$ with the Satake parameter $(z_p^{\cU},\rho_{p}^{\cU})$. Then $\cU(\bK_{1,\fin}^{\xi*})$ consists of $\cH^{+}(\sG_1^\xi(\A_\fin)\sslash \bK_{1,\fin}^{\xi*})$-eigenfunctions with eigencharacter $\otimes_{p\in \fin} C_p^\cU$. The standard $L$-function of $\cU$ is defined as the Euler product $L_\fin(\cU,s):=\prod_{p\in \fin}L(C_p^{\cU};s)$, where $L(C_p^{\cU};s)$ is the local factor given by \cite[(1.16)]{MS98}. 

We can take an orthogonal basis $\cB(\cU;\bK_{1,\fin}^{\xi*})$ of $\cU(\bK_{1,\fin}^{\xi*})$ consisting of eigenfunctions of the involutive  operator $\tau_\fin^\xi$; let $\epsilon_f \in \{\pm 1\}$ be the eigenvalue of $\tau_\fin^{\xi}$ at $f\in \cB(\cU;\bK_{1,\fin}^{\xi*})$. Define $\cB(\cU;\bK_{1,\fin}^{\xi*})^{\pm 1}:=\{f\in \cB(\cU;\bK_{1,\fin}^{\xi*})|\,\epsilon_f=\pm 1\,\}$ and $d^{\pm}(\cU)=\#\cB(\cU;\bK_{1,\fin}^{\xi*})^{\pm 1}$.  

\begin{lem}\label{cUtrace-L} 
It holds that
\begin{align*}
\dim(\cU(\bK_{1,\fin}^{\xi*}))^{-1}\, \tr (\tau_{\fin}^{\xi}|\,\cU(\bK_{1,\fin}^{\xi*}))=\frac{d^{+}(\cU)-d^{-}(\cU)}{d^{+}(\cU)+d^{-}(\cU)}\in \{-1,0,1\}.
\end{align*}
We have either $d^{+}(\cU)=0$, $d^{-}(\cU)=0$, or $d^{+}(\cU)=d^{-}(\cU)$; the last case does not happen if $E_{\fin}(\xi)$ is abelian. 
\end{lem}
\begin{proof} Since $\cU(\bK_{1,\fin}^{\xi*})$ viewed as a representation of $E_\fin(\xi)$ is isomorphic to $\bigotimes_{p\in \fin} \rho_{p}^{\cU}$, we have
$$\dim(\cU(\bK_{1,\fin}^{\xi*}))^{-1}\,\tr (\tau_{\fin}^{\xi}|\,\cU(\bK_{1,\fin}^{\xi*}))
=\prod_{p \in \fin}\dim(\rho_p^{\cU})^{-1}\,\tr(\rho_p^{\cU}(h_p^{\xi})).
$$
Note that $E_p(\xi)=\{1\}$ and $\rho_p^{\cU}$ is trivial for almost all $p$. As recalled in Lemma~\ref{MS98L-1}, $E_p(\xi)$ is isomorphic to $\{1\}$, $\Z/2\Z$ or $D_{2(p+1)}$. If $\rho_p^{\cU}$ is one dimensional, it is evident that $\tr \rho_p^{\cU}$ takes values in $\{-1,+1\}$; this is the case if $E_p(\xi)$ is isomorphic to $\{1\}$ or $\Z/2\Z$. Suppose $E_{p}\cong D_{2(p+1)}$ and $\rho_p^{\cU}$ is two dimensional. If we fix generators $a,b$ of $E_p(\xi)$ such that $b^2=a^{p+1}=1$ and $bab^{-1}=a^{-1}$, then as recalled in Lemma~\ref{ResSO-L}, there exists $1\leq \nu \leq p,\nu\not=\frac{p+1}{2}$ such that $\tr\rho_{p}^{\cU}(a^{j}b)=0$ and $\tr\rho_p^\cU(a^{j})=2\cos(\frac{2\pi j \nu}{p+1})$ for $1\leq j \leq p+1$. Since $\dot h_p^\xi$ is $2$-torsion, it equals either $a^{\frac{p+1}{2}}$, in which case $\tr\rho_p^\cU(\dot h_p^\xi)=2(-1)^{\nu}$, or one of $a^{j}b$'s, in which case $\tr\rho_p^\cU(\dot h_p^\xi)=0$.   
\end{proof}


\subsection{Signature conditions}
Let $\kappa_\infty,\,\varrho_{\eta_1^{+}} \in \sG(\R)$ be as before. We have $\tilde \bK_\infty^{+}=\{1,\kappa_\infty\}\,\bK_\infty$. For any $u \in \tilde \bK_\infty^{+}-\bK_\infty$, the automorphisms ${\rm{Ad}}(u)$ preserves $\bK_\infty$ and $J({\rm{Ad}}(u)\,k,\fz_0)=J(k,\fz_0)$. Thus, $\gamma_\infty$ defines an involutive linear operator $\tau_\infty(u)$ on the space $\fS_l(\bK_{\fin}^*)$ (\cite[\S 3]{Tsud2011-1}) by 
$$
[\tau_{\infty}(u)(F)](g)=F(g u), \quad F\in \fS_l(\bK_\fin^{*}), \,g\in \sG(\A).$$ 
We have two operators $\tau_{\infty}(\kappa_\infty)$ and $\tau_\infty(\varrho_{\eta_1^{+}})$ on $\fS_l(\bK_\fin^*)$ commuting with each other.  

\begin{lem} \label{SignatureLemm}
Let $F\in \fS_l(\bK_\fin^*)$.
\begin{itemize}
\item[(i)] We have $\tau_{\infty}(\kappa_\infty)\tau_{\infty}(\varrho_{\eta_1^{+}})F=(-1)^{l}F$. 
\item[(ii)] Suppose $2\xi\in \cL$. Set $F_1=\tau_\infty(\kappa_\infty)F$. Let $f\in \cV(\xi;\bK_{1,\fin}^{\xi *})$ and set $f_1=\tau_\fin^\xi(f)$. Then, we have $a_{F}^{f_1}(\xi)=a_{F_1}^{f}(\xi)$.
\end{itemize}
\end{lem}
\begin{proof}
(i) This follows from Lemma~\ref{kappa-eta_1+} immediately. 

(ii) From $2\xi\in \cL$, we have the containment $h_\fin^\xi\in \rex_\fin^\xi \bK_\fin^{*}$. Let $r>0$ and $h_{0,\fin}\in \sG_1^\xi(\A_\fin)$ and consider $a_F(h_{0,\fin}h_{\fin}^\xi;\,\xi)\,\cW_l^\xi(\sm(r;\,1)\,b_\infty)$, which equals
\begin{align}
&\int_{V_1(\Q)\bsl V_{1}(\A)}
F(\sn(X)\sm(r;\,h_{0,\fin}\,h_\fin^\xi)\,b_\infty)\,\psi(-\langle \xi,X\rangle)\,\d X.
 \label{JuneL-3-1}
\end{align}
Since $F$ is right $\bK_{\fin}^*$-invariant and left $\sG(\Q)$-invariant and since $h_\fin^\xi\in \rex_{\fin}^\xi\,\bK_{\fin}^{*}$ and $\sm(-1;\cnt^{\sG_1^\xi})\in \sG(\Q)$, \eqref{JuneL-3-1} becomes 
\begin{align}
\int_{V_1(\Q)\bsl V_{1}(\A)}
F(\sm(-1;\cnt^{\sG_1^\xi})\,\sn(X)\,\sm(r;\,h_{0,\fin}\,\rex_\fin^\xi)\,b_\infty)\,\psi(-\langle \xi,X\rangle)\,\d X
 \label{JuneL-3-2}
\end{align}
by the Fourier expansion (\cite[(3.3)]{Tsud2011-1}). By the variable change $\sm(-1;\cnt^{\sG_1^\xi})\,\sn(X)=\sn(X')\,\sm(-1;\cnt^{\sG_1^\xi})$ and then by the relations $\cnt^{\sG_1^\xi}h_{0,\fin}=h_{0,\fin}\cnt^{\sG_1^\xi}$, $\cnt^{\sG}=\sm(-1;\,\cnt^{\sG_1^\xi}\,\rex^{\xi})$, we obtain the equality of the integral \eqref{JuneL-3-2} and the following one. 
\begin{align}
&\int_{V_1(\Q)\bsl V_{1}(\A)}
F(\sn(X)\,\sm(1;h_{0,\fin} )\,\cnt^{\sG}\,\sm(r;\,\rex_\infty^{\xi})\,b_\infty)\,\psi(\langle \xi,X \rangle)\,\d X.
\label{JuneL-3-3}
 \end{align}
Since $\cnt^{\sG}$ belongs to the center of $\sG(\Q)$, it can be moved to the position left to $\sn(X)$ in the argument of $F$; then, by the $\sG(\Q)$-invariance of $F$, we delete $\cnt^{\sG}$ and put $\sm(-1;\,1)$ instead. In this way, the integral \eqref{JuneL-3-3} becomes 
\begin{align*}
\int_{V_1(\Q)\bsl V_{1}(\A)}
F(\sm(-1;\,1)\,\sn(X)\,\sm(1;h_{0,\fin} )\,\sm(r;\,\rex_\infty^\xi)\,b_\infty)\,\psi(\langle \xi,X \rangle)\,\d X.
\end{align*}
By the variable change $\sm(-1;\,1)\,\sn(X)=\sn(X')\,\sm(-1;\,1)$ with $X'=-X$, this becomes
\begin{align}
\int_{V_1(\Q)\bsl V_{1}(\A)}
F(\sn(X)\,\sm(1;h_{0,\fin} )\,\sm(r\,(-1)_\infty;\,\rex_\infty^\xi)\,b_\infty\,\sm((-1)_\fin;\,1))\,\psi(-\langle \xi,X \rangle)\,\d X. 
 \label{JuneL-3-4}
\end{align}
Since $F$ is right $\bK_{\fin}^*$-invariant, $\sm((-1)_\fin;1)\in \bK_{\fin}^*$ is deleted from the argument of $F$. Thus, we finally arrive at the expression of \eqref{JuneL-3-1}  
\begin{align}
\int_{V_1(\Q) \bsl V_{1}(\A)}
F(\sn(X)\,\sm(r;h_{0,\fin} )\,b_\infty\,\kappa_\infty )\,\psi(-\langle \xi,X \rangle)\,\d X
\label{JuneL-3-5}
\end{align}
by using the relation $\kappa_\infty=b_\infty^{-1}\,\sm((-1)_\infty;\,\rex_\infty^\xi)\,b_\infty$. Therefore, \eqref{JuneL-3-5} becomes \\ $a_{F_1}(h_{0,\fin};\,\xi)\,\cW_l^\xi(\sm(r;\,1)\,b_\infty)$. Hence, $a_F(h_{0,\fin}\,h_\fin^\xi;\,\xi)=a_{F_1}(h_{0,\fin};\,\xi)$. Multiplying this with $f(h_0)$ and then taking the integral over $h_0 \in \sG_1^\xi(\Q)\bsl \sG_1^\xi(\A)$, we obtain the desired relation.  
\end{proof}
Recall the involution $\tau_{\infty}(\varrho_{\eta_1^{+}})$ on the finite dimensional Hilbert space $\fS_l(\bK_\fin^*)$. Since it is evidently self-adjoint, the space $\fS_l(\bK_\fin^*)$ is decomposed to an orthogonal direct sum of eigenspaces $\fS_{l}(\bK_\fin^*)^{\pm}$ of $\tau_{\infty}(\varrho_{\eta_1^+})$ with eigenvalues $\pm 1$. Since $\tilde \bK_\infty^{+}=\bK_\infty\cup \varrho_{\eta_1^{+}}\,\bK_\infty$, the space $\fS_l(\bK_\infty^*)^{+}$ coincides with the space of $F\in \fS_{l}(\bK^{*}_\fin)$ such that 
$$
F(gk_\infty)=J(k_\infty,\fz_0)^{-l}F(g), \quad g\in \sG(\A),\,k_\infty\in \tilde \bK_\infty^{+}.
$$ 
For $F\in \fS_l(\bK_\fin^*)^{+}$, define a function ${\rm F}:\cD\times \sG(\A_\fin)\rightarrow \C$ by setting ${\rm F}(\fz,g_\fin)=J(g_\infty,\fz_0)^{l}F(g_\infty g_\fin)$ with $g_\infty\in \sG(\R)^+$ such that $g_\infty\langle \fz_0 \rangle=\fz$. Then the function ${\rm F}$ becomes a holomorphic cusp form of weight $l$ in the sense of \S~\ref{Intro} and the mapping $\fS_l(\bK_\fin^*)^{+}\ni F \rightarrow {\rm F}\in S_l(\bK_\fin^*)$ is a linear isomorphism preserving the Hecke actions.

\subsection{Whittaker-Bessel coefficients} \label{ApprFE-1}
Let $F\in \fS_l(\bK_\fin^*)^{+}$ and $f\in \cV(\xi;\bK_{1,\fin}^{\xi*})$. The $(f,\xi)$-Whittaker-Bessel function associated to $F$ is defined as 
$$
\varphi_{F}^{\xi,f}(g)= \int_{\sG_1^\xi(\Q)\bsl \sG_1^\xi(\A)} f(h_0)\d h_0
\int_{\sN(\Q)\bsl \sN(\A)} F(\sm(1;h_0)ngb_\infty)\psi_{\xi}(n)^{-1}\d n, \quad g\in \sG(\A), 
$$
where $\psi_{\xi}$ is an automorphic character of $\sN(\A)$ defined by
\begin{align}
\psi_{\xi}(\sn(X))=\psi(\langle X,\xi\rangle), \quad \sn(X)\in \sN(\A).
 \label{AutCharsN}
\end{align}
Let $\tilde \fW_l(\xi)$ be the space of $C^\infty$-functions $W:\sG(\R)^{+}\rightarrow \C$ such that $R(\fp^-)W(g)=0$ and $W(ngk)=\psi_\xi(n)J(k,\fz_0)^{-l}W(g)$ for all $(n,g,k)\in \sN(\R)\times \sG(\R)^+\times \tilde \bK_\infty^{+}$. From \cite[Proposition 3]{Tsud2011-1}\footnote{In \cite[Proposition 3]{Tsud2011-1}, we consider a similar space of functions $\fW_l(\eta)$ on $G$. Since $\sG(\R)^+/\tilde\bK_\infty^+=G/\bK_\infty$, the restriction map $W\mapsto W|G$ is bijective.}, the $\C$-vector space $\tilde \fW_l(\xi)$ is generated by the function 
\begin{align}
\cW_l^\xi(g_\infty)=J(g_\infty,\fz_0)^{-l}\exp(2\pi i \langle \xi,g_\infty\langle \fz_0\rangle\rangle), \quad g_\infty \in \sG(\R)^{+}. 
 \label{ArchWhittaker}
\end{align}
By $J(h,\fz)=1$ for all $h \in \sG_1(\R)^{+}$ and $\fz \in \cD$, we can easily confirm that $\cW_l^\xi$ has the left $\sG_1^\xi(\R)$-invariance. From \eqref{ArchWhittaker}, 
\begin{align}
\cW_l^{\xi}(\sm(r;1_m)b_\infty)=r^{-l}\exp(-\sqrt{8\Delta}\pi r), r>0. \label{nxiWhittVal}
\end{align}
Since $g_\infty\mapsto \varphi_{F}^{\xi,f}(g_\infty b_\infty^{-1})$ belongs to the space $\tilde \fW_l(\xi)$, there exists a unique constant $a_{F}^{f}(\xi)\in \C$ (the $(\xi,f)$ Whittaker-Fourier coefficient) such that 
\begin{align}
\varphi_F^{\xi,f}(g_\infty)=a_{F}^{f}(\xi)\,\cW_l^{\xi}(g_\infty b_\infty), \quad g_\infty\in \sG(\R)^+.
 \label{WHBesCoeff}
\end{align}

\subsection{Rankin-Selberg integral} \label{sec-RankinSelbergInt}
Let $\cU$ be as in \S\ref{subsubsec: cU-def} such that $\cU(\bK_{1,\fin}^{\xi*})\not=\{0\}$. Recall the Eisenstein series associated with $f\in \cU(\bK_{1,\fin}^{\xi*})$ is defined as a meromorphic continuation of a holomorphic function on $\Re s>\rho:=\frac{m-1}{2}$ given by the absolutely convergent series 
$$
E(f,s;h)=\sum_{\gamma\in \sP^\xi(\Q)\bsl \sG^\xi(\Q)}\sf^{(s)}(\gamma h), \quad h\in \sG^\xi(\A), 
$$
where $\sf^{(s)}$ is a function on $\sG^\xi(\A)=\sP^\xi(\A)\bK_\fin^{\xi*}\bK_\infty^{\xi}$ defined by the formula $\sf^{(s)}(\sm(t;h_0)nk)=|t|_\A^{s+\rho}f(h_0)$ $(t\in \A^\times,\,h_0\in \sG_1^\xi(\A),n\in \sN^\xi(\A),\,k\in \bK_\fin^{\xi*}\bK_\infty^\xi$.) Let $E^{*}(f,s;h)=L^{*}(\cU,-s)E(f,s;h)$ be the normalized Eisenstein series (\cite[\S 3.6]{Tsud2011-1}), where  
\begin{align*}
&L^*(\cU,s)=L(\cU,s)\,\hat\zeta(2s)^{1-\epsilon}, \qquad 
D_*(s)=\prod_{\substack{0\leq j\leq m-1 \\ j\not=\rho}}(s-\rho+j)
\end{align*}
with $\epsilon\in \{0,1\}$ the parity of $m$ (\cite{MS98}); we collected the analytic properties of the $L$-function $L(\cU,s)$ and the Eisenstein series in \S~\ref{App2-Lftn} and \S~\ref{App2-Eis}. Then for $f\in \cU(\bK_{1,\fin}^{\xi*})$ and $F\in \fS_l(\bK_\fin^{*})$, the Rankin-Selberg integral $Z_{F}^{f*}(s)$ is defined as the integral of $E^{*}(f,s;h)\,F(h b_\infty)$ over $\sG^\xi(\Q)\bsl \sG^\xi(\A)$ (\cite[\S 3.8.2]{Tsud2011-1}). \footnote{Proposition 17 in \cite{Tsud2011-1} is erroneously stated; for correction, we have to modify the statement as follows. After the first sentence of \cite[Proposition 17]{Tsud2011-1}, we add ''Suppose $F$ is an eigenfunction of the involution $\tau_\infty(\rho_{\eta_1^{+}})$ with eigenvalue $+1$." Then the statement of \cite[Proposition 17]{Tsud2011-1} should be extended by Lemma~\ref{SignF-Lem}.} 
\begin{lem} \label{SignF-Lem}
Let $F\in \fS_l(\bK_\fin^{*})$ be an eigenfunction of the involution $\tau_\infty(\rho_{\eta_1^{+}})$ with eigenvalue $-1$. Then $Z_{F}^{f*}(s)$ is identically $0$.   
\end{lem}
\begin{proof} Let $\tilde \bK^{\xi}_\infty$ be the maximal compact subgroup of $\sG^\xi(\R)$. Since $\sP^\xi(\R)\bK_\infty^\xi=\sP^\xi(\R)\tilde \bK_\infty^{\xi}$, we see that $\sf^{(s)}$ is right $\tilde \bK^{\xi}_\infty$-invariant. Let $\e_\infty(F) \in \{\pm 1\}$ be the eigenvalue of $\tau_{\infty}(\varrho_{\eta_1^{+}})$ at $F$. Noting that the element $b_\infty \varrho_{\eta_{1}^{+}} b_\infty^{-1}$ belongs to $\tilde \bK_\infty^{\xi}$, we have
\begin{align*}
Z_{F}^{f*}(s)=\int_{\sG^\xi(\Q)\bsl \sG^\xi(\A)}E^{*}(f,s;hb_\infty \varrho_{\eta_{1}^{+}} b_\infty^{-1})\,F(h b_\infty)\,\d h&=\int_{\sG^\xi(\Q)\bsl \sG^\xi(\A)}E^{*}(f,s;h)\,F(h b_\infty \varrho_{\eta_{1}^{+}})\,\d h
\\
&=\e_\infty(F)\,Z_{F}^{f*}(s).
\end{align*}
From this, we have the conclusion. 
\end{proof}

\begin{cor} \label{SignF-Lem-Cor}
 Let $F\in \fS_l(\bK_\fin^*)$ be an eigenfunction of the involution $\tau_\infty(\kappa_\infty)$. Then 
\begin{itemize}
\item[(i)] $Z_{F}^{f*}(s)$ is identically zero unless $\tau_\infty(\kappa_\infty)F=(-1)^{l}F$.
\item[(ii)] Suppose $2\xi \in \cL$ and $\tau_\infty(\kappa_\infty)F=(-1)^{l}F$. Let $f\in \cV(\xi;\bK_{1,\fin}^{\xi*})$ be an eigenfunction of the involution $\tau_\fin^{\xi}$ with eigenvalue $\epsilon_f\in \{\pm 1\}$. Then $a_{F}^{f}(\xi)=0$ unless $\epsilon_f=(-1)^{l}$. 
\item[(iii)] If $\tau_\infty(\kappa_\infty)F=(-1)^l F$, then $F(g\cnt^{\sG}_\infty)=(-1)^{l}F(g)$ for all $g\in \sG(\A)$. 
\end{itemize} 
\end{cor}
\begin{proof}
The first two assertions follow from  Lemmas~\ref{SignF-Lem} and \ref{SignatureLemm}. To confirm the last assertion, we separate cases. If $m$ is even, then the element $\cnt_\infty^{\sG}$ belongs to $\bK_\infty$ and $J(\cnt_\infty^{G},\fz_0)=-1$. Thus, from the automorphy condition of $F$, we always have $F(g\cnt_\infty^{\sG})=(-1)^{l}F(g)$. If $m$ is odd, then $\cnt_\infty^{\sG}\in \tilde \bK_\infty^{+}-\bK_\infty$. Thus $\cnt_\infty^{\sG} \kappa_\infty \in \bK_\infty$ and $J(\cnt_\infty^{\sG}\kappa_\infty,\fz_0)=+1$; thus $F(g\cnt_\infty^{\sG}\kappa_\infty)=F(g)$. \end{proof}

Recall the basic identity, which relates the Rankin-Selberg integral $Z_{F}^{f*}(s)$ to the adelic Mellin transform of $\varphi_{F}^{\xi,f}(g)$ (\cite[Theorem 2]{Sugano85}, \cite{MS}):
\begin{align}
Z_{F}^{f*}(s)=L^*(f,-s)2^{-\rho}\Delta^{1/2}\fd(\cL_1)^{1/2}\,
\int_{\A^\times}\varphi_{F}^{\xi,f}(\sm(t;1_{m}))|t|_\A^{s-\rho}\,\d^\times t, \quad \Re(s)\gg 0.
 \label{BasicIdentity}
\end{align}
The integral on the right-hand side is shown to be absolutely convergent on a half-plane $\Re(s)\gg 0$ and explicitly computed in terms of the standard $L$-function of $F$. Indeed, as we recalled in \cite[Proposition 17]{Tsud2011-1}, 
\begin{align}
Z_{F}^{f*}(s)=C_{l}^\xi\,a_{F}^{f}(\xi)\,L(F,s+1/2), \quad \Re(s)\gg 0
 \label{ZetaEuler}
\end{align}
with $L(F,s)=\Gamma_{\cL}(l,s)L_\fin(F,s)$ the completed standard $L$-function of $F$, 
\begin{align}
C_{l}^\xi&=2^{\frac{1}{2}(-l-\frac{m}{2}+1)}|Q[\xi]|^{\frac{1}{2}(-l+\frac{m}{2}+\frac{3}{2})}\fd(\cL_1)^{\frac{1}{2}}\,(2^{\epsilon-1}\fd(\cL))^{\frac{1}{4}},
 \label{ZetaEuler-f2}
\end{align}
where $\epsilon\in \{0,1\}$ is the parity of $m$.

\begin{lem} \label{AppFE} Let $F\in \cB_l^{+}$ with $a_{F}^{f}(\xi)\not=0$. Then the function $D_{*}(s)L_\fin(F,s+1/2)$ is entire on $\C$ of order $1$. We have the symmetric functional equation $L(F,s)=L(F,1-s)$. For any compact interval $I\subset \R$, there exists a constant $\kappa>0$ such that  
$$
|D_*(s)\,L_\fin(F,s+1/2)|\ll (1+|s|)^{\kappa}, \quad s\in \cT_{0,I}. 
$$
\end{lem}
\begin{proof} The first two assertions follow from \eqref{ZetaEuler} combined with the holomorphy and the invariance under $s\mapsto -s$ of $D_*(s)E^{*}(f,s;h)$ (see Proposition~\ref{MuraseSugano2}) and Theorem~\ref{T2} (1). Then a standard argument by Phragmen-Lindel\"{o}f convexity principle shows the second claim (\cite[Lemma 5.2]{IwKow}). 
\end{proof}

\section{Double coset decompositions} \label{DouCosDec}

\subsection{The structure of $\sP^\xi(\Q)\bsl \sG(\Q)/\sP(\Q)$}
Let $\XX$ be the set of $\Q$-points $[v]\,(v\in V-\{0\})$ of the projective space $\PP(V)$ such that $Q[v]=0$. Then, by Witt's theorem, the group $\sG(\Q)$ acts on $\XX$ transitively; the stabilizer of the point $[\e_1]\in \XX$ coincides with $\sP(\Q)$. The set $\XX$ is a disjoint union of the following four subsets: 
\begin{align*}
\XX_{00}&=\{[v]\in \XX|,\,\langle v,\xi\rangle=0,\,\langle v,\e_1\rangle=0\}, 
\quad \XX_{01}=\{[v]\in \XX|,\,\langle v,\xi\rangle=0,\,\langle v,\e_1\rangle\not=0\}, \\
\XX_{10}&=\{[v]\in \XX|,\,\langle v,\xi\rangle\not=0,\,\langle v,\e_1\rangle=0\},\quad 
\XX_{11}=\{[v]\in \XX|,\,\langle v,\xi\rangle\not=0,\,\langle v,\e_1\rangle\not=0\}.
\end{align*}
Set $\tilde \XX_{jj'}=\{g\in \sG(\Q)|\,g[\e_1]\in \XX_{jj'}\}$ for $j,j'\in \{0,1\}$. Then $\tilde \XX_{jj'}\,(j,j'\in \{0,1\})$ is a partition of $\sG(\Q)$ by $(\sP^{\xi}(\Q),\sP(\Q))$-invariant subsets of $\sG(\Q)$.   

\begin{prop} \label{DC}
We have
\begin{align*}
\tilde \XX_{00}&=\sP^\xi(\Q)\,\sP(\Q),\quad& \tilde\XX_{01}&=\sP^\xi(\Q)\,\fw_1\,\sP(\Q),  \\
\tilde \XX_{10}&=\sP^\xi(\Q)\,\fw_0\,\sP(\Q), \quad & \tilde \XX_{11}&=\sP^\xi(\Q)\,\sn(\xi)\fw_0\,\sP(\Q).
\end{align*}
\end{prop}
\begin{proof}
It is easy to confirm the equality
\begin{align}
\sP(\Q)=\sP^\xi(\Q)\,\{\sn(x\xi)\sm(1;h)|\,x\in \Q,\,h\in \sG^\xi_1(\Q)\bsl \sG_1(\Q)\}.
\label{DC-f0}
\end{align}
From this, together with the Bruhat decomposition \eqref{BruhatDecom} and the relation
\begin{align}
\sm(x;1_{m})\sn(\xi)\sm(x^{-1};1_{m})=\sn(x\xi), \quad x\in \Q-\{0\}, 
 \label{DC-f-1}
\end{align}
we have that $\sG(\Q)-(\sP(\Q)\cup \sP(\Q)\fw_1\sP(\Q))$ is a union of the sets
$$
\sP^{\xi}(\Q)\sn(\eta)\sm(1;h)\fw_0\sP(\Q), \quad (\eta\in \{0,\xi\},\,h\in \sG_1^\xi(\Q)\bsl \sG_1(\Q)). 
$$
Since $\fw_0$ is commuting with $\sm(1;h)$, we have
$$
\bigcup_{h\in \sG_1^\xi(\Q)\bsl \sG_1(\Q)} \sP^\xi(\Q)\sn(\eta)\sm(1;h)\fw_0 \sP(\Q)=\sP^\xi(\Q)\sn(\eta)\fw_0\sP(\Q)
\quad\text{for $\eta\in \{0,\xi\}$}.
$$ We have the equalities
\begin{align}
\tilde \XX_{10}=\sP(\Q) \fw_1 \sP(\Q)=\sP^\xi(\Q)\fw_1\sP(\Q),  
 \label{DC-f1}
\end{align}
and 
\begin{align}
\sP^\xi(\Q)\sn(\xi)\fw_0\sP(\Q)\cap \sP^{\xi}(\Q)\fw_0 \sP(\Q)=\emp. 
 \label{DC-f2}
\end{align}
Let us show \eqref{DC-f2} by deducing a contradiction from the relation $\sn(\xi)\fw_0=p\fw_0\sm(t;h)\sn(X)$ with $p\in \sP^\xi(\Q)$, $\sm(t;h)\sn(X)\in \sP(\Q)$. Since 
\begin{align*}
\fw_0^{-1}\sn(-\xi)\xi&=\xi+Q[\xi]\e'_1, \\
\sn(-X)\sm(t;h)^{-1}\fw_0^{-1}p^{-1}\xi&=h^{-1}\xi+t\langle X,\xi\rangle\e_1,
\end{align*}
we have the relation $\xi+Q[\xi]\e_1'=h^{-1}\xi+t\langle X,\xi\rangle\e_1$, which is impossible due to $Q[\xi]\not=0$. This completes the proof of \eqref{DC-f2}. To show \eqref{DC-f1}, we note that from \eqref{DC-f0} and \eqref{DC-f-1}, 
$$
\sP(\Q)\fw_1\sP(\Q)=\bigcup_{{\eta \in \{0,\xi\}}}\bigcup_{h \in \sG_1^\xi(\Q)\bsl \sG_1(\Q)} \sP^{\xi}(\Q)\sn(\eta)\sm(1;h)\fw_1\sP(\Q).
$$
Using the relation $\langle h\e_0,\e_1\rangle=\langle\e_0,\e_1\rangle=0$, we have the equality
\begin{align*}
\langle \sn(\eta)\sm(1;h)\fw_1\,\e_1,\xi\rangle=\langle h\e_0,\xi\rangle,
\end{align*}
whose right-hand side is non zero, because the positive definite space $V_1^{\xi}$ does not contain an isotropic vector $h\e_0$. Moreover, 
$$
\langle \sn(\eta)\sm(1;h)\fw_1\e_1,\e_1\rangle=\langle \e_0,\e_1\rangle=0.
$$
Hence $\sn(\eta)\sm(1;h)\fw_1\in \XX_{10}$ for all $\eta\in \{0,\xi\}$ and $h\in \sG_1(\Q)$. From the equality $\XX_{10}=\sP^\xi(\Q)[\fw_1]$ shown in Lemma~\ref{DC-L3}, we have $\sP^\xi(\Q)\sn(\eta)\sm(1;h)\fw_1\sP(\Q)=\sP^{\xi}(\Q)\fw_1\sP(\Q)=\tilde \XX_{10}$. This completes the proof of \eqref{DC-f1}. From the argument so far, we have the disjoint decomposition 
\begin{align}
 \sG(\Q)=\sP(\Q)\bigsqcup \sP^{\xi}(\Q)\fw_{1}\sP(\Q)\bigsqcup \sP^\xi(\Q)\fw_0\sP(\Q)\bigsqcup \sP^\xi(\Q)\sn(\xi)\fw_0\sP(\Q).
 \label{DC-f4}
\end{align}
Since $[\e_1]\in \tilde \XX_{00}$, $[\fw_0\e_1]\in \XX_{01}$, $[\sn(\xi)\fw_0\e_1]\in \XX_{11}$, we have the containments 
$$\sP(\Q)\subset \tilde\XX_{00}, \quad \sP^{\xi}(\Q)\fw_0\sP(\Q)\subset \tilde \XX_{01}, \quad\sP^\xi(\Q)\sn(\xi)\fw_0\sP(\Q)\subset \tilde \XX_{11}.$$
As shown above, we have the equality $\sP^\xi(\Q)\fw_1\sP(\Q)=\tilde \XX_{10}$. From \eqref{DC-f4} and the disjoint decomposition $\sG(\Q)=\bigsqcup_{jj'}\tilde \XX_{jj'}$, all the containments above have to be equalities. This completes the proof. \end{proof}

\begin{lem} \label{DC-L3}
 Set $\e=Q[\xi]\e_0\in V_1$. Then, 
 $\XX_{10}=\sP^{\xi}(\Q)[\e]=\sP^\xi(\Q)\,[\fw_1 \e_1]$. 
\end{lem}
\begin{proof}
The inclusion $\sP^\xi(\Q)[\e]\subset \XX_{10}$ follows from \S~\ref{sec:QUALAT} (iv). Let $[v]\in \XX_{10}$. Since $v \in \e_1^\bot=\Q\e_1+V_1$, it can be written in the form $v=t\e_1+u$ with $t\in \Q$ and $u\in V_1$. Then, $u$ is a non zero isotropic vector in $V_1$ such that $\langle u,\xi\rangle=\langle v,\xi\rangle \not=0$. We decompose $u=a\,\xi+u'$ with $a=(Q[\xi])^{-1}\langle u,\xi\rangle$ and $u'\in V_1^\xi$. Since $u$ is an isotropic vector, we have $Q[u']=-a^2\,Q[\xi]\not=0$. Hence, $\langle u,u'\rangle=Q[u']\not=0$. Set $x=\langle u,u'\rangle^{-1}t$, Then, $\sn(x u')\in \sP^\xi(\Q)$ and 
$$\sn(x u')\,v=t\e_1+\sn(x u')u=t\e_1+u-\langle xu',u\rangle\,\e_1=u. 
$$
This shows $[v]$ and $[u]$ belong to the same $\sP^\xi(\Q)$-orbit. It is easy to confirm that the linear map $\varphi:<\e_1,\xi,u >_\Q \rightarrow <\e_1,\xi,\e>_\Q$ such that $\varphi(\e_1)=\e_1$, $\varphi(\xi)=\xi$ and $\varphi(u)=a\,\e$ is an injective isometry. Applying Witt's theorem, we have an isometry $\tilde\varphi$ of $V$ extending $\varphi$. Then, $\tilde\varphi\in \sP^\xi(\Q)$ and $\tilde\varphi([u])=[\e]$. This shows $[u]$ and $[\e]$ belong to the same $\sP^\xi(\Q)$-orbit. Since $[v]\in \sP^\xi(\Q)\,[u]$ as was shown above, we have $[v]\in \sP^\xi(\Q)\,[\e]$. This completes the proof of the inclusion $\XX_{10}\subset \sP^\xi(\Q)[\e]$. Thus we obtain the equality $\XX_{10}=\sP^\xi(\Q)[\e]$. By $\langle \fw_1\e_1,\xi\rangle=\langle \e_0,\xi\rangle=1$ and $\langle \fw_1\e_1,\e_0\rangle=\langle \e_0,\e_0 \rangle=0$, we see $[\fw_1 \e_1]\in \XX_{10}$. \end{proof}

\subsection{The structure of $\sP^\xi(\Q)\bsl \sG(\Q)/\sN(\Q)$} \label{DCosetSec}

From the Levi decomposition $\sP(\Q)=\sM(\Q)\,\sN(\Q)$, each $(\sP^\xi(\Q),\sP(\Q))$-double coset in $\sG(\Q)$ is expressed as a union of $(\sP^\xi(\Q),\sN(\Q))$-cosets, i.e.,  
\begin{align}
 \sP^\xi(\Q)\su \sP(\Q)=\bigcup_{{t\in \Q^\times}}\bigcup_{\delta\in \sG_1(\Q)} \sP^\xi(\Q)\,\su\,\sm(t;\,\delta)\,\sN(\Q)
 \label{PP=PN}
\end{align}
for any $\su \in \sG(\Q)$. The following lemma tells us how to make a disjoint union from \eqref{PP=PN} by abandoning overlap. 

\begin{lem} 
\begin{itemize}
\item[(1)] \label{DCN-L1} 
The set $\sP(\Q)$ is a disjoint union of subsets
$$
\sP^\xi(\Q)\sm(1;\delta)\sN(\Q), \quad (\delta\in \sG_1^\xi(\Q)\bsl \sG_1(\Q)).
$$
\item[(2)] \label{DCN-L2}
The set $\sP^\xi(\Q)\fw_0\sP(\Q)$ is a disjoint union of the sets
$$
\sP^\xi(\Q)\,\fw_0\,\sm(1;\,\delta)\,\sN(\Q), \qquad (\delta \in \sG^\xi_1(\Q) \bsl \sG_1(\Q)). 
$$
\item[(3)] \label{DCN-L3}
The set $\sP^\xi(\Q) \fw_1 \sP(\Q)$ is a disjoint union of the sets
$$
\sP^\xi(\Q)\,\fw_1\,\sm(t;\,\delta)\,\sN(\Q), \qquad (t\in \Q^\times,\,\delta \in \sP^1_{0}(\Q) \bsl \sG_1(\Q)), 
$$ 
where $\sP_1^0$ is the $\Q$-parabolic subgroup of $\sG_1$ stabilizing the vector $\e_0$ up to a scalar. 
\item[(4)] \label{DCN-L4}
The set $\sP^\xi(\Q)\sn(\xi)\fw_0 \sP(\Q)$ is a disjoint union of the sets
$$
\sP^\xi(\Q)\,\sn(\xi)\fw_0\,\sm(t;\,\delta)\,\sN(\Q), \qquad (t\in \Q^\times,\, \delta \in \sG^\xi_1(\Q) \bsl \sG_1(\Q)). 
$$
\end{itemize}
\end{lem}
\begin{proof} We only give the proofs of (3) and (4) because (1) and (2) are shown in a similar way. 
 
(3) Fix a complete set of representatives $\Gamma$ of $\sP_1^0(\Q)\bsl \sG_1(\Q)$. Let $\delta\in \sG_1(\Q)$ and $t\in \Q^\times$, and write $\delta=q \gamma$ with $q\in \sP_1^0(\Q)$ and $\gamma \in \Gamma$. Let $q=\sm_{\sG_1}(\tau;h)\sn_{\sG_1}(Z)$ with $(\tau,h)\in \Q^\times \times \sG_0(\Q)$ and $Z\in V_0(\Q)$. If we set $X=\langle Z,h^{-1}\alpha\rangle\,\e_0+h^{-1}\alpha-\alpha$, then
$$
\fw_1\,\sm(1;q)\,\sn(X)\,\fw_1\,\e_1=\tau \e_1, \quad 
\fw_1\,\sm(1;q)\,\sn(X)\,\fw_1\, \xi=\xi 
$$
which implies the containment $\fw_1\,\sm(1;q)\,\sn(X)\,\fw_1 \in \sP^{\xi}(\Q)$. This, together with the relation $\delta=q\gamma$ and the obvious containment $\sm(t;\gamma)^{-1}\sn(-X)\sm(t;\gamma) \in \sN(\Q)$, shows the identities 
\begin{align*}
\sP^{\xi}(\Q) \fw_1\sm(t;\delta)\sN(\Q)&=\sP^{\xi}(\Q)\,(\fw_1\sm(1;q)\sn(X)\fw_1^{-1})\,(\fw_1\sm(t;\gamma))\,(\sm(t;\gamma)^{-1}\sn(-X)\sm(t;\gamma))\,\sN(\Q)\\
&=\sP^\xi(\Q)\fw_1\sm(t;\gamma)\sN(\Q),
\end{align*}
which together with \eqref{PP=PN} infers that $\sP^{\xi}(\Q)\fw_1\sP(\Q)$ is a union of cosets $\sP^{\xi}(\Q)\fw_1\sm(t;\gamma)\sN(\Q)$ with $t\in \Q^\times$ and $\gamma \in \Gamma$. To confirm the disjointness of these cosets, it is enough to prove that the equality $\fw_1\sm(t_1;\gamma_1)=p\fw_1\sm(t;\gamma)n$ with $t_1,t\in \Q^\times$, $\gamma_1,\gamma\in \Gamma$, $p\in \sP^\xi(\Q)$ and $n\in \sN(\Q)$ implies $t_1=t$ and $\gamma_1=\gamma$. Set $p=\sm(c;h_0)\sn(Z)$ and $n=\sn(X)$. Then by a direct computation, 
\begin{align*}
&(\fw_1\sm(t_;\gamma_1))^{-1}\e_1=\gamma_1^{-1}\e_0, \quad 
(p\fw_1\sm(t;\gamma)n)^{-1}\e_1=c\{\gamma^{-1}\e_0+\langle X,\gamma^{-1}\e_0\rangle \e_1\}, 
\end{align*}
Since these vectors should be the same, we obtain $\gamma_1^{-1}\e_0=c\gamma^{-1}\e_0$, which means $\gamma\gamma_1^{-1}\in \sP_1^0(\Q)$, or equivalently $\gamma_1=\gamma$. The equality $t_1=t$ is deduced from the identities 
$$
t_1=\langle \e_1', (\fw_1\sm(t_1;\gamma_1))^{-1}\xi\rangle, \quad t=\langle \e_1', (p\sw_1\sm(t;\gamma)n)^{-1}\xi\rangle,
$$
which are shown by $p\xi=\xi$ and $\fw_1^{-1}\xi=a\e_1+\alpha+\e_1'$. 

\smallskip
\noindent
(4) Let $\delta\in \sG_1(\Q)$. If we write $\delta=h\gamma$ with $h\in \sG_1^\xi(\Q)$ and $\gamma \in \sG_1^\xi(\Q)\bsl \sG_1(\Q)$, then 
$$
\sn(\xi)\fw_0\sm(t;h)=\sm(1;h)\sn(h^{-1}\xi)\fw_0\in \sP^\xi(\Q)\sn(\xi)\fw_0,
$$
which shows the equality $\sP^\xi(\Q)\sn(\xi)\fw_0\sm(t;\delta)\sN(\Q)=\sP^\xi(\Q)\sn(\xi)\fw_0\sm(t;\gamma)\sN(\Q)$ for any $t\in \Q^\times$. Hence, from \eqref{PP=PN}, we see that $\sP^\xi(\Q)\sn(\xi)\fw_0 \sP(\Q)$ is a union of cosets $\sP^{\xi}(\Q)\sn(\xi)\fw_0\sm(t;\gamma)\sN(\Q)$ with $t\in \Q^\times$ and $\gamma\in \sG_1^\xi(\Q)\bsl \sG_1(\Q)$. To show that these sets are disjoint with each other, it is enough to confirm that the equality $\sn(\xi)\fw_0\sm(t_1;\gamma_1)=p\sn(\xi)\fw_0\sm(t;\gamma)n$ with $t_1,t\in \Q^\times$, $\gamma_1,\gamma\in \sG_1(\Q)$, $p\in \sP^\xi(\Q)$ and $n\in \sN(\Q)$ implies $t_1=t$ and $\gamma_1\in \sG_1^\xi(\Q)\gamma$. Set $n=\sn(X)$ and let $\tau\in \Q^\times$ be the scalar such that $p^{-1}\e_1=\tau\e_1$. Then a direct computation shows the equalities\begin{align*}
(\sn(\xi)\fw_0\sm(t_1;\gamma_1))^{-1}\e_1=t_1\e_1', \quad 
(p\sn(\xi)\fw_0\sm(t;\gamma)n)^{-1}\e_1=t\tau(\e_1'-X-2^{-1}Q[X]\e_1).
\end{align*}
Since these vectors are the same, we have $X=0$. Similarly, by equating the two vectors
 \begin{align*}
(\sn(\xi)\fw_0\sm(t_1;\gamma_1))^{-1}\xi& =\gamma_1^{-1}\xi+t_1^{-1}Q[\xi]\e_1', \\
 (p\sn(\xi)\fw_0\sm(t;\gamma)n)^{-1}\xi&=t^{-1}
\{\langle X,\gamma^{-1}\xi\rangle -2^{-1}Q[X]Q[\xi]\}\e_1+(\gamma^{-1}\xi-Q[\xi]X)+t^{-1}Q[\xi]\e_1',
\end{align*}
we have $t_1=t$ and $\gamma_1^{-1}\xi=\gamma^{-1}\xi-Q[\xi]X$. Since $X=0$, the desired containment $\gamma\gamma_1^{-1}\in \sG_1^\xi(\Q)$ follows. This completes the proof. 
\end{proof}

\begin{lem} \label{STB-L1}
\begin{itemize}
\item[(1)] For any $\gamma=\sm(1;\,\delta)$ with $\delta \in \sG_1(\Q)$, 
$$
 \sN(\Q)\cap \gamma^{-1}\sP^\xi(\Q)\gamma=\{\sn(Z)|\,Z\in (\delta^{-1}\xi)^\bot\,\}. 
$$
\item[(2)] If $\gamma=\su\sm(t;\delta)$ with $\su\in \{\fw_1,\fw_0,\sn(\xi)\fw_0\}$, $t\in \Q^\times$ and $\delta\in \sG_1(\Q)$, then 
$$
 \sN(\Q)\cap \gamma^{-1}\sP^\xi(\Q)\gamma=\{1\}. 
$$
\end{itemize}
\end{lem}
\begin{proof}
Let $Z\in V_1$. Then, $\gamma\,\sn(Z)\,\gamma^{-1}=\sn(\delta Z)$ belongs to $\sP^\xi(\Q)$ if and only if $\delta Z \in V_1^{\xi}$. This completes the proof of (1). If $\gamma=\su \sm(t;\delta)$ with $\su\in \{\fw_0,\sn(\xi)\fw_0\}$, then $\sN(\Q)\cap \gamma^{-1}\sP^\xi(\Q)\gamma=\{1\}$ follows from $\sN(\Q)\cap \fw_0^{-1}\sP(\Q)\fw_0=\{1\}$. Suppose $\gamma=\fw_1\sm(t;\delta)$. For any $\sn(X)\in \sN(\Q)\cap \gamma^{-1}\sP^\xi(\Q)\gamma$, a computation shows 
$$
\gamma\sn(X)\gamma^{-1}\xi =(a-\langle t\delta X,\alpha\rangle -2^{-1}t^2Q[X])\e_0 +x\e_1+(\alpha+X_0)+x'\e_1'+\e_0',
$$
where we set $t\delta X=x\e_0+X_0+x'\e_0'$. From $\gamma\sn(X)\gamma^{-1}=\xi=a\e_0+\alpha+\e_0'$, we have $x=x'=0$ and $X_0=0$, or equivalently $X=0$. Therefore, $\sN(\Q)\cap \gamma^{-1}\sP^\xi(\Q)\gamma=\{1\}$ as required. 
\end{proof}

\section{The smoothed Rankin-Selberg integral: the spectral side}

Let $\cU$ be an irreducible $\sG_1^{\xi}(\A_\fin)$-submodule of $\cV(\xi)$ (\S~\ref{subsubsec: cU-def}), and $f\in \cU(\,\bK_{1,\fin}^{\xi *})$ be a non-zero vector. We further assume that $f$ is an eigenform of the operator $\tau_\fin^\xi$ defined in \S 1.6.2 with eigenvalue $\epsilon_f\in \{\pm 1 \}$.

\subsection{Real Shintani functions}
First of all, we shall review results in \cite[\S 4]{Tsud2011-1} briefly. For any integer $l>0$, the holomorphic Shintani function of weight $l$ is a smooth function $\Phi_{l}^{\xi}(s):G\rightarrow \C$ depending on a complex number $s\in \C$ defined by the formula\footnote{Erratum: The exponent of $2$ in the formula \cite[(4.2)]{Tsud2011-1} should be $-(s+\rho-l)/2$.}
\begin{align}
\Phi_{l}^\xi(s;\,g)&=J(g,\fz_0)^{-l}\,2^{-(s+\rho-l)/2}\,
 \left(\dfrac{(\xi,g\langle \fz_0 \rangle )}{i\,\Delta^{1/2}}\right)^{s+\rho-l}, \quad g\in G, 
 \label{RealShintani0}
\end{align}
where the complex power $z^{\alpha}$ for $z\in \C^\times$ and $\alpha\in \C$ is defined as\footnote{Erratum: In the first line of \cite[\S 4.3]{Tsud2011-1}, ${\rm Arg}(\log z)$ should be ${\rm Arg}(z)$. } $$
z^{\alpha}=\exp(\alpha(\log |z|+i\,{\rm Arg}(z)) ), \quad {\rm Arg}(z)\in (-\pi,\pi]. 
$$
From \cite[Propositions 18 and 24]{Tsud2011-1}, the function $\Phi_l^\xi(s)$ is characterized as a unique smooth function on $G$ possessing the properties:
\begin{itemize}
\item[(i)] It has the $(\sP^\xi(\R)^0,\bK_\infty)$-equivariance
\begin{align}
\Phi_{l}^\xi(s;\sm(t;h_0)\sn(Z)g k)=J(k,\fz_0)^{-l}|t|^{s+\rho}_\infty \,\Phi_l^\xi(s;g)
 \label{RealShintani-equiv}
\end{align}
for any $\sm(t;h_0)\sn(Z)\in \sP^\xi(\R)^0$, $g\in G$ and $k\in \bK_\infty$.
\item[(ii)] The Cauchy-Riemann condition: $[R(\bar X)\Phi_l^{\xi}(s)](g)=0$ for all $\bar X\in \fp^{-}$ and $g\in G$. 
\item[(iii)] $\Phi_l^\xi(s;b_\infty)=1$. 
\end{itemize}

The function $\Phi_l^\xi(s)$ is extended to $\sG(\R)^{+}$ by the same formula \eqref{RealShintani0} with $g\in \sG(\R)^+$ and then to all of $\sG(\R)$ by demanding 
\begin{align}
\Phi_{l}^{\xi}(s;\sm(-1;1_{m})g)=\Phi_{l}^{\xi}(s;g), \quad g\in \sG(\R)^+. 
\label{RealShintani-neg}
\end{align} 
For any $\kappa_1\in \sG^{\xi}_1(\R)-\sG^{\xi}_1(\R)^0$, the relation $\Phi_l^{\xi}(s;\kappa_1 g)=\Phi_{l}^{\xi}(s;g)$ for all $g\in G$ is easily confirmed by \cite[Proposition 18]{Tsud2011-1}; thus $\Phi^\xi_l(s):\sG(\R)\rightarrow \C$ is the same one obtained in \cite[\S 7.1]{Tsud2011-1}. Recall the vector $v_0^\C$ defined by \eqref{v0C} and set 
\begin{align*}
A(g)=\langle \xi_0^-,gv_0^\C \rangle, \quad B(g)=\langle \e_1,gv_0^{\C}\rangle, \quad g\in \sG(\R). 
\end{align*}

\begin{lem} \label{RealShintani-equi-L}
Set $\chi(g)=1$ if $g\in \sG(\R)^+$ and $\chi(g)=-1$ if $g\in \sG(\R)-\sG(\R)^+$. Then, 
\begin{align}
\Phi_l^\xi(s;g)
&=(-1)^{l}2^{-(s+\rho)/2}\{A(g)\}^{-l}\left(\chi(g)\,i\,\frac{B(g)
}{A(g)} \right)^{-(s+\rho)}
, \quad g \in \sG(\R).
\label{RealShintani} 
\end{align}
The formula \eqref{RealShintani-equiv} holds true for any $\sm(t;h)\sn(Z)\in \sP^\xi(\R)$, $g\in \sG(\R)$ and $k\in \bK_\infty$. 
\end{lem}
\begin{proof} For $g\in \sG(\R)^+$ the formula \eqref{RealShintani} is deduced from \eqref{RealShintani0} in the same way as \cite[Proposition 18]{Tsud2011-1}\footnote{Erratum : The exponent of $2$ in the formula \cite[(4.3)]{Tsud2011-1} should be $-(s+\rho)/2$}. Since $\sm(-1;1_m)\e_1=-\e_1$ and $\sm(-1;1_m)\xi_0^-=\xi_0^{-} $, we have \eqref{RealShintani} for $g\not\in \sG(\R)^{+}$ from \eqref{RealShintani-neg}.  

 Since $\sG_1^\xi(\R)\subset \sG(\R)^+$, an element $\sm(t;h)\sn(Z)$ of $\sP^\xi(\R)$ belongs to $\sG(\R)^+$ if and only if $t>0$. When $g\in \sG(\R)^+$ and $t>0$, then the formula \eqref{RealShintani-equiv} follows from the defining formula \eqref{RealShintani}. When $g\in \sG(\R)^+$ and $t<0$, by \eqref{RealShintani-neg}, we have 
\begin{align*}
\Phi_l^\xi(s;\sm(t;h)\sn(Z)gk)&=\Phi_{l}^\xi(s;\sm(-t;h)\sn(Z)gk) \\
&=|-t|_\infty^{s+\rho}J(k,\fz_0)^{-l}\Phi_{l}^\xi(s;g)=|t|_\infty^{s+\rho}J(k,\fz_0)^{-l}\Phi_{l}^\xi(s;g). 
\end{align*}
The remaining case $g\in \sG(\R)-\sG(\R)^+$ is settled by a similar argument. 
\end{proof}

\begin{lem} \label{P1-L3-L}
\begin{itemize}
\item[(1)] 
Let $g=\sn(x\xi)\sm(r; \epsilon\,1_m)b_\infty$ with $x\in \R$, $r>0$, and $\epsilon\in \{\pm 1\}$. Then 
\begin{align*}
A(g)&=-\epsilon-\frac{\Delta^{1/2}x}{\sqrt{2} r}\,i, \quad 
B(g) =\frac{i}{\sqrt{2}r}.
\end{align*}
\item[(2)] Let $g=\sw_0\,\sn(X)\,\sm(r;\,\epsilon 1_m)\,b_\infty$ with $X=x\xi+Z\,(Z\in V_1^\xi(\R),x\in \R)$, and $\epsilon \in \{\pm 1\}$. Then, 
\begin{align*}
A(g)&=-\epsilon -\frac{\Delta^{1/2}x }{\sqrt{2}r}i, \quad 
B(g)=\frac{ri}{\sqrt{2}} \left\{\frac{Q[Z]}{2r^{2}}
+\left(1+\epsilon\frac{\Delta^{1/2}x}{\sqrt{2}r}i\right)^2
\right\}.
\end{align*}
\item[(3)] Let $g=\sw_1\,\sn(X)\,\sm(r;\,1_m)\,b_\infty$ with $X\in V_1(\R)$ and $r>0$. Then, 
\begin{align*}
A(g)&=\frac{i}{\sqrt{2}r\Delta^{1/2}}
\left\{-\frac{1}{2}Q[X-\alpha]+\sqrt{2}ri\langle X-\alpha,\xi_0^{-}\rangle-\frac{\Delta}{2}-r^{2}\right\}, \\ 
B(g)&=\Delta^{-1/2} +\frac{\langle X-\alpha,\e_0\rangle}{\sqrt{2}r}i,
\end{align*}
where $a\in \Q$ and $\alpha\in V_0$ are the elements in the relation \eqref{xiForm}. 
\item[(4)] Let $g=\sn(\epsilon \xi)\sw_0\sn(X)\,\sm(r;\,1_m)\,b_\infty$ with $\epsilon \in \{\pm 1\}$, $X=Z+x\xi,\,(Z\in V_1^\xi(\R),x\in \R)$. Then,
\begin{align*}
A(g)&=\frac{\epsilon\Delta^{1/2}ri}{\sqrt{2}} \left\{\frac{Q[Z]}{2r^2}+\left(1+\frac{\Delta^{1/2}(-x+\epsilon\Delta^{-1})}{\sqrt{2}r}i\right)^2\right\}+\frac{\epsilon\Delta^{-1/2}}{2\sqrt{2}r}i,
\\
B(g)&=\frac{-ri}{\sqrt{2}}\left\{\frac{Q[Z]}{2r^2}+
\left(1-\frac{\Delta^{1/2}x}{\sqrt{2}r}i\right)^2\right\}. 
\end{align*}
\end{itemize}
\end{lem}
\begin{proof}
From \eqref{v0C} and \eqref{b-infty}, we have $\sm(r;\epsilon \,1_m)b_\infty v_0^\C=\epsilon \xi_0^{-}+i(-r\e_1+r^{-1}\e'_1)/\sqrt{2}$. 

(1) From $g=\sn(x \xi)\,\sm(r;\,1)\,b_\infty$, 
\begin{align*}
B(g)
&=\langle \sn(-x\xi)\e_1 ,
\sm(r;\epsilon \,1_m)b_\infty v_0^\C
\rangle
=\langle \e_1,\epsilon\xi_0^{-}+ \tfrac{i}{\sqrt{2}}({-r\e_1+r^{-1}\e'_1})
\rangle=\tfrac{i}{\sqrt{2}r}, \\
A(g)&=\langle \sn(-x\xi)\xi_0^{-}, \sm(r;\epsilon \,1_m)b_\infty v_0^\C
\rangle
\\
&=\langle \xi_0^{-}-x\Delta^{1/2}\e_1,\epsilon\xi_0^{-}
+\tfrac{i}{\sqrt{2}}(-r\e_1+r^{-1}\e_1') \rangle
=-\epsilon-i\tfrac{\Delta^{1/2}x}{\sqrt{2}r}.  
\end{align*}
(2) From $g=\sw_0\,\sn(X)\,\sm(r;\,\epsilon\,1_m)\,b_\infty$, 
\begin{align*}
A(g) 
&=\langle \sn(-X)\sw_0 \xi_0^{-}, \sm(r;\epsilon \,1_m)b_\infty v_0^\C\rangle \\&=
\langle \xi_0^{-}+\langle X,\epsilon\xi_0^{-}\rangle \e_1, 
\xi_0^{-}+\tfrac{i}{\sqrt{2}}({-r\e_1+r^{-1}\e_1'})\rangle
=\epsilon\langle \xi_0^-, \xi_0^-\rangle +\tfrac{i}{\sqrt{2}r}\,{\langle X,\xi_0^-\rangle}, \\
B(g)&=\langle \sn(-X)\e_1', 
\sm(r;\epsilon \,1_m)b_\infty v_0^\C
\rangle \\
&=\langle \e_1'-X-2^{-1}Q[X]\e_1,\epsilon\xi_0^{-}+\tfrac{i}{\sqrt{2}}({-r\e_1+r^{-1}\e_1'})\rangle 
=-\tfrac{r}{\sqrt{2}}i-\epsilon\langle X, \xi_0^{-}\rangle -\tfrac{i}{2\sqrt{2}r} {Q[X]}.
\end{align*}
(3) Let $g=\sw_1\,\sn(X)\,\sm(r;\,1)\,b_\infty$. Since $\sw_1 \xi_0^{-}=\Delta^{-1/2}(a\e_1+\alpha+\e_1')$, 
\begin{align*}
A(g)&=\langle \sn(-X)\sw_1 \xi_0^{-}, 
\sm(r;\,1_m)b_\infty v_0^\C
\rangle \\
&=\langle 
\Delta^{-1/2}\bigl(a\e_1+\alpha+\langle X,\alpha\rangle\e_1+\e_1'-X-2^{-1}Q[X]\e_1\bigr), \xi_0^{-}+\tfrac{i}{\sqrt{2}}({-r\e_1+r^{-1}\e_1'})
\rangle \\
&=\Delta^{-1/2}\left\{\tfrac{i}{\sqrt{2}r}(a+\langle X,\alpha\rangle-2^{-1}Q[X])-\langle X-\alpha,\xi_0^-\rangle-\tfrac{ri}{\sqrt{2}} \right\}\\
&=\Delta^{-1/2}\left\{\tfrac{-i}{2\sqrt{2}r}Q[X-\alpha]-\langle X-\alpha,\xi_0^-\rangle -\tfrac{\Delta i}{2\sqrt{2}r}-\tfrac{ri}{\sqrt{2}}
\right\}, 
\\
B(g)&=\langle \sn(-X)\e_0, \sm(r;\,1_m)b_\infty v_0^\C\rangle \\
&=\langle \e_0+\langle X,\e_0 \rangle \e_1, \xi_0^{-}+\tfrac{i}{\sqrt{2}r}({-r\e_1+r^{-1}\e_1'}) \rangle 
=\Delta^{-1/2} +\tfrac{i}{\sqrt{2}r}\,{\langle X-\alpha,\e_0\rangle}.
\end{align*}
Note that $\Delta=-2a-Q[\alpha]$, $\langle \alpha,\e_0\rangle=0$ and $\langle \xi_0^-,\e_0\rangle=\Delta^{-1/2}\langle \xi,\e_0\rangle=\Delta^{-1/2}$ from \eqref{xiForm} and \eqref{DelXi0}. 

\noindent
(4) Form $g=\sn(\epsilon \xi)\sw_0\sn(X)\,\sm(r;\,1)\,b_\infty$, 
\begin{align*}
A(g)
&=\langle \sn(-X)\sw_0 \sn(-\epsilon \xi)\xi_0^{-}, \sm(r;,1_m)b_\infty v_0^\C
\rangle \\
&=\langle \xi_0^{-}+\langle X,\xi_0^{-}\rangle \e_1-\epsilon \Delta^{1/2}(\e_1'-X-2^{-1}Q[X]\e_1), \xi_0^{-}+\tfrac{i}{\sqrt{2}}({-r\e_1+r^{-1}\e_1'}) \rangle \\
&=-1+\left(\epsilon \Delta^{1/2}+\tfrac{i}{\sqrt{2}r}\right)\langle X,\xi_0^{-}\rangle+\left(\tfrac{Q[X]}{2r}+r\right)\tfrac{\epsilon\Delta^{1/2}}{\sqrt{2}}i,\\B(g)&=\langle \sn(-X)\sw_0\sn(-\epsilon\xi)\e_1,
\sm(r;\,1_m)b_\infty v_0^\C
\rangle
=\langle \e_1'-X-2^{-1}Q[X]\e_1, \xi_0^{-}+\tfrac{i}{\sqrt{2}}({-r\e_1+r^{-1}\e_1'}) \rangle \\
&=\tfrac{-r}{\sqrt{2}}i-\langle X,\xi_0^-\rangle-\tfrac{i}{2\sqrt{2}r}\,Q[X]. 
\end{align*}
By $x\Delta^{1/2}=-\langle X,\xi_0^{-}\rangle$ and $Q[X]=Q[Z]-\Delta x^2$, a computation shows that these become the required formula. 
\end{proof}

\begin{prop} \label{RealShinExBruhForm}
\begin{itemize}
\item[(1)] Set $g=\sn(x\xi)\sm(r,\epsilon\,1_m)b_\infty$ with $x\in \R$, $r>0$ and $\epsilon\in \{\pm 1\}$. Then, 
 \begin{align*}
\Phi_l^{\xi}(s;g)=\epsilon^lr^{s+\rho} \left(1+\epsilon \tfrac{\Delta^{1/2}x}{\sqrt{2}r}i\right)^{s+\rho-l}. 
\end{align*}
\item[(2)] For $Z\in V_1^{\xi}(\R)$, $x\in \R$, $r>0$ and $\epsilon\in \{\pm 1\}$, set $g=\sw_0\sn(Z+x\xi)\sm(r,\epsilon\,1_m)b_\infty$. Then 
 \begin{align*}
\Phi_l^{\xi}(s;g)=\epsilon ^{l}
r^{-(s+\rho)}\left(1+\epsilon \tfrac{\Delta^{1/2}x}{\sqrt{2}r}i \right)^{-l+s+\rho}\left\{\left(1+\epsilon \tfrac{\Delta^{1/2}x}{\sqrt{2}r}i \right)^2+\tfrac{Q[Z]}{2r^{2}}\right\}^{-(s+\rho)}. 
\end{align*}
\end{itemize}
\end{prop}
\begin{proof}
(1) Since $\Im(B(g)/A(g))=\Im(-\frac{i}{\sqrt{2}r}(\epsilon+\frac{\Delta^{1/2}x}{\sqrt{2}r}i)^{-1})$, the points $g=\sn(x\xi)\sm(r;\epsilon\,1_m)$ $(x\in \R)$ belong to $\sG(\R)^+$ if $\epsilon=1$ and to $\sG(\R)-\sG(\R)^+$ if $\epsilon=-1$. The formula follows from Lemmas~\ref{RealShintani-equi-L} and \ref{P1-L3-L} (1). 

(2) By Lemma~\ref{P1-L3-L} (2), a computation reveals $\epsilon \Im(B(g)/A(g))<0$ and $\Re (A(g))=-\epsilon$, which means $g \in \sG(\R)^{+}$ if and only if $\epsilon=+1$; thus from Lemma~\ref{RealShintani-equi-L}, 
$$
\Phi_{l}^\xi(s:g)=(-1)^{l}2^{-(s+\rho)/2}A(g)^{-l}(\epsilon iB(g)/A(g))^{-(s+\rho)}
$$
Since ${\rm Arg}(\epsilon A(g))\in (-\pi/2,\pi/2)$ and ${\rm Arg}(\epsilon iB(g)/A(g))\in (-\pi/2,\pi/2)$, by our convention on the complex power, we have the relation 
$$
(iB(g))^{-(s+\rho)}=(\epsilon iB(g)/A(g))^{-(s+\rho)}\times (\epsilon A(g))^{-(s+\rho)}. 
$$
Hence $\Phi_l^{\xi}(s;g)=(-1)^{l}2^{-(s+\rho)/2}(iB(g))^{-(s+\rho)}(\epsilon A(g))^{-l+s+\rho}$, which becomes the required formula by a computation. 
\end{proof}

Recall the orthogonal decomposition \eqref{NormFTN-f0}. 
\begin{prop} \label{RealShinExBruhFormW1}
For $g_{Y}=\sw_1\,\sn(Y+\alpha)\,\sm(r;\,1_m)\,b_\infty$ with $Y=y_+\xi_0^{+}+y_{-}\xi_0^{-}+Y_0\,(y_+,y_-\in \R,\,Y_0\in W)$ and $r>0$, 
\begin{align*}
\Phi_l^\xi(s;g_{Y})&=i^{3l}2^{3l/2-2(s+\rho)}\Delta^{l/2}r^{-(s+\rho)+l}\\
&\quad \times 
\left\{Q[Y_0]+y_{+}^2+(iy_{-}+\sqrt{2}r)^{2}+\Delta \right\}^{s+\rho-l}
\left(1+\tfrac{y_{-}-y_{+}}{\sqrt{2}r}i \right)^{-(s+\rho)}. 
\end{align*}
\end{prop}
\begin{proof} By a simple computation with \eqref{xiplus0} and \eqref{xiplus0pm}, the formulas in Lemma~\ref{P1-L3-L} (3) yields  
\begin{align*}
A(g_Y)&=\tfrac{-i}{2\sqrt{2}r\Delta^{1/2}}\{Q[Y_0]+y_{+}^2+(iy_{-}+\sqrt{2}r)^{2}+\Delta\}, \quad 
B(g_Y)=\Delta^{-1/2}\left(1+\tfrac{y_{-}-y_{+}}{\sqrt{2}r}i\right).
\end{align*}
A computation shows $\Im(B(g_0)/A(g_0))=\sqrt{2}r(\Delta/2+r^2)^{-1}>0$; thus $g_{0} \not \in \sG(\R)^{+}$. Since $Y\mapsto g=g_{Y}$ is a continuous mapping from $V_1(\R)$ to $\sG(\R)$, its image is contained in the connected component $g_0G$ of $\sG(\R)$. Hence $g_Y\in g_0 G\subset \sG(\R)-\sG(\R)^+$. Then from Lemma~\ref{RealShintani-equi-L}, 
$$
\Phi_l^{\xi}(s;g_Y)=(-1)^{l}2^{-(s+\rho)/2}A(g)^{-l}(-iB(g)/A(g))^{-(s+\rho)}.
$$
Since $\Re(B(g))=\Delta^{-1/2}>0$, we have ${\rm Arg}(B(g))\in (-\pi/2, \pi/2)$. From $\Im(A(g)/B(g))<0$, we have ${\rm Arg}(iA(g)/B(g))\in (-\pi/2,\pi/2)$. Hence
$$
(iA(g))^{-(s+\rho)}=(iA(g)/B(g))^{-(s+\rho)}\times B^{-(s+\rho)}. 
$$
From this, $\Phi_l^{\xi}(s;g)=(-i)^{l}2^{-(s+\rho)/2}(iA(g))^{-l+s+\rho}B(g)^{-(s+\rho)}$, which becomes the required formula. 
\end{proof}

\begin{lem} \label{RealShinExBruhFormW0xi}
Let $g_{x,Z}^\epsilon=\sn(\epsilon \xi)\sw_0\sn(X)\,\sm(r;\,1_m)\,b_\infty$ with $\epsilon \in \{\pm \}$, $X=Z+x\xi,\,(Z\in V_1^\xi(\R),x\in \R)$ and $r>0$. Then,
\begin{align*}
\Phi_l^{\xi,f}(s;g_{x,Z}^\epsilon)&=\epsilon^l 2^{-(s+\rho-3l)/2}\Delta^{-(s+\rho+l)/2} r^{l}
\left(Q[Z]+\Delta^{-1}+\{\sqrt{2}r-(\Delta^{1/2}x-\epsilon\Delta^{-1/2})\,i\}^{2}\right)^{-l}
\\
&\times \left(
\frac{Q[Z]+(\sqrt{2}r-\Delta^{1/2}x\,i)^2}{Q[Z]+\Delta^{-1}+
\{\sqrt{2}r-(\Delta^{1/2}x-\epsilon\Delta^{-1/2})\,i\}^2}\right)^{-(s+\rho)}.
\end{align*}
\end{lem}
\begin{proof}
From the formulas of $A(g_{x,Z}^{\epsilon})$ and $B(g_{x,Z}^{\epsilon})$ in Lemma~\ref{P1-L3-L} (4), we have
$\epsilon \Im(A(g^{\epsilon}_{0,0})/B(g^{\epsilon}_{0,0}))=-2^{1/2}r^{-1}(1+\Delta^{-1/2}/(2r))<0$, which shows $\epsilon\Im(B(g_{0,0}^\epsilon)/A(g_{0,0}^\epsilon))>0$. By the same reasoning as in the proof of Lemma~\ref{RealShinExBruhFormW1}, we conclude $g_{x,Z}^{+}\in g_{0,0}^{+}G\subset \sG(\R)-\sG(\R)^{+}$ and $g_{x,Z}^{-}\in g_{0,0}^{-}G\subset \sG(\R)^+$ for all $(x,Z)$. Hence from Lemma~\ref{RealShintani-equi-L}, 
$$\Phi_{l}^{\xi}(s;g_{x,Z})=(-1)^{l}2^{-(s+\rho)/2}(A(g_{x,Z}^\epsilon))^{-l}(-i\epsilon B(g^\epsilon_{x,Z})/A(g^\epsilon_{x,Z}))^{-(s+\rho)}.
$$
Then by Lemma~\ref{P1-L3-L} (4), we are done. 
\end{proof}

\subsection{Hecke functions}
Let $\phi \in \cH(\sG(\A_\fin)\sslash \bK^*_\fin)$ be a Hecke function. We set 
\begin{align*}
\Phi_{\fin}^{{f},\xi}(\phi|s;\,g_\fin)&=\fd(\cL_{1}^{\xi})^{-1} \int_{\sG^\xi(\A_\fin)}\sf^{(s)}(h)\phi(h^{-1}g_\fin)\d h, \quad g_\fin \in \sG(\A_\fin).
\end{align*}
with ${\sf}^{(s)}$ being the function on $\sG^\xi(\A_\fin)$ defined as 
$$
\sf^{(s)}(\sm(t;\,h_0)\,\sn(Z)\,k_0)=|t|_\A^{s+\rho}\,f(h_0), \quad t\in \A_\fin^\times,\,h_0\in \sG_1^\xi(\A_\fin),\, Z\in V_1^\xi(\A_\fin),\,k_0 \in \bK_\fin^{\xi*}.
$$
We remark that $\Phi_{\fin}^{f,\xi}(\phi|s)$ with $\phi={\rm ch}_{\bK_\fin^{*}}$ coincides with the function $\Phi_{\fin}^{f,\xi}(s)$ defined in \cite[\S 7.2]{Tsud2011-1}; this is confirmed by the Iwasawa decomposition $\sG^\xi(\Q_p)=\sP^\xi(\Q_p)\bK_{p}^{\xi*}$ (\cite[Proposition 1.2 (i)]{MS98}) together with the relation $\sG_1^\xi(\Q_p)\cap \bK_{1}^*=\bK_{p,1}^{\xi*}$ (\cite[Proposition 2.3]{MS98}) and by noting that the self-dual measure on $V_{1}^\xi(\Q_p)$ yields $\vol(\cL^{\xi}_{1,p})=\fd_p(\cL_{p}^{\xi})^{-1/2}$. Evidently, 
\begin{align}
\Phi_{\fin}^{f,\xi}(\phi|s;\sm(t;h_0)\sn(Z)g_\fin k)=|t|_{\fin}^{s+\rho}\Phi_\fin^{f,\xi}(\phi|s;g_\fin), 
 \label{pAdicTF-equiv}
\end{align}
for all $t\in \A_\fin^\times,\,h_0\in \sG^\xi_1(\Q),\,Z\in V_1^\xi(\A_\fin), \,k\in \bK_{\fin}^{*}$. From definition, there exists a compact set $\cU_{\phi}\subset \sG(\A_\fin)$ depending only on $\phi$ such that 
\begin{align}
\text{$\Phi_{\fin}^{f,\xi}(\phi|s;g)=0$ unless $g\in \sP^{\xi}(\A_\fin)\cU_{\phi}$. }
 \label{pAdicTF-supp}
\end{align}

\begin{lem} \label{pAdicTF-EST}
There exists a compact set $\cK_{\phi}\subset V(\A_\fin)$ depending only on $\phi$ such that $$
|\Phi_\fin^{f,\xi}(\phi|s;g)|\ll \|g^{-1}\e_1\|_\fin^{-(\Re(s)+\rho)}\,\delta(g^{-1}\xi\in \cK_{\phi}), \quad g\in \sG(\A_\fin),\,s\in \C.
$$
\end{lem} 
\begin{proof}
Set $\cK_\phi=\cU_\phi^{-1}\xi$. Then the estimate follows from \eqref{pAdicTF-equiv} and \eqref{pAdicTF-supp} immediately.  
\end{proof}

\subsection{Kernel functions}
For $l\in \N^*$ and $\phi \in \cH^{+}(\sG(\A_\fin)\sslash \bK_\fin^*)$, we define a smooth function ${\mathbf\Phi}_{l}^{f,\xi}(\phi|s)$ on $\sG(\A)$ depending on $s\in \C$ by the formula 
\begin{align}
{\mathbf\Phi}_{l}^{f,\xi}(\phi|s;g)=\Phi_{l}^{f,\xi}(s;g_\infty)\,\Phi_{\fin}^{f,\xi}(\phi|s;g_\fin), \quad g\in \sG(\A).
\label{AdelicTF}
\end{align}
Recall that $\cBB^{+}$ denotes the space of all those entire functions $\beta$ on $\C$ such that for any compact interval $I\subset \R$ and for any $N\in \N$, the estimation $\exp(\pi|\Im(z)|)\,|\beta(z)|\ll (1+|\Im(z)|)^{-N}$, $s\in I+i\,\R$ holds. 

Let $\beta\in \cBB^+$. For $g\in \sG(\A)$, define 
\begin{align*}
\widehat {\mathbf\Phi}_{l}^{f,\xi}(\phi|\beta;g)&=\int_{(c)} \beta(s)\,D_{*}(s)\,L^{*}(\cU,-s)\,{\mathbf \Phi}_{l}^{f,\xi}(\phi|s;\,g)\,\d s, 
\end{align*}
where $(c)$ denotes the contour $\Re(s)=c$ with $c>0$ and $L^{*}(\cU,s)$ the modified $L$-function of $\cU$ (see \S\ref{sec-RankinSelbergInt}). Note that even when $\phi={\rm ch}_{\bK_\fin^*}$ the function $\hat{\mathbf\Phi}_{l}^{f,\xi}(\phi|\beta)$ is different from $\hat{\mathbf \Phi}_{l}^{f,\xi}(\beta,z)$ in that the factor $\beta(s)/(s-z)$ in the integrand of \cite[(7.6)]{Tsud2011-1} is now replaced with $\beta(s)$.

\subsection{Poincar\'{e} series} 
From Lemma~\ref{RealShintani-equi-L} and \eqref{pAdicTF-equiv}, the functions ${\mathbf\Phi}_{l}^{f,\xi}(\phi|s)$ and $\hat{\mathbf\Phi}_{l}^{f,\xi}(\phi|\beta)$ are left $\sP^\xi(\Q)$-invariant. Thus the Poincar\'{e} series 
\begin{align}
\hat{\F}_{l}^{f,\xi}(\phi|\beta;\,g)=\sum_{\gamma \in \sP^\xi(\Q)\bsl \sG(\Q)}
 \hat{\mathbf\Phi}_{l}^{f,\xi}(\phi|\beta;\,\gamma g), \qquad g\in \sG(\A)
 \label{PrS}
\end{align}
is well-defined if it is absolutely convergent. For the proof of convergence, we construct a majorant of ${\mathbf\Phi}_{l}^{f,\xi}(\phi|\beta)$, following the same line of the proof of \cite[Lemma 47]{Tsud2011-1}. Let us recall some definitions necessary to describe the majorant. For any prime $p$ and $n\in \N$, set $G_p(\xi;n)=\{g\in \sG(\Q_p)|g^{-1}\xi \in p^{-n}(\cL_{p}^{*}-p\cL_p^{*})\}$ and $a_{p,n}=\sm(1;\diag(p^{-n},1_{m-2},p^{n}))\in G_p(\xi;n)$. Then, $\sG(\Q_p)$ is evidently a disjoint union of open sets $G_p(\xi,n)\,(n\in \N)$; from \cite[Proposition 2.7]{MS98}, 
\begin{align}
G_p(\xi;n)=\sG^{\xi}(\Q_p)a_{p,n}\bK_{p}^{*}\quad (n\in \N^{*}), \quad G_p(\xi;0)=\sG^{\xi}(\Q_p)\bK_p.
 \label{pAdicCartanDec}
\end{align}
Note that $G_p(\xi;n)=\sG^{\xi}(\Q_p)a_{n,p}\bK_p$ for $n\in \N^*$ is also true due to the obvious right $\bK_p$-invariance of $G_p(\xi;n)$. Set $\fA=\{a_{\fin}=(a_{p,n_{p}})_{p\in \fin}|\,n_{p}=0\,\text{for almost all $p$}\}$. Then 
\begin{align}
\sG(\A_\fin)=\sG^\xi(\A_\fin)\fA\,\bK_\fin=\bigsqcup_{a_\fin \in \fA}\sG^\xi(\A_\fin) a_\fin \bK_{\fin}.
 \label{finCartanDec}
\end{align}
Let $A^{+}=\{a_{\infty}^{(t)}|\,t\in \R\}$ be a one-parameter subgroup of $\sG(\R)$ defined as 
\begin{align*}
&a_{\infty}^{(t)}\eta_1^{+}=(\ch t)\eta_1^{+}+(\sh t)\xi_{0}^{-}, \quad a_{\infty}^{(t)}\xi_0^{-}=(\sh t)\eta_1^{+}+(\ch t)\xi_{0}^{-}, \\
&a_\infty^{(t)}\eta=\eta\quad (\eta \in <\eta_1^{+},\xi_0^{-}>_{\R}^{\bot}).
\end{align*}
Then $\sG(\R)=\sG^\xi(\R)A^{+}b_\infty \bK_\infty$ (\cite[Lemma 20]{Tsud2011-1}). Combining this with \eqref{finCartanDec}, we have a decomposition $\sG(\A)=\sG^\xi(\A)A^{+}b_\infty\, \fA\,\bK_{\infty}\bK_\fin$ by which a general point $g\in \sG(\A)$ is written as
\begin{align}
&g=ha_\infty^{(t)}b_\infty a_{\fin}\,k_\infty k_\fin
 \label{elementgCartanDec}
\end{align}
with $h\in \sG^{\xi}(\A)$, $a_\infty^{(t)}\in A^{+}$, $a_\fin=(a_{p,n_p})_{p\in \fin}\in \fA$ and $k_\infty k_\fin\in \bK_\infty\bK_\fin$.
Recall the definition of the majorant $\Xi_{q_1,q_2,N}^{S}:\sG(\A)\rightarrow \R_+$ introduced in \cite[\S 6]{Tsud2011-1}, where $q_1,q_2,N\in \R$ and $S$ is a finite set of prime numbers. From \cite[(6.5) and (6.6)]{Tsud2011-1}, the value $\Xi_{q_1,q_2,N}^{S}(g)$ at a point \eqref{elementgCartanDec} is given as
\begin{align*}
\Xi_{q_1,q_2,N}^{S}(g)
&=\inf(\|h^{-1}\e_1\|_{\A}^{q_1},\|h^{-1}\e_1\|_{\A}^{-q_2}) \\
&\quad \times \{|Q[\xi]|^{1/2}(\cosh 2t)^{1/2}\}^{-N}\prod_{p\in S}p^{-Nn_p}\prod_{p\in \fin-S}\delta_{n_p,0},
\end{align*}
where $\|\cdot\|_\A$ is the norm function on $V(\R)\times V(\A_\fin)^\#$ defined in \S~\ref{NormFTN}. Without loss of generality, we may suppose that $\phi$ is decomposable function ,i.e., $\phi(g)=\prod_{p}\phi_p(g_p)$ with $\phi_p\in C_{\rm c}^{\infty}(\bK_p^{*}\bsl \sG(\Q_p)/\bK_p^{*})$ with $\phi_p={\rm ch}_{\bK_p}$ for almost all $p$. 
\begin{lem} \label{MMajorant}
Let $r$ be a real number such that $0<r<l-\rho$. Let $S$ be a finite set of prime numbers such that $\phi_p={\rm ch}_{\bK_p}$ and $\bK_p=\bK_p^{*}$ for all $p\in \fin -S$. Then for any small $\epsilon >0$, 
\begin{align*}
|\hat{\mathbf\Phi}_{l}^{f,\xi}(\phi|\beta;g)|\ll_{\epsilon} \Xi_{-(\rho+\epsilon),r+\rho-\epsilon,l-r-\rho}^{S}(g), \quad g\in \sG(\A).
\end{align*}
\end{lem}
\begin{proof}
Let $g\in \sG(\A)$ with the decomposition \eqref{elementgCartanDec}. Then 
\begin{align}
{\mathbf \Phi}_l^{f,\xi}(\phi|s;g)
&=
\int_{\sG^\xi(\A_\fin)}\sf^{(s)}(x)\phi(x^{-1}h_\fin a_\fin k_\fin)\, \d x\times \Phi_l^{f,\xi}(s;h_\infty a_\infty^{(t)}b_\infty k_\infty)
 \notag
\\
&=
\int_{\sG^\xi(\A_\fin)}\sf^{(s)}(h_\fin x)\phi(x^{-1}a_\fin k_\fin)\, \d x\times |t_\infty|_\infty^{s+\rho}\Phi_l^{f,\xi}(s;k_{0,\infty} a_\infty^{(t)}b_\infty k_\infty),
 \label{MMajorant-f1}
\end{align}
where $h_\infty=\sm(t_\infty;h_{0,\infty})n_{0,\infty}\,k_{0,\infty}$ $(t_\infty>0,\,h_{0,\infty}\in \sG^\xi_1(\R),\,n_{0,\infty}\in \sN^\xi(\R),\,k_{0,\infty}\in \bK_{\infty}^{\xi})$ is the Iwasawa decomposition of the point $h_\infty \in \sG^\xi(\R)$. Let $p\in S$; since $G_p(\xi;n)\,(n\in \N)$ is an open covering of $\sG(\Q_p)$ and ${\rm supp}(\phi_p)$ is a compact subset of $\sG(\Q_p)$, there exists $k_p\in \N$ such that ${\rm supp}(\phi_p)\subset \bigcup_{0\leq n\leq k_p}G_{p}(\xi;n)$. From this, for any $N>0$ there exists constants $C,C_N'>0$ such that 
\begin{align*}
\left|\int_{\sG^\xi(\A_\fin)}\sf^{(s)}(h_\fin x)\phi(x^{-1}a_\fin k_\fin)\, \d x\right| &\leq C |t_\fin|_\fin^{s+\rho}\,\prod_{p\in S}\delta(0\leq n_p\leq k_p)\prod_{p\in \fin-S}\delta_{n_p,0} \\
&\leq C_N' |t_\fin|_\fin^{s+\rho}\,\prod_{p\in S}p^{-Nn_p}\prod_{p\in \fin-S}\delta_{n_p,0}
\end{align*}
for $s\in \C$, $h_\fin \in \sG^\xi(\A_\fin)$ and $a_{\fin}=(a_{p,n_p})_{p\in \fin}\in \fA$, where $h_\fin=\sm(t_\fin;h_{0,\fin})n_{0,\fin}k_{0,\fin}$ $(t_\fin \in \A_\fin^\times,\,h_{0,\fin}\in \sG_1^{\xi}(\A_\fin),\,n_{0,\fin}\in \sN^\xi(\A_\fin),\,k_{0,\fin}\in \bK_\fin^{\xi*})$ is the Iwasawa decomposition of the point $h_\fin \in \sG^\xi(\A_\fin)$. Combining this with the majorization of $\Phi_l(k_{0,\infty}a_\infty^{(t)}b_\infty k_\infty)$ shown in \cite[Lemma 27]{Tsud2011-1}, we have
\begin{align}
|{\mathbf \Phi}_{l}^{f,\xi}(\phi|s;g)|\ll \exp\left(\tfrac{\pi}{2}|\Im (s)|\right)\,\|h^{-1}\e_1\|_\A^{-(s+\rho)}\,(\ch t)^{\Re s+\rho-l}\,\prod_{p\in S}p^{-Nn_p}\prod_{p\in \fin-S}\delta_{n_p,0}, \quad (\Re s>0)
 \label{MMajorant-f2}
\end{align}
with the implied constant independent of $s$. Since $\int_{(c)}|\beta(s)|\exp\left(\frac{\pi}{2}|\Im s|\right)\,|\d s|<+\infty$, by taking the contour $\Re s=c$ with $c=r-\epsilon$ and $c=r-\epsilon$ with a small $\epsilon>0$, we obtain the required majorization.
\end{proof}

\begin{lem}\label{MConvergenceL}
Let $l\in \N$, $l>4\rho+1(=2m-1)$ and $\cN \subset \sG(\A)$ be a compact set. Then the series \eqref{PrS} converges absolutely uniformly on $\cN$. 
\end{lem}
\begin{proof}
This follows from Lemma~\ref{MMajorant} with the aid of \cite[Lemma 42]{Tsud2011-1} applied with $q_1=-(\rho+\e)$, $q_2=r+\rho-\epsilon$ and $N=l-r-\rho$. For this argument to work, we need $q_1>1-m$, $q_2>m-1$ and $N>m$ which are satisfied with some $r$ and a small $\epsilon>0$ if $l>4\rho+1$. 
\end{proof}

\begin{prop} \label{HolomrphyPhi}
There exists a positive integer $l_0\,(>4\rho+1)$ such that ${\hat{\Bbb F}_l^{\xi,f}}(\phi|\beta)$ belongs to the space $\fS_l(\bK_\fin^*)^{+}$ if $l\geq l_0$
\end{prop}
The same proof as \cite[Theorem 53]{Tsud2011-1} works with a minor change. We review the proof giving necessary modifications. We consider the series 
\begin{align*}
\hat\EE_l^{\xi,f}(\phi|\beta;g)=\sum_{\delta \in \sP^\xi(\Q)\bsl \sG(\Q)}\hat {\mathbf \Phi}_l^{f,\xi}(\phi|\beta;\delta g), \quad g\in \sG(\A),
\end{align*}
which is absolutely convergent normally on $\sG(\A)$ from Lemma~\ref{MConvergenceL}, and yet another series
\begin{align}
\EE_l^{\xi,f}(\phi|s;g)=\sum_{\delta \in \sP^\xi(\Q)\bsl \sG(\Q)}{\mathbf \Phi}_l^{f,\xi}(\phi|s;\delta g), \quad g\in \sG(\A),
 \label{HolomrphyPhi-f0}
 \end{align} 
depending on $s\in \C$, whose absolute convergence on $\Re s>\rho$ follows from the majorization \eqref{MMajorant-f2}. We have the formula
\begin{align}
\EE_l^{\xi,f}(\phi|s;g)=\sum_{d\in \N} \biggl\{\int_{\sG^\xi(\A_\fin)} \EE(f,s,\varphi_d;hx)\phi(x^{-1}a_\fin)\d x\biggr\}\, J(k_\infty,\fz_0)^{-l}\,\Phi_{l,d}^\xi(s;t), 
 \label{HolomrphyPhi-f1}
\end{align}
for any point $g \in \sG(\A)$ of the form \eqref{elementgCartanDec}, where $\EE(f,s,\varphi_d)$ is the Eisenstein series on $\sG^\xi(\A)$ defined in \cite[\S 4.6]{Tsud2011-1} and the function $\Phi_{l,d}^\xi(s;t)$ is by \cite[(4.25)]{Tsud2011-1}. By the same proof as \cite[Lemma 50]{Tsud2011-1}, the expression \eqref{HolomrphyPhi} is shown to be absolutely convergent for $\Re s>\rho$. Then by taking the contour integral on $\Re s=c$ with $c>\rho$, we obtain  
\begin{align}
\hat\EE_{l}^{\xi,f}(\phi|\beta;g)=J(k_\infty,\fz_0)^{-l}
\sum_{d\in \N} \int_{(c)}\beta(s)D_{*}(s) \biggl\{\int_{\sG^\xi(\A_\fin)}
 \EE^{*}(f,s;\varphi_d;hx) \phi(x^{-1}a_\fin)\,\d x\biggr\} \,\Phi_{l,d}^\xi(s;t)\,\d s
 \label{HolomrphyPhi-f2}
\end{align}
({\it cf}. \cite[Lemma51]{Tsud2011-1}). 
\begin{lem} \label{HolomorphyPhi-L1}
There exists $l_0>3\rho+1$ such that for any $l\geq l_0$, 
$$ \hat\EE_l^{\xi,f}(\phi|\beta)\in L^{1}(\sG^{\xi}(\Q)\bsl \sG(\A)).
$$
\end{lem}
\begin{proof} We mimic the proof of \cite[Lemma 52]{Tsud2011-1} step by step. By \cite[(5.1)]{Tsud2011-1} and \eqref{HolomrphyPhi-f1}, $J(k_\infty,\fz_0)^{l}\EE_{l}^{\xi,f}(\phi|\beta)$ is a sum of the following three integrals: 
\begin{align*}
E_{+}(g)&=\sum_{d\in \N} \int_{(c)}\beta(s)D_{*}(s) L^{*}(\cU,-s) \biggl\{\int_{\sG^\xi(\A_\fin)}
 \sf_{\varphi_d}^{(s)}(hx) \phi(x^{-1}a_\fin)\,\d x\biggr\} \,\Phi_{l,d}^\xi(s;t)\,\d s, \\
E_{-}(g)&=\sum_{d\in \N} \int_{(c)}\beta(s)D_{*}(s) L^{*}(\cU,s) \biggl\{\int_{\sG^\xi(\A_\fin)}
 \sf_{\varphi_d}^{(-s)}(hx) \phi(x^{-1}a_\fin)\,\d x\biggr\} \,\Phi_{l,d}^\xi(s;t)\,\d s, \\
E_{0}(g)&=\sum_{d\in \N} \int_{(c)}\beta(s)D_{*}(s) L^{*}(\cU,-s) \biggl\{\int_{\sG^\xi(\A_\fin)}\EE_{{\rm NC}}^{*}(f,s;\varphi_d;hx)\phi(x^{-1}a_\fin)\,\d x\biggr\} \,\Phi_{l,d}^\xi(s;t)\,\d s,
\end{align*} 
where $c>\rho$. Let $\epsilon>0$ be a small number and set $c_1=\rho-\epsilon$, and fix any real number $N$ such that $N>m/4$. By shifting the contour from $\Re s=c$ to $\Re s=c_1$ with a small $\epsilon>0$, an argument similar to the proof of \cite[Lemma 52]{Tsud2011-1} shows the majorization
\begin{align*}
|E_{\pm }(g)|\ll \|h^{-1}\e_0\|_\A^{-(c_1+\rho)}(\ch t)^{c_1+2N+\rho-l}
\prod_{p\in S}\delta(0\leq n_p\leq k_p)\prod_{p\in \fin -S}\delta_{n_p,0},
\end{align*}
where $S$ and $k_p\,(p\in S)$ are as in the proof of Lemma~\ref{MMajorant}. From this, $\int_{\sG^\xi(\Q)\bsl \sG(\A)}|E^{\pm}(g)|\d g <+\infty$ is shown for any $l>c_1+2N+\rho+1$. To estimate $E_0(g)$, we keep the original contour $\Re s=c>\rho$. In a similar argument as in the proof of \cite[Lemma 52]{Tsud2011-1}, we see that for any $q>0$ there exists an $r\in \N$ such that 
\begin{align*}
|E_0(g)|\ll \|h^{-1}\e_0\|_{\A}^{-(c+\rho-2q)}(\ch t)^{c+\rho+r+1-l}
\prod_{p\in S}\delta(0\leq n_p\leq k_p)\prod_{p\in \fin -S}\delta_{n_p,0}
\end{align*}
for any $g$ of the form \eqref{elementgCartanDec}; then choosing any $q$ such that $q>(c-\rho)/2$ and letting $r_0$ be the corresponding $r\in \N$, the convergence of $\int_{\sG^\xi(\Q)\bsl \sG(\A)}|E_0(g)|\d g$ for $l>c+3\rho+r_0+2$ is shown as in the last part of the proof of \cite[Lemma 52]{Tsud2011-1}. 
\end{proof}
Having Lemma~\ref{HolomorphyPhi-L1}, we follow the argument in \cite[\S 7.5.2]{Tsud2011-1} verbatim to show the following.  
\begin{prop} \label{HolomorphyPhi-L2}
There exists an $l_0>4\rho+1$ such that for any $l\in \N$ with $l\geq l_0$, 
$$\hat{\mathbb F}_l^{\xi,f}(\phi|\beta)\in L^1(\sG(\Q)\bsl \sG(\A)).$$
\end{prop}
Proposition~\ref{HolomrphyPhi} follows from Lemmas~\ref{MConvergenceL} and \ref{HolomorphyPhi-L2}. Indeed, by the absolute convergence the left $\sG(\Q)$-invariance is immediate from the definition \eqref{PrS}. Since the function $\hat {\mathbf \Phi}_l^{f,\xi}(\phi|\beta)$ itself has the $\bK_\fin^*\tilde \bK_\infty^+$-equivariance$$
\hat {\mathbf \Phi}_l^{f,\xi}(\phi|\beta;gk_\fin k_\infty)=J(k_\infty,\fz_0)^{-l} 
\hat {\mathbf \Phi}_l^{f,\xi}(\phi|\beta;g), \quad k_\fin\in \bK_\fin^*,\,k_\infty\in \tilde \bK_\infty^{+}
$$
and satisfies the Cauchy-Riemann condition $R(\fp^{-})\hat {\mathbf \Phi}_l^{f,\xi}(\phi|\beta;g)=0$, the same properties are inherited by $\hat{\mathbb F}_l^{\xi,f}(\phi|\beta)$. Then we finish the proof by Proposition~\ref{HolomorphyPhi-L2} and \cite[Theorem 2]{Tsud2011-1}.

\subsection{The spectral expansion} \label{subsecSPECEXP}
Let $\cB_l^{+}$ be an orthonormal basis of $\fS_l(\bK_\fin^*)^{+}$ fixed in \S~\ref{SpectralParameter}. In view of Lemma~\ref{AppFE}, we define an element of $\fS_l(\bK_\fin^*)^{+}$ depending holomorphically on $s\in \C$ by 
\begin{align}
\F_l^{\xi,f}(\phi|s;g)=
-2\pi^{m/2}\Gamma(m/2)^{-1}\,C_l^\xi\,B_l^{\xi}(s)\,\sum_{F\in \cB_l^{+}}D_*(s)\overline{L(F,\bar s+1/2)\,a_{{F}}^{\bar f}(\xi)}\,\lambda_F(\phi)\,F(g),
 \label{FFSpectExp}
\end{align}
where $C_l^\xi$ is the constant \eqref{ZetaEuler-f2} and $B_l^\xi(s)$ is the entire function defined by \cite[(4.29)]{Tsud2011-1}. Note that the element from $\fS_l(\bK_\fin^*)^{-}$ has no contribution to \eqref{FFSpectExp} from Lemma~\ref{SignF-Lem}.

From now on, we fix an integer $l_0$ as in Proposition~\ref{HolomrphyPhi}.
\begin{prop} \label{SpectralExp-P}
Suppose $l\geq l_0$. Then ${\hat{\Bbb F}_l^{\xi,f}}(\phi|\beta)\in \fS_l(\bK_\fin^*)^{+}$ has the following contour integral expression.
$$ 
{\hat{\Bbb F}_l^{\xi,f}}(\phi|\beta;g)=\int_{(c)}\beta(s)\,\F_l^{\xi,f}(\phi|s;g)\,\d s
$$
with $c>\rho$. 
\end{prop}
\begin{proof} From Proposition~\ref{HolomrphyPhi}, we have the expansion
\begin{align*}
{\hat{\Bbb F}_l^{\xi,f}}(\phi|\beta;g)=\sum_{F\in \cB_l^+}\langle 
{\hat{\Bbb F}_l^{\xi,f}}(\phi|\beta)|F\rangle_{\sG}\,F(g), \quad g\in \sG(\A).
\end{align*}
We compute the inner-product $\langle 
{\hat{\Bbb F}_l^{\xi,f}}(\phi|\beta)|F\rangle_{\sG}$ following the long displayed formula in the proof of \cite[Proposition 56]{Tsud2011-1}. The first 5 lines of the computation up to \cite[(8.4)]{Tsud2011-1} is the same, which gives us the equality
\begin{align*}
\langle 
{\hat{\Bbb F}_l^{\xi,f}}(\phi|\beta)|F\rangle_{\sG}=\int_{(c)} \beta(s)D_*(s)L^{*}(\cU,-s)\,\biggl\{
\int_{\sG^\xi(\A)\bsl \sG(\A)}\d g\int_{\sG^\xi(\Q)\bsl\sG^\xi(\A)}\EE_l^{\xi,f}(s;hg)\bar F(hg)\,\d h
\biggr\}\,\d s,
\end{align*}
where $c>\rho$ and $\EE_{l}^{\xi,f}(\phi|s)$ is defined by the series \eqref{HolomrphyPhi-f0}. By the decomposition \eqref{elementgCartanDec}, the Haar measure $\d g$ on $\sG(\A)$ is decomposed as
$$
\d g=c_{\sG}\,\d h\,(\sh t)^{m-1}(\ch t)\,\d t\,\d\mu(a_\fin)\, \d k
$$
where $c_\sG=2\pi^{m/2}\Gamma(m/2)^{-1}$, $\d h$ is the Haar measure on $\sG^\xi(\A)$, $\d k$ the Haar measure on $\bK_\fin^*\bK_\infty$ with total mass $1$, and $\d\mu(a_\fin)$ is a measure on the discrete space $\fA$ (see \cite[Lemma 20]{Tsud2011-1}). Substituting the expression \eqref{HolomrphyPhi-f1}, we compute the inner-integral of the last formula as follows.
\begin{align}
&\int_{\sG^{\xi}(\Q)\bsl \sG(\A)}\EE^{\xi,f}_l(\phi|s;g)\,\bar F(g)\,\d g
 \label{SpectralExp-P-1}
 \\
&=\int_{\sG^\xi(\A)\bsl \sG(\A)}\d g\int_{\sG^\xi(\Q)\bsl\sG^\xi(\A)}
\sum_{d\in \N} \biggl\{\int_{\sG^\xi(\A_\fin)}\EE(f,s;\varphi_d;hx)\phi(x^{-1}a_\fin)\,\d x\biggr\}\,\Phi_{l,d}^\xi(s;t)\,\bar F(ha_\infty^{(\infty)}b_\infty a_\fin)\,\d h 
 \notag
\\
&=c_{\sG}\sum_{d\in \N} \int_{0}^{\infty}(\sh t)^{m-1}(\ch t)\,\Phi_{l,d}^\xi(s;t)\d t\, 
 \notag
\\
&\quad \times \int_{\fA}\d \mu(a_\fin)\,\int_{\sG^\xi(\A_\fin)}\phi(x^{-1}a_\fin)\,
\biggl\{\int_{\sG^\xi(\Q)\bsl \sG^\xi(\A)}\EE(f,s;\varphi_d;h x)\,\bar F(h a_\infty^{(t)}b_\infty a_\fin) \,\d h\biggr\}\,\d x.
\notag
\end{align}
Since the proof of \cite[Proposition 29]{Tsud2011-1} only needs \cite[Proposition 24]{Tsud2011-1},  the same proof applied to the function $g\mapsto F(gx^{-1}a_\fin)$ with $x\in \sG^\xi(\A_\fin)$ and $a_\fin \in \fA$ shows the formula 
\begin{align*}
&\int_{\sG^\xi(\Q)\bsl \sG^\xi(\A)}\EE(f,s;\varphi_d;h x)\,\bar F(h a_\infty^{(t)}b_\infty a_\fin) \,\d h
\\
&=
L^{*}(\cU,-s)^{-1}\overline{\Phi_{l,d}^{\xi}(-\bar s;t)}\,\int_{\sG^\xi(\Q)\bsl \sG^\xi(\A)}E^{*}(f,s;h)\bar F(hb_\infty x^{-1}a_\fin)\d h 
\end{align*}
for $\Re s>\rho$. Substituting this to the last expression of \eqref{SpectralExp-P-1} and using the relation
\begin{align*}
\int_{\fA}\d\mu(a_\fin)\int_{\sG^\xi(\A_\fin)}\phi(x^{-1}a_\fin)\,\bar F(hb_\infty x^{-1}a_\fin)\d h&=\int_{\sG(\A_\fin)} \phi(g_\fin)\,\bar F(hb_\infty g_\fin)\,\d g_\fin\\
&=(\bar F*\phi)(hb_\infty) =\lambda_F(\phi)\,\bar F(hb_\infty), 
\end{align*}
we obtain the following   
\begin{align*}
\langle 
{\hat{\Bbb F}_l^{\xi,f}}(\phi|\beta)|F\rangle_{\sG}
=c_{\sG}\,\lambda_{F}(\phi)\,\int_{(c)}\beta(s)D_*(s)\,B_{l}^{\xi}(s)
\,\overline{Z_{F}^{\bar f*}(\bar s)}\,\d s
\end{align*}
with 
$$
B_l^{\xi}(s)=\sum_{d\in \N}\int_{0}^\infty(\sh t)^{m-1}(\ch t)\,\Phi_{l,d}^\xi(s;t)\,\overline{\Phi_{l,d}^\xi(-\bar s;t)}\,\d t. 
$$
To complete the proof, we use \cite[Propositions 30 and 17]{Tsud2011-1}.  
\end{proof}

Define $\hat \II_l^{\xi,f}(\phi|\beta,r)$ and $\II_{l}^{\xi,f}(\phi|s;r)$ to be the $(\xi,\bar f)$ Whittaker-Bessel function of $\hat\F_{l}^{\xi,f}(\phi|\beta)$ and $\F_l^{\xi,f}(\phi|s)$ evaluated at the point $\sm(r;1_m)$, respectively (see \S\ref{sec-RankinSelbergInt}), i.e., 
\begin{align}
\II_l^{\xi, f}(\phi|s;r)=\int_{\sG_1^\xi(\Q)\bsl \sG_1^\xi(\A)}\bar f(h_0)\int_{\sN(\Q)\bsl \sN(\A)}\hat \F_l^{\xi,f}(\phi|s;n\sm(r;h_0)b_\infty)\,\psi_\xi(n)^{-1}\d n, 
\label{FrC}
\end{align}
where $\psi_\xi:\sN(\A)\rightarrow \C^\times$ is a character defined by \eqref{AutCharsN}. From \eqref{WHBesCoeff}, we have $\II_{l}^{\xi,f}(\phi|s;r)=\II_l^{\xi,f}(\phi|s)\, \cW_l^\xi(\sm(r;1_m)b_\infty)$ with  
\begin{align}
\II_l^{\xi,f}(\phi|s)=-2\pi^{m/2}\Gamma(m/2)^{-1}\,C_l^\xi\,B_l^{\xi}(s)\,\sum_{F\in \cB_l^{+}}D_*(s)\lambda_F(\phi)\,\overline{L(F,\bar s+1/2)\,|a_{F}^{\bar f}(\xi)}|^2.  
 \label{IISpectEx-f0}
\end{align}
Since $f$ is fixed throughout this article, we abbreviate $\hat\II_l^{\xi,f}(\phi|\beta;r)$ and $\II_l^{\xi,f}(\phi|s)$ to $\hat\II_l^\xi(\phi|\beta;r)$ and $\II_l^\xi(\phi|s)$, respectively.

For any set $I\subset \R$ and $\delta>0$, set 
$$\cT_{I,\delta}=\{s\in \C|\,|\Im s|\geq \delta,\,\Re s\in I\,\}.$$ 
A holomorphic function $h(s)$ on $\C$ is said to be vertically of exponential growth provided that for any compact interval $I\subset \R$ and $\delta>0$ there exists a constant $a>0$ such that $|h(s)|\ll \exp(a|\Im(s)|)$ on $s\in \cT_{I,\delta}$ holds. For example, Stirling's formula shows that a function of the form $\prod_{j=1}^{r}\Gamma(a_js+b_j)^{-1}$ $(r\in \N,\,a_j,b_j \in \C)$ is vertically of exponential growth. 

\begin{lem} \label{IISpectEx} Let $\eta\in \Q^\times\, \xi$. 
The function $\II_l^{\xi,f}(\phi|s)$ is entire and satisfies the functional equation $\II_l^{\xi,f}(\phi|s)=\II_l^{\xi,f}(\phi|-s)$; moreover, it is bounded on any vertical strip with finite width. We have the contour integral expression
\begin{align}
\hat \II_{l}^{\xi,f}(\phi|\beta;r)=\int_{(c)}\beta(s)\II_{l}^{\xi,f}(\phi|s)\,\d s\times \cW_l^{\xi}(\sm(r;1_m)b_\infty) \quad (s>\rho). 
\label{IISpectEx-f2}
\end{align}
\end{lem}
\begin{proof}
Since $D_*(s)=D_{*}(-s)$ is evident from definition, the functional equation of $\II_l^{\xi,f}(\phi|s)$ follows from the functional equation of $a_{F}^{f}(\xi)L(F,s)$ (Lemma~\ref{AppFE}), and the symmetry $B_l^\xi(s)=B_l^{\xi}(-s)$ shown in \cite[Proposition 30 (1)]{Tsud2011-1}. From Lemma~\ref{AppFE} and by Stirling's formula, we have the bound $|a_{f}^{F}(\xi)L(F,s+1/2)|\ll \exp(-a\pi|\Im s|)$ on $\cT_{I,\delta}$ with $a=(1+[m/2])/2$, which combined with the estimate of $|B_l^{\xi}(s)|\ll \exp(\pi|\Im(s)|)$ (\cite[Proposition 30 (1)]{Tsud2011-1}) yields the bound $|\II_{l}^{\xi,f}(\phi|s)|\ll \exp((-a+1)\pi |\Im(s)|)$ on any $\cT_{I,\delta}$. Since $a\geq 1$, this shows that $\II_l^{\eta}(\phi|s)$ is bounded on any vertical strip of finite width. The same argument is applied to $\F_l^{\xi,f}(\phi|s;g)$ to yields the bound $|\F_l^{\xi,f}(\phi|s;g)|\ll \exp((-a+1)\pi|\Im s|)\,(s\in \cT_{I,\delta})$ with the implied constant independent of $g$ in a compact set of $\sG(\A)$. Since the integration domain $\sG^{\xi}_1(\Q)\bsl \sG^\xi_1(\A)\times \sN(\Q)\bsl \sN(\A)$ of $\hat\II_{l}^{\xi,f}(\phi|\beta;r)$ is compact, by Fubini's theorem we exchange the order of integrals to obtain 
\begin{align*}
\hat\II_l^{\xi,f}(\phi|\beta;r)=\int_{(c)}\beta(s)\,\biggl\{\int_{\sG_1^\xi(\Q)\bsl \sG_1^\xi(\A)}\bar f(h_0)\int_{\sN(\Q)\bsl \sN(\A)}\F_l^{\xi,f}(\phi|s;n\sm(r;h_0)b_\infty)\,\psi_\eta(n)\d n  \biggr\}\,\d s. 
\end{align*}
By \eqref{FrC}, we are done.   
\end{proof}

\section{Proof of the main theorem and other consequences} \label{ProofMTHM}
Let $l_0$ be an integer as in Proposition~\ref{HolomrphyPhi}. We fix $l\in \N$ such that $l\geq l_0$. We continue to work with an irreducible $\sG_1^\xi(\A_\fin)$-submodule $\cU\subset \cV(\xi)$ containing a non-zero function $f\in \cU(\bK_{1,\fin}^{\xi*})$ such that $\tau_\fin^\xi(f)=\epsilon_f\,f$. Let $\{(z_p^{\cU},\rho_p^{\cU})\}_{p\in \fin}$ be the spectral parameter of $\cU$ and set $\rho_\fin^{\cU}:=\bigotimes_{p \in \fin}\rho_p^{\cU}$ (see \S\ref{subsubsec: cU-def}). 
Let $\phi \in \cH^{+}(\sG(\A_\fin)\sslash \bK_\fin^*)$ and $\beta\in \cBB^+$. Let $S$ be a finite set of prime numbers satisfying the condition \eqref{conditionS}, so that $\cL_p=\cL_p^*$ and $\cL_{1,p}^{\xi}=\cL_{1,p}^{\xi*}$ (from Lemma~\ref{IndexL}) and thus $\bK_p=\bK_p^*$ and $\bK_{1,p}^\xi=\bK_{1,p}^{\xi*}$ for all $p\in S$. The assumption on $\phi$ that we impose from now on is  
\begin{align}
\text{$\phi(g)=\prod_{p\in \fin}\phi_{p}(g_p)$ with $\phi_p\in \cH^{+}(\sG(\Q_p)\sslash \bK_p^*)$ such that $\phi_p={\rm ch}_{\bK_p}$ for all $p\not\in S$}
 \label{As-phi}
\end{align}
We analyze the function $\hat \II_l^{\xi,f}(\phi|\beta;r)$. From \eqref{PrS} and Proposition~\ref{DC}, 
\begin{align}
\hat\II_l^{\xi,f}(\phi|\beta;r)&=\sum_{\fu\in \{1,\fw_0,\bar\sn(\e),\bar\sn(\xi)\}}\hat\JJ_{l}^{\xi,f}(\fu,\phi|\beta;r) 
 \notag
\\
&=\hat\JJ_{l}^{\xi,f}(1,\phi|\beta;r)+\hat\JJ_{l}^{\xi,f,\rm sing}(\fw_0,\phi|\beta;r)+\hat \JJ_{l}^{\xi,f,\rm reg}(\fw_0,\phi|\beta;r)+\sum_{\fu\in \{\fw_1,\sn(\xi)\sw_0 \}}\hat\JJ_{l}^{\xi,f}(\fu,\phi|\beta;r), 
\label{hatRTF}
\end{align}
where the terms $\hat\JJ_{l}^{\xi,f}(\fu,\phi|\beta;r)$ are defined to be
\begin{align*}
\int_{\sG_1^\xi(\Q)\bsl \sG^\xi_1(\A)} {\bar f}(h_0)\,\d h_0 
\int_{\sN(\Q)\bsl \sN(\A)} \sum_{\gamma \in \sP^\xi(\Q)\bsl[\sP^\xi(\Q)\fu\sP(\Q)]}
 \hat{\mathbf\Phi}_{l}^{f,\xi}(\phi|\beta;\,\gamma\,n\,\sm(r;\,h_0)\,b_\infty)\,\psi_\xi(n)^{-1} \,\d n, 
\end{align*}
and $\hat\JJ_l^{\xi,f}(\fw_0,\phi|\beta;r)$ are written as the sum of $\hat \JJ_{l}^{\xi,f, \rm sing}(\fw_0,\phi|\beta;r)$ and $\hat \JJ_{l}^{\xi,f,\rm reg}(\fw_0,\phi|\beta;r)$ defined as 
\begin{align}
\hat\JJ_{l}^{\xi,f,\rm{sing}}(\fw_0,\phi|\beta;r)=&\int_{\sG_1^\xi(\Q)\bsl \sG^\xi_1(\A)} \bar f(h_0)\,\d h_0\,\int_{\sN(\A)} \sum_{\delta\in \{1,\cnt^{\sG_1}\}}\hat{\mathbf\Phi}_{l}^{f,\xi}(\phi|\beta\,\fw_0\,n\,\sm(r;\,\delta\,h_0)\,b_\infty)\,\psi_{\delta \xi}(n)^{-1} \,\d n
 \label{SINGw0}, \\
\hat\JJ_{l}^{\xi,f,\rm{reg}}(\fw_0,\phi|\beta;r)=&\int_{\sG_1^\xi(\Q)\bsl \sG^\xi_1(\A)} \bar f(h_0)\,\d h_0\sum_{\substack{\delta \in \sG_1^\xi(\Q)\bsl\sG_1(\Q) \\
\delta \{1,\cnt^{\sG_1}\}\cap \sG_1^\xi(\Q)=\emp}}
\int_{\sN(\A)} 
\hat{\mathbf\Phi}_{l}^{f,\xi}(\phi|\beta;\,\fw_0\,n\,\sm(r;\,\delta h_0)\,b_\infty)\,\psi_{\delta\,\xi}(n)^{-1} \,\d n. 
 \label{REGw0}
\end{align}
From Lemmas~\ref{DCN-L1}, we have the disjoint decomposition
$$
\sP^\xi(\Q)\bsl[\sP^\xi(\Q)\fu\sP(\Q)]=\bigcup_{\mu \in \sM(\fu)} \sP^\xi(\Q)\bsl [\sP^\xi(\Q)\fu\mu\sN(\Q)] =\bigcup_{\mu\in \sM(\fu)} \fu \mu \cdot(\sN_{\mu}(\Q)\bsl \sN(\Q))
,$$
where $\sN_{\mu}=\sN\cap (\fu \mu)^{-1}\sP^\xi(\fu \mu)$ and 
$$
\sM(\fu)=
\begin{cases}
\{\sm(1;\,\delta)|\,\delta\in \sG^\xi_1(\Q)\bsl \sG_1(\Q)\,\} \qquad & (\fu\in \{1,\fw_0\}), \\
\{\sm(\tau;\,\delta)|\,\tau\in \Q^\times,\,\delta\in \sP^1_0(\Q)\bsl \sG_1(\Q)\,\} \qquad & (\fu=\sw_1), \\
\{\sm(\tau;\,\delta)|\,\tau\in \Q^\times,\,\delta\in \sG^\xi_1(\Q)\bsl \sG_1(\Q)\,\} \qquad & (\fu=\sn(\xi)\sw_0).
\end{cases}
$$
Applying this, we have 
\begin{align*}
\hat\JJ_{\fu}^{\xi,f}(l;\phi|\beta;r)&=
\int_{\sG_1^\xi(\Q)\bsl \sG^\xi_1(\A)} {\bar f}(h_0)\,\d h_0
\sum_{\mu \in \sM(\fu)}\int_{\sN_{\mu}(\Q)\bsl \sN(\A)} 
\hat{\mathbf\Phi}_{l}^{f,\xi}(\phi|\beta;\,\fu\,\mu\,n\,\sm(r;\,h_0)\,b_\infty)\,\psi_\xi(n)^{-1} \,\d n.
\end{align*}
In the succeeding sections, we shall analyze these integrals further for each coset representative $\fu\in \{1,\fw_0,\fw_1,\sn(\xi)\sw_0 \}$ individually. The terms $\hat\JJ_{l}^{\xi,f}(1,\phi|\beta;r)$ and $\hat \JJ_{l}^{\xi,f,\rm sing}(\fw_0,\phi|\beta;r)$ after summed over $f\in \cB(\cU;\bK_{1,\fin}^{\xi*})$ (see \S\ref{subsubsec: cU-def}) are evaluated exactly under a mild assumption on $\phi$. To describe the results, we need notation.

Let $p\in S$. For a Hecke function $\phi_p\in \cH(\sG(\Q_p)\sslash\bK_p)$, set
\begin{align}
\calW_{p}^{\xi,(z)}(\phi_p;g)=\int_{\sG_1^\xi(\Q_p)}\Omega_{\sG_1^\xi(\Q_p)}^{(z)}(h_0)\d h_0 \int_{\sN(\Q_p)}\phi_p(g^{-1}\sm(1;h_0)^{-1}n^{-1})\,\psi_{\xi,p}(n)\,\d n
, \quad g\in \sG(\Q_p)
 \label{cWphi}
\end{align}
where $\Omega_{\sG_1^{\xi}(\Q_p)}^{(z)}$ is Harish-Chandra's spherical function on $\sG_1^{\xi}(\Q_p)$, and define the Mellin transform of $\calW_p^{\xi,(z)}(\phi_p)$ by 
\begin{align}
\widehat{\calW}_p^{\xi,(z)}(\phi_p;s)=\int_{\Q_p^\times}\calW_p^{\xi,(z)}(\phi_p;\sm(t;1))\,|t|_p^{s-\rho}\,\d^\times t, \quad s\in \C.
 \label{MellincWphi}
\end{align}
Since $\phi_p$ is of compact support on $\sG(\Q_p)$, the Iwasawa decomposition shows that the integral defining $\calW_p^{\xi,(z)}(\phi_p;g)$ is absolutely convergent and that the function $t\mapsto \calW_{p}^{\xi,(z)}(\phi_p;\sm(t;1))$ on $\Q_p^\times$ is of compact support. Thus, \eqref{MellincWphi} converges absolutely for all $s\in \C$ defining an entire function. 
Recall the involutive operator $\tau_{\fin}^{\xi}$ on $\cV(\xi;\bK_{1,\fin}^{\xi*})$ from \S~\ref{subsubsec: xiEigenFtn}. For any $S$ and $\phi \in \cH^{+}(\sG(\A_\fin)\sslash \bK_\fin^{*})$ satisfying \eqref{conditionS} and \eqref{As-phi}, set
\begin{align}
\MM_{l}^{\xi,\cU}(\phi|s)
=&D_*(-s)\,L^*(\cU,-s)\, \delta(2\xi \in \cL_1)\,\biggl\{1+(-1)^{l}\,\frac{\tr(\tau_\fin^{\xi}|\,\cU(\bK_{1,\fin}^{\xi*}))}{\dim(\cU(\bK_{1,\fin}^{\xi*})
)}\biggr\}
\,\frac{(\sqrt{8|Q[\xi]|}\,\pi)^{-s-\rho+l}}{\Gamma(-s-\rho+l)} 
\label{P1-f0} 
\\
&\quad \times \prod_{p\in S}\widehat{\calW}_p^{\xi,(z_p^\cU)}(\phi_p;s).
 \notag
\end{align}
We have the explicit formula of the average of $\hat\JJ_{l}^{\xi,f}(1,\phi|\beta;r)$ and $\hat\JJ_l^{\xi,f,{\rm sing}}(\fw_0,\phi|\beta;r)$ over $f\in \cB(\cU;\bK_{1,\fin}^{\xi*})$ as follows. 
\begin{prop} \label{P1}
The function $s\mapsto \MM_{l}^{\xi,\cU}(\phi|s;r)$ is holomorphic on $\C$ away from a possible simple pole at $s=0$ and is vertically of exponential growth. We have the equalities
\begin{align}
\frac{1}{\dim(\cU(\bK_{1,\fin}^{\xi*})}\sum_{f\in \cB(\cU;\bK_{1,\fin}^{\xi*})} \hat\JJ_l^{\xi,f}(1,\phi|\beta;r)&=\int_{(c)}\beta(s)\,\MM_l^{\xi,\cU}(\phi|s)\,\d s
\times \cW^{\xi}_l(\sm(r;1_m)b_\infty), 
,\label{P1-f1} 
\\
\frac{1}{\dim(\cU(\bK_{1,\fin}^{\xi*})}\sum_{f\in \cB(\cU;\bK_{1,\fin}^{\xi*})}\hat\JJ_l^{\xi,f,{\rm sing}}(\fw_0,\phi|\beta;r)&=\int_{(c)}\beta(s)\,\MM_l^{\xi,\cU}(\phi|-s)\,\d s\times \cW^{\xi}_l(\sm(r;1_m)b_\infty).
\label{P1-f2}
\end{align} 
\end{prop}
The proof of this proposition will be given in \S ~\ref{JJidentity} and \S~\ref{JJw0sing}. 
Next, we describe the results concerning other terms on the right-hand side of \eqref{hatRTF}, which will be proved in \S~\ref{JJw0regular}, \S~\ref{JJbsne}, and in \S~\ref{JJbsnxi}. 

\begin{prop}\label{Prop2} Let $f\in \cB(\cU;\bK_{1,\fin}^{\xi*})$
There exists a holomorphic function $\JJ_{l}^{\xi,f}(\fw_0,\phi|s)$ on the strip  $\Re s \in (\rho,l-3\rho-1)$ such that there exists a constant $C_0>1$ such that for any compact interval $I\subset (\rho,l-3\rho-1)$ and $\delta>0$, 
$$
|\JJ_l^{\xi,f,{\rm reg}}(\fw_0,\phi|s)| \ll \frac{(\sqrt{8\Delta}\pi)^{l-\rho}}{|\Gamma(s+\rho)|\Gamma(l-\rho)}\, C_0^{-l}
$$
uniformly for $s\in \cT_{I,\delta}$ and $l\in \N_{>4\rho+1}^*$, and for any $c\in(\rho,l-3\rho-1)$
$$
\hat\JJ_{l}^{\xi,f,{\rm reg}}(\fw_0,\phi|\beta;r)=\int_{(c)}\beta(s)\,D_*(-s)\,L^*(\cU,-s)\,\JJ_{l}^{\xi,f,{\rm reg}}(\fw_0,\phi|s)\,\d s \times \cW^{\xi}_l(\sm(r;1_m)b_\infty).
$$
\end{prop}

\begin{prop}\label{Prop3} Let $f\in \cB(\cU;\bK_{1,\fin}^{\xi*})$. There exists a holomorphic function $\JJ_{l}^{\xi,f}(\sw_1,\phi|s)$ on the strip $\Re s\in (\rho,l-3\rho-2)$ such that for any compact interval $I\subset (\rho,l-3\rho-2)$, 
$$
|\JJ_l^{\xi,f}(\sw_1,\phi|s)| \ll
 \left|\frac{(2\pi\sqrt{\Delta})^{l}}{\Gamma(-s+l-\rho)}\right|\,
\frac{(N\pi\sqrt{\Delta})^{l}}{\Gamma(l-\rho+1/2)}
$$
uniformly for $s\in \cT_{I,0},\,l\in \N_{>4\rho+1}^*$ with a constant $N>0$; moreover, for any $c\in(\rho,l-3\rho-2)$
$$
\hat\JJ_{l}^{\xi,f}(\sw_1,\phi|\beta;r)=\int_{(c)}\beta(s)\,D_*(-s)\,L^*(\cU,-s)\,\JJ_{l}^{\xi,f} (\sw_1,\phi|s)\,\d s \times \cW^{\xi}_l(\sm(r;1_m)b_\infty). 
$$
\end{prop}

\begin{prop}\label{Prop4}
Let $f\in \cB(\cU;\bK_{1,\fin}^{\xi*})$. There exists $l_1\in \N$ with the following property. For any $l\geq l_1$, there exists a holomorphic function $\JJ_{l}^{\xi,f}(\sn(\xi)\sw_0,\phi|s)$ on the strip $\Re s\in (\rho+1,l-3\rho-1)$ satisfying the estimate
\begin{align*}
|\JJ_l^{\xi,f}(\sn(\xi)\sw_0,\phi|s)|\ll_{\rQ} \frac{(\sqrt{8\Delta}\pi)^{l-\rho}\,(1+|s|)^{2\rho}}{|\Gamma(s+\rho)\,\Gamma(-s+l-\rho)|}\,
\frac{l^{2\rho}(N_1\pi)^{l}}{\Gamma(l-\rho+1/2)}
\end{align*}
for $\Re(s) \in (\rho+1 ,l-3\rho-1)$ and $l\in \N_{>l_1}$ with a constant $N_1>0$; moreover, for any $c\in(\rho+1,l-3\rho-1)$
$$
\hat\JJ_{l}^{\xi,f}(\sn(\xi)\sw_0,\phi|\beta;r)=\int_{(c)}\beta(s)\,D_*(-s)\,L^*(\cU,-s)\,\JJ_{l}^{\xi,f}(\sn(\xi)\sw_0,\phi|s)\,\d s\times \cW^{\xi}_l(\sm(r;1_m)b_\infty).
$$
\end{prop}

Define
\begin{align}
\SS_l^{\xi,\cU}(\phi|s)=\frac{-\Gamma(l-\rho)}{(\sqrt{8\Delta}\pi)^{l-\rho}} 
\left\{\II_{l}^{\xi,\cU}(\phi|s)-\MM_{l}^{\xi,\cU}(\phi|s)-\MM_l^{\xi,\cU}(\phi|-s)
\right\}, s\in \C. 
 \label{ErrorT}
\end{align}
where $\II_l^{\xi,\cU}(\phi|s)$ is defined to be the sum of $\dim(\rho_{\fin}^{\cU})^{-1} \II_l^{\xi,f}(\phi|s)$ (see \eqref{IISpectEx-f0}) over $f\in \cB(\cU;\bK_{1,\fin}^{\xi*}).$

\begin{prop} \label{ErP1} 
We have that $s\mapsto \SS_l^{\xi,\cU}(\phi|s)$ is a holomorphic function of vertically exponential growth on $\C$, which satisfies the functional equation $\SS_{l}^{\xi,\cU}(\phi|-s)=\SS_{l}^{\xi,\cU}(\phi|s)$, $s\in \C$. There exists a constant $C>1$, $q>0$ and $l_1\in \N^*$ such that 
\begin{align}
|e^{s^2}\,\SS_l^{\xi,\cU}(\phi|s)|\ll C^{-l}, \quad \Re s\in [-q,q],\,l\in \N,\,l\geq l_1. 
 \label{ErP1-f0}
\end{align}
\end{prop}
\begin{proof}    
The first two assertions follows from Propositions~\ref{IISpectEx} and \ref{P1}. Note that a possible simple pole at $s=0$ is removable in the expression $\MM_{l}^{\xi,\cU}(\phi|s)+\MM_{l}^{\xi,\cU}(\phi|-s)$. Let $l_1\in \N$ be as in Proposition~\ref{Prop4}. Let $I\subset (\rho+1,l-3\rho-1)$ be any interval. Let $\JJ_{l}^{\xi,\cU,\rm{reg}}(\fw_0,\phi|s)$ be the sum of $\dim(\cU(\bK_{1,\fin}^{\xi*}))^{-1}\JJ_{l}^{\xi,f,\rm{reg}}(\fw_0,\phi|s)$ over $f\in \cB(\cU;\bK_{1,\fin}^{\xi*})$, and define $\JJ_{l}^{\xi,\cU}(\sw_1,\phi|s)$ and $\JJ_l^{\xi,\cU}(\sn(\xi)\sw_0,\phi|s)$ similarly. From Propositions~\ref{Prop2}, \ref{Prop3}, and \ref{Prop4}, the function
$$
A(s)=\SS_l^{\xi,\cU}(\phi|s)+\frac{\Gamma(l-\rho)}{(\sqrt{8\Delta}\pi)^{l-\rho}}D_*(s)L^{*}(\cU,-s)\{\JJ_{l}^{\xi,\cU,\rm reg}(\fw_0,\phi|s)+\JJ_{l}^{\xi,\cU}(\sw_1,\phi|s)+\JJ_l^{\xi,\cU}(\sn(\xi)\sw_0,\phi|s)\}
$$ 
is holomorphic and is vertically of exponential growth on the strip $\Re s\in I$; moreover, $\int_{(c)}\beta(s)\,A(s)\,\d s=0$ for all $\beta\in \cBB^+$, where $c\in I$ is a fixed real number. Thus, the function $t\mapsto e^{(c+it)^2}\,A(c+it)$ on $\R$ is square-integrable and is orthogonal to the set of functions of the form $P(c+it)e^{(c+it)^2}$ ($P(z)\in \C[z]$), which is dense in $L^2(\R)$. Thus, $A(c+it)=0$ for all $t\in \R$. Therefore, $A(s)=0$ or equivalently 
\begin{align*}
\SS_l^{\xi,\cU}(\phi|s)=\frac{-\Gamma(l-\rho)}{(\sqrt{8\Delta}\pi)^{l-\rho}}D_*(s)L^{*}(\cU,-s)
\{\JJ_{l}^{\xi,\cU,\rm reg}(\fw_0,\phi|s)+\JJ_{l}^{\xi,\cU}(\sw_1,\phi|s)+\JJ_l^{\xi,\cU}(\sn(\xi)\sw_0,\phi|s)\}
\end{align*}
identically on the strip $\Re(s)\in (\rho+1,l-3\rho-1)$. First, we shall show the bound \eqref{ErP1-f0} on the strip $\Re(s)\in I$ so that we can use the estimations in Propositions~\ref{Prop2}, \ref{Prop3}, and \ref{Prop4}; thus, $|e^{s^2}\SS_l^{\xi,\cU}(\phi|s)|$ is majorized by 
\begin{align}
&e^{-|\Im(s)|^2}\times |D_*(s)L^{*}(\cU,-s)|n^{2(l+\rho)} \biggl\{\frac{C_0^{-l}}{|\Gamma(s+\rho)|} 
+\frac{\Gamma(l-\rho)\,C_1^l}{|\Gamma(-s+l-\rho)|\Gamma(l-\rho+1/2)}
 \label{ErP1-f1}
\\
&+
\frac{e^{\pi|\Im s|}
\Gamma(l-\rho)\,C_2^{l}}{|\Gamma(-s+l-\rho)|\Gamma(l-\rho+1/2)}
\biggr\}
 \notag
\end{align}
on $\Re(s)\in I$, where $C_0>1$ is a constant and $C_1,C_2$ are positive constants. We have that $D_*(s)L^{*}(\cU,-s)$ is bounded on $\cT_{I,0}$; by Stirling's formula $\Gamma(l-\rho)/\Gamma(l-\rho+1/2)\sim (l-\rho)^{-1/2}=O(1)$ as $l\rightarrow \infty$, and $e^{-|\Im(s)|^2/2}\times e^{{\pi}|\Im s|}|\Gamma(s+\rho)|^{-1}=O(1)$ on $\cT_{I,\delta}$.
Thus \eqref{ErP1-f1} has further majorized by 
\begin{align*}
\frac{e^{-|\Im(s)|^2}\,C_0^{-l}}{\Gamma(s+\rho|}+
\frac{e^{-|\Im(s)|^2}\,C_1^{l}}{|\Gamma(l-s-\rho)|}
+\frac{e^{-|\Im(s)|^2/2}\,C_2^{l}}{|\Gamma(l-s-\rho)|}
\end{align*}
uniformly in $s\in \cT_{I,\delta}$ and $l\in \N\cap(l_0,\infty)$. To estimate the second and the third term of this expression, we apply Lemma~\ref{ErP1-L1}. Thus, we have the following majorant of $n^{-2(l+\rho)}\times e^{s^2}\SS_l^{\xi,\cU}(\phi|s)$ with $s\in \cT_{I,\delta}$ and $l\in \N\cap(l_1,\infty)$ varies. 
\begin{align*}
C_0^{-l}+& 
\biggl\{e^{-|\Im(s)|^2}C_1^{l}+e^{-|\Im(s)|^2/2}C_2^{l}\biggr\}\,\left(\frac{l^{\Re(s)}}{\Gamma(l-\rho)}+\frac{T^{-l}}{|\Gamma(-s-\rho)|}\right). 
\end{align*}
Choose $T>C_0\max(C_1,C_2)$; then by Stirling's formula, this whole expression is majorized by a constant times $C_0^{-l}$. Hence if we set $I=[q',q]$, there exists a constant $B>0$ such that $|e^{s^2}\SS_l^{\xi,\cU}(\phi|s)|\leq B\,C_0^{-l}$ for $\Re s\in [q',q]$ and $l\in \N$, $l\geq l_0$. By the functional equation, the same bound holds true on the strip $\Re s\in [-q,-q']$ also. Since $e^{s^2}\SS_{l}^{\xi,\cU}(\phi|s)$ is vertically of exponential growth on $\C$, the inequality $|e^{s^2}\SS_l^{\xi,\cU}(\phi|s)|\leq B\,C_0^{-l}$ on the union of vertical lines $\Re s=\pm q$ remains valid on the vertical strip $\Re s\in [-q,q]$ by the Phragmen-Lindel\"{o}f principle. 
\end{proof}


\begin{lem} \label{ErP1-L1}
Let $a\in \R$. For any $\cT=\cT_{I,\delta}$ contained in the strip $\Re s\in (\rho,l_0-3\rho-1)$ and for any $T>1$, the following bound holds for all $s\in \cT,\,l\in \N \cap (l_0,\infty)$.
\begin{align*}
\frac{1}{|\Gamma(-s+l+a)|}
\ll_{T} \frac{l^{\Re(s)}}{\Gamma(l+a)}+\frac{T^{-l}}{|\Gamma(-s+a)|}
.
\end{align*}
\end{lem}
\begin{proof} Set $\cT({T})=\cT\cap\{|\Im (s)|\leq T\}$. Since $\cT(T)$ is relatively compact, Stirling's formula shows the bound
$$
\left|{\Gamma(l+a)}/{\Gamma(-s+l+a)}\right|\ll (l+a)^{\Re s}, \quad s\in \cT(T),\,l\in \N\cap(-a,\infty).
$$
Suppose $s=\sigma+it\in \cT-\cT(T)$. Then 
\begin{align*}
 |\Gamma(-s+l+a)/\Gamma(-s+a)|&=\prod_{k=0}^{l}\{t^2+(\sigma+a+k)^2\}^{1/2}\geq |t|^{l+1}\geq T^{l+1}.
\end{align*}
\end{proof}

For convenience, we set 
\begin{align}
{\mathbf\Gamma}(l,s)=
\frac{2\pi^{m/2}\,\Gamma(m/2)^{-1}\,\Gamma(l-\rho)}{(\sqrt{8\Delta}\pi)^{l-\rho}}\,B_l^{\xi}(s)\, C_l^\xi \frac{\Gamma_\cL(l,s+1/2)}{\Gamma_{\cL_1^\xi}(1+s)},\label{GammaFCT}
\end{align}
where $\Gamma_{\cL}(l,s)$ is the common gamma factor for $L(F,s)$ with $F\in \cB_l^{+}$ (see \cite[3.4]{Tsud2011-1}), $\Gamma_{\cL_1^\xi}(s)$ the gamma factor for $L(f,s)$ (see \cite[3.6]{Tsud2011-1}), $B_l^\xi(s)$ is the function defined by \cite[(4.29)]{Tsud2011-1} and $C_l^\xi$ the constant in \cite[Proposition 17]{Tsud2011-1}. 
\begin{lem} \label{ErP2}
Suppose that $L_{\fin}(\cU,s)$ is regular at $s=1$. Then we have that $\SS_l^{\xi,\cU}(\phi|0)$ equals $\Gamma_{\cL_1^\xi}(1)\,D_*(0)$ times the expression
\begin{align*}
\frac{{\mathbf\Gamma}(l,0)}{\dim(\rho_\fin^{\cU}) }\,\sum_{F\in \cB_l^{+}}
\sum_{f\in \cB(\cU;\bK_{1,\fin}^{\xi*})}|a_{F}^{f}(\xi)|^2\,L_\fin(F,1/2)-
\,\frac{{\rm{CT}}_{s=0}(L(\cU,1-s)\,\hat\zeta(1-2s)^{1-\epsilon})}{\Gamma_{\cL_1^\xi}(1)} \prod_{p\in S}\,\widehat{\cW}_p^{\xi,(z_p^\cU)}(\phi_p;0)
, 
\end{align*}
 and that
\begin{align*}
{\mathbf\Gamma}(l,0)\sim 
\left(\tfrac{\pi}{4}\right)^{\rho}\,(2^{-1}\fd(\cL))^{1/2}\,4^{-1}\,
(4\pi\sqrt{2\Delta})^{2\rho-2l+1}\,\Gamma(2l-\rho)\,l^{-m}, \qquad l\rightarrow +\infty. 
\end{align*}
\end{lem}
\begin{proof}
The first formula follows from \eqref{IISpectEx-f0} and \eqref{P1-f0}. From \cite[Proposition 30 (2)]{Tsud2011-1}, we have
\begin{align*}
{\mathbf\Gamma}(l,0)=
\tfrac{\Gamma(l-\rho-1/2)\,\Gamma(l-2\rho)}{\Gamma(l-\rho/2)\,\Gamma(l-\rho/2+1/2)}\,\Gamma(2l-\rho)\,\left(\tfrac{\pi}{4}\right)^{\rho}\,(2^{-1}\fd(\cL))^{1/2}\,4^{-1}\,(4\pi\sqrt{2\Delta})^{2\rho-2l+1}.
\end{align*}
Since $\Gamma(l+a)/\Gamma(l+b)\sim l^{a-b}$ as $l\rightarrow +\infty$ (\cite[p.12]{MOS}) for $a,\,b\in \R$, the first factor of the right-hand side asymptotically equals $l^{-m}$ as $l\rightarrow +\infty$. 
\end{proof}

\begin{lem} \label{ErP6} Suppose that $L_\fin(\cU,s)$ is regular at $s=1$. 
\begin{align*}
{\Gamma_{\cL_1^\xi}(1)}^{-1}\, 
{\rm{CT}}_{s=0}(L(f,1-s)\,\hat\zeta(1-2s)^{1-\epsilon})
=\begin{cases} {\rm{CT}}_{s=1}\,L_\fin(f,1), \qquad &m\equiv 1 {\pmod 2}, \\
 L_\fin'(f,1)-d_{\cL}(\xi)\,L_\fin(f,1), \qquad &m\equiv 0 {\pmod 2},
\end{cases}
\end{align*}
where
$$
d_\cL(\xi)=\Gamma_{\cL_1^\xi}'(1)/\Gamma_{\cL_1^\xi}(1)=
\tfrac{-1}{2} \log(2^{-1}\fd(\cL_1^\xi))+\tfrac{m-2}{2}\log(2\pi)-\sum_{j=1}^{m/2-1}\tfrac{\Gamma'}{\Gamma}\left(\tfrac{m+1}{2}-j\right).
$$
\end{lem}
\begin{proof}
This is shown by a direct computation. 
\end{proof}

\begin{prop} Let $S$ be a finite set of prime numbers and $\phi \in \cH^{+}(\sG(\A_\fin)\sslash \bK_\fin^*)$ satisfying the conditions \eqref{conditionS} and \eqref{As-phi}. Let $\cU \subset \cV(\xi)$ be an irreducible $\sG_{1}^{\xi}(\A_\fin)$-submodule such that $\cU(\bK_{1,\fin}^{\xi*})\not=\{0\}$. Fix an orthonormal basis $\cB(\cU;\bK_{1,\fin}^{\xi*})$ of $\cU(\bK_{1,\fin}^{\xi*})$ consisting of eigenfunction of the involutive operator $\tau_\fin^{\xi}$ and set $\cB(\cU;\bK_{1,\fin}^{\xi*})^{\pm 1}=\{f\in \cB(\cU;\bK_{1,\fin}^{\xi*})|\tau_{\fin}^\xi(f)=\pm f\}$. Let $\{(z_p^{\cU},\rho_p^{\cU})\}_{p\in \fin}$ be the spectral parameter of $\cU$ and set $\rho_{\fin}^{\cU}=\otimes_{p\in \fin}\rho_{p}^{\cU}$. Set $\chi(\cU)=\dim(\cU(\bK_{1,\fin}^{\xi*})^{-1}\tr(\tau_{\fin}^{\xi}|\,\cU(\bK_{1,\fin}^{\xi*}))$. Suppose that $L_\fin(\cU,s)$ is regular at $s=1$. Then for $l\in \N$ with $l\geq l_0$, 
\begin{align}
&\tfrac{{\mathbf \Gamma}(l)}{4l^m} \sum_{F\in \cB_l^+} \lambda_{F}(\phi)\,L_\fin(F,1/2) \sum_{f\in \cB(\cU;\bK_{1,\fin}^{\xi*})^{(-1)^{l}}} |\fa_F^f(\xi)|^2
 \label{preMaintheorem-f1}
\\
&=c_\cL(\xi,\cU)\,\dim(\rho_{\fin}^{\cU})\,\{1+(-1)^{l}\,\chi(\cU)\}
\prod_{p\in S}\widehat{\cW}_p^{\xi,(z_p^{\cU})}(\phi_p;0) +\Ocal(C^{-l}), \quad (l\rightarrow +\infty),
 \notag
\end{align}
where $C>1$ is a constant independent of $l$. 
\end{prop}
\begin{proof}
This follows from Lemmas~\ref{ErP2} and \ref{ErP6} combined with the estimate \eqref{ErP1-f0} applied with $s=0$. We use Corollary~\ref{SignF-Lem-Cor} (ii) to reduce the summation range of $f$ to $\cB(\cU;\bK_{1,\fin}^{\xi*})^{(-1)^{l}}$. 
\end{proof}
To complete the proof of Theorem~\ref{Maintheorem}, it remains to evaluate the quantity $\widehat{\cW}_p^{\xi,(z)}(\phi_p;0)$, which is accomplished in the next subsection. 

\subsection{Evaluation of the $S$-factor} \label{EVSfactro}
In this section, we fix a prime $p\in S$ to work with, where $S$ is a finite set of primes satisfying the condition \eqref{conditionS}. From \S\ref{SpectralParameter}, recall the points $d(t,h)\in \sG$ which is defined by means of the Witt decomposition of $V(\Q_p)$ together with a set of isotropic vectors $e_{\pm j}\,(1\leq j \leq \ell_p)$ and a maximal $\Q_p$-anisotropic subspace $W_p$; then define $\sT$ to be the maximal $\Q_p$-split torus of $\sG$ consisting of all those points $d(t,1)$ ($t=(t_j)_{j=1}^{\ell_p}\in ({\rm GL}_1)^{\ell_p})$. Let $W^{\sG}_{\Q_p}:={\rm Norm}_{\sG}(\sT)/Z_{\sG}(\sT)$ and $W^{\sG^0}_{\Q_p}:={\rm Norm}_{\sG^0}(\sT)/Z_{\sG^0}(\sT)$ are the Weyl groups of $\sG$ and $\sG^0$, respectively. Obviously $W_{\Q_p}^{\sG^0}$ is viewed as a subgroup of $W_{\Q_p}^{\sG}$. 
For $\nu \in \fX_p$, let $\Omega_{\sG(\Q_p)}^{(\nu)}$ and $\Omega_{\sG^{0}(\Q_p)}^{(\nu)}$ be the Harish-Chandra's spherical functions of $\sG(\Q_p)$ and $\sG^0(\Q_p)$, respectively. Then from Corollary~\ref{ResSO-Cor}
\begin{align}
\Omega_{\sG(\Q_p)}^{(\nu)}(h)=\tfrac{1}{2}\biggl(\Omega_{\sG^0(\Q_p)}^{(\nu)}(h)+\Omega_{\sG^{0}(\Q_p)}^{(\ss \nu)}(h)\biggr), \quad h\in \sG^0(\Q_p),
 \label{OSOspheftn}
\end{align}
where $\ss\in W_{\Q_p}^{\sG}$ is the transposition of the basis vectors $e_1$ and $e_{-1}$. If $m$ is odd, or $m$ is even and $\dim W_p=2$ then $\ss\in W_{\Q_p}^{\sG^0}$, which yields $W_{\Q_p}^{\sG^0}=W_{\Q_p}^{\sG}$ and $\Omega_{\sG(\Q_p)}^{(\nu)}|\sG^0(\Q_p)=\Omega_{\sG^0(\Q_p)}^{(\nu)}$. If $m$ is even and $\dim W_p=0$, then $W_{\Q_p}^{\sG}=W_{\sQ_p}^{\sG^0}\rtimes \{1,\ss\}$. For $\phi \in \cH(\sG(\Q_p)\sslash \bK_p)$, its restriction to $\sG^0(\Q_p)$, say $\phi^{\sG_0}$, belongs to the Hecke algebra $\cH(\sG^0(\Q_p)\sslash \bK_p\cap \sG^{0}(\Q))$; then the Fourier transform $\widehat{\phi^{\sG_0}}(\nu)$ is defined by means of the spherical function $\Omega_{\sG^0(\Q_p)}^{(\nu)}$ and the Haar measure on $\sG^0(\Q_p)$ such that $\vol(\sG^0(\Q_p)\cap \bK_p)=1$. Since $\bK_p$ contains an element of $\sG(\Q_p)-\sG^0(\Q_p)$, we easily have the equality $\hat \phi(\nu)=\widehat{\phi^{\sG_0}}(\nu)$. By the inversion formula (\cite{Macdonald}), 
\begin{align}
\phi_p(g)=\int_{\fX_p^0}\widehat{\phi_p^{\sG^0}}(\nu)\,\Omega_{\sG^0(\Q_p)}^{(-\nu)}(g)\,\d\mu_p^{\rm Pl}(\nu) , \quad \phi_p \in \cH(\sG(\Q_p)\sslash \bK_p),\,g\in \sG^{0}(\Q_p),  
 \label{FInvForm}
\end{align}
where $\d\mu_p^{\rm Pl}(\nu)$ is the spherical Plancherel measure on $\fX_p^0$ given by 
\begin{align}
\d\mu_p^{\rm Pl}(\nu)={Q_p}
\left|\prod_{\alpha \in \Sigma^{+}(\sT_p,\sG^0)} \frac{\zeta_p(\langle \alpha,\nu\rangle+1)}{\zeta_p(\langle \alpha,\nu \rangle)}\right|^{2}\,
\prod_{j=1}^{\ell_p}(\log p)^{-1}\d \nu_j,  
 \label{Intro-f2}
\end{align}
where $\Sigma^{+}(\sT,\sG^0)$ is a positive system of $\sT$-roots on $\sG^0$,  $\zeta_p(z)=(1-p^{-z})^{-1}$ and $Q_p>0$ is a constant. Set $\sH=\sG^\xi_1$ for simplicity. Let $\pi^{\sG^0}_p(\nu)$ $(\nu \in \fX_p)$ and $\pi^{\sH^{0}}_p(z)$ $(z\in \fX_p(\xi))$ be the spherical representations of the special orthogonal groups $\sG^0(\Q_p)$ and $\sH^0(\Q_p)$, and ${\bf A}_p(\nu)\in {}^L (\sG^0)$ and ${\bf A}_p(z)\in {}^L(\sH^0)$ their Satake parameters, respectively. 
Then, the various local $L$-factors to be cooperated in the formula \eqref{SpectMeasure} are defined as follows: 
\begin{align*}
L(s,\pi^{\sH^0}_p(z)\boxtimes \pi_p^{\sG^0}(\nu))&=\det(1-{\bf A}_p(z_p)\otimes {\bf A}_p(\nu)\,p^{-s})^{-1}, \\
L(s,\pi_p^{\sG^0}(\nu);{\rm Std})&=\det(1-{\bf A}_p(\nu)\,p^{-s})^{-1}, \\
L(s,\pi_p^{\sG^0}(\nu);{\rm Ad})&=\det(1-{\rm Ad}({\bf A}_p(\nu))\,p^{-s})^{-1}, \end{align*}
where ${\rm Ad}:{}^L{(\sG^0)} \rightarrow \GL({\rm Lie}({\widehat \sG}(\C))$ is the adjoint representation. In the same way, the local $L$-factors $L(s,\pi_p^{\sH^0}(z);{\rm Std})$ and $L(s,\pi_p^{\sH^0}(z);{\rm Ad})$ are defined. Note that these $L$-factors, which are a priori invariant under $W_{\Q_p}^{\sG^0}$ and $W_{\Q_p}^{\sH^0}$, possesses a larger invariance by Weyl groups $W_{\Q_p}^{\sG}$ and $W_{\Q_p}^{\sH}$. We also remark that the for a character $\lambda_p$ of $\cH(\sG(\Q_p)\sslash \bK_{p})$ with the Satake parameter $\nu \in \fX_p/W_{\Q_p}^{\sG}$ the $L$-factor $L(\lambda_p,s)$ (given by \eqref{Def-LocalLfactor}) coincides with $L(s,\pi_p^{\sG^0}(\nu);{\rm Std})$. 
$$
\Delta_{\sG^0,p}=\begin{cases}
\prod_{j=1}^{(m+1)/2}\zeta_p(2j), \quad (\text{$m$ is odd}), \\
\prod_{j=1}^{m/2}\zeta_p(2j)\,L((m+2)/2,\chi_{K_p/\Q_p}), \quad (\text{$m$ is even}),
\end{cases}
$$
where $K_p$ is the discriminant quadratic field of the pair $(V(\Q_p),V^{\xi}_1(\Q_p))$ (\cite[\S 2.1]{Liu}). From the explicit formula, all the local $L$-factors defined above are non-negative for $s\in \R$ if $z\in \fX_p^{0}(\xi)$ and $\nu \in \fX_p^0$. We need the notion of the stable integral on $\sN(\Q_p)$ (\cite{LapidMao}, \cite{Liu}). From \cite[Proposition 3.1]{Liu}, for any $h_0\in \sH^0(\Q_p)$ and $g\in \sG^0(\Q_p)$, the function $\varphi(n)=\Omega_{\sG^{0}(\Q_p)}^{(\nu)}(\sm(1;h_0)n g)\,\psi_\xi(n)$ on $\sN(\Q_p)$ is compactly supported after averaging. Thus there exists an open compact subgroup $\cU_0(h_0,g) \subset \sN(\Q_p)$ such that for any open compact subgroup $\cU$ such that $\cU_0(h_0,g) \subset \cU$ the integral $\int_{\cU}\varphi(u)\d u$ is independent of $\cU$; the stable integral $\int_{\sN(\Q_p)}^{\rm st}\varphi(n)\d n$ is defined to be this common value of integral. 

\begin{lem}
Suppose $z\in\fX_p^{0}(\xi)$ and $\nu \in \fX_p^0$ so that $\pi_{p}^{\sH^0}(z)$ and $\pi_{p}^{\sG^0}(\nu)$ are both tempered. Then for any $g\in \sG^0(\Q_0)$, 
\begin{align}
\int_{\sH^0(\Q_p)}
|\Omega_{\sH(\Q_p)}^{(z)}(h_0)|\,
\left|\int_{\sN(\Q_p)}^{\rm st}
\Omega_{\sG^{0}(\Q_p)}^{(-\nu)}(\sm(1;h_0) n g)\,\psi_\xi(n)\,\d n \right|\,\d h_0<+\infty.
 \label{Liu-f1}
\end{align}
We have 
\begin{align}
\int_{\sH^0(\Q_p)}\Omega_{\sH^0(\Q_p)}^{(z)}(h_0)\,\biggl\{
\int_{\sN(\Q_p)}^{\rm st}\Omega_{\sG^{0}(\Q_p)}^{(-\nu)}(\sm(1;h_0) n)\psi_\xi(n)\,\d n\biggr\} \,\d h_0
=\frac{2^{-1}\Delta_{\sG^{0},p}\,L(1/2,\pi^{\sH^0}(z)\boxtimes \pi^{\sG^0}(-\nu))}
{L(1,\pi^{\sH^0}(z);{\rm Ad})\,L(1,\pi^{\sG^0}(-\nu);{\rm Ad})}. 
 \label{Liu-f2}
\end{align}
\end{lem}
\begin{proof} 
\cite[Theorems 2.1 and 2.2]{Liu}. Note that our measure $\d h_0$ on $\sH^0(\Q_p)$ is such that $\vol(\sH^0(\Q_p)\cap \bK_p)=1/2$. The factor $2^{-1}$ on the right-hand side of \eqref{Liu-f2} is due to this.  
\end{proof}

\begin{prop} \label{MellinWphi}
 Let $z\in \fX_p^{0+}(\xi)$ with $|\Re\,z_j|<1/2$ for all $1\leq j\leq \ell_p(\xi)$, and $s\in \C$ with $\Re(s)>-1/2$. Set $\sigma_p=\pi_p^{\sH^0}(z)$. For any $\phi_p\in \cH(\sG(\Q_p)\sslash \bK_p)$, 
\begin{align}
\widehat{\calW}_p^{\xi,(z)}(\check \phi_p;s)=
\int_{\fX_p^{0}}\widehat{\phi_p^{\sG_0}}(\nu)\,\frac{\Delta_{\sG^0,p}\,L(1/2,\sigma_p\boxtimes \pi_p^{\sG^0}(\nu))}{L(1,\sigma_p;{\rm Ad})\,L(1,\pi_p^{\sG^0}(\nu);{\rm Ad})}\frac{L(s+1/2,\pi_p^{\sG^0}(\nu);{\rm Std})}{L(s+1,\sigma_p;{\rm Std})\,\zeta_p(2s+1)^{1-\epsilon}}\,\d\mu_p^{\rm Pl}(\nu),
 \label{MellinWphi-f0}
\end{align}
where $\check \phi_p(g)=\phi_p(g^{-1})$, and $\epsilon\in \{1,0\}$ is the parity of $m$. 
\end{prop}
\begin{proof} First we suppose that $z$ belongs to the tempered locus $\fX_p^{0}(\xi)$. Since $\phi_p$ is of compact support on $\sG(\Q_p)$, by the Iwasawa decomposition, there exists an open compact subgroup $\cN_0(g)\subset \sN(\Q_p)$ such that the integral domain $\sN(\Q_p)$ in \eqref{cWphi} is replaced with any $\cN$ containing $\cN_0(g)$. Moreover, since $\sH$ is unramified over $\Q_p$, $\bK_{1,p}^{\xi}$ contains an element of $\sH(\Q_p)-\sH^0(\Q_p)$. Hence we can replace the domain of the $h$-integral in \eqref{cWphi} to $\sH^{0}(\Q_p)$ after multiplying $2$. Then substituting \eqref{FInvForm} to \eqref{cWphi} and by the relation $\hat\phi(\nu)=\widehat{\phi^{\sG_0}}(\nu)$, we have
\begin{align*}
\calW_p^{\xi,(z)}(\check \phi_p;g)&=2\int_{\sH^0(\Q_p)}\Omega_{\sH(\Q_p)}^{(z)}(h_0)\,\d h_0\int_{\cN_0(g)} \biggl\{\int_{\fX_p^0}\hat{\phi_p} (\nu)\Omega_{\sG^0(\Q_p)}^{(-\nu)}(\sm(1;h_0) n g)\,\d\mu_p^{\rm Pl}(\nu)\biggr\}\,\psi_\xi(n)\,\d n. 
\end{align*}
Since $\cN_0(g)$ and $\fX_p^0$ are both compact, we exchange the order of integrals to see that the right-hand side becomes
\begin{align*}
2\int_{\sH^0(\Q_p)}
\Omega_{\sH(\Q_p)}^{(z)}(h_0)\,\d h_0 \int_{\fX_p^0}\hat{\phi_p
} (\nu) \biggl\{ \int_{\cN_0(g)} \Omega_{\sG^0(\Q_p)}^{(-\nu)}(\sm(1;h_0) n g)\,\psi_\xi(n)\,\d n \biggr\}\,\d\mu_p^{\rm Pl}(\nu).
\end{align*}
For each $h_0\in \sH^0(\Q_p)$, the $n$-integral is unchanged when $\cN_0(g)$ is enlarged to any open compact subgroup $\cN$ containing $\cN_0(g)\,\cU_0(h_0,g)$; then the $n$-integral over such $\cN$ is identified with the stable integral by definition. Therefore, we obtain the first equality of  
\begin{align*}
\calW_p^{\xi,(z)}(\check \phi_p;g)
&=2\int_{\sH^0(\Q_p)}\Omega_{\sH(\Q_p)}^{(z)}(h_0)\,\d h_0 \int_{\fX_p^0}\hat\phi_p(\nu) \biggl\{ \int_{\sN(\Q_p)}^{\rm st} \Omega_{\sG^0(\Q_p)}^{(-\nu)}(\sm(1;h_0) n g)\,\psi_\xi(n)\,\d n \biggr\}\,\d\mu_p^{\rm Pl}(\nu) \\
&=2\int_{\fX_p^0}\hat\phi_p(\nu) \biggl\{\int_{\sH^0(\Q_p)}\Omega_{\sH(\Q_p)}^{(z)}(h_0)\,\d h_0 \, \int_{\sN(\Q_p)}^{\rm st} \Omega_{\sG^0(\Q_p)}^{(-\nu)}(\sm(1;h_0) n g)\,\psi_\xi(n)\,\d n \biggr\}\,\d\mu_p^{\rm Pl}(\nu) \\
&=\int_{\fX_p^0}\hat\phi_p(\nu) \{B^{(z,\nu)}(g)+B^{(\ss_0 z,\nu)}(g)\} \d \mu_{p}^{\rm Pl}(\nu),
\end{align*}
where $\ss_0\in W_{\Q_p}^{\sH}$ is the $H$ equivalent of $\ss$ and $$
B^{(z,\nu)}(g)=\int_{\sH^0(\Q_p)}\Omega_{\sH^0(\Q_p)}^{(z)}(h_0)\,\d h_0 \, \int_{\sN(\Q_p)}^{\rm st} \Omega_{\sG^0(\Q_p)}^{(-\nu)}(\sm(1;h_0) n g)\,\psi_\xi(n)\,\d n.
$$
In the above computation, the second equality is legitimized by \eqref{Liu-f1} and the last equality follows from the $H$ equivalent of the relation \eqref{OSOspheftn}. When viewed as a function in $g\in \sG^0(\Q_p)$, $B^{(z,\nu)}(g)$ possesses the properties:
\begin{itemize}
\item[(i)] $B^{(z,\nu)}(\sm(1;u) n g k)=\psi_\xi(n)^{-1}B^{(z,\nu)}(g)$ for $(u,n,g,k)\in (\sH^0(\Q_p)\cap \bK_p)\times \sN(\Q_p)\times \sG^0(\Q_p)\times \sG^0(\Q_p)\cap \bK_{p}
$. 
\item[(ii)] $\phi_0 *_{\sH^0(\Q_p)}B^{(z,\nu)}=\hat \phi_0(z)\,B^{(z,\nu)}$ for $\phi_0\in \cH(\sH^0(\Q_p)\sslash \bK_p\cap \sH^0(\Q_p))$. 
\item[(iii)] $B^{(z,\nu)}*_{\sG^0(\Q_p)}\phi=\hat\phi(\nu)\,B^{(z,\nu)}$ for $\phi \in \cH(\sG^0(\Q_p)\sslash \sG^0(\Q_p)\cap \bK_p)$.
\end{itemize}
By \cite[Theorem 0.5]{MSK}, we have that the $\C$-vector space of all the functions satisfying these three properties form a one dimensional space containing a unique function $B_0^{(z,\nu)}$ such that $B^{(z,\nu)}_0(1)=1$. Thus there exists a constant $C\in \C$ such that $B^{(\nu)}(g)=C\,B_0^{(\nu)}(g)$ for all $g\in \sG^0(\Q_p)$. The value $C=B^{(z,\nu)}(1)$ is given by \eqref{Liu-f2}. Hence
\begin{align*}
\calW_p^{\xi,(z)}(\check \phi_p;g)&=\int_{\fX_p^0}\hat\phi_p(\nu) \,\{B^{(z,\nu)}(g)+B^{(\ss_0z,\nu)}(g)\}\,\d \mu_p^{\rm Pl}(\nu)
\\
&=\int_{\fX_p^0}\hat\phi_p(\nu)\, \frac{2^{-1}\Delta_{\sG^{0},p}L(1/2,\pi_p^{\sH^0}(z)\boxtimes \pi_p^{\sG^0}(-\nu))}
{L(1,\pi_p^{\sH^0}(z);{\rm Ad})\,L(1,\pi_p^{\sG^0}(-\nu);{\rm Ad})}\,\{B_0^{(z,\nu)}(g)+B_0^{(\ss_0z,\nu)}(g)\} \,\d \mu_p^{\rm Pl}(\nu).
\end{align*}
The Mellin transform of $B_0^{(z,\nu)}$ is computed by \cite[Proposition 2]{Sugano85}, which yields
 \begin{align*}
\int_{\Q_p^\times} B_0^{(z,\nu)}(\sm(t;1))|t|_p^{s-m/2}\d^\times t
=\frac{L(s,\pi_p^{\sG^0}(\nu);{\rm Std})}{L(s+\frac{1}{2},\pi_p^{\sH^0}(z);{\rm Std})} \zeta_p(2s)^{\epsilon-1}
\end{align*}
for $\Re s\gg 0$. Note that the last quantity remain unchanged when $z$ is replaced with $\ss_0 z$. Since $\fX_p^0$ is compact, by changing the order of integrals and by applying this formula, we obtain the formula \eqref{MellinWphi-f0} for $s\in \C$ with sufficiently large real part. The right-hand side of \eqref{MellinWphi-f0} defines a holomorphic function on $\Re(s)>-1/2$. Thus, by analytic continuation, the same formula is true for any $s\in \C$ with $\Re s>-1/2$. This completes the formula for $z\in \fX_p^{0}(\xi)$. The condition $z\in \fX_p^{0}(\xi)$ is weakened to $|\Re z_j|<1/2\,(1\leq j\leq \ell_p(\xi))$ by analytic continuation again. 
\end{proof}

\subsection{Proof of Theorem~\ref{MAINTHM3} and Corollary~\ref{MAINTHM2}} 
Set $W=W_{S}^{\sG}$, $\cH_S=\bigotimes_{p\in S}\cH(\sG(\Q_p)\sslash \bK_p)$, $\mu_l=\Lambda_l^{\xi,f}$, and ${\mathbf\Lambda}={\mathbf\Lambda}^{\xi,(z_S^{\cU})}$ for simplicity. From the assumption, the point $z=z_S^{\cU}:=(z_p^\cU)_{p\in S}$ belongs to the tempered locus $\fX_S^{0}(\xi)$. The space $C(\fX_S^{0+}/W)$ is endowed with the topology of the uniform convergence. From \eqref{SpectMeasure}, we have the formula
$$
{\bf \Lambda}(\tilde \alpha)=\int_{\fX_S^{0}/W}\tilde \alpha(\nu)\,\fD(\nu)\,\d\mu_S^{\rm Pl}(\nu), \quad \tilde \alpha \in C(\fX_S^{0+}/W)
$$
with 
$$
\fD(\nu):=\prod_{p\in S}\frac{\Delta_{\sG^0,p}\,\zeta_p(1)^{\epsilon-1}}{L(1,\pi_p^{\sH}(z);{\rm Ad})\,L(1,\pi_p^{\sH^0}(z);{\rm Std})}\times\frac{L\left(\frac{1}{2},\pi_p^{\sH^0}(z)\boxtimes \pi_p^{\sG^0}(\nu)\right)L\left(\frac{1}{2},\pi_p^{\sG^0}(\nu);{\rm Std}\right)}{L(1,\pi_p^{\sG^0}(\nu);{\rm Ad})}.
$$
As noted in \S\ref{EVSfactro}, $\fD(\nu)\geq 0$ for all $\nu\in \fX_S^0$. 
Since $\fD(\nu)$ is a continuous function on the compact space $\fX_S^{0}/W$, we have $C=\max_{\nu \in \fX_S^{0}}\fD(\nu)<\infty$ and ${\bf \Lambda}(\tilde\alpha) \leq C\mu_{S}^{\rm Pl}(\tilde\alpha)$ for any non negative function $\tilde\alpha$, where $\mu_S^{\rm Pl}=\otimes_{p\in S}\mu_p^{\rm Pl}$. Let $\tilde \alpha\in C(\fX_S^{0+}/W)$ and $\e>0$. From \cite[Theorem 7.3]{Sauvageot}, we can find $\phi_1,\phi_1\in \cH_S$ such that 
\begin{align*}
&|\widehat{\phi}(\nu)-\tilde \alpha(\nu)| \leq \widehat{\phi_2}(\nu)\quad (\nu \in \fX_S^{0+}), \quad \mu_{S}^{\rm Pl}(\widehat{\phi_2})\leq \frac{1}{4(C+1)}\,\e.
\end{align*} 
Hence
$$
|{\mathbf\Lambda}(\widehat{\phi_1}-\tilde \alpha) |\leq {\bf \Lambda}(|\widehat {\phi_1}-\tilde \alpha|) \leq C\,\mu_S^{\rm Pl}(|\widehat {\phi_1}-\tilde \alpha|) \leq C\,\mu_S^{\rm Pl}(\widehat{\phi_2})
\leq C\times \frac{\e}{4(C+1)}<\frac{\e}{4}
.
$$ 
From Theorem~\ref{Maintheorem}, there exists $L>0$ such that 
\begin{align*}
|\mu_l(\widehat{\phi_i})-{\mathbf\Lambda}(\widehat{\phi_i})|\leq \frac{\e}{8}\quad(i=1,2)
\end{align*}
for all $l>L$. 
Define 
$$
\mu_{l}^{\flat}(\alpha)=\frac{{\bf \Gamma}(l)}{4l^m}\sum_{F\in \cB_l^{+}(\flat)}\sum_{f\in \cB(\cU;\bK_{1,\fin}^{\xi*})} L_\fin(F,1/2)\,|\fa_{F}^{f}(\xi)|^2, \quad \alpha \in C(\fX_S^{0+}/W_S) 
$$
and $\mu_{l}^{\natural}=\mu_{l}-\mu_{l}^{\flat}$. Suppose \eqref{NTvanishing} were true; then 
\begin{align*}
|\mu_{l}^{\flat}(\widehat{\phi_2})|+|\mu_{l}^{\flat}(\widehat{\phi_1})-\mu_{l}^{\flat}(\tilde\alpha)|\leq 2\,\max_{\nu\in \fX_S^{0+}}|\widehat{\phi_2}(\nu)|\times \frac{{\bf \Gamma}(l)}{4l^m}\sum_{f\in \cB(\cU;\bK_{1,\fin}^{\xi*})}\sum_{F\in \cB_l^{+}(\flat)}|L_\fin(F,1/2)|\,|\fa_{F}^{f}(\xi)|^2
\end{align*}
tends to $0$ as $l\rightarrow \infty$. Thus there exists $L_1>0$ such that 
$$
|\mu_l^{\flat}(\widehat{\phi_2})|+|\mu_l^{\flat}(\widehat{\phi_1}-\tilde\alpha)|\leq \frac{\e}{8}
$$
for $l>L_1$. Suppose the statement \eqref{non-negativityL} holds true. Then the measure $\mu_l^{\natural}$ is non-negative for all $l>L$; thus, for $l>\max(L,L_1)$, 
\begin{align*}
|\mu_l(\widehat{\phi_1})-\mu_l(\tilde \alpha)|
&\leq \mu_{l}^{\natural}(|\widehat{\phi_1}-\tilde\alpha|)+|\mu_l^{\flat}(\widehat{\phi_1}-\tilde\alpha)|
\\
&\leq \mu_{l}^{\natural}(\widehat{\phi_2})+| \mu_l^{\flat}(\widehat{\phi_1}-\tilde\alpha)|
\\
&=\mu_{l}(\widehat{\phi_2})-\mu_{l}^{\flat}(\widehat{\phi_2})+|\mu_l^{\flat}(\widehat{\phi_1}-\tilde\alpha)|
\\
&\leq \mu_l(\widehat{\phi_2})+(|\mu_l^\flat(\widehat{\phi_2})|+|\mu_l^{\flat}(\widehat{\phi_1}-\tilde\alpha)|)
\\
&\leq {\bf \Lambda}(\widehat{\phi_2})+\frac{\e}{8} +\frac{\e}{8}
\\
&\leq C\,\mu_{S}^{\rm Pl}(\widehat {\phi_2})+\frac{\e}{4}
\\
&\leq C\times \frac{\e}{4(C+1)}+\frac{\e}{4}<\frac{\e}{4}+\frac{\e}{4}=\frac{\e}{2}.
\end{align*}
Therefore, 
{\allowdisplaybreaks\begin{align*}
|(\mu_l-{\mathbf\Lambda})(\alpha)|&\leq 
|\mu_l(\widehat{\phi_1}-\tilde \alpha)|+|\mu_l(\widehat{\phi_1})-{\bf \Lambda}(\widehat{\phi_1})|+|{\bf \Lambda}(\widehat{\phi_1}-\tilde\alpha)|
\\
&\leq \e/2+\e/4+\e/4=\e 
\end{align*}}for all $l>\max(L,L_1)$. Hence $\mu_l(\tilde\alpha)  \rightarrow {\mathbf\Lambda}(\tilde\alpha)$ as $l\rightarrow \infty$. This completes the proof of Theorem~\ref{MAINTHM3}. Let us prove Corollary~\ref{MAINTHM2}. Since the support of ${\mathbf\Lambda}$ is evidently $\fX_S^0/W$, there exists an element $\alpha \in C(\fX_S^0/W)$ supported on $\cN/W$ such that ${\mathbf\Lambda}(\tilde \alpha)\not=0$, where $\tilde \alpha\in C(\fX_S^{0+}/W)$ is the extension of $\alpha$ by $0$ outside $\fX_S^0$. We choose $\e>0$ such that $\e<|{\mathbf\Lambda}(\tilde\alpha)|/2$, then for $l>L$, we have $\mu_l(\tilde \alpha)\not=0$ due to $|\mu_l(\tilde\alpha)|\geq |{\mathbf\Lambda}(\tilde\alpha)|-\e>|{\mathbf\Lambda}(\tilde\alpha)|/2>0$. Hence there exist $F\in \cB_l^+$ and $f\in \cB(\cU;\bK_{1,\fin}^{\xi*})$ such that $\alpha(\nu_S(F))L_\fin(F,1/2)|\fa_{F}^{f}(\xi)|^2\not=0$. \qed

\section{The proof of Proposition~\ref{P1}} \label{JJidentity}
From Lemma~\ref{STB-L1}, $\sN_{\mu}=\Ad(\mu^{-1})\,\sN^\xi$ for $\mu=\sm(1;\,\delta)\in \sM(1)$. Hence, $\hat\JJ_{l}^{\xi,f}(1,\phi|\beta;r)$ with $f\in \cB(\cU;\bK_{1,\fin}^{\xi*})$ equals 
{\allowdisplaybreaks\begin{align*}
&\int_{\sG_1^\xi(\Q)\bsl \sG^\xi_1(\A)} \bar f(h_0)\,\d h_0
\sum_{\delta \in \sG_1^\xi(\Q)\bsl \sG_1(\Q)}
\int_{\sN^{\xi}(\Q)\bsl \sN(\A)} 
\hat{\mathbf\Phi}_{l}^{f,\xi}(\phi|\beta;\,n\,\sm(r;\,\delta h_0)\,b_\infty)\,\psi_{\delta^{-1}\xi}(n)^{-1} \,\d n
\\
&=
\int_{\sG_1^\xi(\Q)\bsl \sG^\xi_1(\A)} \bar f(h_0)\,\d h_0
\sum_{\delta \in \sG_1^\xi(\Q)\bsl \sG_1(\Q)}
 \int_{\sN^\xi(\Q)\bsl \sN^\xi(\A)} 
\psi_{\delta^{-1}\xi}(n_1)^{-1}\, \d n_1 \\
&\quad \cdot \int_{\sN^{\xi}(\A)\bsl \sN(\A)}\hat{\mathbf\Phi}_{l}^{f,\xi}(\phi|\beta\,;\,n\,\sm(r;\,\delta h_0)\,b_\infty)\,\psi_{\delta^{-1}\xi}(n)^{-1} \,\d n.
\end{align*}}The integral of $\psi_{\delta^{-1}\xi}$ on $\sN^\xi(\Q)\bsl \sN^\xi(\A)$ is $1$ if and only if $\delta^{-1}\xi$ is proportional to $\xi$, or equivalently $\delta^{-1}\xi=\pm \xi$; otherwise, the integral is $0$. Thus, only the elements of $\sG_1^\xi(\Q)\bsl \sG_1(\Q)$ represented by the identity and $\cnt^{\sG_1}$ contribute non trivially in the summation. Since $V_1=\Q\xi\oplus V_1^\xi$, the mapping $x \mapsto \sn(x \xi)$ is a measure preserving bijection from $\A$ onto $\sN^\xi(\A)\bsl \sN(\A)$. Thus, formally changing the order of integrals, we obtain the following formula, whose proof will be given at the end of this subsection. 
{\allowdisplaybreaks
\begin{align}
\hat\JJ_l^{\xi,f}(1,\phi|\beta;r)&=\int_{(c)}\beta(s)\,D_*(s)\,L^*(\cU,-s)\,
\left(\cJ_l^{\xi,f,+}(\phi|s;r)+\cJ_l^{\xi,f,-}(\phi|s;r)\right)\,\d s,
 \label{P1-1}
\end{align} }with $c>\rho$, where $\cJ^{\xi,f, \pm}_l(\phi|s;r)$ is the integral of $\bar f(h_{0})\,{\mathbf\Phi}_{l}^{f,\xi}(\phi|s;\,\sn(x \xi)\,\sm(t;\,\delta^{\pm}\,h_0))\,\psi(x\,\langle \xi,\xi\rangle])$ over $(x,h_0)\in \A\times (\sG_1^\xi(\Q)\bsl \sG_1^\xi(\A))$ with $\delta^{+}=1$ and $\delta^{-}=\cnt^{\sG_1}$. The integral $\cJ^{\xi,f, \pm}_l(\phi|s)$ becomes the product of the following two integrals
{\allowdisplaybreaks
\begin{align}
\cJ_{l}^{\xi,\pm}(s;r)&=\int_{\R} \Phi_l^{\xi}(s;\,\sm(1;\delta_\infty^{\pm})\,\sn(x_\infty \xi)\,\sm(r;\,1)\,b_\infty)\,\exp(-2\pi i x_\infty \langle \xi,\xi\rangle)\,\d x_\infty,
 \notag
\\
\cJ_{\fin}^{\xi,f,\pm}(\phi|s)&=\int_{\sG^\xi_1(\Q)\bsl \sG_1^\xi(\A_\fin)}\bar f(h_{0,\fin})\,\d h_{0,\fin}\int_{\A_\fin} \Phi_\fin^{f,\xi}(\phi|s;\,\sm(1;\delta_\fin^{\pm}h_{0,\fin})\,\sn(x_\fin \xi))\,\psi_{\fin}(-x\,\langle \xi,\xi\rangle)\,\d x_\fin. 
\label{JJidentity-padic-f0}
\end{align}}Let $\cJ_l^{\xi,\cU,\pm}(s;r)$ denote the sum of $\dim(\rho_{\fin}^{\cU})^{-1}\,\cJ_l^{\xi,f,\pm}(s;r)$ over $f\in \cB(\cU:\bK_{1,\fin}^{\xi*})$.  
\begin{lem} \label{P1-L2} 
For any compact interval $I\subset \R$, we have the estimation
{\allowdisplaybreaks\begin{align}
\int_{\sG^\xi_1(\Q)\bsl \sG_1^\xi(\A_\fin)}|f(h_{0,\fin})|\d h_{0,\fin}\,\int_{\A_\fin} |\Phi_\fin^{f,\xi}(\phi|s;\,\sn(x_\fin \xi)\,\sm(1;\,\delta_{\fin}^\pm\,h_{0,\fin}))|
\,\d x_\fin\ll 1, \quad \Re s\in I.  
 \label{P1-L2-f01}
\end{align}}Moreover, 
{\allowdisplaybreaks\begin{align}
\cJ_\fin^{\xi,\cU,-}(\phi|s)&=\dim(\rho_\fin^\cU)^{-1}\,\tr(\tau_{\fin}^{\xi}|\cU(\bK_{1,\fin}^{\xi*}))\,\cJ_\fin^{\xi,\cU,+}(\phi|s), \\
\cJ_\fin^{\xi,\cU,+}(\phi|s)&=|Q[\xi]|_\fin^{-1}\,\delta(2\xi \in \cL_1)\,\prod_{p\in S}\widehat{\calW}_p^{\xi,(z_p)}(\phi_p;s).
\label{P1-L2-f00}
\end{align}} 
\end{lem}
\begin{proof}
By the normalization of the Haar measure on $\sG^\xi(\A_\fin)$ (\cite[\S 3.8]{Tsud2011-1}), we have
{\allowdisplaybreaks
\begin{align*}
\cJ_{\fin}^{\xi,f,\pm}(\phi|s) 
&=\fd(\cL_{1}^{\xi})^{-1}\int_{\sG_1^\xi(\Q)\bsl \sG_1^\xi(\A_\fin)}\bar f(h_{0,\fin})\d h_{0,\fin} \int_{\A_\fin}\biggl\{\int_{\A_\fin^\times}\int_{\sG_1^\xi(\A_\fin)}|t|_\A^{-s+\rho}f(u^{-1}) \\
&\quad \times 
\int_{\sN^\xi(\A_\fin)} \phi(\sm(t;u)n_0 \sm(1;\delta^{\pm}_\fin h_{0,\fin})\,\sn(x\xi))\d n_0 \biggr\}\,\psi_\fin(-x\langle \xi,\xi\rangle)\,\d^\times t\,\d n_0\,\d x
\\
&=\int_{\sG_1^\xi(\Q)\bsl \sG_1^\xi(\A_\fin)}\bar f(h_{0,\fin})\d h_{0,\fin}
\int_{\A_\fin^\times}\int_{\sG_1^\xi(\A_\fin)}|t|_\fin^{-s+\rho}f(u^{-1})\d^\times \d u\,\\
&\quad \times 
\fd(\cL_{1}^{\xi})^{-1/2}
\int_{\A_\fin\times V_1^\xi(\A_\fin)} \phi(\sm(t;u\delta^{\pm}_\fin h_{0,\fin})\,\sn(Z+x\xi))\,\psi(-\langle Z+x\xi,\xi\rangle)\,\d Z\,\d x,
\end{align*}}where $\d Z$ is the self-dual Haar measure on $V_1^\xi(\A_\fin)$ with respect to the bi-character $\psi_\fin(\langle Z,Z'\rangle)$. From the computation above, noting that $\sG_1^\xi(\Q)\bsl \sG_1^\xi(\A_\fin)$ is compact, we see that the integral in \eqref{P1-L2-f01} is bounded from above by 
{\allowdisplaybreaks\begin{align*}
\int_{\A_\fin^\times} \d^\times t  \int_{\sG_1^\xi(\A_\fin)} \d u \int_{\A_\fin\times V_1^\xi(\A_\fin)} |t|_\fin^{-\Re (s)+\rho}|\phi( \sm(t;u\delta^{\pm}_\fin h_{0,\fin})\,\sn(Z+x\xi))|\,\d Z\,\d x.
\end{align*}}Since $\phi$ is of compact support on $\sG(\A_\fin)$, by the Iwasawa decomposition on $\sG(\A_\fin)=\sP(\A_\fin)\bK_\fin$, the integral domain is restricted to  a compact set of $(t,u,x,Z)\in \A_\fin^\times \times \sG_1^\xi(\A_\fin)\times \A_\fin\times V_1^\xi(\A_\fin)$. From this, the estimate \eqref{P1-L2-f01} is evident. 
Let $\d X$ be the self-dual Haar measure on $V_{1}(\A_\fin)$ with respect to the bi-character $\psi_\fin(\langle X,X'\rangle)$. Then $\d X=|Q[\xi]|_\fin^{1/2}\,\d x\,\d Z$ for $X=x\xi+Z\,(x\in \A_\fin,\,Z\in V_1^\xi(\A_\fin))$ is easily confirmed. 
By Lemma~\ref{IndexL}, 
{\allowdisplaybreaks
\begin{align}
\cJ_{\fin}^{\xi,\cU,\pm}(\phi|s)=&
\dim(\rho_{\fin}^{\cU})^{-1}\sum_{f\in \cB(\cU;\bK_{1,\fin}^{\xi*})}
\int_{\sG_1^\xi(\Q)\bsl \sG_1^\xi(\A_\fin)}\bar f(h_{0,\fin})\d h_{0,\fin}\int_{\A_\fin^\times}\int_{\sG_1^\xi(\A_\fin)}|t|_\fin^{-s+\rho}f(u^{-1})\d^\times \d u\,
 \notag
\\
&\quad \times \fd(\cL_{1}^{\xi})^{-1/2}|Q[\xi]|_\fin^{-1/2} \int_{V_1(\A_\fin)} \phi(\sm(t;u h_{0,\fin}\delta^{\pm}_\fin)\,\sn(X))\,\psi(-\langle X,\xi\rangle)\,\d X
 \notag
\\
&=|Q[\xi]|_\fin^{-1/2}\fd(\cL_{1}^{\xi})^{-1/2}\fd(\cL_{1})^{-1/2} \int_{\A_\fin^\times} \calW_{\fin}^\xi(\phi|s;\sm(t;\delta_\fin^{\pm}))\,|t|_\fin^{s-\rho}\,\d^\times t
 \notag
\\
&=|Q[\xi]|_\fin^{-1}\fd(\cL_{1})^{-1}
\int_{\A_\fin^\times} \calW_{\fin}^\xi(\phi|s;\sm(t;\delta_\fin^{\pm}))\,|t|_\fin^{s-\rho}\,\d^\times t
 \label{P1-L2-f0}
\end{align} }where 
$$
\calW_\fin^\xi(\phi;g)=\int_{\sG_1^\xi(\A_\fin)}\Psi_{\cU}(u^{-1})\d u\int_{\sN(\A_\fin)} \phi(g^{-1}\sm(1;u)n)\,\psi_{\xi}(n)^{-1}\,\d n, \quad g\in \sG(\A_\fin)
$$
with 
$$
\Psi_{\cU}(u)=\dim(\rho_{\fin}^{\cU})^{-1}\sum_{f\in \cB(\cU;\bK_{1,\fin}^{\xi*})}\int_{\sG_1^\xi(\Q)\bsl \sG_1^{\xi}(\A)}\bar f(h_{0})\,f(h_{0}u)\,\d h_{0}, \quad u \in \sG_1^\xi(\A).
$$
Obviously, $\Psi_\cU$ is bi-$\bK_{1,\fin}^{\xi*}$-invariant function on $\sG_1^\xi(\A_\fin)$; moreover, we have $\Psi_\cU(v^{-1}uv)=\Psi_\cU(u)$ for all $v\in \bK_{1,\fin}^\xi$, and $\Psi_{\cU}(1)=1$. The function $\Psi_\cU$ satisfies the same Hecke eigenequation $ \Psi_\cU*\varphi=C_\fin^{\cU} (\varphi) \Psi_\cU$ for all $\varphi\in \cH^{+}(\sG_1^\xi(\A_\fin)\sslash \bK_{1,\fin}^{\xi*})$, where $C_\fin^{\cU}:\cH^{+}(\sG_1^\xi(\A_\fin)\sslash \bK_\fin^{\xi*}) \rightarrow \C$ is the eigencharacter which is decomposed to the product $\prod_{p} C_p^{\cU}$ of characters $C_p^{\cU}$ of $\cH^{+}(\sG_1^\xi(\Q_p)\sslash \bK_{1,p}^{\xi *})$. From \cite[Lemma 1.5]{MS98}, we have the product formula
\begin{align}
\Psi_\cU(u)=\prod_{p\in \fin}\omega_{p}^{\cU}(u_p), \quad u \in \sG_1^\xi(\A_\fin) 
\label{P1-L2-f1}
\end{align} 
where $\omega_{p}^{\cU}$ is the zonal spherical function on $\sG_1^\xi(\Q_p)$ associated with the eigencharacter $C_p^{\cU}$ studied in \cite[\S 1.3]{MS98}. Therefore, 
\begin{align*}
\calW_\fin^\xi(\phi;g)=\prod_{p\in \fin}\calW_p^{\xi}(\phi_p;g_p), 
\end{align*}
with 
$$
\calW_p^{\xi}(\phi_p;g_p)=\int_{\sG_1^\xi(\Q_p)}\omega_p^{\cU}
(u^{-1})\d u\int_{\sN(\Q_p)} \phi_p(g_p^{-1}\sm(1;u)n)\,\psi_{\eta,p}(n)^{-1}\,\d n, \quad g_p\in \sG(\Q_p),
$$
where the Haar measure on $\sN(\Q_p)$ is normalized so that $\vol(\sN(\Q_p)\cap \bK_p)=1$. Let us show 
$$\prod_{p\in \fin-S} \calW_p^{\xi}(\phi_p|s;\sm(t_p;\delta_p))
=\fd(\cL_{1})\,\delta(2\xi \in \cL_1)
\,\begin{cases} 
\epsilon(\cU) \quad &(\delta=\delta^{-}),
\\
1 \quad &(\delta=\delta^{+})
\end{cases} 
$$
for $t=(t_p)\in \A_\fin^\times$, where $\epsilon(\cU)=\dim(\rho_{\fin}^{\cU})^{-1}\,\tr(\tau_\fin^\xi|\,\cU(\bK_{1,\fin}^{\xi*})$. We consider the case $\delta=\delta^{-}(=\cnt^{\sG_1})$. 
Let $p \in \fin-S$; then since $\phi_p={\rm ch}_{\bK_p^*}$, by the Iwasawa decomposition $\sG(\Q_p)=\sP(\Q_p)\bK_p^*$ (\cite[Proposition 1.2 (i)]{MS98}), we have that 
\begin{align}
\calW_p^{\xi}(\phi_p;\sm(t_p;\delta_p^{\pm}))=\delta(t_p\in \Z_p^\times)\int_{(\sG_1^\xi(\Q_p)\cap (\bK_{1,p}^*\delta_p^{\pm})}\omega^{\cU}_p (u^{-1})\d u \int_{\sN(\Q_p)\cap \bK_p^*}\psi_{\xi,p}(n_p)^{-1}\d n_p.
 \label{P1-L2-f2}
\end{align}
The set $\sG_1^\xi(\Q_p)\cap (\bK_{1,p}^*\delta_p^{-})$ is non empty if and only if $\delta_p^{-}=\cnt_p^{\sG_1}\in \sG_{1}^\xi(\Q_p)\bK_{1,p}^*$, or equivalently $2\xi \in \cL_{1,p}$ by \eqref{cnt-ref} and Lemma~\ref{JuneL-0}; under this condition, $\sG_1^\xi(\Q_p)\cap (\bK_{1,p}^*\delta_p^{-})=\sG_1^\xi(\Q_p)\cap (h_p^{\xi}\bK_{1,p}^{*})$ where $h_{p}^\xi\in \sG_1^\xi(\Q_p)$ is the element in \eqref{rex-dec}. Since $\omega_{p}^{\cU}$ is right $\bK_{1,p}^{\xi*}$-invariant and since the Haar measure on $\sG_1^\xi(\Q_p)$ is normalized so that $\vol(\bK_{1,p}^{\xi*})=1$, the $u$-integral is computed to be $\delta(2\xi\in \cL_{1,p})\,\omega_p^{\cU}(h_p^\xi)$. The character $\psi_{\xi,p}$ is trivial on $\sN(\Q_p)\cap \bK_p^{*}\,(\cong \cL_{1,p})$; thus
$$
\int_{\sN(\Q_p)\cap \bK_{p}^{*}}\psi_{\xi,p}(n)\d n_p=\vol(\cL_{1,p}^*;\d n_p)=
\fd_p(\cL_{1}). 
$$
Therefore, 
\begin{align*}
\prod_{p\in \fin-S}\calW_p^{\xi}(\phi_p;\sm(t_p;\delta_p^{-}))
&=\prod_{p\in \fin -S} \fd_p(\cL_{1})\,\delta(t_p\in \Z_p^\times)\,\delta(2\xi \in \cL_{1,p})\,\omega_p^{\cU}(h_p^\xi)
\\
&=\fd(\cL_1) \delta(2\xi \in \cL_1)\,\prod_{p\in \fin-S} \delta(t_p\in \Z_p^\times)\,\omega_p^{\cU}(h_p^\xi), 
\end{align*}
because $\xi\in \cL_{1,p}^{*}=\cL_{1,p}\subset 2^{-1}\cL_{1,p}$ for $p\in S$. Set $h_\fin^\xi=(h_{p}^{\xi})_{p\in \fin}$. Under the condition $2\xi\in \cL_{1}$, from \eqref{P1-L2-f1} and the eigen equation $\tau_\fin^{\xi}(f)(u):=f(uh_{\fin}^{\xi})=\epsilon_f\,f(u)$, 
\begin{align*}
\prod_{p \in \fin -S}\omega_{p}^{\cU}(h_p^\xi)=\Psi_{\cU}(h_{\fin}^\xi)=\dim(\rho_\fin^{\cU})^{-1} \sum_{f\in \cB(\cU;\bK_{1,\fin}^{\xi*})} \epsilon_{f}
=\dim(\rho_\fin^{\cU})^{-1}\,\tr(\tau_{\fin}^{\xi}|\,\cU(\bK_{1,\fin}^{\xi*}))=:\epsilon(\cU). 
\end{align*}
Note that $h_p^{\xi}=1$ and $\omega_p^{\cU}(h_p^\xi)=1$ for $p\in S$. This completes the proof of \eqref{P1-L2-f2} when $\delta=\delta^{-}$. The other case $\delta=\delta^{+}(=1)$ is much simpler. Indeed, from \cite[Proposition 2.3]{MS98}, we have the equality $\sG_1^\xi(\Q_p)\cap \bK_{1,p}^{*}=\bK_{1,p}^{\xi*}$ which allows us to compute the $u$-integral in \eqref{P1-L2-f2} to be $1$. The remaining part is the same as before. 

For $p\in S$, $\cW_p^{\xi}(\phi_p;g)$ coincides with \eqref{cWphi}.  
Thus from \eqref{P1-L2-f2} we have  
\begin{align*}
\int_{\A_\fin^\times}\calW_\fin^\xi(\phi;\sm(t;\delta^{\pm}))|t|_\fin^{s-\rho}\d^\times t
=\fd(\cL_{1})\,\delta(2\xi \in \cL)\,\prod_{p\in S}\widehat{\calW}_p^{\xi,(z_p)}(\phi_p;s) \begin{cases} \epsilon(\cU)
\quad &(\delta=\delta^{-}), \\ 
1 \quad &(\delta=\delta^{+}).
\end{cases}
\end{align*}
Substituting this to \eqref{P1-L2-f0}, we have \eqref{P1-L2-f00} as required. 
\end{proof}

\begin{lem}\label{P1-L3}
Let $0<c<l-\rho$. Then,   
\begin{align}
\int_{\R} |\Phi_l^{\xi}(s;\,\sn(x_\infty \xi)\,\sm(r;\,\delta_\infty^{\pm})\,b_\infty)|\,\d x_\infty \ll 1, \quad \Re(s)=c. 
 \label{P1-L3-f1}
\end{align}
We have  
$$
\cJ_{l}^{\xi,+} (s)=(-1)^l\,\cJ_{l}^{\eta,-} (s)=|Q[\xi]
|_\infty^{-1}\, \frac{(\sqrt{8|Q[\xi]|}\,\pi)^{-s-\rho+l}}{\Gamma(-s-\rho+l)}\,n^{l-s-\rho-1}\times \cW^{\xi}_l(\sm(r;1_m)b_\infty).
$$
\end{lem}
\begin{proof}
From Proposition~\ref{RealShinExBruhForm} (1), we have 
\begin{align*}
\cJ_{1}^{\xi,\epsilon}(s)&=\epsilon^{l}
r^{s+\rho}\,\int_{\R}\left(1+\epsilon i\,(2^{-1}\Delta)^{1/2}xr^{-1}\right)^{s+\rho-l}
\exp(-2\pi i\,\epsilon \tau \langle \xi,\xi\rangle \,x)\,\d x
\\
&=\epsilon^l 2^{1/2}\,\Delta^{-1/2}\,r^{s+\rho+1}
\int_{\R} (1+iy)^{s+\rho-l}\,\exp(2\pi i \sqrt{2\Delta}\,\tau r\, y)\,\d y.
\end{align*}
Here, the second equality is obtained by the change of variable $x=(2\Delta^{-1})^{1/2}\,r y$. Using the formula \cite[3.382.6]{GR} to evaluate the last integral and by \eqref{nxiWhittVal}, we have the second statement of Lemma~\ref{P1-L3}. A similar computation yields the following majorant of the integral \eqref{P1-L3-f1}.   
\begin{align*}
&2^{1/2}\,\Delta^{-1/2}\,r^{\Re(s)+\rho+1}
\int_{\R} (1+y^2)^{(\Re(s)+\rho-l)/2}\,\d y. 
\end{align*}  
Thus, we obtain the first part of Lemma~\ref{P1-L3}. 
\end{proof}

\noindent
{\it Proof of the formula \eqref{P1-f1}} : From the first assertion of Lemmas~\ref{P1-L2} and \ref{P1-L3}, combined with \cite[Corollary 16]{Tsud2011-1}, we can exchange the order of integrals by Fubini's theorem to have the equality \eqref{P1-1}. We conclude the proof of Proposition~\ref{P1} by \eqref{P1-1} and the second assertions of Lemmas~\ref{P1-L2} and \ref{P1-L3}.

\section{The proof of Proposition~\ref{P1} : the singular term from $\sw_0$} \label{JJw0sing}
We continue to work with $\cU$. Let $f\in \cB(\cU;\bK_{1,\fin}^{\xi*})$. 
For any $\delta \in \sG_1^\xi(\Q)\bsl \sG_1(\Q)$ we set
\begin{align}
\cJ_{l}^{\xi,\delta}(\sw_0,s)&=
\int_{\sN(\R)} \Phi_{l}^\xi(s;\,\sw_0\,n\,\sm(r;\,\delta_\infty
)\,b_\infty)\,\psi_{\delta\xi,\infty}(n)^{-1}\,\d n, 
 \label{cJdelsw0-inf}
\\
\cJ_\fin^{\eta,f,\delta}(\sw_0,\phi|s)&=
\int_{\sG_1^\xi(\Q)\bsl \sG_1^\xi(\A_\fin)}\bar f(h_{0})\d h_{0} \int_{\sN(\A_\fin)} \Phi_{\fin}^{f,\xi}(\phi|s;\,\sw_0\,n\,\sm(1;\,\delta_\fin \,h_{0}))\,\psi_{\delta\xi,\fin}(n)^{-1} \,\d n.
 \label{cJdelsw0-fin}
\end{align}
From \eqref{SINGw0},  
\begin{align}
\hat\JJ_{l}^{\xi,f,\rm{sing}}(\sw_0,\phi|\beta)=
\int_{(c)}\beta(s)\,D_*(s)\,L^*(\cU,-s)\,\bigl(\cJ_{l}^{\xi,f,+}(\sw_0,\phi|s;r)+\cJ_l^{\xi,f,-}(\sw_0,\phi|s;r)\bigr)\,\d s,
 \label{P2-1}
\end{align}
with $c>\rho$, where 
$$
\cJ_l^{\xi,f,\pm}(\sw_0,\phi|s)=\cJ_{\fin}^{\xi,f,\delta^{\pm}}(\sw_0,\phi|s)\,\cJ_{l}^{\xi,\delta^{\pm}}(\sw_0, s;r)
$$
with $\delta^{+}=1$ and $\delta^{-}=\cnt^{\sG_1}$. To obtain \eqref{P2-1}, we argue by Fubini's theorem using the estimates in the following lemmas \ref{P2-L1} and \ref{P2-L2}.

\begin{lem} \label{P2-L1} 
\begin{align*}
\int_{V_{1}(\A_\fin)} |\Phi_{\fin}^{f,\xi}(\phi|s;\,w_0\,\sn(X_\fin)\,\sm(1;\,\delta_\fin^{\pm}\,h_{0,\fin}))|\,\d X_\fin \ll 1, \quad \Re(s)=c.
 \end{align*}
Moreover, for $\Re s>\rho$, 
\begin{align*}
\cJ_{\fin}^{\xi,f, \pm }(\sw_0,\phi|s)&=\cJ_{\fin}^{\xi,f,\pm}(\phi|-s)\times |Q[\xi]|^{-1}
\frac{L^*_{\fin}(\cU,s)}{L^*_{\fin}(\cU,s+1)}.
\end{align*}
\end{lem}
\begin{proof} We have that $\cJ_\fin^{\xi,f,\pm}(\sw_0,\phi|s)$ equals
\begin{align*}
&\fd(\cL_1^{\xi})^{-1}\fd(\cL_1)^{1/2}
\int_{\sG_1^\xi(\Q)\bsl \sG_1^\xi(\A_\fin)}\bar f(h_0)\,\d h_0  \\
&\quad \times 
\int_{V_1(\A_\fin)} \int_{\sG^{\xi}(\A_\fin)} \sf^{(s)}(h)\phi(h^{-1}\sw_0 \sn(X) \sm(1;\delta^{\pm}_\fin h_{0,\fin}))\psi_\fin(\langle X,\xi\rangle)\,\d h \,\d X
\\
&=\fd(\cL_1^{\xi})^{-1}\fd(\cL_1)^{1/2}|Q[\xi]|_\fin^{1/2} 
\int_{\sG_1^\xi(\Q)\bsl \sG_1^\xi(\A_\fin)}\bar f(h_0)\,\d h_0  \\
&\quad \times \int_{\sG^{\xi}(\A_\fin)}\d h \, \int_{V_1^\xi(\A_\fin)} \sf^{(s)}(\sw_0\sn(Z) h)\,\d Z\, \int_{\A_\fin} \phi(h^{-1}\sn(x\xi)\sm(1;\delta^{\pm}_\fin h_{0,\fin}))\psi_\fin(\langle \xi,\xi \rangle x)\,\d x,
\end{align*}
using the relation of the self-dual measures $\d X=|Q[\xi]|_\fin^{1/2}\d x\,\d Z$ to show the second equality. Note that $\sf^{(s)}$ depends on our $f\in \cB(\cU;\bK_{1,\fin}^{\xi*})$. From \cite[Corollary 1.10]{MS98}, 
\begin{align*}
\fd(\cL_1^{\xi})^{1/2}\int_{V_1^\xi(\A_\fin)} \sf^{(s)}(\sw_0 \sn(Z)h)\,\d Z=\frac{L_\fin^*(\cU,s)}{L^{*}_\fin(\cU,s+1)}\,\sf^{(-s)}(h),\quad h\in \sG^{\xi}(\A_\fin), \,\Re s>\rho.
\end{align*}
where $\d Z$ is the self-dual Haar measure on $V_1^\xi(\A_\fin)$ so that $\cL_{1,\fin}^{\xi}$ has measure $\fd(\cL_{1}^{\xi})^{-1/2}$. Thus, the last expression of $\cJ_l^{\eta,\pm}(\sw_0,\phi|s)$ equals the product of
\begin{align*}
\fd(\cL_1^{\xi})^{-1/2}\fd(\cL_1)^{1/2}|Q[\xi]|^{-1/2}
\frac{L_\fin^*(\cU,s)}{L^{*}_\fin(\cU,s+1)},
\end{align*}
which equals $|Q[\xi]|^{-1}L^*(\cU,s)L^{*}(\cU,s+1)^{-1}$ by Lemma~\ref{IndexL}, and \begin{align*}
\fd(\cL_1^{\xi})^{-1}
\int_{\sG_1^\xi(\Q)\bsl \sG_1^\xi(\A_\fin)}\bar f(h_0)\,\d h_0  
\int_{\sG^{\xi}(\A_\fin)}\sf^{(-s)}(h)\,\d h  \int_{\A_\fin} \phi(h^{-1}\sn(x\xi)\sm(1;\delta^{\pm}_\fin h_{0,\fin}))\psi_\fin(\langle \xi,\xi \rangle x)\,\d x,\end{align*}
which equals $\cI_{l}^{\xi,f,\pm}(\phi|-s)$ from definition \eqref{JJidentity-padic-f0}. 
\end{proof}

\begin{lem} \label{P2-L2}
Let $c>\rho-l$.  
\begin{align}
\int_{V_{1}(\R)} |\Phi_{l}^\xi(s;\,\sw_0\,\sn(X_\infty)\,\sm(r;\,\delta_\infty^{\pm})\,b_\infty)|\,\d X_\infty\ll 1, \quad \Re(s)=c.
 \label{P2-L2-1}
\end{align}
For $\Re(s)>\rho-l$,  
\begin{align*}
\cJ_l^{\xi,+}(\sw_0,s;r)&=(-1)^{l}\,
\cJ_{l}^{\xi,-}(\sw_0,s;r)\\
&=\frac{(2\pi)^{\rho} \Gamma(s)}{\Gamma(s+\rho)}\,\frac{(\sqrt{8|Q[\xi]}\pi)^{s-\rho+l}}{\Gamma(s+l-\rho)}\,\fd(\cL_1^\xi)^{-1/2}\,\cW_l^{\xi}(\sm(r;1_m)b_\infty).
 \end{align*}
\end{lem}
\begin{proof} Note that the self-dual measure $\d X$ on $V_{1}(\R)=\R^{m}$ is decomposed as $\d X=|Q[\xi]|_\infty^{1/2} \d x\,\d X_0$ as $X=x\,\xi+X_0$ ($x\in \R$, $X_0 \in V_{1}^\xi(\R)$) with $\d X_0$ the Euclidean Haar measure on $V_{1}^{\xi}(\R)$. We fix an orthonormal basis of $V^{\xi}_1(\R)$ to identify $V_{1}^{\xi}(\R)$ with the the standard Euclidean space $\R^{m-1}$ with norm $\|\cdot\|^2$. Then, noting that $Q[\xi]<0$, from Proposition~\ref{RealShinExBruhForm} (2), we have 
\begin{align*}
\cI_{l}^{\eta,+}(\sw_0, s;r)\fd(\cL_1)^{1/2} 
&=|Q[\xi]|_\infty^{1/2} \int_{\R} \left(1+ \tfrac{\sqrt{\Delta} x}{\sqrt{2}r}\,i\right)^{s+\rho-l}\,
\exp(-2\pi i\langle \xi,\xi \rangle x ) \,\d x
\\
&\quad \times \int_{V_{1}^\xi(\R)} \left\{r\left(1+ \tfrac{\sqrt{\Delta} x}{\sqrt{2}r}\,i\right)^2+\|{(2r)}^{-1/2}X_0\|^2\right\}^{-(s+\rho)}\,\d X_0.
\end{align*}
The $X_0$-integral is computed by the formula \eqref{P4-L3-SL1-0}. Thus, the integral $\cJ_{l}^{\eta,+}(\sw_0,s;r)\fd(\cL_1)^{1/2}$ equals 
\begin{align*}
&(2r)^{(m-1)/2}\, r^{-s}\,\Delta^{1/2}\,\pi^\rho \frac{\Gamma(s)}{\Gamma(s+\rho)}\,\int_{\R} 
\left(1+\tfrac{\sqrt{\Delta} x}{\sqrt{2}r}\,i\right)^{-s+\rho-l}\,\exp(-2\pi i 
\langle \xi,\xi \rangle x ) \,\d x\\
&=\Delta^{-1/2} \, 2^{m/2}\,r^{-s+\rho+1}\,\Delta^{1/2}\,\pi^\rho \frac{\Gamma(s)}{\Gamma(s+\rho)}\,
\int_{\R} 
\left(1+i\,y\right)^{-s+\rho-l}\,\exp(2\pi i \sqrt{2\Delta}\,r\,y) \,\d y.
\end{align*}
Using \cite[3.382.6]{GR} to evaluate the $y$-integral, we obtain  
\begin{align*}
\cJ_l^{\eta,+}(\sw_0,s;r)&=\fd(\cL_1)^{-1/2}|Q[\xi]|_\infty^{-1/2}
\times (2\pi)^{\rho}\frac{\Gamma(s)}{\Gamma(s+\rho)}
\frac{(\sqrt{8|Q[\xi]|}\pi)^{s-\rho+l}}{\Gamma(s-\rho+l)}\,\ee_l(r). 
\end{align*}
By Lemma~\ref{IndexL}, $\fd(\cL_1)^{-1/2}|Q[\xi]|_\infty^{-1/2}=\fd(\cL_1^{\xi})^{-1/2}=\fd(\cL_1^\xi)^{-1/2}$. Note the formula \eqref{nxiWhittVal}. This completes the proof of the second part of Lemma~\ref{P2-L2}. By the same way as above, we have that the integral \eqref{P2-L2-1} equals
\begin{align*}
\Delta^{-1/2}2^{m/2}\,r^{-\Re(s)+\rho+1}\,\pi^\rho \frac{\Gamma(\Re(s))}{\Gamma(\Re(s)+\rho)}\,
\int_{\R} 
|\left(1+i\,y\right)^{-s+\rho-l}| \,\d y.
\end{align*} 
From this, the first assertion is obvious. 
\end{proof}

\noindent
{\it Proof of the formula~\eqref{P1-f2}} : From Lemmas~\ref{P2-L1}, \ref{P1-L2} and \ref{P2-L2}, 
\begin{align*}
&\cJ_l^{\xi,\delta^{+}}(\sw_0,\phi|s;r)
\\
&=
\frac{(2\pi)^{\rho}\Gamma(s)}{\Gamma(s+\rho)}\frac{(\sqrt{8|Q[\xi]|}\pi)^{s-\rho+l}}{\Gamma(s+l-\rho)}\fd(\cL_1^\xi)^{-1/2}\frac{L^{*}_\fin(\cU,s)}{L_\fin^{*}(\cU,s+1)}\delta(2\xi \in \cL_1) \\
&\quad \times 
\prod_{p\in S}\widehat{\calW}_p^{\xi,(z_p)}(\phi_p;-s)\cW_l^{\xi}(\sm(r;1_m)b_\infty).
\end{align*}
By a direct computation using the explicit form of the gamma factor of $L^{*}(\cU,s)$ and the functional equation $L^{*}(\cU,-s)=L^*(\cU,s+1)$ (\cite[\S 3.6]{Tsud2011-1}), we confirm the equality 
\begin{align*}
(2\pi)^{\rho}\frac{\Gamma(s)}{\Gamma(s+\rho)}D_{*}(s)L^{*}(\cU,-s)\fd(\cL_1^\xi)^{-1/2}\frac{L_\fin^{*}(\cU,s)}{L_\fin^*(\cU,s+1)}
=D_{*}(-s)L^{*}(\cU,s). 
\end{align*}
Hence 
$$
D_*(s)L^*(\cU,-s)\,\{\cJ_l^{\xi,f,\delta^{+}}(\sw_0,\phi|s;r)+\cJ_l^{\xi,f,\delta^{-}}(\sw_0,\phi|s;r)\}=\MM_l^\xi(\phi|-s)\times\cW_l^{\xi}(\sm(r;1_m)b_\infty),
$$
which combined with \eqref{P2-1} completes the proof. 

\section{The proof of Proposition~\ref{Prop2}: the regular term from $\sw_0$} \label{JJw0regular}
Starting with the formula \eqref{REGw0}, formally exchanging order of the summation and the integrals, we obtain  
\begin{align}
\hat\JJ_{l}^{\xi,f,{\rm reg}}(\fw_0,\phi|\beta;r)&= \int_{\sG_1^\xi(\Q)\bsl \sG_1^\xi(\A_\fin)}\bar f(h_0)\,\left\{
\int_{(c)} \beta(s) D_*(s)L^*(\cU,-s)\,\fJ_l^{\xi,f,{\rm reg}}(\sw_0,\phi|s,h_0;r)
\,\d s\right\}\,\d h_0
 \label{JJw0regular-f0}
\end{align}
with \begin{align}
\fJ^{\xi,f,\rm{reg}}_l(\sw_0,\phi|s,h_0;r)=\sum_{\delta\in \sG^\xi_1(\Q)\bsl[\sG_1(\Q)-\{1,\cnt^{\sG_1}\}\,\sG_1^\xi(\Q)]} \cJ_{l}^{\xi,\delta}(\sw_0,s;r)\cJ_\fin^{\xi,f,\delta}(\sw_0,\phi|s,h_0), \qquad h_0\in \sG_1^\xi(\A), 
 \label{P4-f100}
\end{align}
where $\cJ_l^{\xi,\delta}(\sw_0,s;r)$ and $\cJ_\fin^{\xi,f,\delta}(\sw_0,\phi|s,h_0)$ are the integrals defined by \eqref{cJdelsw0-inf} and \eqref{cJdelsw0-fin}, respectively. The proof of the equality \eqref{JJw0regular-f0} and the absolute convergence of the series \eqref{P4-f100} will be given later. We fix a compact subset $\cN\subset \sG_1^\xi(\A_\fin)$ such that $\sG_1^\xi(\A_\fin)=\sG_1^\xi(\Q)\cN$ once and for all.  

\begin{lem} \label{P4-L1} 
Let $c>\rho$. Then, 
\begin{align*}
|\cJ_\fin^{\xi,f,\delta}(\sw_0,\phi|s,h_0)|\ll 1, 
\quad \delta \in \sG_1(\Q),\,h_0\in \cN,\, \Re(s)=c. 
\end{align*}
There exists a compact set $\cK_\phi'\subset V_1(\A_\fin)$ such that $\cJ^{\xi,f,\delta}_\fin(\sw_0,\phi|s,h_0)=0$ for all $\Re s>\rho$ and $h_0\in \cN$ unless $\delta_\fin^{-1}\xi \in \cK_\phi'$. 
\end{lem}
\begin{proof} From a computation as in the proof of Lemma~\ref{P2-L1}, we have 
\begin{align*}
|\cJ_\fin^{\xi,f,\delta}(\sw_0,\phi|s,h_0)|&\leq \frac{|L^*(\cU,s)|}{|L^*(\cU,s+1)|} \int_{\A_\fin}|\Phi_{\fin}^{\xi,f}(\phi|-s;\sn(x\xi)\sm(1;\delta_\fin h_0))|
\d x. 
\end{align*}
A computation shows the relations
\begin{align*}
(\sn(x\xi)\sm(1;\delta_\fin h_0))^{-1}\xi&=h_0^{-1}\delta_\fin^{-1}\xi+x Q[\xi]\,\e_1, \\
(\sn(x\xi)\sm(1;\delta_\fin h_0))^{-1}\e_1&=\e_1.
\end{align*}
Let $\cK_\phi \subset V(\A_\fin)$ be the compact set as in Lemma~\ref{pAdicTF-EST}. Then there exists a compact set $\cK_\phi''\subset V_1(\A_\fin)$ and $N\in \N^*$ such that $(\sn(x\xi)\sm(1;\delta_\fin h_0))^{-1}\xi\in \cK_\phi$ implies $h_0^{-1}\delta_\fin^{-1}\xi \in \cK_{\phi}''$ and $x\in N^{-1}\hat \Z$. Set $\cK_\phi'=\cN\cK_\phi''$. Then from the estimate in Lemma~\ref{pAdicTF-EST}, 
$$
|\Phi_{\fin}^{f,\xi}(\phi|-s;\sn(x\xi)\sm(1;\delta_\fin h_0))|\ll \delta(\delta_\fin^{-1}\xi\in \cK_\phi',\,x\in N^{-1}\hat\Z), \quad h_0\in \cN,\, x \in \A_\fin. 
$$
\end{proof}

For $\eta \in V_1(\R)$, let $\eta^{\#}$ denote the orthogonal projection of $\eta$ to $V_{1}^\xi(\R)$ and $\|\eta^\#\|=Q[\eta^\#]^{1/2}$. For $\delta\in \sG_1(\Q)$, the vector $\delta\xi_0^{-}\in V_1(\R)$ is decomposed as   
$$
\delta\xi_0^{-}=a(\delta)\,\xi_0^{-}+(\delta\xi_0^{-})^\# \quad {\text{with $a(\delta)=-\langle \delta \xi_0^{-}, \xi_0^{-}\rangle$}}. 
$$From the relation $-1=Q[\xi_0^-]=Q[\delta\,\xi_0^{-}]=-a(\delta)^2+\|(\delta\xi_0^{-})^{\#} \|^2$, we have \begin{align*}
\|(\delta\xi_0^{-})^\#\|<|a(\delta)|, \qquad 1=|a(\delta)|^2-\|(\delta\xi_0^{-})^\#\|^2. 
\end{align*}

\begin{lem} \label{P4-L3-SL2}
For any $\delta \in \sG_1(\Q)$, we have $|a(\delta)|\geq 1$; the equality holds if and only if $\delta\in \{1,\cnt^{\sG_1}\}\,\sG_1^\xi(\Q)$. 
\end{lem}
\begin{proof}
Since $\|(\delta\xi_0^-)^\# \|\geq 0$, this follows from $1=|a(\delta)|^2-\|(\delta\xi_0^{-})^\# \|^2$ immediately. 
\end{proof}

\begin{lem} \label{P4-L4} 
Let $\cK'\subset V_1(\A_\fin)$ be a compact set. Then, there exists a positive number $\epsilon=\epsilon_{\cK'}$ such that $|a(\delta)|\geq 1+\epsilon$ for any $\delta\in \sG_1(\Q)-\{1,\cnt^{\sG_1}\} \sG_1^\xi(\Q)$ satisfying $\delta_\fin^{-1}\xi \in \cK'$.  
\end{lem}
\begin{proof}
Let $\cM_1\subset V_{1}$ be a $\Z$-lattice such that $\cL_{1}\otimes \hat\Z +\cK' \subset \cM_{1}\otimes \hat\Z$. Then, $\xi \in \cM_1$. Let $N\in \N$ be an integer such that $N\cM_1\subset \cL_1$. Then, $ Q[\xi]\, N\,a(\delta) \in \Z$ for any $\delta\in \sG_1(\Q)$ with $\delta_\fin^{-1}\xi \in \cK'$. Thus, $\{a(\delta)|\,\delta\in \sG_1(\Q)-\{1,\cnt^{\sG_1}\}\sG_1^\xi(\Q),\,\delta_\fin^{-1}\xi \in \cK'\,\}$ is a discrete subset of $\R_{+}$. Combining this observation with Lemma~\ref{P4-L3-SL2}, we have the conclusion. 
\end{proof}

For $s\in \C$ and for $\eta \in V_{1}(\R)$ such that $Q[\eta]<0$ and $\eta^\#\not=0$, set
\begin{align}
\Psi_l^{(s)}(\eta)=i\,\pi^{\rho-l}\,2^{3(\rho-l+1)/2}\,\Gamma(l-\rho)
\, \int_{(c)} z^{\rho-l}\,K_s(2\sqrt{2}\| \eta^\#\|\pi z)\,\exp(-2\sqrt{2}\pi \,\langle \xi^{-}_0,\eta\rangle\,z)\,\d z 
 \label{P4-f1}
\end{align}
with $c>0$, where $K_s(z)$ is the $K$-Bessel function. Note that the integral converges absolutely since the integrand is bounded by $(1+|\Im(z)|)^{-l+\rho}$ on $\Re(z)=c>0$, which is integrable for $l>4\rho+1$. We should also note that Cauchy's integral formula shows that the integral is independent of $c$.

\begin{lem} \label{P4-L3} Let $s\in \C$. If $\langle \xi_0^{-},\eta\rangle>0$, then $\Psi_l^{(s)}(\eta)=0$ identically. If $\langle \xi_0^{-},\eta\rangle<0$, then 
\begin{align*}
\Psi_{l}^{(s)}(\eta)&=2^{-l+\rho+1}
\|\eta^\# \|^{-s}
(\sqrt{-Q[\eta]})^{s+l-\rho-1}
\int_{-1}^{+1}(1-t^2)^{l-\rho-1}
\left(\frac{-\langle \xi_0^{-},\eta\rangle}
{\sqrt{-Q[\eta]}}+t\right)^{s-l+\rho}\,\d t.
\end{align*}
\end{lem}
\begin{proof}
Set $a=-\langle \xi_0^{-},\eta\rangle$. Suppose $a<0$. Then, from the asymptotic $K_{s}(z)\sim \sqrt{\pi/2z}\,e^{-z}$ $(z\rightarrow \infty,\,\Re(z)>0)$ (\cite[p.139]{MOS}), the integrand of \eqref{P4-f1} is bounded by $C\,|z|^{\rho-l-1/2}\,\exp(-q\,\Re(z))$ with some constants $C,\,q>0$ on $\Re(z)>0$. Thus, by letting $c \rightarrow +\infty$, we obtain $0$ as the integral value of \eqref{P4-f1}. In the remaining part of the proof, we suppose $a>0$. From the formula $K_s(z)=\frac{1}{2}\int_0^\infty \exp(-(v+v^{-1})z/2)\,v^{s-1}\,\d v$, 
\begin{align}
&\int_{(c)} z^{\rho-l}\,K_s(2\sqrt{2}\|\eta^\#\|\pi z)\,\exp(2\sqrt{2}\pi \,a\,z)\,\d z,
 \notag
\\
&=\tfrac{1}{2}\int_{0}^{+\infty} \{\int_{(c)} z^{\rho-l}\exp\left(2\sqrt{2}\pi \,\{ a-\tfrac{v+v^{-1}}{2}\,\|\eta^\#\|\}\,z\right)\,\d z\} \,v^{s-1}\,\d v.
 \label{P4-L3-1}
\end{align}
For $p\in \R-\{0\}$ and $b>0$, we have
\begin{align*}
\int_{(c)} z^{-b}\,\exp(pz)\,\d z&=
-2\pi i\times \begin{cases} 
 0 \qquad &(p<0), \\
 p^{b-1}\,\Gamma(b)^{-1} \qquad &(p>0).
\end{cases}
\end{align*}
by the formula \cite[3.382.7]{GR}. Apply this with $p=2\sqrt{2}\pi \,\{ a-\frac{v+v^{-1}}{2}\,\|\eta'\|\}$ and $b=l-\rho$; the condition $p>0$ is equivalent to $v_{-}(\eta)<v<v_{+}(\eta)$ with $v_{\pm}(\eta)=({-\langle \xi_0^{-},\eta \rangle \pm  \sqrt{-Q[\eta]} }){\|\eta^\#\|}^{-1}$. From \eqref{P4-f1}, \eqref{P4-L3-1}, 
\begin{align*}
\Psi_l^{(s)}(\eta)=\left(\tfrac{\|\eta^\#\|}{2}\right)^{l-\rho-1} \int_{v_-(\eta)}^{v_+(\eta)}\{(v_{-}(\eta)-v)(v_{+}(\eta)-v)\}^{l-\rho-1}v^{s-l+\rho}\d v. 
\end{align*}
To obtain the formula in the lemma, we only have to set $v=({-\langle \xi_0^{-},\eta \rangle +\sqrt{-Q[\eta]}\,t }){\|\eta^\# \|}^{-1}$; then $(v_{-}(\eta)-v)(v_{+}(\eta)-v)=(\sqrt{-Q[\eta]}/\|\eta^\#\|)^2\,(1-t^2)$ and the interval $[v_{-}(\eta),v_{+}(\eta)]$ is mapped linearly to $[-1,+1]$. 
\end{proof}

For $\epsilon>0$, set $${\bf Y}(\epsilon)=\{\eta\in V_1(\R)|\,Q[\eta]=-1,\,-\langle \xi_0^{-},\eta\rangle\geq 1+\epsilon\,\}.$$

\begin{cor}\label{P4-L5}
 For any compact interval $I\subset \R_+^\times$ and for any $\epsilon>0$, 
$$
|\Psi_{l}^{(s)}(\eta)|\ll_{\epsilon,I} 2^{-l}(l-\rho)^{-1/2}|\langle \xi_0^{-},\eta\rangle|^{-l+\rho}, \quad l\in \N\cap (\rho,+\infty),\, \Re(s) \in I,\,\eta \in {\bf Y}(\epsilon).
$$
\end{cor}
\begin{proof} Set $c=\Re(s)$ and $a_\eta=-\langle \eta,\xi_0^{-}\rangle$ for $\eta \in V_1(\R)$. Noting the relation $|a_\eta|^2-\|\eta^\#\|^2=1$ for $\eta \in {\bf Y}(\epsilon)$, from Lemma~\ref{P4-L3}, we have
\begin{align*}
|\Psi^{(s)}_l(\eta)|
&\leq 2^{-l+\rho+1}(a_\eta^2-1)^{-c/2} (a_\eta+1)^{c} \int_{-1}^{+1}(1-v^2)^{l-\rho-1}(a_\eta+v)^{-l+\rho}\d v
\\
&\leq 2^{-l+\rho+1}((a_\eta+1)/(a_\eta-1))^{c/2}\,(a_\eta-1)^{-l+\rho}\int_{-1}^{1}(1-v^2)^{l-\rho-1}\d v
\end{align*}
for $c>0$, $l>\rho$. The last $v$-integral is easily computed as $B(1/2,l-\rho)=\sqrt{\pi}\Gamma(l-\rho)/\Gamma(l-\rho+1/2)$, which is $O((l-\rho)^{-1/2})$ by Stirling's formula. Since $|a_\eta|\geq 1+\epsilon$ for $\eta \in {\bf Y}(\epsilon)$, $(a_\eta+1)/(a_\eta-1) \asymp 1$ and $a_\eta-1\asymp a_\eta$. 
\end{proof}

The integral $\cJ^{\xi,\delta}_l(\sw_0,s;r)$ is evaluated in terms of $\Psi_l^{(s)}(\delta\xi_0^{-})$. 
\begin{lem} \label{P4-L2}
Let $\delta\in \sG_1(\Q)-\{1,\cnt^{\sG_1}\}\,\sG_1^\xi(\Q)$ and $s\in \C$ such that $\Re(s)>0$. Then,  
\begin{align*}
\cJ_l^{\xi,\delta}(\sw_0,s;r)
&=
\frac{-(\sqrt{8\Delta}\pi)^{l}}{\Gamma(l-\rho)}
\,\frac{\|\sqrt{2\Delta}\pi\,(\delta\xi_0^{-})^\#\|^{s}}{(2\Delta)^{\rho/2}\,\Gamma(s+\rho)}\,\Psi_l^{(s)}(\chi_\delta\,\delta\xi_0^{-1}) \times \cW_l^{\xi}(\sm(r;1_m)b_\infty)
\end{align*}
with $\chi_\delta=\sgn\langle \delta\xi_0^{-},\xi_0^{-}\rangle$. 
\end{lem}
\begin{proof}
For $\eta\in V_{1}(\R)$, consider the integral 
$$
W(\eta;\,g)=\int_{V_{1}(\R)} \Phi_l^\xi(s;\,\fw_0\,\sn(X)\,g)\,\exp(-\langle \eta,X\rangle)\,\d X, \quad g\in \sG(\R)^+. 
$$
Granting the absolute convergence of this integral for a while, we have that the function $W(\eta,g)$ belongs to the space $\tilde \fW_l(\eta)$ (see \S\ref{ApprFE-1}).  Thus, there exists a constant $C$ independent of $g$, such that $W(\eta;\,g)=C\,\cW_l^\eta(g)$ for all $g\in \sG(\R)^+$. To show the absolute convergence of the integral and evaluate the constant $C$, it suffices to examine $W(\eta;\,g)$ for $g=\sm(r;1_m)\,b_\infty$ $(r>0)$. By Lemma~\ref{RealShinExBruhForm} (2),  \begin{align*}
W(\eta;\,\sm(r;1_m)\,b_\infty)
&=\Delta^{1/2} \int_{\R}\int_{V_{1}^\xi(\R)}
\left(1+\tfrac{ x}{\sqrt{2} r}\,i\right)^{s+\rho-l} \, 
\left\{r\left(1+\tfrac{ x}{\sqrt{2} r}\,i\right)^2+\tfrac{1}{2r}\,Q[X_0]\right\}^{-(s+\rho)}\\
&\qquad \times\exp(-2\pi i\,\langle \xi_0^{-},\eta\rangle \,x) 
\, \exp(-2\pi i \langle X_0,\eta\rangle )\,\d x\,\d X_0
\end{align*}
By the variable change $X_0\mapsto (2r)^{-1/2}X_0$ and by using Lemma~\ref{P4-L3-SL1}, we have that the last expression of $W(\eta;\,\sm(r;\,\epsilon 1_m)\,b_\infty)$ equals
\begin{align*}
&\epsilon^{l} \Delta^{1/2}\,(2r)^{\rho}\,2\pi^{\rho}\,(2^{1/2} r^{1/2}\pi\|\eta^\#\|)^{s}\,\Gamma(s+\rho)^{-1}\,r^{-s/2} 
\\
&\times \quad \int_{\R}
\left(1+\tfrac{x}{\sqrt{2} r}\,i\right)^{\rho-l}K_s\left(2\sqrt{2}r\pi\left(1+\tfrac{x}{\sqrt{2} r}\,i\right)\,\|\eta^\#\|\right)\,
\exp(-2\pi i\, \langle \xi_0^{-},\eta\rangle \,x)\,\d x
 \end{align*} 
By an obvious variable change, this becomes
\begin{align*}
&\epsilon^{l} \Delta^{1/2}\,(2r)^{\rho}\,2\pi^{\rho}\,(2^{1/2} r^{1/2}\pi\|\eta^\#\|)^{s}\,\Gamma(s+\rho)^{-1}\,r^{-s/2} \\
&\quad \times 2^{1/2}\,r\,(-ir^{l-\rho-1})\,\exp(2\sqrt{2}\,\pi \langle \xi_0,\eta\rangle )
\int_{(r)} z^{\rho-l}\,K_s(2\sqrt{2}\pi\|\eta^\#\|z)\,\exp(-2\sqrt{2}\pi\, \langle \xi_0^{-},\eta\rangle 
\,z)\,\d z
\\
&={-2^{(3l-\rho+s)/2}\pi^{l+s}\Delta^{1/2}}{\Gamma(s+\rho)^{-1}\Gamma(l-\rho)^{-1}}\,\|\eta^\#\|^{s}\,\Psi_l^{(s)}(\eta) \times \cW^{\eta}_l(\sm(r;1_m)b_\infty),\end{align*}
where $\cW_l^\eta(\sm(r;1_m)b_\infty)=r^{l}\exp(\sqrt{2}\,\pi \langle \xi_0,\eta\rangle)$ from \eqref{ArchWhittaker}. We have to show the absolute convergence of $W(\eta;\,\sm(r;\,g_1)\,b_\infty)$. Set $\sigma=\Re(s)$. We have 
\begin{align*} 
&\int_{V_{1}(\R)}\left|
\left(1+\tfrac{x}{\sqrt{2} r}\,i\right)^{s+\rho-l} \, 
\left\{r\left(1+\tfrac{x}{\sqrt{2} r}\,i\right)^2+\tfrac{1}{2r}\,Q[X_0]\right\}^{-(s+\rho)}\right|\,\d X
\\
&\ll \exp(\pi|\Im(s)|)\,
\int_{\R}\int_{V_{1}^\xi(\R)} \,\left(1+\tfrac{x^2}{2r^2}\right)^{\sigma+\rho-l}\,\left|r\left(1+\tfrac{x}{\sqrt{2} r}\,i\right)^2+\tfrac{1}{2r}\,Q[X_0]\right|^{-(\sigma+\rho)}\,\d x\,\d X_0
\end{align*}
Thus, it suffices to have the convergence of the integral
\begin{align}
\int_{\R}\int_{0}^{\infty} \,\left(1+{x^2}\right)^{\sigma+\rho-l}
\,\left|\left(1+{x}\,i\right)^2+v^2\right|^{-(\sigma+\rho)}\,v^{m-2}\,\d x\,\d v. \label{qqqqq}
\end{align}
Let $\sigma>0$. Then, by the easily confirmed inequality $\left|\left(1+{x}\,i\right)^2+v^2\right|\geq \{1+(v^2-x^2)^2\}^{1/2}\gg (1+x^2)\,(1+u^2)$ with $v=u\,|x|$, we have that \eqref{qqqqq} is majorized by 
$$
\biggl(\int_{\R} (1+x^2)^{-(\sigma+\rho+l)}|x|^{m-1}\,\d x\biggr)\,\biggl(
\int_{0}^{\infty} (1+u^2)^{-(\sigma+\rho)}\,u^{m-2}\,\d u\biggr),
$$
whose convergence for $\sigma>0$ is obvious. To complete the proof of the formula in the lemma, we note the relation $\cJ_l^{\xi,\delta}(\sw_0,s;r)=W(\delta\xi;\,\sm(r;\,\delta)\,b_\infty)$. The point $\sm(r;\delta)$ belongs to $\sG(\R)^+$ if and only if $\langle \delta\xi_0^{-},\xi_0^{-}\rangle<0$. 
We have $\cW_l^{\delta \xi}(\sm(r;\,\delta)\,b_\infty)=\ee_l(r)$ if $\langle\delta\xi_0^-,\xi_0^-\rangle <0$. When $\langle \delta\xi_0^{-},\xi_0^{-}\rangle>0$, by the relation $W(\xi;\sm(1;-1_m)g)=W(-\xi;g)$ which is confirmed by the second statement of Lemma~\ref{RealShintani-equi-L}, we have $\cJ_l^{\xi,\delta}(\sw_0,s;r)=W(-\delta\xi;\,\sm(r;\,-\delta)\,b_\infty)$. We remark that $\Psi_{l}^{(s)}(\chi_{\delta}\delta\xi)=(\Delta^{1/2})^{l-\rho-1}\Psi_{l}^{(s)}(\chi_\delta\delta\xi_0^{-})$ from the formula in Lemma~\ref{P4-L3}. 
\end{proof}
The application of Fubini's theorem to show the identity \eqref{JJw0regular-f0} is justified by the next Corollary.  
\begin{cor}
For any $\delta\in \sG_1(\Q)-\{1,\cnt^{\sG_1}\}\,\sG_1^\xi(\Q)$ and $h_0\in \cN$, \begin{align*}
\int_{(c)}|\beta(s)|\,|D_*(s)\,L^*(\cU,-s)|\,\{\int_{\sN(\A)}|\mathbf\Phi_l^{\xi,f}(\phi|s\;,\fw_0\,n\,\sm(r;\,\delta h_0)\,b_\infty)|\,\d n\}\,|\d s|<\infty. 
\end{align*}
\end{cor}
\begin{proof} 
From the obvious modification of the proof of Lemma~\ref{P4-L1} and \ref{P4-L2}, the $n$-integral, regarded as a function in $s$, turns out to be bounded on $\Re(s)=\sigma$. The remaining factors of the integrand is of rapid decay by \cite[Corollary 16]{Tsud2011-1}. Thus, the conclusion follows. \end{proof}

For a lattice $\cM_1 \subset V_1$, set $\cO(\xi)=\sG_1(\Q)\xi-\{\pm \xi\}$ and 
$$
 {\bf Z}_{\cM_1}(s)=\sum_{\substack{Y \in \cM_1\cap \cO(\xi) \\ \langle Y,\xi\rangle<0,\,}} \,|\langle Y,\xi\rangle/Q[\xi]|^{-s}, \qquad s\in \C.
$$
\begin{lem} \label{P4-L7}
Let $\cM_1\subset V_1$ be a lattice. Then, the series ${\bf Z}_{\cM_1}(s)$ converges absolutely if $\Re(s)>\rho-1/2$. Moreover, there exists a constant $C_{\cM_1}>1$ such that for any $\sigma>\rho-1/2$, 
$${\bf Z}_{\cM_1}(\sigma+l)={\mathcal O}(C_{\cM_1}^{-l}), \qquad l\in \N.
$$
\end{lem}
\begin{proof}
 The first assertion is proved in the same way as \cite[Lemma 68]{Tsud2009}. The second statement follows from the convergence of ${\bf Z}_{\cM_1}(\sigma)$; we can take $C_{\cM_1}=\min\{|\langle Y,\xi\rangle/Q[\xi]|\,Y \in \cM_1 \cap \sG_1(\Q)\xi-\{\xi\},\,\langle \xi,Y \rangle<0\,\}$. The inequality $C_{\cM_1}>1$ follows from Lemma~\ref{P4-L4}. 
 \end{proof}

\begin{lem} \label{P4-L8}
Let $\cN \subset \sG_1^\xi(\A)$ be a compact set and $\sigma \in (\rho,l-3\rho-1)$. Then, the series \eqref{P4-f100} converges absolutely and uniformly for $h_{0}\in \cN$ and $\Re(s)=\sigma$. Moreover, there exists a lattice $\cM_1$ in $V_1$ depending on $\cN$ such that  
\begin{align*}
&\sup_{h_0\in \cN} \,|\fJ^{\xi,f,\rm{reg}}_{l}(\sw_0,\phi|s,\,h_0;r)|\ll \frac{(\sqrt{8\Delta}\pi)^{l-\rho}}{\Gamma(l-\rho)}\,{\bf Z}_{\cM_1}(-\Re(s)+l-\rho)\,
\times \cW_l^{\xi}(\sm(r;1_m)b_\infty)
\end{align*}
for $\Re(s)=\sigma$ and for any sufficiently large $l$ with the implied constant independent of $s$, $r>0$, and $l$.  
\end{lem}
\begin{proof}
Let $\cM_1$ be as in the proof of Lemma~\ref{P4-L4}. Then, from Lemmas~\ref{P4-L1} and \ref{P4-L4}, the vectors $Y=\delta^{-1}\eta$ with $\delta\in \sG_1(\Q)-\{1,\cnt^{\sG_1}\}\,\sG_1^\xi(\Q)$ such that $\cJ_\fin^{\xi,f,\delta}(\sw_0,\phi|s,h_0)\cJ_l^{\xi,\delta}(\sw_0,s;r)\not=0$ for some $h_0\in \cN$ and $\delta\in \sG_1(\Q)$ belongs to the set 
$${\bf Y}(\epsilon)\cap \{Y \in \cM_1\cap \cO(\xi)|\,\langle Y ,\xi\rangle<0\,\}$$
with some $\epsilon>0$. From $\|(\delta\xi_0^-)^\#\|^2=a(\delta)^2-1\asymp a(\delta)^2$ for $\delta\xi \in {\bf Y}(\epsilon)$. The required estimation follows from this remark, combined with Corollary~\ref{P4-L5} and Lemma~\ref{P4-L7}.   
\end{proof}

\noindent
{\it The proof of Proposition~\ref{Prop2}} : The formula in Proposition~\ref{Prop2} is obtained from \eqref{JJw0regular-f0} with 
\begin{align}
\JJ_{l}^{\xi,f,\rm reg}(\sw_0,\phi|s;r)=\int_{\sG_1^\xi(\Q)\bsl \sG_1^\xi(\A_\fin)}\bar f(h_0)\,\fJ_l^{\xi,f,{\rm reg}}(\sw_0,\phi|s,h_0;r)\,\d h_0.
\label{P4-f10}
\end{align}
The necessary order exchange of the contour integral in $s$ and the summation in $\delta$ is allowed by Lemma~\ref{P4-L8} with $\cN$ such that $\sG^\xi_1(\Q)\cN=\sG_1^\xi(\A)$. Then, the estimation in Proposition~\ref{Prop2} follows from Lemmas~\ref{P4-L8} and \ref{P4-L7}. \qed
  
%

\section{The proof of Proposition~\ref{Prop3}} \label{JJbsne}
Let $\su\in \{\sw_1,\sn(\xi)\sw_0\}$. Since $\sN_\mu$ is trivial for any $\mu=\fw_1\,\sm(\tau;\,\delta)\in \sM(\su)$ (Lemma~\ref{STB-L1} (2)), the term $\hat\JJ_l^{\xi,f}(\fu,\phi|\beta;r)$ is equal to 
\begin{align}
\int_{\sG_1^\xi(\Q)\bsl \sG_1^\xi(\A)} \bar f(h_0)\,\d h_0\,\sum_{\tau\in \Q^\times} \sum_{\delta \in {\sR}(\Q) \bsl \sG_1(\Q)} 
\int_{\sN(\A)} \hat{\mathbf{\Phi}}_l^{f,\xi}(\phi|\beta; \su\,\sm(\tau;\,\delta)\,n\,\sm(r;\,h_0)\,b_\infty)\,\psi_{\xi}(n)^{-1}\,\d n. 
 \label{P5-f1}
\end{align} 
where $\sR=\sP_1^{0}$ if $\su=\fw_1$ and $\sR=\sG_1^\xi$ if $\su=\sn(\xi)\fw_0$. Formally changing the order of summation and integrals, we have
\begin{align}
\hat\JJ_l^{\xi,f}(\fw_1,\phi|\beta;r)=
\int_{\sG_1^\xi(\Q)\bsl \sG_1^\xi(\A_\fin)} \bar f(h_0)\biggl\{ \int_{(c)} \beta(s)D_*(s)L^*(\cU,-s)\fJ_{l}^{\xi,f}(\su,\phi|s,h_0;r)\,\d s \biggr\}\,\d h_0,
 \label{P5-f0}
\end{align}
where
\begin{align}
\fJ_l^{\xi,f}(\su,\phi|s,h_0;r)=\sum_{\tau\in \Q^\times} \sum_{\delta \in \sR(\Q) \bsl \sG_1(\Q)} \cJ_{l}^{\xi,(\tau,\delta)}(\su|s;r)\cJ_\fin^{\xi,f,(\tau, \delta h_0)}(\su,\phi|s), \quad h_0 \in \sG_1^\xi(\A_\fin) 
 \label{Pp5-f2}
\end{align}
with 
\begin{align}
\cJ_\fin^{\xi,f,(\tau, h_\fin)}(\fu,\phi|s)&=\fd(\cL_1)^{1/2} \int_{V_1(\A_\fin)} {\mathbf \Phi}_\fin^{f,\xi}(\phi|s,\su \sn(X) \sm(\tau;h_\fin))\,\psi_\fin(-\tau^{-1}\langle h_\fin \xi,X\rangle)\, \d X, 
 \label{Pp5-f3}
\\
\cJ_l^{\xi,(\tau, h_\infty)}(\fu|s;r)&=\fd(\cL_1)^{-1/2}\int_{V_1(\R)} {\mathbf \Phi}_l^{\xi}(s;\su\sn(X) \sm(r\tau;h_\infty)b_\infty)\,\psi_\infty(-\tau^{-1}\langle h_\infty \xi,X\rangle)\,\d X
 \label{Pp5-f4}
\end{align}
for $h_\fin\in \sG_1(\A_\fin)$ and $h_\infty\in \sG_1(\R)$ and $\d X$ is the self-dual measure on $V_1(\A_\fin)$ or on $V_1(\R)$. In this section, we study these integrals in detail for the case $\su=\fw_1$ to prove the equality \eqref{P5-f0} showing the absolute convergence of the series \eqref{Pp5-f2} and its estimation on the way. Similar analysis for $\fu=\sn(\xi)\fw_0$ will be done in the next section.

\begin{lem} \label{P5-L2}
The integral $\cJ_{\fin}^{\xi,f,(\tau,h)}(\fw_1,\phi|s)$ converges absolutely locally uniformly in $(s,h)\in \C\times \sG_1(\A_\fin)$. There exists $N\in \N^*$ depending only on $\phi$ such that $\cJ_\fin^{\xi,(\tau,h)}(\fw_1,\phi|s)=0$ for all $(s,h)$ unless $\tau\in N^{-1}\Z-\{0\}$. There exists a constant $C_0>0$ such that 
\begin{align*}
|\cJ_\fin^{\xi,f,(\tau, h)}(\fw_1,\phi|s) |\leq C_0 \,\|h^{-1}\e_0\|_\fin^{-(\Re(s)+\rho)}\end{align*}
for $\Re(s)>-\rho$, $\tau \in N^{-1}\Z-\{0\}$, and $h \in \sG_1(\A_\fin)$. 
\end{lem}
\begin{proof} 
A computation shows the equalities 
\begin{align*}
(\fw_1\sn(X) \sm(\tau;h_\fin))^{-1}\xi&=(2\tau)^{-1}(-Q[X-\alpha]+Q[\xi])\,\e_1-h^{-1}(X-\alpha)+\tau\,\e_1', \\
(\fw_1\sn(X) \sm(\tau;h_\fin))^{-1}\e_1&=h^{-1}\e_0+\tau\langle X,\e_0\rangle\,\e_1.
\end{align*}
Let $\cK_\phi\subset V(\A_\fin)$ be a compact set as in Lemma~\ref{pAdicTF-EST}. There exists a compact set $\cK'_\phi\subset V_1(\A_\fin)$ and $N\in \N^*$ such that $(\fw_1\sn(X) \sm(\tau;h_\fin))^{-1}\xi\in \cK_\phi$ implies $h^{-1}(X-\alpha)\in \cK_\phi'$ and $\tau\in N^{-1}\Z-\{0\}$. Moreover, $\|h^{-1}\e_0+\tau\langle X,\e_0\rangle\,\e_1\|_\fin \geq \|h^{-1}\e_0\|_\fin$. Therefore, when $\Re(s)+\rho>0$, Lemma~\ref{pAdicTF-EST} provides us with the majorization  
$$
|\Phi_\fin^{f,\xi}(\phi|s;\fw_1\sn(X) \sm(\tau;h_\fin)|\ll \|h^{-1}\e_0\|^{-(\Re(s)+\rho)}_\fin\,\delta(\tau\in N^{-1}\Z-\{0\}, \,h^{-1}(X-\alpha)\in \cK_\phi').$$
From this the assertions of the lemma is evident. 
\end{proof}

\begin{lem}\label{P5-L4}
Let $\tau\in \Q^\times$ and $h\in \sG_1(\R)$. 
\begin{itemize}
\item[(1)] The integral $\cJ_l^{\xi,f, (\tau,h)}(\sw_1|s;r)$ converges absolutely locally uniformly in $(s,h)$ if $\Re(s)\in (0,l-2\rho-1/2)$. For any fixed $(\tau,h)$, 
$$
\int_{V_1(\R)} |\Phi_l^{\xi}(s;\,\sw_1\,\sn(Y)\,\sm(\tau r;\,h)\,b_\infty)|
\,\d Y \ll \exp\left(\tfrac{\pi}{2}|\Im(s)|\right), \quad (\Re(s)\in (0,l-2\rho-1/2)). 
$$
\item[(2)] Let $\Re(s)\in (0,l-2\rho-1/2)$. If $\langle h \xi,\e\rangle>0$, we have $\cJ_l^{\xi,f, (\tau,h)}(\sw_1|s;r)=0$. If $\langle h\xi,\e\rangle<0$, then  $\cJ_l^{\xi,(\tau,h)}(\sw_1|s;r)$ equals
\begin{align*}
&4\Delta^{-1/4}\pi\, i^{3l} (\sqrt{8\Delta}\pi)^{l-\rho-s}\frac{|\langle \e_0,h\xi\rangle |^{-(s+\rho)}}{\Gamma(-s+l-\rho)} \\
&\,\times e^{-2\pi i \tau^{-1}\langle h\xi,\alpha\rangle}\,|\tau|^{s+2\rho+1/2}\,J_{l-\rho-1/2}(2\pi\Delta^{1/2}|\tau|^{-1})\,\cW_l^{\xi}(\sm(r;1_m)b_\infty),
\end{align*}
where $J_\nu(z)$ is the $J$-Bessel function, and $\alpha \in V_0(\R)$ is as in \eqref{xiForm}.
\end{itemize}
\end{lem}
\begin{proof} 
First we consider the integral 
$$
W(\eta,s;g)= \int_{V_1(\R)} \Phi_l^{\xi}(s;\,\sw_1\,\sn(X)\,g)\,\psi(-\langle \eta,X \rangle)\,\d X, \qquad g \in \sG(\R)^+ 
$$
with $\eta \in V_{1}(\R)$. Let $g=\sm(\tau;\,h)b_\infty$. Then, from \cite[Lemma 27]{Tsud2011-1}, 
$$
|\Phi_l^{\xi}(s;\,\sw_1\,\sn(X)\,g)|\leq (\ch t)^{\Re(s)+\rho-l}\,\exp\left(\tfrac{\pi}{2}|\Im(s)|\right)
$$
if $\sw_1\,\sn(X)\,g\in \sG^\xi(\R)\,a_\infty^{(t)}\,b_\infty \bK_{\infty}$ and $\Re(s)>0$. From \cite[Lemma 37]{Tsud2011-1} and Lemma~\ref{P1-L3-L} (3), we have that $\ch^2 t=|A(\sw_1\sn(X)g)|^2$ is expressed as a degree $4$ polynomial of coordinates of $Y=h\,(X-\alpha)$. 
Therefore, the ratio $(\ch^2 t) /(1+\|Y\|)^{4}$, regarded as a function in $Y$, is bounded on $V_1(\R)=\R^{m}$ and never attains zero. Thus, there exists a constant $C_0>0$ such that $\ch^2 t \geq C_0\,(1+\|Y\|)^{4}$ for all $Y\in V_1(\R)$. Hence, if $\Re(s)+\rho-l<0$ and $\Re(s)>0$, 
\begin{align*}
\int_{X\in \R^m} |\Phi_l^{\xi}(s;\,\fw_1\,\sn(X)\,\sm(\tau;h)b_\infty)|\,\d X
&\leq C_1\,\exp\left(\tfrac{\pi}{2}|\Im(s)|\right)\,
\int_{0}^{+\infty} (1+u)^{2(\Re(s)+\rho-l)}\,u^{m-1}\,\d u
\end{align*}
with $C_1=\vol(\bS^{m-1})\,C_0^{(\Re(s)+\rho-l)/2}$. The last integral is convergent if $\Re(s)<l-2\rho-1/2$. Thus, we obtain the absolute convergence of $W(\eta,s;\,g)$. By the same reasoning as in the proof of Lemma~\ref{P4-L2}, there exists a constant $\phi(s)$ independent of $g$ such that $W(\eta,s;\,g)=\phi(s)\,\cW_l^\eta(g)$ for $g\in \sG(\R)^+$. To determine $\phi(s)$, it suffices to examine $W(\eta,s;\sm(r;1)\,b_\infty)$ with $r>0$. By Proposition~\ref{RealShinExBruhFormW1}, 
\begin{align}
&W(\eta,s;\sm(r;1)\,b_\infty) 
 \label{P5-L4-1}
\\ &=
i^{3l}2^{3l/2-2(s+\rho)}\Delta^{l/2}r^{-(s+\rho)+l}
\int_{\R^m} \left\{Q[Y_0]+y_{+}^2+(iy_{-}+\sqrt{2}r)^{2}+\Delta \right\}^{s+\rho-l}\left(1+\tfrac{y_{-}-y_{+}}{\sqrt{2}r}i \right)^{-(s+\rho)}
 \notag
\\
&\quad\times e^{-2\pi i \langle \eta,\alpha\rangle}\,\exp(-2\pi i \langle Y_0,\eta''\rangle)\,\exp(-2\pi i c_{+}y_+)\,\exp(2\pi i c_{-}y_{-})
\,\d Y_0 \,\d y_{+}\,\d y_{-}. 
 \notag
\end{align}
Here, $Y_0 \in W\cong \R^{m-2}$, $y_{+},y_{-}\in \R$ and $\eta=\eta''+c_{+}\xi_{0}^{+}+c_{-}\xi_0^{-}$ with $\eta''\in W$ and $c_\pm \in \R$. If we further set $z=\sqrt{2}r+i(y_{-}-y_{+})$, then \eqref{P5-L4-1} becomes 
\begin{align*}
&i^{3l}2^{3l/2-2(s+\rho)}\Delta^{l/2}r^{-(s+\rho)+l}\times (-i)\exp(-2\sqrt{2} \pi c^{-} r)\,e^{-2\pi i \langle \eta,\alpha\rangle}\\
&\quad \times
\int_{Y_0 \in \R^{m-2}} \int_{y^{+} \in \R} \int_{\Re(z)=\sqrt{2}r} 
\left(\tfrac{z}{\sqrt{2}r}\right)^{-(s+\rho)}\{\|Y_0\|^2+z^2+2izy_{+} +\Delta\}^{s+\rho-l}\\
&\qquad \times e^{-2\pi i \langle Y_0,\eta''\rangle}\,e^{2\pi i(c_{-}-c_{+})y_{+}}\,e^{2\pi c_{-}z}\,\d z\,\d y^{+}\,\d Y_0.
\end{align*}
By Lemma~\ref{P5-L6} (1), this equals 
\begin{align*}
&i^{3l}2^{3l/2-2(s+\rho)}\Delta^{l/2}r^{-(s+\rho)+l}\times (-i)\exp(-2\sqrt{2} \pi c^{-} r)\times 2\pi^{\rho-1}\tfrac{\Gamma(-s+l-2\rho+1)}{\Gamma(-s+l-\rho)}\,e^{-2\pi i \langle \alpha,\eta\rangle} \\
&\quad \times \int_{\Re z=\sqrt{2}r}\int_{y_{+}\in \R} \left(\tfrac{z}{\sqrt{2}r}\right)^{-(s+\rho)} \left(\int_{0}^{\infty} (u^2+z^2+\Delta+2izy_{+})^{s-l+2\rho-1} \cos(2\pi\|\eta''\|u)\d u\right) \\
&\qquad \times 
e^{2\pi i (c_--c_+)y_{+}}\,e^{2\pi c_{-}z} \,\d z\,\d y_{+}. 
\end{align*} 
Suppose $c^{+}-c^{-}\not=0$. Then, by \cite[3.382.6]{GR}, the $y_{+}$-integral is calculated as
\begin{align*}
(2z)^{s-l+2\rho-1}\times \tfrac{-2\pi\,|2\pi(c^{+}-c^{-})|^{-s+l-2\rho}}{\Gamma(l-2\rho-s+1)}\,\exp\left(2\pi (c^{+}-c^{-})\tfrac{u^2+z^2+\Delta}{2z}\right)
\end{align*}
if $c^{+}-c^{-}<0$ and to zero otherwise. Summing up the argument so far, we obtain 
\begin{align}
&W(\eta,s;\,\sm(r;\,1)\,b_\infty)
 \label{P5-L4-2} 
\\
&=i^{3l+1}
2^{3(l-s-\rho)/2+1}\pi^{-(s+\rho)+l} \Delta^{l/2} r^{l}
\,\exp(-2\sqrt{2} \pi c^{-} r)\,\tfrac{|c_+-c_-|^{l-s-2\rho}\Gamma(-s+l-2\rho+1)}{\Gamma(-s+l-\rho)\,\Gamma(l-2\rho-s+1)}e^{-2\pi i \langle \alpha,\eta\rangle}
 \notag
\\
&\quad\times \,\int_{\Re(z)=\sqrt{2}r} z^{-l+\rho-1}\,\exp\left(\pi(c^{+}-c^{-})\left(\tfrac{\Delta}{z}+z\right)\right) \d z 
 \notag
\\
&\quad \cdot \int_{0}^{+\infty} 
\exp(\pi(c^{+}-c^{-})z^{-1}u^2)\,\cos(2\pi \|\eta''\|u)\,\d u
\notag
\end{align}
if $c^{+}-c^{-}<0$, and $W(\eta,s;\,\sm(r;\,1)\,b_\infty)=0$ otherwise. From now on, we suppose $c^{+}-c^{-}<0$. By computing the $u$-integral in \eqref{P5-L4-2} by Lemma~\ref{P5-L6} (2) and then using Lemma~\ref{P5-L6} (3) to compute the $z$-integral, we have that $W(\eta,s;\,\sm(r;\,1)\,b_\infty)$ equals 
\begin{align*}
&
i^{3l+1}2^{3l/2-s-\rho+1}\pi^{-(s+\rho)+l} \Delta^{l/2} r^{l}
\,\exp(-2\sqrt{2} \pi c^{-} r)\,\tfrac{|c_+-c_-|^{l-s-2\rho-1/2 }}{\Gamma(-s+l-\rho)} e^{-2\pi i \langle \alpha,\eta\rangle}\\
&\quad\times \,\int_{\Re(z)=\sqrt{2}r} z^{-l+\rho-1/2}\,
\exp\left(\pi\left\{\tfrac{\Delta(c^{+}-c^{-})}{ z}+\tfrac{\|\eta\|^2\,z}{c^{+}-c^{-}}\right\} \right)\, \d z 
\\
&=i^{3l}2^{3(l-s-\rho)/2+2}\pi^{-(s+\rho)+l+1} \Delta^{-(s+\rho)/2+l/2-1/4}r^{l}\,\exp(-2\sqrt{2} \pi c^{-} r)\,\tfrac{|\langle \e_0,\eta\rangle |^{-(s+\rho)}
}{\Gamma(-s+l-\rho)}e^{-2\pi i \langle \alpha,\eta\rangle} \\
&\quad \times \|\eta\|^{l-\rho-1/2}J_{l-\rho-1/2}(2\pi\|\eta\|\Delta^{1/2}).
\end{align*}
Note that $\|\eta\|^2=\|\eta''\|^2+c_{+}^2-c_{-}^2$ and $c_{+}-c_{-}=-\Delta^{1/2}\langle \e_0,\eta\rangle$. We also note that $\cW_l^\xi(\sm(r;\,1)\,b_\infty)=r^l\,\exp(-2\pi\, 2^{1/2}\Delta^{-1/2}\langle \xi,\xi\rangle r)$ (\cite[Proposition 3]{Tsud2011-1}). This completes the evaluation of $\phi(s)$. By the formula 
$$
\cJ_l^{\xi,(\tau,h)}(\sw_1|s;r)
=W(\tau^{-1}\,h\xi ,s;\,\sm(\tau r;\,h)\,b_\infty),
$$
we are done.   
\end{proof}

\noindent
{\bf Remark} : In (2), the case $\langle h\xi,\e_0\rangle=0$ never happen. Indeed, since $h\xi \in V_1(\R)$ and $Q[h\xi]=Q[\xi]<0$ for $h\in \sG_1(\R)$, the space $(h\xi)^\bot \cap V_1(\R)$ is positive definite, which means the isotropic vector $\e_0 \in V_1(\R)$ can not be orthogonal to $h\xi$. 

\begin{lem}
Let $q\in \N^*$. Then we have the estimate
\begin{align*}
&|\cJ_l^{\xi,f,(\tau,h)}(\sw_1,s)|
\ll_{q} \frac{|\langle \e_0,h\xi\rangle|^{-(\Re(s)+\rho)}}
{|\Gamma(-s+l-\rho)|}\,(\sqrt{8\Delta}\pi)^{l} \frac{(\pi\sqrt{\Delta})^{l}
|\tau|^{\Re(s)+3\rho-l+1}}{\Gamma(l-\rho+1/2)}
\end{align*}
for $l\in \N^*\cap [\rho+q+1,\infty)$, $\Re(s)\in (0,l-2\rho-1/2)$, $\tau\in \Q^\times$ and $h\in \sG_1(\R)$, where the implied constant is independent of $(s,\tau,h,l)$.  
\end{lem}
\begin{proof} 
From \eqref{JBesselEst3}, 
$$
|J_{l-\rho-1/2}(x)|\ll (x/2)^{l-\rho-1/2}\Gamma(l-\rho+1/2)^{-1}, \quad x>0, l\in \N^{*}\cap (\rho+1, \infty)
$$
The required bound follows from these and the exact formula in Lemma~\ref{P5-L4} (2). 
\end{proof}

\begin{lem} \label{P5-L8}
 Let $c\in (\rho,l-3\rho-1)$. Then, for any $\tau\in \Q^\times$ and $h\in \sG_1(\A)$, 
\begin{align*}
\int_{(c)}\,{|\beta(s)|}\,|D_*(s)\,L^*(\cU,-s)|\,
\biggl\{\int_{\sN(\A)}|{\mathbf\Phi}_l^{\xi,f}(s;\,\fw_1\,n\,\sm(r \tau;\,h)\,b_\infty)|\,\d n\biggr\}\,|\d s|<+\infty. 
\end{align*}
\end{lem}
\begin{proof} Note that the factor $|D_*(s)L^{*}(\cU,-s)|=|D_*(s)L^{*}(\cU,s+1)$ is of exponential decay on $\Re(s)=c\,(>\rho)$ due to the gamma factor. The necessary convergence follows from this remark together with Lemmas~\ref{P5-L2} and \ref{P5-L4} (1).  
\end{proof}

Recall that $\cN\subset \sG_1^\xi(\A_\fin)$ is a compact set such that $\sG_1^\xi(\A_\fin)=\sG_1^\xi(\Q)\,\cN$. 

\begin{lem} \label{P5-L7}
Let $\epsilon>0$. On the strip $\Re(s) \in (\rho,l-2\rho-1/2)$, the double series \eqref{Pp5-f2} converges absolutely. There exists $N\in \N^*$ such that for any $c \in (\rho,l-3\rho-2)$, we have the estimate 
\begin{align}
&\sup_{h_0 \in \cN } \,|\fJ_{l}^{\xi,f}(\fw_1,\phi|s,h_0)|
\ll \left|\frac{(\sqrt{8\Delta}\pi)^{l}}{\Gamma(-s+l-\rho)}\right|\,\frac{(\pi\sqrt{\Delta}N)^{l}}{\Gamma(l-\rho+1/2)}
 \label{P5-L7-1}
\end{align} 
for any $\Re(s)=c$ and for any $l>100\rho$ with the implied constant independent of $(l,s,n)$. 
\end{lem}
\begin{proof} Let $h\in \sG_1^\xi(\A)$. From Lemmas~\ref{P5-L2} and \ref{P5-L4}, the series $\sum_{\tau}\sum_{\delta}|\cJ_{l}^{\xi,(\tau,\delta h_\infty)}(\sw_1,s)\cJ_\fin^{\xi,f,(\tau,\delta h_\fin)}(\fw_1,\phi|s)|$ is majorized by the product of\\ $\left|\tfrac{(\sqrt{8\Delta}\pi)^{l}\,\Gamma(-s+l+1)}{\Gamma(-s+l-\rho)\,\Gamma(-s+l-2\rho+1)}\right|$ and 
\begin{align*}
&\biggl\{\sum_{\delta\in \sP_1^0(\Q)\bsl \sG_1(\Q)} |\langle \xi,\delta^{-1}\e_0\rangle|_\infty^{-(c+\rho)}\|\delta^{-1}\e_0\|_\fin^{-(c+\rho)}\biggr\}\times \frac{(\pi\sqrt{\Delta})^{l}}{\Gamma(l-\rho+1/2)}
\sum_{\tau \in N^{-1}\Z} |\tau|_\infty^{-l+c+3\rho+1}
\end{align*}
uniformly in $(l,s)$ such that $\Re(s)=c$ and $l>4\rho+1$. From the convergence of the spherical Eisenstein series on $\sG_1$, the $\delta$-summation is finite if $c>\rho-1/2$. 
It is evident that $\sum_{\tau \in N^{-1}\Z} |\tau|_\infty^{-l+c+3\rho+1}\ll N^l$ uniformly in $l$. \end{proof}

\noindent
{\it The proof of Proposition~\ref{Prop3}} :
By Lemma~\ref{P5-L8}, we change the order of $n$-integral and the contour integral in $s$ in \eqref{P5-f1} by Fubini's theorem. Thus to prove \eqref{Pp5-f2} by another application of Fubini's theorem, it suffices to have the convergence of the integral 
\begin{align*}
&\int_{(c)}|\beta(s)|\,|D_*(s)\,L^*(\cU,-s)| \\
&\quad \int_{\cN}|f(h_0)|\,\{\sum_{\tau\in \Q^\times}\sum_{\delta \in \sP^1_0(\Q)\bsl \sG_1(\Q)} |\cJ_l^{\xi,(\tau,\delta)}(\sw_1,s)\cJ_\fin^{\xi,f,(\tau,\delta h_0)}(\sw_1,\phi|s)|\,\}\,\d h_0\,|\d s|.
\end{align*} 
which is already seen in the proof of Lemma~\ref{P5-L7}. 

\section{The proof of Proposition~\ref{Prop4}} \label{JJbsnxi}
In this section, we study the integrals \eqref{Pp5-f3} and \eqref{Pp5-f4} to prove the identity \eqref{P5-f0} and the absolute convergence of the series \eqref{Pp5-f2}, which leads us to the proof of Proposition~\ref{Prop4}. 

\begin{lem} \label{P6-L1} Let $\tau\in \Q^\times$ and $h\in \sG_1(\A_\fin)$. 
The integral $\cJ_{\fin}^{\xi,f,(\tau,h)}(\sn(\xi)\sw_0,\phi|s)$ converges absolutely and locally uniformly in $(s,h)$. There exists $N_1\in \Q^\times$ such that integral $\cJ_{\fin}^{\xi,f,(\tau,h)}(\sn(\xi)\sw_0,\phi|s)$ is zero unless $\tau\in N_1^{-1}{\Z}-\{0\}$. For $\Re(s)+\rho\geq 0$, we have the estimate
\begin{align}
|\cJ_{\fin}^{\xi,f,(\tau,h)}(\sn(\xi)\sw_0,\phi|s)| \ll (\|h^{-1}\xi\|_\fin\,|\tau|_\fin)^{
-(\Re(s)+\rho)}
, \qquad \tau \in N_1^{-1}\Z-\{0\},\,\eta\in \N^*\,\xi,\, h\in \sG_1(\A_\fin). 
 \label{P6-L1-1}
\end{align}
\end{lem}
\begin{proof}
By a computation, we have
\begin{align*}
(\sn(\xi)\sw_0\sn(X)\sm(\tau;h))^{-1}\xi&=\tau^{-1}(-2^{-1}Q[X]Q[\xi]+\langle \xi,X\rangle)\,\e_1+h^{-1}(\xi-Q[\xi]\,X)+\tau Q[\xi]\,\e_1',\\
(\sn(\xi)\sw_0\sn(X)\sm(\tau;h))^{-1}\e_1&=-2^{-1}Q[X]\tau^{-1}\,\e_1-h^{-1}X+\tau\,\e_1'.
\end{align*}
Let $\cK_\phi\subset V(\A_\fin)$ be a compact set as in Lemma~\ref{pAdicTF-EST}. There exists a compact set $\cK'_\phi\subset V_1(\A_\fin)$ and $N\in \N^*$ such that $(\sn(\xi)\fw_0\sn(X) \sm(\tau;h_\fin))^{-1}\xi\in \cK_\phi$ implies $h^{-1}(\xi-X)\in \cK_\phi'$ and $\tau Q[\xi] \in N^{-1}\Z-\{0\}$. Let $p$ be a prime number. Then by $|\tau|_p \leq |N_1|_p^{-1}$ with $N_1=NQ[\xi]$,  
\begin{align*}
\|-2^{-1}Q[X]\tau^{-1}\,\e_1-h^{-1}X+\tau\,\e_1'\|_p  
&\geq \|-h^{-1}X+\tau\,\e_1'\|_p =\max(\|h^{-1}X\|_p,|\tau|_p) \\
&\gg |\tau|_p \,\max(\|h^{-1}X\|_p ,1) \gg \max(\|h^{-1}\xi\|_p,1)|\tau|_p, 
\end{align*}
where the last majorization is confirmed by noting that $h^{-1}\xi-h^{-1}X$ remains in a compact set. Therefore, when $\Re(s)+\rho>0$, Lemma~\ref{pAdicTF-EST} provides us with the majorization  
$$
|\Phi_\fin^{f,\xi}(\phi|s;\sn(\xi)\fw_0 \sn(X) \sm(\tau;h_\fin)|\ll (|\tau|_\fin\|h^{-1}\xi\|_\fin)^{-(\Re(s)+\rho)}\,\delta(\tau\in N_1^{-1}\Z-\{0\}, \,h^{-1}(X-\xi)\in \cK_\phi').$$
Bounding the integral on $\sN(\A_\fin)$ by volume $\vol(h\cK_\phi')=\vol(\cK_\phi')$, we have the required bound from this. 
\end{proof}

For $s\in \C$ with $\Re(s)\in (-\rho,l-\rho)$, $l\in \N^*$, $a\in \R^\times_+$ and $b\in \R$, set 
\begin{align}
\Ical_l^{(s)}(a,b)=\int_{0}^{1}(1-x)^{s+\rho-1}x^{-(s+1/2)}J_{l-\rho-1/2}\left(2\pi ax\right) \exp\left(2\pi ib x \right)\,\d x.
 \label{Ftn-V}
\end{align}


\begin{lem} \label{P6-L5}
\begin{itemize}
\item[(1)] The integral $\cJ_l^{\xi,(\tau,h)}(\sn(\xi)\sw_0|s;r)$ converges absolutely and locally uniformly in $(s,\tau,h)$ if $\Re(s)\in (-\rho,l-2\rho-1/2)$.  
\item[(2)] Set $\epsilon={\rm sgn}(\tau\langle h\xi_0^{-},\xi_0^{-}\rangle)$. Then $\cJ_l^{n\xi,(\tau,h)}(\sn(\xi)\sw_0|s;r)$ equals 
\begin{align*}
\Delta^{-1}\,\pi^{\rho+1} 
\frac{(\sqrt{8\Delta}\pi)^{l-\rho}\,({2^{-1}\Delta})^{s/2}}{\Gamma(s+\rho)\Gamma(-s+l-\rho)}\,{|\tau|}^{\rho+1/2}\,
{\mathcal I}_l^{(s)}\left(\frac{1}{|\tau|}, \left|\frac{\langle h\xi,\xi\rangle}{\tau}\right|\right)\times \cW_{l}^{\xi}(\sm(r;1_m)b_\infty).
\end{align*}
\end{itemize}
\end{lem}
\begin{proof}
For $\epsilon \in \{\pm 1\}$, set
$$
W^{\epsilon}(\eta,s;\,g)=\int_{\R^{m}}\Phi_l^{\xi}(s;\,\sn(\epsilon \xi)\sw_0\,\sn(X)\,g)\,\psi_\infty(-\langle X,\eta\rangle)\,\d X, \qquad g\in \sG(\R). 
$$
By the same reasoning as in the proof of Lemma~\ref{P5-L4}, we can show that this integral converges absolutely when $\Re(s)<l-2\rho-1/2$ and is decomposed as
 $W^{\epsilon}(\eta,s;\,g)=\phi^\epsilon(s)\,\cW_l^\eta(g)$ ($g\in \sG(\R)^{+}$) with a function $\phi^\epsilon (s)$ in $s$ independent of $g$. To determine $\phi^\epsilon(s)$, we compute $W^{\epsilon}(\eta,s;\,\sm(r;\,1_m)\,b_\infty)$ for $r>0$. From Lemma~\ref{RealShinExBruhFormW0xi}, 
\begin{align*}
W^{\epsilon}(\eta,s;\,\sm(r;\,1_m)\,b_\infty) &=
2^{-(s+\rho-3l)/2}\Delta^{(s+\rho-l)/2} r^{l}\int_{\R} \varphi^\epsilon (x)\,\exp(2\pi i c x\Delta^{1/2} )\,\d x,
\end{align*}
where $c=-\langle \eta,\xi\rangle/\sqrt{\Delta}$, $\eta^\#=\eta-c\,\xi^-_0$ and \begin{align*}
\varphi^\epsilon(x)&=\epsilon^l\int_{\R^{m-1}} 
\left(\tilde A^\epsilon(x,Z)\right)^{-l}\left(\tilde B^\epsilon(x,Z)/{\tilde A^\epsilon(x,Z)}\right)^{-(s+\rho)}
\,\exp(-2\pi i \langle Z,\eta^\#\rangle) \,\d Z.
\end{align*}
with 
\begin{align*} 
\tilde A^\epsilon(x,Z)&=Q[Z]+\Delta^{-1}+\{\sqrt{2}r+(\Delta^{1/2}x+\epsilon\Delta^{-1/2})\,i\}^{2}, 
\tilde B^\epsilon(x,Z)=Q[Z]+(\sqrt{2}r+\Delta^{1/2}x\,i)^2.
\end{align*}
Set $z=\sqrt{2}r+\Delta^{1/2}xi$ and $\|Z\|=Q[Z]^{1/2}$. Then 
$$\tilde A^\epsilon(x,Z)=\tilde B^\epsilon (x,Z)+2i\epsilon \Delta^{-1/2}z=\|Z\|^2+z^2+2i\epsilon \Delta^{-1/2}z$$
is easily confirmed. By this, we have the expression
\begin{align*}
\varphi^\epsilon(x)=\epsilon^{l}\int_{\R^{m-1}}\left(1+\frac{2i\epsilon \Delta^{-1/2}z}{\|Z\|^2+z^2}\right)^{s+\rho-l}(\|Z\|^2+z^2)^{-l}\exp(-2\pi i \langle Z,\eta^\#\rangle)\,\d Z,\end{align*}
which coincides with the left-hand side of \eqref{P6-L4-1} with $v=\eta^\#$, $\alpha=s+\rho-l$ and $T=\Delta^{1/2}$. For a while, assume that $\sqrt{2\Delta}\,r>2$, or equivalently $\Re(z)>2T^{-1}$, and that $\eta^\#\not=0$. Then, from Lemma~\ref{P6-L4}, we obtain
\begin{align*}
\varphi^\epsilon(x)&=2\epsilon^l\,\pi^l\,\|\eta^\#\|^{l-\rho} 
\,z^{\rho-l}
\sum_{n=0}^{\infty}\left(\begin{smallmatrix} {s+\rho-l} \\ {n} \end{smallmatrix}\right)\,(2\epsilon \Delta^{-1/2}i\pi \|\eta^\#\|)^{n}\,{K_{n-\rho+l}(2\pi\|\eta^\#\|\,z)}{\Gamma(n+l)}^{-1}.
\end{align*}
Hence, $\int_{\R} \varphi^\epsilon(x)\,\exp(2\pi i c\Delta^{1/2} x)\,\d x$ equals
\allowdisplaybreaks{
\begin{align*}
&2\epsilon^{l}\pi^l\,\|\eta^\#\|^{l-\rho}\sum_{n=0}^{\infty}
\left(\begin{smallmatrix} {s+\rho-l} \\ {n} \end{smallmatrix}\right)
\,(2i\pi \epsilon \|\eta^\#/\sqrt{\Delta}\|)^{n}\,{\Gamma(n+l)}^{-1} 
\\
&\quad \times \int_{x\in \R} z^{\rho-l}
{K_{n-\rho+l}(2\pi\|\eta^\#\|\,z)}
\exp(2\pi i \Delta^{1/2} cx)\,\d x
\\
&=2\epsilon^{l}\pi^l\,\|\eta^\#\|^{l-\rho}\sum_{n=0}^{\infty}
\left(\begin{smallmatrix} {s+\rho-l} \\ {n} \end{smallmatrix}\right)
\,(2i\pi \epsilon\|\eta^\#/\sqrt{\Delta}\|)^{n}\,{\Gamma(n+l)}^{-1} 
\,\exp\left({-2\sqrt{2} \pi rc}\right)
\\
&\quad \times \frac{-i}{\sqrt{\Delta}}
\int_{\Re(z)=\sqrt{2} \,r}
z^{\rho-l}
{K_{n-\rho+l}(2\pi\|\eta^\#\|\,z)}
\exp(2\pi c z)\,\d z
\\
&=2\epsilon^l\pi^l\,\|\eta^\#\|^{l-\rho}\sum_{n=0}^{\infty}
\left(\begin{smallmatrix} {s+\rho-l} \\ {n} \end{smallmatrix}\right)
\,(2i\pi \epsilon \|\eta^\# /\sqrt{\Delta}\|)^{n}\,{\Gamma(n+l)}^{-1} 
\,\exp\left({-2\sqrt{2} \pi r c}\right)
\\
&\quad \times  \frac{-2^{(\rho-l+1)/2}\,i}{\sqrt{\Delta}}
\frac{\Psi_{l}^{(n-\rho+l)}(\eta)}{i\,\pi^{\rho-l}\,2^{3(\rho-l+1)/2}\,\Gamma(l-\rho)}.
\end{align*}}
Here $\Psi_{l}^{(s)}(\eta)$ is the integral defined by \eqref{P4-f1}. From the last expression, invoking Lemma~\ref{P4-L3}, we have that the integral $\int_{\R} \varphi^\epsilon( x)\,\exp(2\pi i c \Delta^{1/2} x)\,\d x$ is zero unless $c>0$, in which case it equals
\begin{align}
&\frac{-2^{l-\rho-1}\,\pi^{2l-\rho}\,\Delta^{-1/2}}{\Gamma(l-\rho)}
\,\|\eta^\#\|^{l-\rho}\sum_{n=0}^{\infty}
\left(\begin{smallmatrix} {s+\rho-l} \\ {n} \end{smallmatrix}\right)
\,(2i\pi \epsilon \|\eta^\#/\sqrt{\Delta}\|)^{n}\,{\Gamma(n+l)}^{-1} 
 \label{P6-L5-3}
\\
&\quad \times\exp(-2\sqrt{2}\pi c\,r)
\,
2^{-l+\rho+1}\|\eta^\#\|^{-n-l+\rho}|Q[\eta]|^{(n-2\rho+2l-1)/2}\int_{-1}^{1}(1-t^2)^{l-\rho-1}\left(\frac{c}{|Q[\eta]|^{1/2}}+t\right)^{n}\,\d t.
 \notag
\end{align}
Using $\Gamma(x)\,\Gamma(1-x)=\pi/\sin (\pi x)$, we easily have
\begin{align*}
&\sum_{n=0}^{\infty}
\left(\begin{smallmatrix} {s+\rho-l} \\ {n} \end{smallmatrix}\right)
\,{\Gamma(n+l)}^{-1}v^{n}
=\tfrac{-1}{\Gamma(l)}\,{}_1F_1(-s+l-\rho,l;-v).
\end{align*}
From \cite[Proposition 3]{Tsud2011-1}, $\cW_l^\eta(\sm(r;\,1)\,b_\infty)=r^{l}\,\exp(-2\sqrt{2}\pi c\,r))$. Thus, if $ c>0$ and $\sqrt{2\Delta}\,r>2$, \eqref{P6-L5-3} becomes the product of
\begin{align*}
\epsilon^{l} \tfrac{\pi^{2l-\rho}\,\Delta^{-1/2}|Q[\eta]|^{l-\rho-1/2}}{\Gamma(l-\rho)\,\Gamma(l)}
\,r^{-l} \cW_l^\eta(\sm(r;\,1)\,b_\infty)
\end{align*}
and 
\begin{align*}
\int_{-1}^{+1} (1-t^2)^{l-\rho-1}\,{}_1F_1\left(-s+l-\rho,l;\,-2\pi\epsilon i 
(\sqrt{-Q[\eta]/\Delta}\,t+\tfrac{c}{\sqrt{\Delta}}) \right)\,\d t. 
\end{align*}
All in all, under the assumption $\eta^\#\not=0$ we obtain $W^\epsilon(\eta,s;\,\sm(r;\,1)\,b_\infty)=\phi^\epsilon(s)\,{\mathcal W}_l^\eta(\sm(r;\,1)\,b_\infty)$ with
\begin{align*}
\phi^\epsilon (s)&=\epsilon^{l}
{2^{(3l-s-\rho)/2}\,\pi^{2l-\rho}\,\Delta^{(s+\rho-l-1)/2}}{\Gamma(l-\rho)^{-1}\,\Gamma(l)^{-1}}\,|Q[\eta]|^{l-\rho-1/2}\\
& \times \delta(\langle \eta,\xi\rangle<0)\, 
\int_{-1}^{+1} (1-t^2)^{l-\rho-1}\,{}_1F_1\left(-s+l-\rho,l;\,\tfrac{-2\pi\epsilon i}{\sqrt{\Delta}}(\sqrt{-Q[\eta]}\,t-\langle \xi^{-}_0,\eta\rangle)\right)\,\d t.
\end{align*}
Since this expression is continuous on the set of $\eta\in V_1(\R)$ with $Q[\eta]<0$, the formula remains true even when $\eta^\#=0$, i.e., when $\eta$ is proportional to $\xi$. By a simple variable change, we have the relation 
\begin{align}
\cJ_l^{\eta,(\tau,h)}(\sn(\xi)\fw_0|s;r)=W^{+}(\tau^{-1}h\eta ,s;\sm(r\tau;\,h)\,b_\infty). 
 \label{cJ=W+}
\end{align}
If $\tau\langle h\xi_0^{-},\xi_0^{-}\rangle<0$, then the point $\sm(r\tau;h)b_\infty$ belongs to $\sG(\R)^{+}$. 
If $\tau\langle h\xi_0^{-},\xi_0^{-}\rangle>0$, then $\sm(-r\tau;h)b_\infty\in \sG(\R)^{+}$. From the second claim of Lemma~\ref{RealShintani-equi-L}, by using the element $\sm(-1;1_m)\in \sP^\xi(\R)$, we have the relation
$$W^{+}(\tau^{-1}h\eta ,s;\sm(r\tau;\,h)\,b_\infty)=W^{-}(-\tau^{-1}h\eta ,s;\sm(-r\tau;\,h)\,b_\infty).  
$$
We complete the proof by Lemma~\ref{P6-L6}. 
\end{proof}

\begin{lem} \label{P6-L6}
If $\Re(s) \in (-\rho,l-\rho)$, then 
\begin{align}
&\int_{-1}^{1}(1-t^2)^{l-\rho-1}{}_1F_1\left(-s+l-\rho,l;\,
{2\pi i }(at+b)\right)\,\d t 
 \label{P6-L6-0}
\\
&=\frac{\sqrt{\pi}\,\Gamma(l)\,\Gamma(l-\rho)}{\Gamma(s+\rho)\,\Gamma(-s+l-\rho)}
\left({\pi |a|}\right)^{-l+\rho+1/2}\,{\Ical}_l^{(s)}\left(|a|, b\right)
\label{P6-L6-1}
\end{align}
for any $a\in \R^\times$ and $b\in \R$.
\end{lem}
\begin{proof} By the integral representation 
\begin{align*}
&{}_1F_1\left(-s+l-\rho,l;\,{2\pi i }(at+b)\right) \\
&=\tfrac{\Gamma(l)}{\Gamma(-s+l-\rho)\,\Gamma(s+\rho)}\,\int_{0}^{1} 
\exp\left({2\pi i }(at+b)\,x \right)\,x^{-s+l-\rho-1}(1-x)^{s+\rho-1}\,\d x, 
\end{align*}valid on the region $-\rho<\Re(s)<l-\rho$ (\cite[p.274]{MOS}), we have that \eqref{P6-L6-0} equals $\frac{\Gamma(l)}{\Gamma(-s+l-\rho)\,\Gamma(s+\rho)}$ times the integral
\begin{align*}
\int_{0}^{1} 
\left\{\int_{-1}^{+1}(1-t^2)^{l-\rho-1}\,\exp\left({2\pi i }a\,t\,x\right)\d t\right\}\,e^{2\pi i bx}\,x^{-s+l-\rho-1}(1-x)^{s+\rho-1}\,\d x, 
\end{align*}
The $t$-integral is evaluated by \cite[3.384.1]{GR} as
\begin{align*}
2^{2(l-\rho)-1}\,B(l-\rho,l-\rho)\,e^{{2\pi i }|a|x}\,
{}_1F_1\left(l-\rho,2l-2\rho;\,{4\pi i}|a| x\right),
\end{align*}
which equals $\sqrt{\pi}\,\Gamma(l-\rho)\,\left({\pi |a|x}\right)^{-l+\rho+1/2}\,J_{l-\rho-1/2}\left(2{\pi x}|a|\right)$ by a formula on \cite[p.283]{MOS}.
\end{proof}

\begin{prop} \label{P6-L10}
Let $N\in \N^*$. For $q \in \N$, there exist $l_1\in \N$ such that 
\begin{align*}
|\cJ_l^{\xi,(\tau,h)}(\sn(\xi)\fw_0|s;r)|&\ll(1+|s|)^{q}
\frac{l^{q}\,(\sqrt{8\Delta}\pi)^{l-\rho}|\tau|^{-l+2\rho+1}(|\tau|^{-q}+|\tau|^{2q})}{|\Gamma(s+\rho)\,\Gamma(-s+l-\rho)|\Gamma(l-\rho+1/2)}\, 
\,\left|{\langle h\xi,\xi\rangle}\right|^{-q}
\end{align*}
holds for $l\geq l_1,\,\Re(s)\in (q-\rho+1,l-\rho-q-1),\,\tau\in N^{-1}\Z-\{0\}$ and $h\in \sG_1(\R)$, with the implied constant independent of $(l,s,\tau,h,r)$.\end{prop}
\begin{proof}
This follows from Lemmas~\ref{P6-L5} (2) and \ref{P6-L9}.  
\end{proof}

\begin{lem} \label{P6-L14}
There exists $l_1\in \N^*$ such that the series \eqref{Pp5-f2} with $\su=\sn(\xi)\fw_0$ converges absolutely for $\Re(s)\in (\rho+1,l-3\rho-1)$ and $l\geq l_1$. We have the majoration 
\begin{align*}
\sup_{h_0 \in \cN}\,|\fJ_{l}^{\xi,f}(\sn(\xi)\fw_0,\phi|s,h_0;r) |\ll 
(1+|s|)^{2\rho}\,\frac{l^{2\rho}\,(\sqrt{8\Delta}\pi)^{l-\rho}}{|\Gamma(s+\rho)\,\Gamma(-s+l-\rho)|\Gamma(l-\rho+1/2)}
\end{align*}
for $\Re(s) \in (\rho+1 ,l-3\rho-1)$ and $l\in \N_{>l_1}$ with the implied constant independent of $(h_0, s,l)$. 
\end{lem}
\begin{proof}
By the Cartan decomposition of $\sG_1(\R)$ it is easy to confirm that $|\langle \xi, h\xi\rangle/\Delta|_\infty \asymp \|h^{-1}\xi\|_\infty$ for $h\in \sG_1(\R)$. Thus, from Lemma~\ref{P6-L10} with $q=2\rho=m-1$ and Lemma~\ref{P6-L1}, we have the following majorant of $|\cJ_\fin^{\xi,f,(\tau,\delta h_0)}(\sn(\xi)\fw_0,\phi|s)\cJ_l^{\xi,(\tau,\delta)}(\sn(\xi)\fw_0|sr)|$ ($\tau\in N_1^{-1}\Z-\{0\}$, $\delta \in \sG_1(\Q),\,h_0\in \cN$): 
$$ 
(1+|s|)^{q}\,\frac{l^{2\rho}\,(\sqrt{8\Delta}\pi)^{l-\rho}
}{|\Gamma(s+\rho)\,\Gamma(-s+l-\rho)|\Gamma(l-\rho+1/2)}\,|\tau|^{5\rho-l+1}
\ff_{2\rho,\Re(s)+\rho}(\delta h_0)
$$
where $\ff_{N_1,N_2}:\sG_1(\A)\rightarrow \R_{+}$ is a function defined as in \S~\ref{ConvergenceLemma-1}. From this majorization, we obtain the conclusion by applying Lemma~\ref{P6-L12}. 
\end{proof}

\noindent
{\it The proof of Proposition~\ref{Prop4}} : Having all the necessary estimates on our hands, we see that the same argument as in the last part of \S~\ref{JJbsne} works in this case. \qed

\subsection{Convergence lemmas} \label{ConvergenceLemma-1}
Recall that $\sG_1^\xi(\R)={\rm Stab}_{\sG_1(\R)}(<\xi_0^{-}>_\R)$ is a maximal compact subgroup of $\sG_1(\R)$ such that $b_\infty^{-1} \sG_1^\xi(\R) b_\infty\subset \bK_\infty$ (see \S~\ref{RealLie}). 
Define a one parameter subgroup $a_{G_1}^{(t)}$ $(t\in \R)$ in $G_1$ by 
\begin{align*}
&a_{G_1}^{(t)}\,\xi_0^{+}=(\ch t)\,\xi_{0}^{+}+(\sh t)\,\xi_0^{-}, \qquad 
a_{G_1}^{(t)}\,\xi_0^{-}=(\sh t)\,\xi_{0}^{+}+(\ch t)\,\xi_0^{-}, \\
&a_{G_1}^{(t)}|< \xi_{0}^{+},\xi_{0}^{-}>_\R^{\bot} ={\rm id},
\end{align*}
where $\xi_0^{+}=-\Delta^{1/2}\e_0-\xi_0^{-}$ (see \S~\ref{NormFTN}). Set $A_{G_1}^{+}=\{a_{G_1}^{(t)}|\,t\in \R_{+}\,\}$. Then, we have the Cartan decomposition $\sG_1(\R)=\sG_1^\xi(\R)\,A_{G_1}^{+}\,\sG_1^\xi(\R)$. Let $\sP_0^1$ be the parabolic subgroup of $\sG_1$ stabilizing $\Q \e_0$, and $\sN_0^{1}$ the unipotent radical of $\sP_0^{1}$. Set $\sm_{\sG_1}(t;h_0)=\diag(t,h_0,t^{-1})\in \sG_1$ for $t\in{\bf GL}_1$ and $h_0 \in \sG_0$. As we recalled for $\sG$ in \cite[6.0.3]{Tsud2011-1}, from \cite[Proposition 2.7]{MS98}, we have a disjoint decomposition $\sG_1(\A_\fin)=\sG_1^\xi(\A_\fin)\,\fA_{\sG_1}\,\bK_{1,\fin}$ where 
$$
\fA_{\sG_1}=\{a_{\fin}=(\sm_{\sG_1}(p^{-n_p};\,1))_{p\in \fin}|\,n_p=0\quad{\text{for almost all $p\in \fin$}}\,\}. 
$$
Fix a compact set $\cN\subset \sG_1^\xi(\A)$ such that $\sG_1^\xi(\A)=\cN\,\sG_1^\xi(\Q)$. All in all, we see that the set $\cN\,A_{G_1}^{+}\,\fA_{\sG_1}\,\bK_{1}$ contains a representatives of the quotient space $\sG_1^\xi(\Q)\bsl \sG_1(\A)$, where we set $\bK_1=\bK_{1,\fin}\,\sG_1^\xi(\R)$. Let $\cS_{\sG_1}$ be a Siegel domain of $\sG_1(\A)$ such that $\sG_1(\A)=\sG_1(\Q)\,\cS_{\sG_1}$ and $\cS_{\sG_1}\bK_{1}=\cS_{\sG_1}$. Let $t_{\sG_1}:\sG_1(\R)\rightarrow \R_+$ be the function which is left $\sN_0^1(\A)$-invariant and right $\bK_1$-invariant and such that $t_{\sG_1}(\sm_{\sG_1}(t;\,h_0))=|t|_\A$ for any $(t,h_0)\in \A^\times \times \sG_0(\A)$. We have the norms $\|\cdot\|_\infty$ on $V(\R)$ and $\|\cdot\|_\fin$ on $V(\A_\fin)^{\#}$ defined in \S~\ref{NormFTN}. Note that $\|b_\infty^{-1} k_1 v\|_\infty=\|b_\infty^{-1} v\|_\infty$ for $v\in V_1(\R)$ and $k_1\in \sG_1^\xi(\R)$. 

\begin{lem} \label{P6-L13}
Let $\cU\subset \sG_1(\A)$ be a compact set. For $h\in \fS_{\sG_1}$ and $g\in \cN\,A_{G_1}^+\fA_{\sG_1}\,\bK_1\cap \sG_1(\Q) h \cU$, 
\begin{align*}
t_{\sG_1} (h) \ll \|b_\infty^{-1} g^{-1}_\infty\xi\|_\infty \,\|g_\fin^{-1}\xi\|_\fin\end{align*}
with the implied constant independent of $h$ and $g$. 
\end{lem}
\begin{proof}
Let $g=\omega\,a_{1,\infty}^{(t)}a_{1,\fin}\,k_1=\delta h u$ with $\omega \in \cN$, $t\in \R_+$, $a_{\fin}=(\sm_{\sG_1}(p^{-n_p};\,1))\in \fA_{\sG_1}$, $k_1\in \bK_1$, $\delta \in \sG_1(\Q)$, $h \in \cS^1$ and $u \in \cU$. Then, from $h=\delta^{-1}\omega a_{G_1}^{(t)}a_{\fin}\,k_1\,u^{-1}$, 
\begin{align}
t_{\sG_1}(h)^{-1}&=\|h^{-1}\e_0\|_\A =\|uk_{1}^{-1}a_{G_1}^{(-t)}a_{\fin}^{-1}\,\omega^{-1}\delta\,\e_0\|_\A
 \notag
\\
&\asymp \|a_{G_1}^{(-t)}a_{\fin}^{-1}\,\omega^{-1}\delta\,\e_0\|_\A
 \label{P6-L13-1}
\\
&\gg \{e^{-t}\,\prod_{p\in \fin}p^{-n_p}\}\,\|\omega^{-1}\delta \,\e_0\|_\A
 \notag
\\
&\asymp \{e^{-t}\,\prod_{p\in \fin}p^{-n_p}\}\,\|\delta \,\e_0\|_\A
 \label{P6-L13-2}
\\
&\gg e^{-t}\,\prod_{p\in \fin}p^{-n_p}
 \label{P6-L13-3}
\\
& \asymp \| b_\infty^{-1} a_{G_1}^{(-t)}\xi\|_\infty^{-1}\, \|a_{\fin}^{-1}\xi\|_\fin^{-1}
 \asymp \|b_\infty^{-1} g^{-1}_\infty\xi\|_\infty^{-1}\,\|g_\fin^{-1}\xi\|_\fin^{-1}.
 \notag
\end{align}
Here, to have the majoration \eqref{P6-L13-1} and \eqref{P6-L13-2}, we note that the elements $uk_1^{-1}$ and $\omega$ vary in a compact set; the estimation \eqref{P6-L13-3} follows from the rationality of points $\delta \e_0$.
\end{proof}

For $N_1,N_2\in \N$, consider a function $\ff_{N_1,N_2}:\sG_1(\A) \rightarrow \R$ defined by 
\begin{align*}
\ff_{N_1,N_2}(h)=\|h_\infty^{-1}\xi\|_\infty^{-N_1}\,\|h_\fin^{-1}\xi\|_\fin^{-N_2}, \quad h\in \sG_1(\A). 
\end{align*}
It is evident that $\ff_{N_1,N_2}$ is positive and is left $\sG_1^\xi(\A)$-invariant and right $\bK_{1}$-invariant. Thus, the following summation makes sense. \begin{align}
\fF_{N_1,N_2}(h)=\sum_{\delta \in \sG_1^\xi(\Q)\bsl \sG_1(\Q)} \ff_{N_1,N_2}(\delta h), \qquad h\in \sG_1(\A).
 \label{P6-L12-2}
\end{align}

\begin{lem} \label{P6-L12}
Suppose $N_2>m-1$ and $N_1>m-2$. The series \eqref{P6-L12-2} converges and the following estimation holds. 
\begin{align}
\fF_{N_1,N_2}(h)\ll_{\epsilon} t_1(h)^{2(m-2)-N_1+\epsilon}, \quad h\in \cS_{\sG_1}.
 \label{P6-L12-3}
\end{align}
\end{lem}
\begin{proof}
Let $\cU_\infty$ be a compact neighborhood of unity in $\sG_1(\R)$. Then, $\|b_\infty^{-1} u_\infty^{-1}h_\infty^{-1}\xi\|_\infty\asymp \|b_\infty^{-1} h_\infty^{-1}\xi\|_\infty$ $(u_\infty \in \cU_\infty,\,h_\infty \in \sG_1(\R)$). Thus, if we set $\cU=\cU_\infty\,\bK_{1,\fin}$, there exists a constant $C>0$ such that$$
C\,\ff_{N_1,N_2}(h)\leq \ff_{N_1,N_2}(hu), \qquad h\in \sG(\A),\,u \in \cU.
$$
From this, by a similar way to \cite[Lemma 42]{Tsud2011-1}, we obtain 
\begin{align}
C\,\vol(\cU)\,\fF_{N_1,N_2}(h) \leq \int_{\sG_1^\xi(\Q)\bsl \sG_1(\A)}\ff_{N_1,N_2}(y)\,\biggl\{\sum_{\delta \in \sG_1(\Q)} \chi_{\cN}(h^{-1}\delta y)\biggr\}\,\d y
 \label{P6-L12-1}
\end{align}
for any $h\in\sG_1(\A)$, where $\chi_\cU$ denotes the characteristic function of $\cU$. Invoking the bound
\begin{align*}
\sum_{\delta \in \sG_1(\Q)} \chi_{\cN}(h^{-1}\delta y) \ll \chi_{\sG_1(\Q)\,h\,\cU}(y)\,t_{\sG_1}(h)^{m-2}, \qquad y\in \sG_1(\A),\,h \in \cS_{\sG_1},
\end{align*}
we get the following estimation from \eqref{P6-L12-1}. 
\begin{align*}
\fF_{N_1,N_2}(h) \ll t_1(h)^{m-2}\,\int_{\sG_1^\xi(\Q)\bsl (\sG_1(\Q)\,h\,\cU)} \ff_{N_1,N_2}(y)\,\d y, \qquad h\in \cS_{\sG_1}. 
\end{align*}
From Lemma~\ref{P6-L13}, there exists a constant $C_1>0$ such that the integral on the right-hand side is majorized by 
\begin{align*}
&\int_{\substack{y\in \cU\,A_1^{+}\fA_1\,\bK_1 \\ 
 C_1\,t_{\sG_1}(h)\leq \|b_\infty^{-1} y_\infty^{-1}\xi\|_\infty\,\|y_\fin^{-1}\xi\|_\fin}} \ff_{N_1,N_2}(y)\,\d y
 \\
&
=\sum_{{a_{\fin}\in \fA_{\sG_1}}} \left\{
\int_{\max(0,\log\{C_1\,t_{\sG_1}(h)\,\|a_{\fin}^{-1}\xi\|_\fin^{-1}\})}^{+\infty}  \|b_\infty^{-1} a_{G_1}^{(-t)}\xi\|_\infty^{-N_1}\, (\sh t)^{m-2} \,\d t\right\}\, 
\mu(a_{\fin})\,\|a_{\fin}^{-1}\xi\|_\fin^{-N_2}
\\
&\ll \sum_{{a_{\fin}\in \fA_{\sG_1}}} \{\max(1,C_1\,t_{\sG_1}(h)\,\|a_{\fin}^{-1}\xi\|_\fin^{-1})\}^{-(N_1-m+2)}\,
\mu(a_{\fin})\,\|a_{\fin}^{-1}\xi\|_\fin^{-N_2},
\end{align*}
if $N_1>m-2$, where $\mu(a_{\fin})=\vol(\sG_1^\xi(\A_\fin)\bsl \sG_1^\xi(\A_\fin)\,a_{\fin}\,\bK_{1,\fin})$. The number of $a_\fin$ such that $C_1 t_{\sG_1}(h)\|a_\fin^{-1}\xi\|_\fin^{-1}\geq 1$ is estimated from above by $\log t_{\sG_1}(h)$. Since $\mu(a_{\fin})\ll \prod_{p \in \fin} p^{(m-2)n_p}$ for $a_{\fin}=(\sm_0(p^{-n_p};\,1))\in \fA_{\sG_1}$ ({\it cf}. \cite[Lemma 39]{Tsud2011-1}), the last series in $a_{\fin}$ is majorized by
\begin{align*}
t_{\sG_1}(h)^{-N_1+m-2}
\biggl\{ \log t_{\sG_1}(h)+\prod_{p\in \fin} \{\sum_{n_p=0}^{\infty} p^{(-N_2+m-2)n_p}\}\biggr\} &
\ll_{\epsilon} t_{\sG_1}(h)^{-N_1+m-2+\epsilon},
\end{align*}
when $N_2>m-1$. This completes the proof. 
\end{proof}

\section{Appendix 1} \label{APP}
In this section, for convenience, we collect miscellaneous results on archimedean integrals involving Bessel functions which are needed in the main body of the article.

We endow the $n$-dimensional Euclidean space $(\R^{n},\langle\,,\,\rangle)$ with the standard Lebesgue measure. Let $\d \omega$ be the volume form on the $(n-1)$-sphere ${\bf S}^{n-1}=\{X\in \R^n|\,\|X\|=1\}$.  

\begin{lem} \label{APP-L0}
\begin{itemize}
\item[(1)] Let $\eta_1\in \R^{n}$ be a unit vector. Then for any $q\in \N$, 
\begin{align*}
\int_{{\bf S}^{n-1}}\langle \omega,\eta_1\rangle^{q}\d \omega=\delta(q\in 2\N)\,\frac{2\pi^{n/2}\,q!}{2^{q}\,\Gamma((n+q)/2)\,(q/2)!}
\end{align*}
In particular, $\vol({\bf S}^{n-1})=2\pi^{n/2}/\Gamma(n/2)$. 
\item[(2)] For any $\eta\in \R^n-\{0\}$, 
$$
\int_{\bS^{n-1}}\exp(-2\pi i  \langle \eta,\omega\rangle)\,\d \omega=
2\pi \|\eta\|^{1-n/2}\,J_{n/2-1}(2\pi\|\eta\|).
$$
\end{itemize}
\end{lem}
\begin{proof}
(1) Set $T(q)=
\int_{{\bf S}^{n-1}}\langle \omega,\eta_1\rangle^{q}\d \omega$ for $q\in \N$. We compute the integral $I(q)=\int_{\R^n}\exp(-\|X\|^2)\langle X,\eta_1\rangle^{q}\,\d X$ in two ways. By the polar coordinates $X=\rho \omega\,(\rho>0,\,\omega\in {\bf S}^{n-1})$, we have $\d X=\rho^{n-1}\d \rho\,\d \omega$; hence, 
\begin{align*}
I(q)&=\biggl\{\int_{0}^{\infty}\exp(-\rho^2)\rho^{n+q-1}\,\d \rho\biggr\}\,\biggl\{\int_{{\bf S}^{n-1}}\langle \omega,\eta_1\rangle^{q}\,\d \omega\biggr\}
=T(q)\times \tfrac{1}{2}\Gamma\left(\tfrac{n+q}{2}\right).
\end{align*}
Let $\{\eta_j\}_{j=1}^{n}$ be an orthonormal basis of $\R^{n}$. For $X=\sum_{j=1}^{n}x_j\eta_j$ with $x_j \in \R$, we have
\begin{align*}
I(q)&=\biggl\{\int_{\R}e^{-x_1^2}x_1^{q}\d x_1\biggr\}\,\prod_{j=2}^{n}
\int_{\R} \exp(-x_j^{2})\,\d x_j 
=\delta(q\in 2\N)\Gamma\left(\tfrac{q+1}{2}\right)\times \pi^{(n-1)/2}.
\end{align*}
Comparing the two formulas of $I(q)$, we obtain
$$
T(q)=\delta(q\in 2\N)\,2\pi^{(n-1)/2}{\Gamma\left(\tfrac{q+1}{2}\right)}{\Gamma\left(\tfrac{n+q}{2}\right)}^{-1}, \quad q\in \N.
$$
Since $\pi^{-1/2}\Gamma(\tfrac{q+1}{2})=\frac{q!}{2^q\,(q/2)!}$ for any $q\in \N$, we are done. 

(2) Let $\eta=y\,\eta_1$ with $y>0$ and $\eta_1\in {\bf S}^{n-1}$. Then by (1), \begin{align*}
\int_{\bS^{n-1}}\exp(-2\pi i  \langle \eta,\omega\rangle)\,\d \omega
&=\sum_{k=0}^{\infty}\frac{(-2\pi i y)^{2k}}{(2k)!}\,T(2k)
\\
&=2\pi^{n/2}\sum_{k=0}^{\infty}\frac{(-1)^{k}(\pi y)^{2k}}{k!\,\Gamma(n/2+k)}
=2\pi\,y^{1-n/2}\,J_{n/2-1}(2\pi y). 
\end{align*}
\end{proof}

Recall $m\in \N$ with $m\geq 3$ and $\rho=\frac{m-1}{2}$.
\begin{lem} \label{P4-L3-SL1} 
Let $\eta\in \R^{m-1}$. 
For $a,\,s\in \C$ such that $\Re(a)>0$, $\Re(s)>-\frac{1}{2}(\rho+\frac{1}{2})$, \begin{align}
&\int_{\R^{m-1}}(\|Z\|^2+a^2)^{-(s+\rho)}\,\exp(-2\pi i \langle \eta,Z\rangle )\,\d Z
 \notag
 \\
&={\pi^{\rho}}{\Gamma(s+\rho)}^{-1}
\begin{cases}
a^{-2s}\Gamma(s)  \quad &(\eta=0), \\
2\,(\pi \|\eta\|)^{s}\,{a^{-s}}\,K_{s}(2\pi \,a\|\eta\|) \quad &(\eta\not=0), 
\end{cases}
  \label{P4-L3-SL1-0}
\end{align}
with the integral on the left-hand side being absolutely convergent. 
\end{lem}
\begin{proof}
By analytic continuation, it suffices to prove the formula for $a\in \R$ with $a>0$. Suppose $\eta=0$. By the polar coordinates on $\R^{m-1}$, the integral \eqref{P4-L3-SL1-0} equals
\begin{align*}
&a^{-2s}\vol({\bf S}^{m-2})\,\int_{0}^{\infty} (y^{2}+1)^{-(s+\rho)}y^{m-2}\d y
\end{align*}
By the variable change $x=(1+y^{2})^{-1}$, the $y$-integral becomes the beta-integral $2\int_{0}^{1}x^{s-1}(1-x)^{\rho-1}\d x=2\Gamma(\rho)\Gamma(s)\Gamma(s+\rho)^{-1}$. Using the formula $\vol({\bf S}^{m-2})$ recalled above, we are done. 

Suppose $\eta\not=0$. In the same way as above, by Lemma~\ref{APP-L0} (2), we have 
\begin{align*}
&\int_{0}^{+\infty} (y^2+a^2)^{-(s+\rho)}\,y^{m-2}\,\d y \,\biggl\{\int_{\bS^{m-2}}\exp(-2\pi i \langle \eta,\omega\rangle\,y)\,\d \omega\biggr\} 
\\
&=2\pi^{\rho}\,(\pi\|\eta\|)^{1-\rho}\,\int_{0}^{\infty}(y^2+a^2)^{-(s+\rho)}\,
J_{\rho-1}(2\pi\|\eta\|y)\,y^{\rho}\,\d y.
\end{align*}
Applying the formula \cite[6.565.4]{GR} to compute the last integral, we are done. 
\end{proof}

\begin{lem} \label{P5-L6}
\begin{itemize}
\item[(1)]
For $\alpha,\,A \in \C$ such that $\rho-1<-2\Re(\alpha)$ and $\Re(A)>0$ and for $\eta \in \R^{m-2}-\{0\}$, 
\begin{align*}
\int_{\R^{m-2}} (\|Y\|^2+A)^{\alpha}\,\exp(2\pi i \langle Y,\eta\rangle)
\d Y=
2\pi^{\rho-1}\tfrac{\Gamma(-\alpha-\rho+1)}{\Gamma(-\alpha)}\,\int_{0}^{+\infty} (u^2+A)^{\alpha+\rho-1}\,\cos(2\pi \|\eta\|u)\,\d u.
\end{align*}
\item[(2)] For $A,\,B>0$, 
\begin{align*}
\int_{0}^{+\infty} \exp(-B u^2)\,\cos(A u)\,\d u =(\pi B^{-1})^{1/2}\,\exp(-A^2/(4B)). 
\end{align*}
\item[(3)] For $q>0$, $A>0$, $B<0$ and $\sigma>0$,  
\begin{align*}
\int_{(\sigma)} 
\exp(-(Az^{-1}+Bz))\,z^{-q}\,\d z=-2\pi i(A|B|^{-1})^{(1-q)/2}\,J_{q-1}(2\sqrt{|AB|}).
\end{align*}
\end{itemize}
\end{lem}
\begin{proof} The formula in (1) follows from Lemma~\ref{P4-L3-SL1} and the formula in \cite[p.85 (the last line)]{MOS}. The formula in (2) is elementary. The formula in (3) is obtained by expanding $\exp(-A/z)$ to the Taylor series $\sum_{n=0}^{\infty}(-A/z)^{n}/n!$ and then by using the formula $\int_{(\sigma)}z^{-(q+n)}\exp(-Bz)\,\d z=-2\pi i \delta(B<0)\,|B|^{q+n-1}\,\Gamma(q+n)^{-1}$ (see \cite[3.382.6]{GR}). \end{proof}

\begin{lem}\label{P6-L4.5} 
Let $z\in \C$ and $T>0$ be such that $\Re(z)>{2}T^{-1}$. Then, 
$$\left|{2iz}/({z^2+u^2})\right|<T \quad {\text{for any $u \in \R$}}.$$ 
\end{lem}
\begin{proof} Let us show $2|z|<|u-iz|\,|u+iz|$. If $\theta$ denotes the angle of $u+iz$ and $u-iz$ (i.e., $\theta={\rm{Arg}}((u+iz)/(u-iz))$), then 
$$|2z|^2=|u+iz|^2+|u-iz|^2-2|u+iz|\,|u-iz|\,\cos\theta 
\leq (|u+iz|+|u-iz|)^2.
$$
Hence, it suffices to prove $|u+iz|+|u-iz|<T|u-iz|\,|u+iz|$. Set $v_{\pm}=|u\pm iz|$. Then, $v_\pm^2=(u\mp \Im\, z)^2+(\Re\, z)^2\geq (\Re\, z)^2>4T^{-2}$. Thus, $(1/v_+)+(1/v_-)<T/2+T/2=T$, or equivalently $|u+iz|+|u-iz|<T\,|u-iz|\,|u+iz|$.
\end{proof}

\begin{lem} \label{P6-L4}
Let $v\in \R^{m-1}-\{0\}$, $l\in \N$, $\alpha,\,z\in \C$ and $T>0$ be such that $\Re(z)>2T^{-1}$, $l>\rho$. Then 
\begin{align}
&\int_{\R^{m-1}} \left(1+\tfrac{-2iz T^{-1}}{\|Z\|^2+z^2}\right)^{\alpha}
\,(\|Z\|^2+z^2)^{-l}\,\,\exp(-2\pi i \langle Z,v\rangle)\,\d Z
 \label{P6-L4-1}
\\
&=2\pi^l\,\|v\|^{l-\rho} 
z^{\rho-l}
\,\sum_{n=0}^{\infty}\left( \begin{smallmatrix} {\alpha} \\ {n} \end{smallmatrix}\right) \,(-2i\pi \|v\|T^{-1})^{n}\,{\Gamma(n+l)}^{-1}{K_{n-\rho+l}(2\pi\|v\|z)}. 
 \notag
\end{align}
\end{lem}
\begin{proof}
From Lemma~\ref{P6-L4.5}, we can expand $(1+w)^{\alpha}$ to the Taylor series in $w=\frac{-2\pi i zT^{-1}}{\|Z\|^2+z^2}$ for $\Re z>2T^{-1}$. Then by Lemma~\ref{P4-L3-SL1}, the integral on the left-hand side of \eqref{P6-L4-1} is computed as
\begin{align*}
&\sum_{n=0}^{\infty}\left( \begin{smallmatrix} {\alpha} \\ {n} \end{smallmatrix}\right) (-2\pi i z T^{-1})^{n} \int_{\R^{m-1}} (\|Z\|^2+z^2)^{-(n+l)}\exp(-2\pi i \langle Z,v \rangle)\,\d Z
\\
&=\sum_{n=0}^{\infty}\left( \begin{smallmatrix} {\alpha} \\ {n} \end{smallmatrix}\right) (-2\pi i z T^{-1})^{n}\times 2z^{-(n-\rho+l)}(\pi\|v\|)^{n-\rho+l}\tfrac{\pi^{\rho}}{\Gamma(n+l)}K_{n-\rho+l}(2\pi \|v\|z). 
\end{align*}  
\end{proof}
We study the function $\Ical_l^{(s)}(a,b)$ defined by the formula \eqref{Ftn-V}. 

\begin{lem} \label{P6-L7}
Let $\Re(s)<l-\rho-1$ and $b\not=0$. For any $q\in \N$ such that $q\leq l-\rho-\Re(s)-1$ and $q\leq \Re(s)+\rho-1$, 
\begin{align*}
\Ical_l^{(s)}(a,b)&=(2\pi ib)^{-q}\,\int_{0}^{1} 
\tfrac{\d^{q}}{\d x^q}\bigl\{x^{-(s+1/2)}(1-x)^{s+\rho-1}\,
J_{l-\rho-1/2}\left(2\pi ax \right)\bigr\} \,\exp\left(2\pi ib x\right)\,\d x.
\end{align*}
\end{lem}
\begin{proof} Note $J_\nu(x) \sim (x/2)^{\nu}$ as $x\rightarrow +0$. By using the formula 
\begin{align}
J'_\nu(x)=2^{-1}(J_{\nu-1}(x)-J_{\nu+1}(x))
 \label{P6-L7-f0}
\end{align} (\cite[\S 3.1.1(p.67)]{MOS}), it is not hard to confirm that for $q \leq l-\rho-\Re(s)-1$ and $q\leq \Re(s)+\rho-1$ all the derivatives of $x^{-(s+1/2)}(1-x)^{s+\rho-1}J_{l-\rho-1/2}(2\pi ax)$ up to $q$ vanish at $x=0$. Thus, the formula is proved by successive application of integration by parts. 
\end{proof}

We use uniform bound 
\begin{align}
|J_{\nu}(x)| \leq 2\Gamma(\nu+1)^{-1}(x/2)^{\nu} \quad (x>0,\,\nu>0)
\label{JBesselEst3}
\end{align}
which is shown by the integral representation \cite[8.411 8]{GR}. 

\begin{lem} \label{P6-L8}
Let $q\in \N$. Then, there exist $l_1\in \N$ such that 
\begin{align*}
&\left|\tfrac{\d^{q}}{\d x^q}\bigl\{x^{-(s+1/2)}(1-x)^{s+\rho-1}\,J_{l-\rho-1/2}(2\pi ax)\bigr\}\right| 
\ll_{q} (1+|s|)^{q} \frac{l^{q}\,a^{l-\rho-1/2}(a^{2q}+a^{-q})}{\Gamma(l-\rho+1/2)}
\end{align*}
holds uniformly for $x\in (0,1)$, $a>0$, $l\geq l_1$, $\Re(s)\in (q-\rho+1, 
l-\rho-q-1)$ and $a>0$. 
\end{lem}
\begin{proof} Let $j,p,k\in \N$ with $j+p+k=n$. We have the bound
\begin{align*}
&\left|\left(\tfrac{\d}{\d x}\right)^{j}(1-x)^{s+\rho-1}\right|\ll_{j} (1+|s|)^{j}\,(1-x)^{\Re(s)+\rho-1-j}, \\ 
&\left|\left(\tfrac{\d}{\d x}\right)^{p}x^{-(s+1/2)}\right|\ll_{p}(1+|s|)^{p}\,
x^{-\Re(s)-1/2-p}.
\end{align*}
By a successive application of the formula \eqref{P6-L7-f0}, using \eqref{JBesselEst3}, we have
\begin{align*}
\left|\left(\tfrac{\d}{\d x}\right)^{k}J_{l-\rho-1/2}(2\pi ax) \right| 
&\ll_{k} \max(1,a^{k})\sum_{u=-k}^{k} {a^{l-\rho-1/2+u}}\Gamma(l-\rho+1/2+u) \\
&\ll_{k}a^{l-\rho-1/2}(a^{2q}+a^{-q})\,l^{q}\,\Gamma(l-\rho+1/2)^{-1}
\end{align*}
All the exponents of powers $x$ and $1-x$ occurring above are non-negative due to the condition on $q$; thus these powers on $(0,1)$ are bounded trivially by $1$. Then by Leibnitz' rule, 
\begin{align*}
&\left|\tfrac{\d^{q}}{\d x^q}\biggl\{x^{-(s+1/2)}(1-x)^{s+\rho-1}\,J_{l-\rho-1/2}(ax)\right|\ll_{\rQ,q} \frac{a^{l-\rho-1/2}(a^{2q}+a^{-q})\,l^{q}}{\Gamma(l-\rho+1/2)}(1+|s|)^{q}
\end{align*}
uniformly in $x\in (0,1)$ and sufficiently large $l$ and $\Re(s)\in(q-\rho+1, 
l-\rho-q-1)$. 
\end{proof}

From Lemmas~\ref{P6-L7} and \ref{P6-L8}, we obtain the following uniform bound of ${\Ical}_l^{(s)}(a,b)$. 

\begin{lem} \label{P6-L9}
Let $q\in \N$. Then, there exist $l_1\in \N$ such that
\begin{align*}
|{\mathcal I}_l^{(s)}(a,b)|&\ll_{q} 
{(1+|\Im(s)|)^{q}}\,\left|b\right|^{-q}\frac{l^{q}a^{l-\rho-1/2}(a^{2q}+a^{-q})}{\Gamma(l-\rho+1/2)}
\end{align*}
holds for $l\geq l_1,\,\Re(s)\in (q-\rho+1,l-\rho-q-1),\,a\in \R^\times_+$ and $b\in \R^\times $, with the implied constant independent of $(l,s,a,b)$
\end{lem}

\section{Appendix 2: Representations generated by Hecke eigenforms}  

\subsection{Local theory} \label{sec:LocalTheory}
Let $p$ be a prime number. Let $Q$ be a non-degenerate quadratic space of dimension $m$ over $\Q_p$. A $\Z_p$-lattice $\cL\subset V$ is said to be maximal integral if $\cL$ is a maximal element of the set of $\Z_p$-lattices $\cM\subset V$ such that $2^{-1}Q[\cM] \subset \Z_p$. We fix such an $\cL$ and denote by $\cL^*$ its dual lattice in $(V,Q)$. Let $\sG={\rm O}(Q)$ be the orthogonal group of $(V,Q)$ and set 
$$
\bK=\{k\in \sG(\Q_p)|\,k\,\cL=\cL\}, \quad \bK^{*}=\{k\in \bK|\,kX-X\in \cL\,(\forall X\in \cL^{*}\}. 
$$
Let $\cH=\cH(\sG(\Q_p)\sslash \bK)$ be the Hecke algebra for the pair $(\sG(\Q_p),\bK)$ and set 
$$
\cH^{+}=\cH^{+}(\sG(\Q_p)\sslash \bK^{*})=\{\phi \in \cH|\,\phi(u gu^{-1})=\phi(g)\, (u\in \bK)\}.
$$
Fix a Witt decomposition of $\cL$ as
$$
\cL=\bigoplus_{j=1}^{\nu}(\Z_p v_j\oplus \Z_p v_j')\oplus \cM,
$$
where $\{v_j,\,v_j'\,(j=1,\dots,\nu)\}$ is a set of isotropic vectors such that $Q(v_i,v_j')=\delta_{ij}$ and $\cM=\{Z\in V_0|2^{-1}Q[Z]\in \Z_p\}$ is a unique maximal integral lattice in $V_0$, the anisotropic kernel of $V$; $\nu$ is the Witt index of $Q$. Let $Q_0=Q|V_0$ and $\sG_0={\rm O}(Q_0)$ viewed as a subgroup of $\sG$. We define $\bK_0=\{h\in \sG_0(\Q_p)|h\,\cM=\cM\}$ and $\bK_0^{*}=\{k\in \bK_0|\,kZ-Z\in \cM\,(\forall Z\in \cM^{*})\}$, where $\cM^*$ is the dual lattice of $\cM$ in $(V_0,Q_0)$. We have that $\bK_0=\sG_0(\Q_p)$ and that $\bK^*$ (resp. $\bK_0^*$) is a normal subgroup of $\bK$ (resp. $\sG_0(\Q_p)$) of finite index and the natural inclusion $\bK_0\hookrightarrow \bK$ induces an group isomorphism $\bK_0/\bK_0^*\cong \bK/\bK^*$. Set $E:=\bK_0/\bK_0^{*}$. The structure of $E$ is classified by \cite[Proposition 1]{MS98} quoted below: 
\begin{lem} \label{MS98L-1}
Define
$$
\partial_{\cL}:=\dim_{\Z_p/p\Z_p}(\{X\in \cL^{*}|\,2^{-1}Q[X]\in p^{-1}\Z_p\}/\cL).
$$
Then $\partial_{\cL}\in \{0,1,2\}$; the group $E$ is isomorphic to $\{1\}$, $\Z/2\Z$, and $D_{2(p+1)}$ according to $\partial_\cL$ equals $0$, $1$, and $2$, respectively. Here $D_{2l}:=(\Z/l\Z) \rtimes (\Z/2\Z)$ is the dihedral group of order $2l$.  
\end{lem}

Note that $\sG$ is unramified over $\Q_p$ if and only if $\partial_{\cL}=0$. Let $\sB$ be the Borel subgroup of $\sG$ stabilizing the $\Q_p$-isotropic flag $\langle v_1,\dots,v_j\rangle_{\Q_p}\,(1\leq j \leq \nu)$ and $\sN$ the unipotent radical of $\sB$. For $(t_1,\dots,t_\nu)\in ({\bf GL}_1)^{\nu}$ and $u \in \sG_0$, let $d(t_1,\dots,t_\nu;u)$ denote the element of $\sG$ such that $v_j \mapsto t_j\,v_j$, $v'_j \mapsto t_j^{-1}\,v_j'$ ($1\leq j\leq \nu$) and $X\mapsto u(X)$ for $X\in V_0$; such elements form a closed $\Q_p$-subgroup $\sH$ of $\sG$ such that $\sB=\sH\sN$. For ${\bf r}=(r_j)\in \Z^{\nu}$ and $u\in \sG_0(\Q_p)$, set
$$
\varpi_{{\bf r},u}:=d(p^{r_1},\dots,p^{r_\nu};u) \in \sG(\Q_p).
$$
Then from \cite[Proposition 1.2]{MS98}, we have the Iwasawa decompositon 
\begin{align}
\sG(\Q_p)=\sN(\Q_p)\sH(\Q_p)\bK^*=\bigsqcup_{{\bf r}\in \Z^\nu, \e\in E}\sN(\Q_p)\varpi_{{\bf r},\e}\,\bK^*.
 \label{HNK*Iwasawa}
\end{align}
For $f\in \cH$, define
$$
\Phi_f(h):=\delta_{\sB}(h)^{1/2}\int_{\sN(\Q_p)}f(hn)\,\d n, \quad h\in \sH(\Q_p),
$$
where $\delta_{\sB}$ is the modulus character of $\sB(\Q_p)$ and $\d n$ denotes the Haar measure on $\sN(\Q_p)$ such that $\vol(\sN(\Q_p)\cap \bK)=1$; note that $\sN(\Q_p)\cap \bK=\sN(\Q_p)\cap \bK^*$. 

Let $W_{\sG}$ be the Weyl group of $(\sT,\sG)$ defined to be the quotient group of the normalizer of $\sT(\Q_p)$ in $\sG(\Q_p)$ by $\sT(\Q_p)$, where $\sT$ is the maximal $\Q_p$-split torus in $\sH$. Recall the following result by Murase-Sugano: 

\begin{thm}\label{MS98thm} Let $B:=\C[E]$ be the group algebra of $E$. Then there exists a $\C$-algebra isomorphism $\Phi:\cH \rightarrow B[X_1^{\pm 1},\dots X_\nu^{\pm 1}]^{W_\sG}$ such that 
$$
\Phi(f)=\sum_{{\bf r}\in \Z^\nu}\sum_{\e\in E}\Phi_f(\varpi_{{\bf r},\e})\,\e\,\prod_{j=1}^{\nu}X_j^{r_j}, \quad f\in \cH. 
$$
We have the $\Phi(\cH^{+})=Z(B)[X_{1}^{\pm 1}, \dots,X_{\nu}^{\pm 1}]^{W_\sG}$ and $\Phi(f_{\e})=\e$ for $\e\in E$, where $f_\e\in \cH$ denotes the characteristic function of $\bK^*\varpi_{{\bf 0},\e}\bK^*=\bK^*\,\varpi_{{\bf 0},\e}$. 

The subalgebra $\cH^+$ coincides with the center of $\cH$ and $\cH=\cH^{+}\,\langle f_{\e}|\e\in E\rangle_\C$. 
\end{thm}
\begin{proof} \cite[Theorem 1.3]{MS98}. 
\end{proof}

For a $\nu$-tuple $\lambda=(\lambda_j)_{j=1}^{\nu}$ of uniramified characters of $\Q_p^\times$ and an irreducible representation $(\rho,V_\rho)$ of $E$, let 
$\cI_{\lambda,\rho}$ be the smooth representation of $\sG(\Q_p)$ defined by letting $\sG(\Q_p)$ act by the right-translations on $I(\lambda,\rho)$, the space of smooth functions $f:\sG(\Q_p)\rightarrow V_\rho$ satisfying the condition 
$$
f(hn g)=\delta_{\sB}(h)^{1/2}\{\prod_{j=1}^{\nu}\lambda_j(a_j)\}\,\rho([u])\,f(g), \quad g\in \sG(\Q_p)
$$
for any $h=d(a_1,\dots,a_\nu;u)\in \sH(\Q_p)$ and $n\in \sN(\Q_p)$, where $[u]\in E$ is the class of $u\in \sG_0(\Q_p)=\bK_0$. Since $\bK^*$ is a normal subgroup of $\bK$, we have a representation of $E$ on the space of $\bK^*$-fixed vectors $I(\lambda,\rho)^{\bK^{*}}$ such that $(\e\cdot f)(g)=f(g\,\varpi_{{\bf 0},\e})$ for $\e\in \bK_0/\bK_0^{*})$. Recall that the zonal-spherical function $\omega_{\lambda,\rho}$ belonging to $(\lambda,\rho)$ is defined as 
$$
\omega_{\lambda, \rho}(g)=\frac{1}{\dim(\rho)}\int_{\bK^*}\tr(\phi_{\lambda,\rho}(ku))\,\d k, \quad g\in \sG(\Q_p),
$$
where $\d k$ is the normalized Haar measure on $\bK^*$, and $\phi_{\lambda,\rho}:\sG(\Q_p)\rightarrow {\rm End}(V_\rho)$ is a function given by the Iwasawa decomposition \eqref{HNK*Iwasawa} as 
$$
\phi_{\lambda,\rho}(hnk)=\rho(\e)\delta(h)^{1/2}\prod_{j=1}^{\nu}\lambda_j(a_j)
 $$
for $h=d(a_1,\dots,a_\nu;\e)\in \sH(\Q_p)$, $n\in \sN(\Q_p)$ and $k\in \bK^*$. There exists a unique $\C$-algebra homomorphism $C_{\lambda,\rho}:\cH^{+}\rightarrow \C$ such that 
$$
C_{\lambda,\rho}(\phi)\,{\rm id}_{V_\rho}=\int_{\sG(\Q_p)}\phi_{\lambda,\rho}(ugu^{-1})\,f(ug^{-1}u^{-1})\,d g, \quad \phi \in \cH^{+}.
$$
We quote a result by Murase-Sugano: 

\begin{thm} \label{MS98thm2}
Let $\XX$ be the set of equivalence classes of tuples $(\{\lambda_j\}_{j=1}^{\nu},\rho)$ consisting of unramified quasi-characters $\lambda_j$ of $\Q_p^\times$ and an irreducible complex representation $\rho$ of $E$, where two such tuples $(\{\lambda_j\},\rho)$ and $(\{\lambda_j'\},\rho')$ are defined to be equivalent when $\{\lambda_j\}$ and $\{\lambda_j'\}$ belongs to the same $W_{\sG}$-orbit and $\rho$ and $\rho'$ are isomorphic. 
\begin{itemize}
\item[(1)] Then $[(\lambda,\rho)]\mapsto C_{\lambda,\rho}$ defines a well-defined bijection from $\XX$ onto ${\rm Hom}_{\C-{\rm alg}}(\cH^{+},\C)$. 
\item[(2)] 
For each $[(\lambda,\rho)]\in \XX$, let $\Omega_{[(\lambda,\rho)]}$ denote the set of all $\C$-valued functions $\omega$ on $\sG(\Q_p)$ such that 
\begin{itemize}
\item[(i)] $\omega(u_1 g u_2)=\omega(g)$ for $u_1,u_2\in \bK^*$.
\item[(ii)] $\omega(u gu^{-1})=\omega(g)$ for $u\in \bK$.
\item[(iii)] $\omega(1)=1$. 
\item[(iv)] $\omega*\phi=C_{\lambda,\rho}(\phi)\,\omega$ for $\phi \in \cH^{+}$. \end{itemize}
Then $\Omega_{[\lambda,\rho)]}=\{\omega_{\lambda,\rho}\}$. 
\end{itemize}
\end{thm}
\begin{proof}
The first part (1) is \cite[Theorem 1.8]{MS98}. The assertion (2) follows from \cite[Lemma 1.7 and Lemma 1.5]{MS98}. 
\end{proof}

For $\rho\in \widehat {E}$, let $e_{\rho}$ be the fundamental idempotent of $\rho$, i.e., 
$$
e_{\rho}:=\dim(\rho)\,\sum_{\e\in E}\chi_{\rho}(\e)\,\e \quad\in \C[E],
$$
where $\chi_{\rho}(\e)=\tr(\rho(\e))$ is the character of $\rho$. Since $E$ is a quotient group of $\bK$, $e_{\rho}$ is viewed as an element of $\cH$ supported in $\bK$.

\begin{lem}\label{SphftEQ-L} Let $(\lambda,\rho)\in \XX$. 
We have the functional equation:
\begin{align}
\int_{\bK}\omega_{\lambda,\rho}(ugu^{-1}g')\d u =\omega_{\lambda,\rho}(g)\,\omega_{\lambda,\rho}(g'), \quad g,\,g'\in \sG(\Q_p),
 \label{FESphericalft}
\end{align}
where $\d u$ is the Haar measure on $\bK$ such that $\vol(\bK)=1$. Moreover,
\begin{align*}
L(e_{\rho^\vee})\,\omega_{\lambda,\rho}=R(e_{\rho})\,\omega_{\lambda,\rho}=\omega_{\lambda,\rho}
\end{align*}
\end{lem}
\begin{proof} Indeed, from $C_{\lambda,\rho}(\phi)=\omega_{\lambda,\rho}*\phi(1)$ and the multiplicativity $C_{\lambda,\rho}(\phi_1*\phi_2)=C_{\lambda,\rho}(\phi_1)\,C_{\lambda,\rho}(\phi_2)$, the difference of the left-hand side and the right-hand side of the desired equality \eqref{FESphericalft}, say $\Phi(g,g')$, satisfies  
{\allowdisplaybreaks\begin{align*}
&\int_{\sG(\Q_p)}\int_{\sG(\Q_p)}
\Phi(g_1,g_2)\,\phi_1(g_1^{-1})\,\phi_2(g_2^{-1})\,\d g_1\d g_2 \\
&=\int_{\sG(\Q_p)\times \sG(\Q_p)}
\int_{\bK}
 \omega_{\lambda,\rho}(g_1g_2)
\phi_1(u^{-1}g_1^{-1}u)\,\phi_2(g_2^{-1})\,\d u\,\d g_1\,\d g_2
\\
&\quad -
\biggl(\int_{\sG(\Q_p)}\omega_{\lambda,\rho}(g_1) \phi_1(g_1^{-1})\d g_1\biggr)
\,\biggl(\int_{\sG(\Q_p)}\omega_{\lambda,\rho}(g_2)\phi_2(g_2^{-1})\,\d g_2 \biggr) \\
&=\int_{\sG(\Q_p)\times \sG(\Q_p)}\omega_{\lambda,\rho}(g_1g_2)
\phi_1(g_1^{-1})\,\phi_2(g_2^{-1})\,\d u\,\d g_1\,\d g_2 \\
&\quad -\biggl(\int_{\sG(\Q_p)}\omega_{\lambda,\rho}(g_1) \phi_1(g_1^{-1})\d g_1\biggr)\,\biggl(\int_{\sG(\Q_p)}\omega_{\lambda,\rho}(g_2)\phi_2(g_2^{-1})\,\d g_2 \biggr) 
\\
&=[\omega_{\lambda,\rho}*(\phi_1*\phi_2)](1)-(\omega_{\lambda,\rho}*\phi_1)(1)\times (\omega_{\lambda,\rho}*\phi_2)(1)
\\
&=C_{\lambda,\rho}(\phi_1*\phi_2)-
C_{\lambda,\rho}(\phi_1)\,C_{\lambda,\rho}(\phi_2)=0
\end{align*}}for all $\phi_1,\phi_2\in \cH^+$. Since $\Phi(u_1 g_1u_1^{-1},u_2 g_2 u_2^{-1})=\Phi(g_1,g_2)$ for all $u_1,u_2\in \bK$ is easily confirmed, we conclude $\Phi(g_1,g_2)=0$ by \cite[Lemma 1.5]{MS98}.

We have 
\begin{align*}
L(e_{\rho^\vee})\omega_{\lambda,\rho}(g)&=\sum_{\e\in E}\chi_{\rho^\vee}(\e)\, 
\int_{\bK^*}\tr(\phi_{\lambda,\rho}(k\e^{-1} g)\,\d k
\\
&=\sum_{\e\in E}\chi_\rho(\e^{-1})\, 
\int_{\bK^*}\tr(\phi_{\lambda,\rho}(\e^{-1} k g)\,\d k
\\
&=\sum_{\e\in E}\chi_\rho(\e)\, 
\int_{\bK^*}\tr(\rho(\e)\,\phi_{\lambda,\rho}(k g)\,\d k,
=\int_{\bK^*} \tr\biggl( \sum_{\e\in E} \chi_\rho(\e)\rho(\e)\,\phi_{\lambda,\rho}(kg) \biggr)\,\d k.
\end{align*} 
Since $\sum_{\e\in E} \chi_{\rho}(\e)\rho(\e)=\dim(\rho)^{-1}\,{\rm id}_{V_\rho}$, the last quantity equals $\omega_{\lambda,\rho}(g)$ as desired. 
\end{proof}
Let $V(\omega_{\lambda,\rho})$ be the set of the finite $\C$-linear combinations of right translations of $\omega_{\lambda,\rho}$; $V(\omega_{\lambda,\rho})$ is a $\C$-subspace of $\C$-valued functions on $\sG(\Q_p)$. By letting the group $\sG(\Q_p)$-act on $V(\omega_{\lambda,\rho})$ by the right-translation, we define a representation of $\sG(\Q_p)$ on $V(\omega_{\lambda,\rho})$. The space $V(\omega_{\lambda,\rho})$ also carries an action of $E\cong \bK/\bK^{*}$ induced by the left-translation of functions by $\bK$. Thus $V(\omega_{\lambda,\rho})$ is viewed as a $E\times \sG(\Q_p)$-module.  
\begin{lem} \label{IrredDoubleModSPft}
The $E\times \sG(\Q_p)$-module $V(\omega_{\lambda,\rho})$ is smooth and irreducible.
\end{lem}
\begin{proof} {\it cf}. \cite[p.151]{Cartier}. Since $g\mapsto \tilde \omega_{\lambda,\rho}(g x)$ is invariant by $x\bK^*x^{-1}$, the smoothness of $V(\omega_{\lambda,\rho})$ is evident. Let $f\in V(\omega_{\lambda,\rho})$ be a non-zero element. Then there exists a complex numbers $c_i\,(1\leq i \leq l)$ and elements $x_i\,(1\leq i\leq l)$ of $\sG(\Q_p)$ such that 
$$
f(g)=\sum_{i}c_i \,\omega_{\lambda,\rho}(g x_i)
, \quad g\in \sG(\Q_p).
$$
From this and \eqref{FESphericalft}, we have   
\begin{align*}
\int_{\bK} f (ugu^{-1}g')\d u =\omega_{\lambda,\rho}(g)\,f(g'), \quad g,g'\in \sG(\Q_p). 
\end{align*}
Choose $g'\in \sG(\Q_p)$ such that $f(g')\not=0$. Then 
$$
\omega_{\lambda,\rho}=f(g')^{-1} \int_{\bK}L(u^{-1})R(u^{-1}g')f\,\d u
=\frac{1}{\# E}\,f(g')^{-1}\sum_{\e\in E} L(\e^{-1})R(\e^{-1}1_{\bK^{*} g'}).
$$
Thus $\omega_{\lambda,\rho}$ belongs to the cyclic $E\times \sG(\Q_p)$-submodule $V(f)$ generated by $f\,(\not=0)$. Hence $V(\omega_{\lambda,\rho})=V(f)$. This shows the irreducibility of $V(\omega_{\lambda,\rho})$ as an $E\times \sG(\Q_p)$-module.  
\end{proof}

\begin{lem} The map $f\mapsto f(1)$ yields an $E$-isomorphism from $I(\lambda,\rho)^{\bK^*}$ onto $V_\rho$. 
\end{lem}
\begin{proof}
This follows from the Iwasawa decomposition \eqref{HNK*Iwasawa}. 
\end{proof}
Thus there exists a unique irreducible smooth $\sG(\Q_p)$-module $\pi_{\lambda,\rho}$ with the following property: For any Jordan-H\"{o}lder sequence $\{0\}=V_0\subset V_1 \subset\cdots \subset V_{l-1}\subset V_l=I(\lambda,\rho)$ of the $\sG(\Q_p)$-module $I(\lambda, \rho)$, we have $\pi_{\lambda,\rho}\cong V_j/V_{j-1}$ where $j$ is the unique index $j$ such that $V_{j}^{\bK^*}/V_{j-1}^{\bK^*}$ is isomorphic to $\rho$. 

\begin{lem} \label{DoublemodSPHft}
As a representation of $E\times \sG(\Q_p)$, $V(\omega_{\lambda,\rho})$ is isomorphic to $\rho^{\vee}\boxtimes \pi_{\lambda,\rho}$.  
\end{lem}
\begin{proof} From Lemma~\ref{SphftEQ-L}, 
\begin{align}
\text{$L(e_{\rho^{\vee}})\,f=f$ for all $f\in V(\omega_{\lambda,\rho})$,}
\label{++}
\end{align}
 because any such $f$ is a finite $\C$-linear combination of right-translations of $\omega_{\lambda,\rho}$. From Lemma~\ref{IrredDoubleModSPft}, there exists an irreducible representation $\rho'$ of $E$ and a smooth irreducible representation $\pi'$ of $\sG(\Q_p)$ such that $V(\omega_{\lambda,\rho})$ is isomorphic to $\rho'\boxtimes \pi'$ as representations of $E\times \sG(\Q_p)$. From \eqref{++}, we have $\rho'\cong \rho^\vee$. 
From definition, it is easy to confirm that for any $\phi \in C_{\rm c}^{\infty}(\sG(\Q_p))$, $\phi_{\lambda,\rho}*\phi=0$ implies $\omega_{\lambda,\rho}*\phi=0$. Hence there exists a $\sG(\Q_p)$-homomorphism from $V(\phi_{\lambda,\rho})$ to $V(\omega_{\lambda,\rho})$, where $V(\phi_{\lambda,\rho})$ denotes the cyclic $\sG(\Q_p)$-submodule $V_{\rho^\vee}\otimes I(\lambda,\rho)$ generated by $\phi_{\lambda,\rho} \in V_{\rho^\vee}\otimes I(\lambda,\rho)^{\bK^*}$. Since $V(\omega_{\lambda, \rho})$ viewed as a $\sG(\Q_p)$-module is isomorphic to $\dim(\rho)$ copies of $\pi'$ , $\pi'$ must occur in a subquotient of $V(\phi_{\lambda,\rho})$ on one hand.

Since $\omega_{\lambda,\rho}\in V(\omega_{\lambda,\rho})$ satisfies $R(e_{\rho})\omega_{\lambda,\rho}=\omega_{\lambda,\rho}$, $R(\C[E])\omega_{\lambda,\rho}\subset V(\omega_{\lambda, \rho})^{\bK^{*}}$ is $\rho$-isotypic. As a $R(\C[E])$-module, $V(\omega_{\lambda,\rho})^{\bK^*}$ is a direct sum of $\dim(\rho)$ copies of $(\pi')^{\bK^{*}}$. Hence the $\C[E]$-module $(\pi')^{\bK^{*}}$ must contain $\rho$ on the other hand. 

Therefore, $\pi'\cong \pi_{\lambda,\rho}$ because $\pi_{\lambda,\rho}$ is the unique subquotient of $I(\lambda,\rho)$ which contains $\rho$ when viewed as an $E$-module.    
\end{proof}

\begin{prop} \label{IrredCycMod}
Let $(\pi,{\cV}_\pi)$ be a unitary representation of $\sG(\Q_p)$ on a Hilbert space ${\cV}_\pi$. Let $\varphi_0$ be a unit vector of $\cV_\pi$ such that 
$\pi(k)\varphi_0=\varphi_0$ for all $k\in \bK^*$ and 
\begin{align}
\pi(\phi)\varphi_0=C(\phi)\,\varphi_0, \quad \phi\in \cH^{+}
 \label{IrredCycMod-f1}
\end{align}
with a $\C$-algebra homomorphism $C:\cH^+\rightarrow \C$. Let $(\lambda,\rho)$ be the Satake parameter of $\lambda$, i.e., $C=C_{\lambda,\rho}$. Let $V(\varphi_0)$ be the smallest smooth $\sG(\Q_p)$-submodule containing $\varphi_0$. Then $V(\varphi_0)$ is irreducible and is isomorphic to $\pi_{\lambda,\rho}$. 
\end{prop}

When $\cH^{+}=\cH$, which is equivalent to $\partial_{\cL}\in \{0,1\}$ (Lemma~\ref{MS98L-1} and Theorem~\ref{MS98thm}), the first part of the proposition is known by \cite[Proposition 1.2]{NPS}. To cover the general case, we slightly modify the proof of \cite[Proposition 1.2]{NPS}. Let $U({\varphi_0})$ be the $\bK$-span of $\varphi_0$. 
Since $\varphi_0\in {\cV}_\pi^{\bK^*}$ and $\bK^*$ is normal in $\bK$, we have $U({\varphi_0})\subset {\cV}_\pi^{\bK^*}$. Thus we obtain a well-defined action of $E\cong \bK/\bK^*$ on $U({\varphi_0})$.  

\begin{lem} \label{IrredCycMod-L}
We have that
\begin{itemize}
\item[(1)] any element of $U({\varphi_0})$ satisfies the same $\cH^+$-eigenequation \eqref{IrredCycMod-f1} as $\varphi_0$, and that  
\item[(2)] $U({\varphi_0})$ is an irreducible $E$-module and $V(\varphi_0)^{\bK^*}=U(\varphi_0)$.
\end{itemize}
\end{lem}
\begin{proof} The assertion (1) is proved by the relation $\phi(ugu^{-1})=\phi(g)$ $(u\in \bK)$ readily. From the defining formula of $\omega_{\lambda,\rho}$, it is easy to have the formula
\begin{align}
\omega_{\lambda,\rho}(\varpi_{{\bf 0},\e})=\frac{1}{\dim(\rho)}\tr(\rho(\e)), \quad \e\in E=\bK_0/\bK^*_0.
 \label{IrredCycMod-L-f1}
\end{align}
Define
$$
\omega(g):=\frac{1}{\# E}\sum_{\e\in E} \langle \pi(g \tilde \e)\varphi_0|\pi(\tilde \e)\varphi_0 \rangle, \quad g\in \sG(\Q_p)
$$
where $\tilde \e=\varpi_{{\bf 0},\epsilon}$ for $\e\in E$. Then it is easy to see that $\omega$ possessed the properties (i), (ii), (ii) and (iv)  in Theorem~\ref{MS98thm2}. Hence $\omega(g)=\omega_{\lambda,\rho}(g)$ for all $g\in \sG(\Q_p)$. Applying this for $g=\tilde \e$ with $\e\in E$ and by \eqref{IrredCycMod-L-f1}, we have
\begin{align}
\omega(\tilde \e)=\frac{1}{\dim(\rho)}\chi_{\rho}(\e), \quad \e\in E.
 \label{IrredCycMod-L-f2}
\end{align}
where $\chi_\sigma(\e)=\tr(\sigma(\e))$ is the character of any representation $\sigma$ of $E$. Let $U(\varphi_0)=\bigoplus_{\nu=1}^{t}U_\nu$ be an $E$-irreducible decomposition. We fix an orthonormal basis $\{v_\alpha^{(j)}\}_{\alpha \in I_\nu}$ of $U_\nu$ and write
$$
\pi(\tilde \e)\varphi_0=\sum_{j=1}^{\nu}\sum_{\alpha \in I_\nu} \langle \pi(\tilde \e)\varphi_0|v_\alpha^{(\nu)}\rangle\,v_\alpha^{(\nu)}. 
$$
Using this and by the orthogonality relation of matrix coefficients, we compute $\omega(\tilde \e)$ as
\begin{align*}
\omega(\tilde \e)&=\frac{1}{\# E}\sum_{\eta\in E} \sum_{\nu,\mu=1}^{t}\sum_{\alpha \in I_\nu, \beta\in I_\mu} \langle \pi(\tilde \eta)\varphi_0|v_\alpha^{(\nu)}\rangle \overline{\langle \pi(\tilde \eta)\varphi_0|v_\beta^{(\mu)}\rangle}\langle \pi(\tilde \e)v_\alpha^{(\nu)}|v_{\beta}^{(\mu)}\rangle 
\\
&= \sum_{\nu=1}^{t}\sum_{(\alpha,\beta) \in I_\nu \times I_\nu}\biggl\{\frac{1}{\# E}\sum_{\eta\in E} \langle \pi(\tilde \eta)\varphi_0|v_\alpha^{(\nu)}\rangle \overline{\langle \pi(\tilde \eta)\varphi_0|v_\beta^{(\nu)}\rangle} \biggr\}\,
\langle \pi(\tilde \e)v_\alpha^{(\nu)}|v_{\beta}^{(\nu)}\rangle
\\
&=\sum_{\nu=1}^{t}\sum_{(\alpha,\beta) \in I_\nu \times I_\nu}\biggl\{\frac{1}{\dim(U_\nu)} \langle \varphi_0|\varphi_0\rangle\,\langle v_\beta^{(\nu)}|v_\alpha^{(\nu)}\rangle \biggr\}\,
\langle \pi(\tilde \e)v_\alpha^{(\nu)}|v_{\beta}^{(\nu)}\rangle
\\
&=\sum_{\nu=1}^{t}\frac{1}{\dim(U_\nu)} \sum_{\alpha \in I_\nu}
\langle \pi(\tilde \e)v_\alpha^{(\nu)}|v_{\beta}^{(\nu)}\rangle
=\sum_{\nu=1}^t \frac{1}{\dim(U_\nu)} \chi_{U_\nu}(\e). 
\end{align*}
Comparing this with \eqref{IrredCycMod-L-f2}, by the linear independence of characters of $E$, we obtain $t=1$ and $\rho \cong U_1=U(\varphi_0)$. 
\end{proof}

\noindent
{\it Proof of Proposition~\ref{IrredCycMod}} : Let $W$ be a $\sG(\Q_p)$-subspace of $V(\varphi_0)$ and set $W'=V(\varphi_0)\cap W^{\bot}$. Then we have a $\sG(\Q_p)$-decomposition $V(\varphi_0)=W\oplus W'$. Write $\varphi_0=\varphi_{W}+\varphi_{W'}$ with $\varphi_{W}\in W$ and $\varphi_{W'}\in W'$. Since $\varphi_{W}\in V(\varphi_0)$, we can find a function $\phi\in C_{\rm c}^{\infty}(\sG(\Q_p))$ such that $\varphi_{W}=\pi(\phi)\,\varphi_0$. Since both $\varphi_0$ and $\varphi_W$ are $\bK^*$-fided vectors, we have $\varphi_{W}=\pi(\phi_0)\varphi_0$ with 
$$
\phi_0(g)=\int_{\bK_*}\int_{\bK^*}\phi(k_1 gk_2)\,\d k_1\,\d k_2, \quad g\in \sG(\Q_p). 
$$
By $\cH=\langle f_\e|\e\in E\rangle_{\C}\cH^+$ (Theorem~\ref{MS98thm}) and by \eqref{IrredCycMod-f1}, the vector $\pi(\phi_0)$ is a linear combination of vectors $\pi(f_\e)\varphi_0\in U(\varphi_0)$. Hence $\varphi_W\in W^{\bK^*}\cap U(\varphi_0)$. Suppose $\varphi_W\not=0$. Then $W^{\bK^*}\cap U(\varphi_0)\not=\{0\}$. Since $U(\varphi_0)$ is irreducible as an $E$-module (Lemma~\ref{IrredCycMod-L}), we must have $U(\varphi_0)\cap W^{\bK^*}=U(\varphi_0)$, or equivalently $U(\varphi_0)\subset W^{\bK^*}$. Thus $\varphi_0\in W$, which implies $V(\varphi_0)=W$. The argument also shows $U(\varphi_0)=V(\varphi_0)^{\bK^*}$. Suppose $\varphi_{W}=0$, then $\varphi_0=\varphi_{W'}\in W'$, which implies $V(\varphi_0)=W'$ and $W=\{0\}$.  

Let us show that there exists a well-defined $\sG(\Q_p)$-homomorphism from $V(\omega_{\lambda,\rho})$ onto $V(\varphi_0)$ sending $\omega_{\lambda,\rho}$ to $\varphi_0$. Since $V(\omega_{\lambda,\rho})$ is a cyclic smooth $\sG(\Q_p)$-module generated by $\omega_{\lambda,\rho}$ it suffices to confirm that for any $\phi \in C_{\rm c}^{\infty}(\sG(\Q_p))$, $\omega_{\lambda,\rho}*\check \phi=0$ implies $\pi(\phi)\,\varphi_0=0$. By the equality $\omega=\omega_{\lambda,\rho}$ shown above, 
\begin{align*}
\sum_{\e\in E}\|\pi(\phi)\pi(\tilde \e)\varphi_0\|^2&=
\sum_{\e\in E}\langle \pi(\phi^**\phi)\pi(\tilde \e)\varphi_0|\pi(\tilde \e)\varphi_0\rangle 
\\
&=\omega*(\phi^**\phi)^{\vee}(1)
\\
&=\omega*(\phi^**\phi)^{\vee}(1)
=(\omega_{\lambda,\rho}*\check \varphi)*\check {\phi^{*}}(1)=0.
\end{align*}
This yields $\pi(\phi)\pi(\tilde \e)\varphi_0=0$ for all $\e\in E$; in particular we have $\pi(\phi)\varphi_0=0$ as desired. 

From Lemma~\ref{DoublemodSPHft}, as a $\sG(\Q_p)$-module $V(\omega_{\lambda,\rho})$ is $\pi_{\lambda,\rho}$-isotypic. Hence $V(\varphi_0)\cong \pi_{\lambda,\rho}$. \qed

\begin{cor} \label{IrredCycMod-cor}
 Let $(\pi,V_\pi)$ be an irreducible smooth unitarizable representation of $\sG(\Q_p)$ such that $V_{\pi}^{\bK^*} \not=\{0\}$. Then
\begin{itemize}
\item[(i)] As a $\C[E]$-module, $V_{\pi}^{\bK^*}$ is irreducible. 
\item[(i)] There exists a $\C$-algebra homomorphism $C_{\pi}:\cH^{+} \rightarrow \C$ such that $\pi(\phi)|V_\pi^{\bK^{*}}$ is the scalar operator $C_{\pi}(\phi)$ for any $\phi \in \cH^{+}$. 
\item[(ii)] Let $(\lambda,\rho) \in \XX$ be the Satake parameter of $C_{\pi}$. Then the representation of $E$ on $V_\pi^{\bK^*}$ belongs to the class $\rho$ and $\pi\cong \pi_{\lambda, \rho}$. 
\end{itemize}
\end{cor}
\begin{proof} Since $\pi$ is irreducible smooth and $V^{\bK^*}\not=\{0\}$, the $\cH$-module $V_{\pi}^{\bK^*}$ is known to be irreducible and finite dimensional over $\C$ (\cite[]{}). Hence by Schur's lemma, the center $\cH^{+}$ of $\cH$ acts on $V_\pi^{\bK^*}$ by a character $C_\pi:\cH^{+}\rightarrow \C$. Fix a non zero vector $\varphi_0$ in $V_\pi^{\bK^*}$; then $V(\varphi_0)=V_\pi$ by the irreducibility of $\pi$. Let $(\lambda,\rho)$ be the Satake parameter of $C_\pi$. Then by Proposition~\ref{IrredCycMod}, we have $\pi=V(\varphi_0)\cong \pi_{\lambda,\rho}$. In particular $V_{\pi}^{\bK^{*}}\cong (\pi_{\lambda,\rho})^{\bK^*}$ is irreducible as a $\C[E]$-module belonging to the class $\rho$. 
\end{proof}

\subsection{Restriction to the special orthogonal group} 
Let $\sigma\in \sG(\Q_p)$ with $\det \sigma=-1$. Then $\sG(\Q_p)=\sG^{0}(\Q_p) \rtimes\{1,\sigma\}$, where $\sG^{0}={\rm SO}(Q)$ is the special orthogonal group. Let $\pi$ be an irreducible smooth representation of $\sG(\Q_p)$. We define $\pi\otimes \det$ to be the representation $(\pi\otimes \det)(g):=\det(g)\,\pi(g)$ $(g\in \sG(\Q_p))$ on the space $V_\pi$. Let $\pi|\sG^{0}(\Q_p)$ be the restriction of $\pi$ to $\sG^{0}(\Q_p)$; it is a smooth representation of $\sG^{0}(\Q_p)$. Since $\sG^{0}(\Q_p)$ is a normal open subgroup of index $2$, from \cite[Lemma 2.1]{GelbartKnapp}, we have two possibilities : 
\begin{itemize}
\item[(i)] Suppose $\pi \otimes \det\cong \pi$; then the restriction $\pi|\sG^{0}(\Q_p)$ remain irreducible. 
\item[(ii)] Suppose $\pi \otimes \det \not \cong \pi$; then the restriction $\pi|\sG^{0}(\Q_p)$ is decomposed to a direct sum of two irreducible smooth representations $\pi^{+}$ and $\pi^{-}$ of $\sG^{0}(\Q_p)$ such that $(\pi^+)^{\sigma}\cong \pi^{-}$ and $\pi^{+} \not\cong \pi^{-}$. Here $(\pi^+)^{\sigma}$ is the $\sigma$-twist of $\pi^{+}$, i.e., $(\pi^{+})^{\sigma}(h)=\pi^{+}(\sigma^{-1} h\sigma)$ for $h\in \sG^{0}(\Q_p)$.
\end{itemize}

\begin{lem} \label{ResSO-L}
Let $(\lambda,\rho)\in \XX$. Suppose $\bK_0^{*}$ contains an element of $\sG(\Q_p)-\sG^{0}(\Q_p)$
Then $\pi_{\lambda,\rho}|\sG^0(\Q_p)$ is irreducible. 
\end{lem}
\begin{proof} The space $(\cI_{\lambda,\rho}\otimes \det)^{\bK^*}$ consists of functions $f:\sG(\Q_p)\rightarrow V_\rho$ such that 
$$
f(hnk)=\rho(\e)\,\delta_{\sB}(h)^{1/2}\prod_{j=1}^{\nu}\lambda_j(a_j)\,\det (k), \quad h=d(a;\e)\in \sH(\Q_p),\,n\in \sN(\Q_p),\,k \in \bK^{*}.
$$ 
If $\bK_0^{*}$ contains an element $u$ with $\det u=-1$, then $f(1)=\rho([u])f(1)=f(d(1,\dots,1;u))=\det(u)f(1)=-f(1)$ implies $f(1)=0$; thus $(\cI_{\lambda,\rho}\otimes \det)^{\bK^*}=\{0\}$. Since $\pi_{\lambda,\rho}\otimes \det$ is a subquotient of $\cI_{\lambda,\rho}\otimes \det$, we also have $(\pi_{\lambda,\rho}\otimes \det)^{\bK^{*}}=\{0\}$, whereas $\pi_{\lambda,\rho}^{\bK^{*}} \cong V_\rho$. Hence $\pi_{\lambda,\rho}\otimes \det$ can not be isomorphic to $\pi_{\lambda,\rho}$ as $\sG(\Q_p)$-modules. From (i) recalled above, we have that $\pi_{\lambda,\rho}|\sG^0(\Q_p)$ remains irreducible. 
\end{proof}

\begin{cor} \label{ResSO-Cor}
Let $\pi=\pi_{\lambda,\rho}$ with $(\lambda,\rho)\in \XX$ be an irreducible unitary representation of $\sG(\Q_p)$ such that $\pi^{\bK^*}\not=\{0\}$, and $n(\pi)$ the number of irreducible summands of $\pi|\sG^{0}(\Q_p)$. Then $n(\pi)=1$ unless $\dim(V_0)=0$, or $(\dim(V_0),\partial_{\cL})(2,2)$ and $\rho$ is two dimensional. For these exceptional cases, we have $\pi(\pi)=2$; it $\pi_0 \subset \pi|\sG^{0}(\Q_p)$ is an irreducible summand, then $\pi|\sG^0(\Q_p)=\pi_0\oplus \pi_0^{\sigma}$ and $\pi_0\not\cong \pi_0^\sigma$ for any $\sigma\in \sG(\Q_p)-\sG^0(\Q_p)$. 
\end{cor}
\begin{proof} Suppose $\dim(V_0)>0$; then arguing case by case by the classification list \cite[(1.2)]{MS98}, we can find an element $\sigma\in \bK^{*}_0$ with $\det(\sigma)=-1$ except for the case $(\dim(V_0),\partial_{\cM})=(2,2)$, in which case we indeed have $\bK_0^*\subset (\sG_0)^{0}(\Q_p)$. Thus the conclusion in this case follows from Lemma~\ref{ResSO-L}. 

Suppose $(\dim(V_0),\partial_{\cM})=(2,2)$. From \cite[(1.2)]{MS98}, $E=\bK_{0}/\bK_0^*\cong D_{2(p+1)}$. Then exists $\sigma\in \bK_0$ such that $\det \sigma=-1$ and $\det$ induces a non-trivial character $\chi$ of $E$. If we set $\rho'=\rho\otimes \chi$, then $\cI_{\lambda,\rho}\otimes \det\cong \cI_{\lambda,\rho'}$. From this we obtain $\pi_{\lambda,\rho}\otimes \det \cong \pi_{\lambda,\rho'}$. Hence $\pi_{\lambda,\rho}\otimes \det \cong \pi_{\lambda,\rho}$ if and only if $\rho\cong \rho'$, which in turn happens if and only if $\rho$ is two dimensional. Indeed, from the proof of \cite[Proposition 1.1]{MS98}, we can take a realization $E=\{a^{j}b^{k}|\,0\leq j\leq p,\,k\in\{0,1\}\}\cong D_{2(p+1)}$ with elements $a\in \sG^{0}(\Q_p)\cap \bK_0$ and $b\in \bK_0-\sG^0(\Q_p)$ obeying the basic relations $a^{p+1}=e,\,b^{2}=2, bab=a^{-1}$. The $2$-dimensional irreducible characters of $D_{2(p+1)}$ is exhausted by $\chi_{\nu}\,(1\leq \nu \leq p, \, \nu\not=(p+1)/2)$ given as 
$$
\chi_{\nu}(a^{j})=2\cos\left(\tfrac{2\pi \nu j}{p+1}\right), \quad \chi_\nu(a^{j}b)=0.
$$
Since $\chi(a^{j})=+1$ and $\chi(a^{j}b)=-1$, we have $\chi_{\nu}\chi=\chi_{\nu}$ as desired. By (ii) recalled above, we are done. 

Suppose $\dim(V_0)=0$. Then $\bK_0^*=\bK_0=\{1\}$, $\bK^*=\bK$, and $\nu=\dim(V)/2 \geq 1$. The space $\XX$ is reduced to the Weyl group orbits of $\nu$-tuples $\lambda=\{\lambda_j\}$ of unramified characters $\lambda_j$ of $\Q_p^\times$, and the corresponding representation $\pi_{\lambda}$ is the unique subquotient of $\cI_\lambda:={\rm Ind}_{\sB(\Q_p)}^{\sG(\Q_p)}(\chi_\lambda)$, where $\chi_{\lambda}$ is the character of $\sB(\Q_p)$ such that $\chi_{\lambda}(d(a_1,\dots,a_\nu))=\prod_{j}\lambda_j(a_j)$. Since $\sB(\Q_p)\subset \sG^0(\Q_p)$, we have $\det|\sB(\Q_p)$ is trivial; thus ${\rm Ind}_{\sB(\Q_p)}^{\sG(\Q_p)}(\chi_\lambda)\otimes \det \cong {\rm Ind}_{\sB(\Q_p)}^{\sG(\Q_p)}(\chi_\lambda\,(\det|\sB(\Q_p))={\rm Ind}_{\sB(\Q_p)}^{\sG(\Q_p)}(\chi_\lambda)$. From this $\pi_{\lambda}\otimes \det\cong \pi_{\lambda}$. Hence from (ii) recalled above, $\pi_{\lambda}|\sG^{0}(\Q_p)$ has two irreducible summands permuted by $\sigma$-twist.

\end{proof}

\section{Appendix 3} 
The aim of this appendix is to provide proofs of \cite[Proposition 14]{Tsud2011-1}, which corresponds Corollary~\ref{VBLtn} and Theorem~\ref{T2}, and \cite[Lemma 15]{Tsud2011-1}, which corresponds Proposition~\ref{CONV-B}. To make this section independent of the other part of this article, we abandon all the former notation and state our results in a general setting borrowed from \cite{MS98}.

\subsection{Automorphic representations} \label{subsec:AutoRep}
Let $m\in\N$ with $m>1$, and $V=\Q^{m-1}$ the $(m-1)$-dimensional $\Q$-vector space of column vectors. Let $Q$ be a non-singular symmetric matrix which is even-integral; we view $V$ as a quadratic space over $\Q$ whose associated bilinear form is $\langle X,Y\rangle={}^tX\,Q\,Y$, where $X,\,Y\in V$. Let $\sG=O(Q)$ be the orthogonal group of $Q$, which we regard as a $\Q$-algebraic group; thus, for any $\Q$-algebra $R$, 
$$
\sG(R)=\{h\in \GL_{m-1}(R)|\,{}^t h\,Q\,h=Q\,\}.  
$$
As in \cite{MS94} and \cite{MS98}, we assume that the $\Z$-lattice $\cL=\Z^{m-1}$ is maximal integral with respect to the bilinear form $\langle\,,\,\rangle$. In particular, $\cL$ is contained in the dual lattice $\cL^*$. Let $p$ be a prime number. The $\Q_p$-bilinear extension of $\langle \,,\,\rangle$ to $V_p=V\otimes_\Q \Q_p$ is denoted by the same symbol. Let $\nu_p$ denote the Witt index of $(V_p,Q)$ and $n_{0,p}$ the dimension of maximal $\Q_p$-anisotropic subspace of $V_p$, so that $m-1=2\nu_p+n_{0,p}$. We set $\bK_p=\sG(\Q_p)\cap \GL_{m-1}(\Z_p)$. Then $\bK_p$ is a maximal compact subgroup of $\sG(\Q_p)$. Let $\bK^*_p=\bK_{Q,p}^*$ be the kernel of the natural homomorphism $\bK_p\rightarrow {\rm Aut}(\cL_p^*/\cL_p)$. 
Let $\cH_p=\cH(\sG(\Q_p)\sslash \bK_p)$ be the Hecke algebra for the pair $(\sG(\Q_p),\bK_p)$ and set 
$$
\cH_p^{+}=\cH^{+}(\sG(\Q_p)\sslash \bK_p^{*})=\{\phi \in \cH_p|\,\phi(u gu^{-1})=\phi(g)\, (u\in \bK_p)\}.
$$
By the convolution product with respect to a Haar measure on $\sG(\Q_p)$ normalized so that $\vol(\bK_p^{*})=1$, $\cH_p$ is an associative $\C$-algebra with a unit $1_{\bK_p^*}$ and that $\cH_p^{+}$ coincides with its center (Theorem~\ref{MS98thm}). Let $\cH_\fin$ (resp. $\cH^{+}_\fin$) be the restricted tensor product of $\cH_{p}\,(p\in \fin)$ (resp. $\cH_p^+$), i.e., it is a finite $\C$-linear combinations of pure tensors $\phi=\otimes_{p}\phi_p$ with $\phi_p \in \cH_p$ (resp. $\phi_p \in \cH_p^{+}$) such that $\phi_p=1_{\bK_p^*}$ for almost all $p$; such element is viewed as a function on $\sG(\A_\fin)$ as $\phi(g)=\prod_{p \in \fin}\phi_p(g_p)$ for $g=(g_p)\in \sG(\A_\fin)$. Let $E_p=\bK_p/\bK_p^{*}$ for each $p \in \fin$ and set $E_\fin:=\prod_{p \in \fin}E_p$. Note that $E_p=\{1\}$ for almost all $p$ where $\partial_{\cL_p}=0$ (Lemma~\ref{MS98L-1}). Hence $E_\fin$ is a finite group. Let $\XX_p$ denote the set of Satake parameters of $\sG(\Q_p)$, i.e., the totally ty of Weyl group orbits of pairs $(\lambda,\rho)$ of $\nu_p$-tuples of unramified characters $\lambda=(\lambda_j)_{j=1}^{\nu_p}$ of $\Q_p^\times$ and an equivalence class $\rho$ of irreducible representations of $E_p$. 

Suppose $\sG(\R)$ is compact. This implies $\sG$ is $\Q$-anisotropic so that the quotient space $\sG(\Q)\bsl \sG(\A)$ is compact. We denote by $\cV_Q$ the space of all those complex valued functions on $\sG(\Q)\bsl \sG(\A)/\sG(\R)$ viewed aas a $\sG(\A_\fin)$-module by the right-translation $R$. Let $\cV_Q(\bK_\fin^*)$ be the $\bK_\fin^*$-fixed vectors of $\cV_Q$; we have a well-defined action of $E_\fin$ on $\cV_Q(\bK_\fin^*)$. 

\begin{prop} \label{AutoRep-P}
Let $f\in \cV_Q$ be a function which satisfies the following conditions: 
\begin{itemize}
\item[(i)] $f(gk_\fin )=f(g)$ for all $k\in \bK_{\fin}^{*}$. 
\item[(ii)] There exists a $\C$-algebra homomorphism $C_\fin=\otimes_{p}C_p$ of $\cH_\fin^+$ such that $R(\phi)f=C_\fin(\phi)\,f$ for all $\phi \in \cH_\fin^{+}$. 
\end{itemize}
Then the cyclic $\sG(\A_\fin)$-submodule $V(f)\subset \cV_Q$ generated by $f$ is irreducible and is isomorphic to the restricted tensor product $\bigotimes_{p \in \fin}\pi_{\lambda_p,\rho_p}$, where $(\lambda_p,\rho_p)\in \XX_p$ is the Satake parameter of $C_p$. The $E_\fin$-representation $V(f)^{\bK_\fin^*}$ is isomorphic to $\bigotimes_{p \in \fin}\rho_{p}$. 
\end{prop}
\begin{proof} We may suppose $\|f\|=1$. Let $\d g_\fin=\otimes_{p \in \fin}\d g_p$ be a Haar measure on $\sG(\A_\fin)$ with $\d g_p$ the Haar measure on $\sG(\Q_p)$ normalized by demanding $\vol(\bK_p^{*})=1$. Let $\d g_\infty$ be the Haar measure on $\sG(\R)$ such that $\vol(\sG(\R))=1$ and $\d g$ the product measure on $\sG(\A)=\sG(\R)\times \sG(\A_\fin)$. Then $\tilde\cV:=L^{2}(\sG(\Q)\bsl \sG(\A),\d g)$ endowed with the action of $\sG(\A)$ by the right-translation $R$ yields a unitary representation of $\sG(\A)$. Since $\sG(\Q)\bsl \sG(\A)$ is compact, we have $V(f)\subset \tilde \cV$.  Define 
$$
\omega(g)=\frac{1}{\# E_\fin}\sum_{\e\in E_\fin} \langle R(g)\,R(\e)f|R(\e)f\rangle, \quad g\in \sG(\A).  
$$
Then it is easy to confirm that $\omega$ is a right $\sG(\R)$-invariant function on $\sG(\A)$ which satisfies the properties: $\omega(k_1gk_2)=\omega(g)$ ($k_1,k_2\in \bK_\fin^*)$, $\omega(ugu^{-1})=\omega(g)$ $(u \in \bK_\fin)$, $\omega(1)=\|f\|^{2}=1$. Moreover, since $R(\e)\,\varphi$ with $\e\in E_\fin$ satisfies the same $\cH_\fin^+$-eigenequation as $f$ in (ii), we see that the function $\omega$ satisfies the equation $R(\phi)\omega=C_\fin(\phi)\,\omega$ for $\phi \in \cH^{+}_\fin$. By Theorem~\ref{MS98thm2}, we have the equality 
$$
\omega(g)=\prod_{p\in \fin}\omega_{\lambda_p,\rho_p}(g_p), \quad g=(g_p)_{p\leq \infty}\in \sG(\A).
$$
From this and \eqref{IrredCycMod-L-f1}, we have $\omega(\tilde \e)={\dim(\rho_\fin)}^{-1} \chi_{\rho_\fin}(\e)$ for all $\e \in E_\fin$, where $\tilde \e\in \bK_\fin$ is a representative of $\e\in E_\fin$ and $\rho_\fin=\otimes_{p \in \fin}\rho_p$ is the tensor product representation of $E_\fin$. Then we follow the remaing part of the argument in Lemma~\ref{IrredCycMod-L} verbatim to see that $U(f):=\C[E_\fin]f$, which is well-defined by (i), is isomorphic to $\rho_\fin$ (in particular irreducible) as $E_\fin$-modules and that $V(f)\cap \cV_Q({\bK_\fin^*})=U(f)$.  

Let $W$ be a $\sG(\A)$-invariant subspace of $V(f)$ and set $W'=V(f)\cap W^{\bot}$. Then $W'$ is also $\sG(\A)$-stable subspace of $V(f)$ which fits in the orthogonal direct sum decomposition $V(f)=W\oplus W'$. Let $f=\varphi_{W}+\varphi_{W'}$ with $\varphi_{W}\in W$ and $\varphi_{W'}\in W'$. Since $W$ and $W'$ are $\sG(\A)$-stable and $f$ is $\sG(R)\bK_\fin^*$-invariant, so is the vector $\varphi_{W}$. Since $\varphi_{W}\in V(f)$, we can find a finite number of pure tensors $\phi^{(i)}=\otimes_{p \in \fin}\phi_p^{(i)} \in C_{\rm c}^{\infty}(\sG(\A_\fin))$ and complex numbers $c_i$ such that  $\varphi_{W}=\sum_{i}c_i R(\phi^{(i)})\,f
$. Since both $f$ and $\varphi_{W}$ are $\bK_\fin^{*}$-invariant, we obtain $\varphi_{W}=\sum_{i}c_iR(\phi_0^{(i)})\,f$ with $\phi_0^{(i)}=\otimes_{p \in \fin}\phi_{0,p}^{(i)}\in \cH_\fin$ whose $p$-factor is 
$$
\phi_{0,p}^{(i)}(g_p)=\int_{\bK_p^{*}}\int_{\bK_p^*}\phi_p^{(i)}(kg_pk')\d k\,\d k'.
$$ 
For almost all $p$, we have $\phi_{0,p}^{(i)}\in \cH_p^+$. For $p$ at which $\phi_p^{(i)}\not\in \cH_p^{+}$, we apply Theorem~\ref{MS98thm} to write it as a linear combination of products $f_{\e_p}*\psi_p$ with $\psi_p \in \cH_p^{+}$ and the characteristic function $f_{\e_p}$ of $\e_p\,\bK_p^{*}$ with $\e_p\in E_p$. Thus, we obtain a finite linear expression $\varphi_W=\sum_{j}\sum_{\e\in E_\fin} c'_j\,R(f_{\e}*\psi_j)\varphi_0$ with $\psi_j \in \cH_{\fin}^+$ and $f_{\e}=\otimes_{p\in \fin}f_{\e_p}$. By the $\cH^{+}_\fin$-eigenequation of $f$, this yields the equality $\varphi_{W}=\sum_{j}\sum_{\e} c_j'C_\fin(\psi_j)\,R(\e)f$, which shows the containment $\varphi_W \in U(f)\cap W^{\bK_\fin^*}$. Suppose $\varphi_{W}\not=0$. Then $U(f)\cap W^{\bK_\fin^*}\not=\{0\}$. Since $U(f)$ is $E_\fin$-irreducible, we must have $U(f)\cap W^{\bK_\fin^*}=U(f)$, or equivalently $U(f)\subset W^{\bK_\fin^*}$. Hence $f \in W \subset V(f)$, which implies $V(f)=W$. Suppose $\varphi_{W}=0$; then $f=\varphi_{W'}\in W'$ implies $V(f)=W'$, and thus $W=\{0\}$. This completes the proof of the irreducibility of $V(f)$. 

Since $\cV^{\sG(\R)\,\cU}$ is finite dimensional for any open compact subgroup $\cU \subset \sG(\A_\fin)$, we see that the representation $V(f)$ of $\sG(\A_\fin)$ is admissible. Note that for almost all $p$, the Hecke algebra $\cH_p=\cH_p^{+}$ is commutative (Theorem~\ref{MS98thm}). Thus from \cite[Theorem 2]{Flath}, there exists a family of irreducible smooth representations $\{\pi_p\}_{p\leq \infty}$ together with a family of $\bK_p$-fixed vectors $\xi_p\in V_{\pi_p}-\{0\}$ for almost all $p \in \fin$, such that $V(f)$ is isomorphic to the restricted tensor product ${\bigotimes}_{p\leq \infty}\pi_{p}$. Hence 
$$
\biggl({\bigotimes}_{p\leq \infty}\pi_{p} \biggr)^{\bK_\fin^*}\cong V(f)^{\bK_\fin^*}=U(f).
$$
As shown above, $\cH_\fin^{+}$ acts on $V(f)^{\bK_\fin^*}=U(f)$ by the character $C_\fin=\otimes_{p \in \fin}C_p$ and $V(f)^{\bK_\fin^*}$ is an irreducible $E_\fin$-module isomorphic to $\rho_\fin$. Hence for each $p \in \fin$ the algebra $\cH_p^{+}$ acts on the invariant part $\pi_p^{\bK_p^*}$ through the character $C_p=C_{\lambda_p,\rho_p}$ and $\pi_{p}^{\bK_p^*}$ is an irreducible $E_p$-module isomorphic to $\rho_p$. Note that $\pi_p$ is unitarizable. By Corollary~\ref{IrredCycMod-cor}, we obtain $\pi_p\cong \pi_{\lambda_p,\rho_p}$. Then the last statement is evident from the defining property of $\pi_{\lambda,\rho_p}$. \end{proof}

\subsection{$L$-functions of definite orthogonal groups} \label{App2-Lftn}
We review the main result of \cite{MS98} introducing notation. Let $p$ be a prime number. For a character $C_p:\cH^+(\sG(\Q_p)\sslash \bK_p^*)\rightarrow \C$ with Satake parameter $({\boldsymbol\lambda}=(\lambda_j)_{j=1}^{\nu_p},\rho)\in \XX_p$, its local standard $L$-factor is defined by 
$$
L_p(C_p;s)=\prod_{j=1}^{\nu_p}\{(1-\lambda_j(p)\,p^{-s})(1-\lambda_j(p)^{-1}\,p^{-s})\}^{-1}\,A_{\rho,p}(s), 
$$
where $A_{\rho,p}(s)$ is the modification factor given by \cite[Formula (1.18)]{MS98}. We only note that if $\cL_p=\cL_p^*$ then $E_p$ is trivial and the factor $A_{\rho,p}(s)=1$ if $n_{0,p}\in \{0,1\}$ and $A_{\rho,p}(s)=(1-p^{-2s})^{-1}$ if $n_{0,p}=2$. Given a $\sG(\A_\fin)$-submodule $\cU$ of $\cV_Q$, let $\cU(\bK_\fin^*)$ denote the $\bK_\fin^*$-fixed vectors in $\cU$. Suppose $\cU(\bK_\fin^*)\not=\{0\}$. Then from Proposition~\ref{AutoRep-P}, for any prime $p$ there exists a character $C_p^{\cU}$ of $\cH^+(\sG(\Q_p)\sslash \bK_p^*)$ satisfying $f*\check \phi =C^{\cU}_p(\phi)\,f$ for all $f\in \cU(\bK_\fin^*)$ and $\phi\in \cH^+(\sG(\Q_p)\sslash \bK_p^*)$. Then, the standard $L$-function of $\cU$ is defined to be the Euler product of $L_p(C_p^{\cU};s)$ over all primes $p$ 
\begin{align}
L_{\fin}(\cU,s)=\prod_{p\in \fin}L_p(C_p^\cU;s),
 \label{STLftn}
\end{align}
which is shown to be convergent on the half plane $\Re(s)>(m-1)/2$. Indeed, since $f \in \cU(\bK_\fin^*)$ is square-integrable on $\sG(\Q)\bsl \sG(\A)$, the zonal spherical function $\omega_{\lambda_p,\rho_p}$ (\cite[1.3]{MS98}) should be of positive type; from this fact, we obtain the trivial estimation $p^{-(m-3)/2}\leq |\lambda_{p,j}(p)|\leq p^{(m-3)/2}$ for the Satake parameters $(\{\lambda_{p,j}\},\chi_p)$ of $C_p^{\cU}$ at almost all primes $p$, which in turn yields the absolute convergence region $\Re(s)>(m-1)/2$ of the Euler product \eqref{STLftn}. The completed $L$-function of $\cU$ is defined by 
$$
 L(\cU,s)=L_\fin(\cU,s)\,\Gamma_Q(s),
$$
with 
$$
\Gamma_Q(s)=\prod_{j=1}^{[(m-1)/2]}\Gamma_\C(s-j+(m-1)/2)\,
\begin{cases} 
(\det Q)^{s/2} \quad & \text{$m$ : odd} ,\\
(2^{-1}\,\det Q)^{s/2} \quad & \text{$m$ : even}.
\end{cases}
$$
When a non-zero function $f$ in $\cU(\bK_{\fin}^*)$ is given, we occasionally write $L(f,s)$ for $L(\cU,s)$.

\begin{prop} $($ \cite[Theorem 4.1]{MS98} $)$ \label{MuraseSugano1}
Let $\cU$ be an irreducible $\sG(\A_\fin)$-submodule of $\cV_Q$ such that $\cU(\bK_{\fin}^*)\not=\{0\}$. Then, the function $L(\cU,s)$ is continued to a meromorphic function on $\C$ satisfying the functional equation $L(\cU,s)=L(\cU,1-s)$. When $m=2$, $L(\cU,s)$ is entire. When $m\geq 3$, $L(\cU,s)$ is holomorphic except at possible simple poles at $s=(m-1)/2-j\,(0\leq j\leq m-2,\,j\in \Z)$, It has a pole at $s=(m-1)/2$ if and only if $\cU$ contains a non-zero constant function. 
\end{prop}
From Proposition~\ref{MuraseSugano1}, the function $D_{m-1}(s)\,L(\cU,s)$ is entire on $\C$, where
\begin{align}
D_{m-1}(s)=\prod_{j=0}^{m-2}\left(s-\tfrac{m-1}{2}+j\right).
\label{Poly}
\end{align}
Recall that an entire function $\phi(s)$ on $\C$ is said to be of finite order if there exists a constant $a\geq 0$ such that $|\phi(s)|\ll \exp(|s|^a)\, s\in \C.$ The infinimum of the numbers $a$ such that this estimate holds is called the order of $\phi(s)$. In this section, We provide a proof of the following.

\begin{thm} \label{T1}
Let $\cU$ be as above. The entire function $D_{m-1}(s)\,L(\cU,s)$ is of order $1$. 
\end{thm}

Here are several consequences of the theorem. For any $\epsilon\geq 0$ and for any interval $I\subset \R$, set $\cT_{\epsilon,I}=\{s\in \C|\,\Re(s)\in I,\,|\Im(s)|\geq \epsilon\,\}.$

\begin{prop} \label{CONV-B}
Let $\cU$ be as in Theorem~\ref{T1}. For any $\epsilon>0$, the estimation
\begin{align*}
|L_\fin(\cU,s)|\ll |\Im(s)|^{\kappa(\Re(s))}, \qquad s \in \cT_{\epsilon,\R}
\end{align*}
holds with the exponent 
$$ \kappa(\sigma)=\max\{1,\,(m-1-2\sigma+2\epsilon)\,\left[\tfrac{m-1}{2}\right],\,
(m-2+2\epsilon)\,\left[\tfrac{m-1}{2}\right]\}.
$$
\end{prop}
\begin{proof} A standard argument by means of the Phragm\'{e}n-Lindel\"{o}f convexity principle proves this. We need Theorem~\ref{T1}.  
\end{proof}
A meromorphic function $\phi(s)$ on $\C$ is said to be bounded in vertical strips of finite width if for any compact interval $I$, there exists $\epsilon>0$ such that $\phi(s)$ is holomorphic and bounded on $\cT_{\epsilon,I}$.

\begin{cor}\label{VBLtn}
Let $f$ be as in Theorem~\ref{T1}.
Then, for any $n\in \N$ and for any polynomial $P(s)$, the function 
$$P(s)\,\tfrac{\d^{n}}{\d s^{n}} \{D_{m-1}(s)L(f,s)\}$$
 is bounded in vertical strips of finite width.  
\end{cor}
\begin{proof}
We prove this by induction on $n$. First consider the case $n=0$. By Stirling's formula, there exists a constant $N_1$ such that 
$$
|L_\infty(\cU,s)|\ll (1+|\Im(s)|)^{N_1}\, \exp\left(-\tfrac{\pi}{2}\left[\tfrac{m-1}{2}\right]\,|\Im(s)|\right), \qquad s\in \cT_{\epsilon,I}.
$$
Thus, $L_\infty(\cU,s)=O(\exp(-a_1\,|\Im (s)|))$ on $\cT_{\epsilon,I}$ for some $a_1>0$. From Proposition~\ref{CONV-B}, $L_\fin(\cU,s)$ has a polynomial bound on $\cT_{\epsilon,I}$. Hence, $P(s)\,L(\cU,s)$ is $O(\exp(-a_0\,|\Im(s)|))$ on $\cT_{\epsilon,I}$ for any $a_0\in (0,a_1)$, a fortiori bounded on $\cT_{\epsilon,I}$. 

For $\sigma_0>0$, set $S(\sigma_0)=\cT_{0,[-\sigma_0,\sigma]}$. Take a sufficiently large $R>0$ and consider the rectangle $S(\sigma_0,R)=S(\sigma_0)\cap \{|\Im(s)|\leq R\}$. Let $P(s)$ be a polynomial. From Cauchy's integral formula applied to the entire function $\tilde L(s)=P(s)\,D_{m-1}(s)\,L(\cU,s)$, the value $\tilde L^{(n)}(s)$ is expressed as a contour integral of ${\tilde L(z)}/{(z-s)^{n+1}}$ along the boundary $\partial S(\sigma_0,R)$ endowed with the counterclockwise orientation. The integral along the horizontal edges of $\partial S(\sigma_0,R)$ vanish in the limit $R\rightarrow +\infty$ due to the estimation $\tilde L(z)=O(\exp(-a\,|\Im(z)|))$ $(a>0)$ established above. Hence,
\begin{align*} 
\tilde {L}^{(n)}(s)=\frac{n\,!}{2\pi i} \int_{\partial S(\sigma_0)} \frac{\tilde L(z)}{(z-s)^{n+1}}\,\d z, \qquad s\in S(\sigma_0). 
\end{align*}
From this, $\tilde L^{(n)}(s)$ is bounded on $S(\sigma_1)$ for any $\sigma_1<\sigma_0$. By Leibniz's formula, $\tilde L^{(n)}(s)$ is a sum of $P(s)\,\{D_{m-1}(s)\,L(\cU,s)\}^{(n)}$ and a $\C$-linear combination of $P^{(j)}(s)\,\{D_{m-1}(s)\,L(\cU,s)\}^{(n-j)}$ with $j<n$. By induction assumption, we conclude $P(s)\{D_{m-1}(s)\,L(\cU,s)\}^{(n)}$ is bounded on $S(\sigma_1)$. 
\end{proof}

\subsection{Eisenstein series}\label{App2-Eis}
We continue to keep the setup in \S~\ref{App2-Lftn}. Set
\begin{align*}
V_1&=\left[\begin{smallmatrix} \Q \\ V \\ \Q \end{smallmatrix} \right],  &
\cL_1&=\left[\begin{smallmatrix} \Z \\ \cL \\ \Z \end{smallmatrix} \right], 
& Q_1&=\left[\begin{smallmatrix} {} & {} & {1} \\ {} & {Q} & {} \\ {1} & {} & {} \end{smallmatrix} \right] 
\end{align*}
Then, $Q_1$ is a non-singular and even integral symmetric matrix with signature $(m+,1-)$. By the obvious inclusion $V \hookrightarrow V_1$, we view $V$ as a $\Q$-subspace of $V_1$. The $\Q$-bilinear form on $V$ is extended to the ambient space $V_1$ by setting $\langle X, Y \rangle={}^t X\,Q_1\,Y$ for $X,\,Y \in V_1$. The vectors $
\e_0=\left[\begin{smallmatrix} 1 \\ 0_{m-1} \\ 0 \end{smallmatrix}\right]$ and $\e_0'=\left[\begin{smallmatrix} 0  \\ 0_{m-1} \\ 1 \end{smallmatrix}\right]$ in $V_1$ are isotropic vectors spanning a hyperbolic plane. We obviously have $V_1=V\oplus <\e_0,\e_0'>_\Q$. Let $\sG_1$ be the orthogonal group of $Q_1$. The stabilizer $\sP$ of the isotropic line $\Q\e_0$ is a maximal $\Q$-parabolic subgroup of $\sG_1$. For any $\Q$-algebra $R$, the set $\sP(R)$ consists of all the matrices of the form $\sm(r;h)\,\sn(X)$ with 
\begin{align*}
\sm(r;h)&={\rm diag}( r, {h}, {r^{-1}}) , \quad \sn(X)=\left[\begin{smallmatrix} 1 & {-{}^tX Q} & {-2^{-1}Q[X]} \\ {} & {1_{m-1}} & {X} \\ {} & {} & {1} \end{smallmatrix} \right]
\end{align*}
for $r \in R^\times$, $h\in \sG(R)$ and $X\in V(R)$. 
We have the Levi decomposition $\sP=\sM\,\sN$, where $\sM$ is the Levi subgroup such that $\sM(R)=\{\sm(r;h)|\,r\in R^\times, \, h\in \sG(R)\,\}$ and $\sN$ the unipotent radical of $\sP$. 

The $\Z$-lattice $\cL_1$ is maximal integral with respect to the bilinear form $\langle\,,\,\rangle$. For any prime $p$, the construction in \S\ref{App2-Lftn} gives us an open compact subgroup $\bK_{1,p}^*$ of the maximal compact subgroup $\bK_{1,p}=\sG_1(\Q_p)\cap \GL_{m+1}(\Z_p)$ in $\sG_1(\Q_p)$. We have the Iwasawa decomposition $\sG_1(\Q_p)=\sP(\Q_p)\,\bK_{1,p}^{*}$. Let $\bK_{1,\fin}$ be the direct product of $\bK_{1,p}$ over all prime numbers $p$. Let $G_1\cong{\rm SO}_0(m,1)$ be the identity component of $\sG_1(\R)\cong {\rm O}(m,1)$. Let $\cD_1$ be the connected component of $\{Y\in V_{1,\R}|\,2^{-1}Q_1[Y]=-1\,\}$ containing the point $Y_0^{-}=\e_0-\e_0'$. The domain $\cD_1$ is an orbit of the point $Y_0^{-}$ by the natural action of $G_1$ on $V_{1,\R}$ and is isomorphic to an $m$-dimensional real hyperbolic space. Let $\bK_{1,\infty}$ denote the stabilizer of $Y_0^{-}$ in $G_1$; then $G_1/\bK_{1,\infty}\cong \cD_1$ and $\bK_{1,\infty}\cong  SO(m)$. For latter purpose, we need to consider the conjugate subgroup $\bK_{1,\infty}^{[r_0]}=\sm(r_0;\,1)\,\bK_{1,\infty}\,\sm(r_0;\,1)^{-1}$ for various $r_0>0$. We have the Iwasawa decomposition $\sG_1(\R)=\sP(\R)\,\bK_{1,\infty}^{[r_0]}$. Let $f\in \cV_Q$ and $s\in \C$. By the Iwasawa decomposition $\sG_1(\A)=\sP(\A)\,\bK_{1,\fin}^*\,\bK_{1,\infty}^{[r_0]}$, there exists a unique function $\sf^{(s)}$ on $\sG_1(\A)$ such that $\sf^{(s)}(g)=|t|_\A^{(s+(m-1)/2}\,f(h)$ for any $g\in \sm(t;\,h)\,\sN(\A)\,\bK_{1,\fin}\bK_{1,\infty}^{[r_0]}$ with $t\in \A^\times$ and $h\in \sG(\A)$. Then, the Eisenstein series $E(f,s;\,g)$ is defined by the series
\begin{align}
E(f,s;\,g)=\sum_{\gamma \in \sP(\Q)\bsl \sG_1(\Q)} \sf^{(s)}(\gamma g),\qquad g\in \sG_1(\A),
 \label{Eis}
\end{align}
which is absolutely convergent on $\Re(s)>(m-1)/2$. The holomorphic function $E(f,s;\,g)$ in $s$ has a meromorphic continuation to the whole complex plane (\cite[Theorem IV 1.8]{MW}). At a regular point $s\in \C$, the function $g\mapsto E(f,s;\,g)$ is a  $\bK_{1,\fin}\,\bK_{1,\infty}^{[r_0]}$-invariant automorphic form on $\sG_1(\A)$ (\cite[Proposition IV 1.9]{MW}). If $f\in \cV_Q$ is a simultaneous Hecke eigenform which generates an irreducible submodule $\cU\subset \cV_Q$, we set 
\begin{align}
L^*(\cU,s)=r_0^{-s/2}\,L(\cU,s)\,\hat\zeta(2s)^{1-\epsilon}
\label{normalizingfactor}
\end{align}
with $\epsilon\in \{0,1\}$ the parity of $m$, and introduce the normalized Eisenstein series by 
\begin{align*}
E^*(f,s;\,g)=L^*(\cU,-s)\,E(f,s;\,g), \qquad g\in \sG_1(\A), 
\end{align*}
where $\hat \zeta(s)=\Gamma_\R(s)\,\zeta(s)$ is the completed Riemann zeta function. 
\begin{prop} $($ \cite[Proposition 4.3]{MS98} $)$ \label{MuraseSugano2}
The normalized Eisenstein series $E^*(f,s;\,g)$ has an analytic continuation to the whole complex plane as a meromorphic function in $s$ and satisfies the functional equation $E^{*}(f,-s;\,g)=E^{*}(f,s;\,g)$. It is holomorphic away from possible poles at $
 s_j=\tfrac{m-1}{2}-j,\quad (j\in \Z,\,0\leq j\leq m-1)$. The right most point $s=(m-1)/2$ is a pole if and only if $f$ is a constant function. 
\end{prop}
Set $\sP(\A)^{1}=\sM(\A)^{1}\,\sN(\A)$ with $\sM(\A)^{1}=\{\sm(t;\,h)|\,r\in \A^{1},\,h\in \sG(\A)\,\}$. For a positive $t>0$, set $[t]=\sm(t;\,1)$ and let $T$ be the split torus in $G_1$ consisting of all such points. For any positive number $c$, we set $T^{+}_c=\{\,[t]\,|\,t\in \R,\,t\geq c\,\}.$ Recall that any subset $\fS\subset \sG_1(\A)$ of the form $\fS=\omega\,T^+_c \,\bK_{1,\fin}\bK_{1,\infty}^{[r_0]} $ with a relatively compact subset $\omega\subset \sP(\A)^1$ and a positive real number $c$ is called a Siegel set in $\sG_1(\A)$. For $g\in \sG_1(\A)$, let $\sm(t(g);1)$ denote the $T$-component of $g$ along the decomposition $\sG_1(\A)=\sP(\A)^1\,T\,\bK_1^{[r_0]}$. Note that the modulus function of $\sP(\A)$ restricted to $T$ is given by $\sm(t;\,1) \mapsto |t^{m-1}|$. By the reduction theory, there exists a Siegel set $\fS$ such that $\sG_1(\A)=\sG_1(\Q)\,\fS$. From Proposition~\ref{MuraseSugano2}, the function $D_{m}(s+1/2)\,E^*(f,s;\,g)$ is entire on $\C$ for any $g\in \C$, where $D_{m}(s)$ is the polynomial \eqref{Poly}. Here is our main theorem of this section; it provides a bound of the normalized Eisenstein series on $\fS$ by a majorant independent of the imaginary part of the spectral parameter $s$.  

\begin{thm} \label{T2}
Let $f\in \cV_Q$ be a simultaneous eigenform of the Hecke algebra $\cH^+(\sG(\A_\fin)\sslash \bK_\fin^*)$. Then, 
\begin{itemize}
\item[(1)] For any $\epsilon>0$ and for any compact set $\ccU\subset \sG_1(\A)$, the following estimation holds:
$$
|D_{m}(s+1/2)\,E^*(f,s;\,g)|\ll \exp(|s|^{1+\epsilon}),\qquad s\in \C,\,g\in \ccU.$$
\item [(2)] For any compact interval $I\subset \R$, there exists $\delta>0$ such that for any $\epsilon>0$, the following estimation holds:
$$
|D_m(s+1/2)\,E^{*}(f,s;\,g)]|\ll_{\epsilon} t(g)^{|\Re(s)|+(m-1)/2+\epsilon},
\quad s\in \cT_{\delta,I}, \,g\in \fS.
$$
\end{itemize}
\end{thm}

\subsection{Proofs}
In this section, we prove Theorems~\ref{T1} and \ref{T2} simultaneously by induction on $m$, the matrix size of $Q$. We fix a symmetric matrix $Q$ of size $m$ satisfying the conditions \S\ref{App2-Lftn} and suppose the statement of Theorem~\ref{T1} is true for any Hecke eigen form $f\in \cV_Q$. We show that the statement (1) of Theorem~\ref{T2} is true for any $f\in \cV_Q$. Let $r$ be an integer such that $r>[(m-1)/2]$. Recall the polynomial \eqref{Poly}; we set $D(s)=D_m(s+1/2)$ for simplicity. 
 
\begin{lem} \label{L2norm-MaassSelberg}
Suppose $f\in \cV_Q$ is a simultaneous Hecke eigenform generating an irreducible sub module $\cU\subset \cV_Q$ such that $D_{m-1}(s)\,L(\cU,s)$ is of order $1$. Let $T>1$. Then, for any $a>1$, for any $\epsilon>0$ and for any compact interval $I\subset [0,+\infty)$, and for any $Z\in U(\fg_{1,\infty})$, 
\begin{align*}
\|D(\nu)\,\wedge^{\T}R(Z)E^*(f,\nu;-)\|_{\sG_1}\ll \exp(|\nu|^{a}), \quad \nu \in \cT_{\epsilon,I}.
\end{align*}
\end{lem}
\begin{proof} We consider the case $Z=1$ and $I\subset (0,+\infty)$. Set $M(\nu)=L^*(\cU,\nu)/L^*(\cU,-\nu)$. 
 The Maass-Selberg relation asserts that for any $\nu=\sigma+i\,t$ with $\sigma,\,t\in \R$ such that $\sigma\not=0$, the equality
\begin{align*}
\int_{\cS} |\wedge^\T E(f,\nu;\,y)|^2\,\d y
&=
\frac{\T^{2\sigma}-\T^{-2\sigma}}{2\sigma}
+\frac{\T^{-2\sigma}}{2\sigma}\,\left\{1-|M(\nu)|^2\right\}
+\frac{\overline {M(\nu)}\,T^{2i t}-M(\nu)\,\T^{-2i t}}{it}
\end{align*}
holds for any $\T>0$. From this $\|\nu\,D(\nu)\,\wedge^{\T}E^*(f,\nu;-)\|^2_{\sG_1}$ is the sum of the following three terms. 
\begin{align*}
I_1(\nu)&=|F(-\nu)|^2\,\frac{T^{2\sigma}-\T^{-2\sigma}}{2\sigma}, \\
I_2(\nu)&=\frac{\T^{-2\sigma}}{2\sigma}\{|F(-\nu)|^2-|F(\nu)|^2\}, \\
I_3(\nu)&=\frac{(-1)^{m}}{it}\{F(-\nu)\,\bar F(\nu)\,T^{2it}-\bar F(-\nu)\,F(\nu)\,T^{2it}\},
\end{align*}
where $F(\nu)=\nu\,D(\nu)\,L^*(\cU,\nu)$. From Proposition~\ref{MuraseSugano1} and from our assumption, $F(\nu)$ is an entire function of order $1$. Thus, for any $a>1$, the following estimates hold. 
\begin{align*}
|F(\nu)|\ll \exp (|\nu|^a), \quad |F'(\nu)|\ll \exp(|\nu|^a), \quad \nu \in \C.
\end{align*}
From this, it is obvious that the function $I_3(\nu)$ is majorized by $\exp(|\nu|^a)$ on $\cT_{\epsilon,I}$ for any $a>1$. Since 
\begin{align*}
|I_1(\nu)|&=|F(-\nu)|^2 \,\left|\frac{\log \T}{2\sigma}\int_{-2\sigma}^{2\sigma}T^{x}\,\d x\right|^2 
\leq |F(-\nu)|^2\,(2\log \T)\,T^{|2\sigma|}, 
\end{align*}
$I_1(\nu)$ is also majorized by $\exp(|\nu|^a)$ on $\cT_{\epsilon,I}$. Since $\overline{F(\nu)}=F(\bar \nu)$, we have 
\begin{align*}
I_2(\nu)&=T^{-2\sigma}\frac{F(-\bar \nu)\,F(-\nu)-F(\nu)\,F(\bar \nu)}{\nu+\bar \nu}
\\
&=-T^{-2\sigma} \left(F(-\nu)\,\frac{F(-\bar\nu)-F(\nu)}{(-\bar\nu)-\nu}+F(\nu)\,\frac{F(-\nu)-F(\bar \nu)}{-\nu-\bar\nu}\right)
\\
&=-T^{-2\sigma}\, \left(
F(-\nu) \frac{1}{(-\bar \nu)-\nu} \int_{[\nu,-\bar\nu]}F'(w)\,\d w
+F(\nu)\int_{[\bar\nu, -\nu]}F'(w)\,\d w\right),
\end{align*}
where $[a,b]$ denotes the line segment connecting $a$ to $b$. Thus, 
\begin{align*}
|I_2(\nu)|&\leq T^{-2\sigma}\{|F(-\nu)| \sup_{w\in [\nu, -\bar \nu]}|F'(w)|
+|F(\nu)|\,\sup_{w\in [\bar \nu,-\nu]}|F'(w)|\} \ll \exp(|\nu|^a), \quad \nu \in \cT_{\epsilon,I}. 
\end{align*}

\end{proof}

\subsubsection{A consequence of Lemma~\ref{L2norm-MaassSelberg}}\label{Cor-L2norm-MaassSelberg}
For any $a>1$ and for any Hecke eigenform $f\in \cV_Q$ such that $L$-function $D_{m-1}(s)L(\cU,s)$ is of order $1$, and for any compact set $\cN\subset \sG_1(\A)$, we show the bound 
\begin{align}
|D(\nu)\,E^*(f,\nu;\,g)|\ll \exp(|\nu|^{a}), \quad g\in \cN,\, \nu \in \C.
 \label{Proof-Thm2-0}
\end{align}
Let $\cV(f)$ be the minimal closed $\sG(\A)$-invariant subspace of $L^2(\sG(\Q)\bsl \sG(\A))$ containing $f$. Let $\cI(f,\nu)$ be the space of all the smooth functions $\varphi:\sG_1(\A)\rightarrow \C$ with the equivalence $\varphi(\sm(t;1)n g)=|t|_\A^{\nu+\rho}\,\varphi(g)$ for any $t\in \A^\times$ and $n\in \sN(\A)$ such that the function $h\mapsto \varphi(\sm(1;h)k)$ belongs to the space $\cV(f)$ for all $k\in \bK$. By the right-translation, the group $\sG(\A)$ acts continuously on the Frechet space $\cI(f,\nu)$. Then by the automatic continuity theorem of Casselman and Wallach, for any $\nu \in \cT_{1,[0,2\rho]}$ the normalized Eisenstein series gives a continuous $\sG(\A)$-intertwining map $\cE^{*}$ form $\cI(f,\nu)$ to the space of smooth functions on $\sG(\Q)\bsl \sG(\A)$ of moderate growth. For any $Z\in U(\fg_{1,\infty})$, we have $\cE^*(g,R(Z)\sf^{(\nu)})=R(Z)E^{*}(f,\nu,g)$. By the same reasoning as in \cite[\S 5.3]{GelbartLapid}, invoking Lemma~\ref{L2norm-MaassSelberg}, we show the following statement: For any right $\bK_{1,\fin}\bK^{[r_0]}_{1,\infty}$-invariant compact set $\ccU\subset \sG_1(\A)$ and for any $Z\in U(\fg_{1,\infty})$ there exists a constants $C>0$ such that 
\begin{align*}
|D(\nu)\,\cE^*(g, R(\phi)R(Z)\sf^{(\nu)})|\leq C\,\max_{g\in \sG_1(\A)}|\phi(g)|\,\exp(|\nu|^a)
\end{align*}
for all $\phi \in C_{\rm c}(\sG_1(\A))$ supported in $\cU$, $g\in \cN$ and $\nu \in \cT_{1,[0,2\rho]}$. As in \cite[p.637]{GelbartLapid}, we can find $\phi_1,\phi_2\in C_{\rm c}(\sG_1(\A))$ and $Z\in U(\fg_{1,\infty})$ such that $$E^*(f,\nu;g)=\cE^*(g,R(\phi_1)\sf^{(\nu)})+\cE^*(g,R(\phi_2)R(Z)\sf^{(\nu)}),$$
which combined with the estimate above yields the desired bound \eqref{Proof-Thm2-0}. 

The estimation \eqref{Proof-Thm2-0} is extended to the strip $\Re(\nu)\in[0,2\rho]$, because this strip is a union of $\cT_{1,[0,2\rho]}$ and a compact set. On the region $\Re(\nu)\geq 3\rho/2$, where the series $E(f,\nu;\,g)$ converges absolutely, we estimate
\begin{align}
|D(\nu)\,E^{*}(f,\nu;\,g)|
\ll |D(\nu)\,L^{*}(\cU,-\nu)|\, \sum_{\gamma \in \sP(\Q)\bsl \sG_1(\Q)}\|\gamma^{-1}\e_1\|^{-(\Re(\nu)+\rho)}, \quad g\in \cN,\,\Re(\nu)\geq 3\rho/2. 
 \label{Proof-Thm2-1}
\end{align}
Here, $\|\,\|$ is the height function on the space of primitive adeles in $V_{1,\A}=\A^{m+1}$. Since $\cI=\{\gamma \in \sP(\Q)\bsl \sG_1(\Q)|\,\|\gamma^{-1}\e_0\|\leq 1\,\}$ is a finite set, we have
$$
\sum_{\gamma \in \sP(\Q)\bsl \sG_1(\Q)}\|\gamma^{-1}\e_1\|^{-(\Re(\nu)+\rho)}.
\ll O(b^{-(\Re(\nu)+\rho)})+\sum_{\gamma \in \sP(\Q)\bsl \sG_1(\Q)}\|\gamma^{-1}\e_1\|^{-5\rho/2}
$$
with $b=\inf\{\|\gamma^{-1}\e_0\|\,|\,\gamma \in \cI\,\}$. Thus, the infinite series on the right hand side of \eqref{Proof-Thm2-1} is certainly bounded by $\exp(|\nu|^a)$ on the region $\Re(\nu)\geq 3\rho/2$ for any $a>1$. From  our assumption, the factor $D(\nu)\,L^*(\cU,-\nu)$ is bounded by $\exp(|\nu|^a)$. Therefore, the estimation \eqref{Proof-Thm2-0} is extended to the half plane $\Re(\nu)\geq 0$. Invoking the functional equation $D(\nu)\,E^*(f,\nu;\,g)=(-1)^{m}\,D(-\nu)\,E^*(f,-\nu;\,g)$, the estimation is eventually extended to the whole complex plane. \qed

\subsubsection{The proof of Theorem~\ref{T1} and Theorem~\ref{T2}}
We prove Theorem~\ref{T1} by induction on $m$. If $m=2$, then the statement follow from the fact that $(2^{-1}Q)^{-s/2}\,L(\cU,s)$ is a finite product of terms of the form $(1\pm p^{-(s-1/2)})$ (see \cite[p.92]{MS98}). 

Let us suppose the statement of Theorem~\ref{T1} is true for any positive definite symmetric matrix of size $m-1$ satisfying the conditions in \S~\ref{App2-Lftn}. Let $Q$ be a positive definite symmetric matrix of size $m$ satisfying the conditions in \S~\ref{App2-Lftn}. We prove the assertion of Theorem~\ref{T1} for a simultaneous Hecke eigenform $F\in \cV_Q$ on the orthogonal group ${\bf O}(Q,\A)$. Fix $g_0\in {\bf O}(S,\A)$ such that $F(g_0)\not=0$. From \cite[\S 3]{MS98}, there exists the following objects. 
\begin{itemize}
\item a $\Z$-basis $\xi_j\,(1\leq j\leq m)$ of $\cL=\Z^{m}$ such that the upper left $(m-1)\times (m-1)$ block of the matrix $T=({}^t\xi_i Q \xi_j)_{1\leq i,j \leq m}$ is even integral and maximal. 
\item $\tilde F \in \cV_{T}$ such that $\tilde F(1)=F(g_0)\not=0$ and $L(\tilde F,s)=F(F,s)$.
\end{itemize}
Let 
$$
T=\left[\begin{smallmatrix} S & -S\alpha \\ {}^t\alpha S & {-2a}\end{smallmatrix}\right], \quad 
S_1=\left[\begin{smallmatrix} {} & {} & {1} \\ {} & {S} & {} \\ {1} & {} &{}
\end{smallmatrix}\right], \qquad 
\xi=\left[\begin{smallmatrix} a \\ \alpha \\ 1 \end{smallmatrix}\right]
$$
and $\sG={\bf O}(S)$, $\sH={\bf O}(T)$, and $\sG_1={\bf O}(S_1)$. As in \cite[\S 2]{MS98}, we define embeddings $\iota_0:\sG\rightarrow \sH$ and $\iota:\sH\rightarrow \sG_1$ so that $\iota(\sH)$ coincides with the stabilizer of $\xi$ in $\sG_1$. Since $F_1(1)\not=0$, there exists a simultaneous Hecke eigen form $f\in \cV_{S}$ such that $\langle\, \tilde F\circ \iota_0\,|\,f\,\rangle_{\sG}\not=0$. Since the matrix size of $S$ is $m-1$, the entire function $D_{m-1}(s)\,L(f,s)$ is of order $1$ by induction assumption. From the result of \S~\ref{Cor-L2norm-MaassSelberg}, for any $a>1$, the estimation 
\begin{align}
D_{m-1}(s)\,E^*(f,s;\iota(h)) \ll \exp(|s|^a), \quad s\in \C, \, h \in \sH(\Q)\bsl \sH(\A) 
 \label{IND-ASSM}
\end{align}
holds. From \cite[Theorem 4.4, Theorem 2.11]{MS98}, there exists a constant $C>0$ such that 
\begin{align*}
\int_{\sH(\Q)\bsl \sH(\A)} \tilde F(h)\,E^*(f,s-1/2;\,\iota(h))\,\d h_1=C\,\langle \,\tilde F\circ \iota_0 \,|\,f\,\rangle_{\sG} \, L(\tilde F,s). 
\end{align*}
Therefore, applying \eqref{IND-ASSM}, we obtain  
\begin{align*}
|D_m(s)\,L(\tilde F,s)|&\ll \int_{\sH(\Q)\bsl \sH(\A)}|\tilde F(h)|\,|D_{m}(s)\,E^{*}(f,s-1/2;\,\iota(h))|\,\d h
\\
&\ll \exp(|s|^{a}), \quad s\in \C 
\end{align*}
for any $a>1$. Since $L(F,s)=L(\tilde F,s)$, we are done. This completes the proof of Theorem~\ref{T1}. Then, from \S~\ref{Cor-L2norm-MaassSelberg}, the proof of Theorem~\ref{T2} (1) is also completed. 

On the convergence region $\Re \nu>\frac{m-1}{2}$, we have the inequality $|E(f,\nu;g)|\leq (\max|f|)\,E({\bf 1},\Re \nu;g)$ for all $g\in \sG_1(\R)$, where ${\bf 1}$ is the constant function $1$ on $\sG(\A)$. By the estimate $|E({\bf 1},\Re \nu;g)|\ll t(g)^{\Re \nu+(m-1)/2}$ on the Siegel set $\fS$, we have $|E(f,\nu;g)|\ll t(g)^{\Re \nu+(m-1)/2}$ ($g\in \fS)$ as well. Hence by Corollary~\ref{VBLtn}, for any compact interval $I$ contained in $(\frac{m-1}{2},+\infty)$ there exists $C$ such that 
\begin{align}
|D_{m}(\nu+1/2)E^*(f,\nu;g)|\leq C\, t(g)^{|\Re \nu|+(m-1)/2}, \quad g\in \fS,\,\nu \in \cT_{0,I}.
\label{ProofT2(2)-1}
\end{align}
Since the function $D_{m}(s+1/2)E^{*}(f,\nu;g)$ is invariant by the variable change $\nu\mapsto -\nu$, the same inequality holds true for $g\in \fS$ and $\nu \in \cT_{\delta,-I}$. We fix $g\in \fS$. Then by the bound in Theorem~\ref{T2} (1), we can apply the Phragm\'{e}n-Lindel\"{o}f convexity principle to the entire function $\nu\mapsto D_m(\nu+1/2)E^{*}(f,\nu;g)$, we have that the same inequality \eqref{ProofT2(2)-1} is extended to the smallest vertical strip containing $\cT_{0,I}\cup \cT_{0,-I}$. \qed

\end{document}